\documentclass[12pt]{amsart}

\newlength{\minipagewidth}
\usepackage{amsthm,amsmath,amssymb,stmaryrd,verbatim,dsfont,color}
\usepackage[all]{xy}
\usepackage{wrapfig}
\usepackage[english]{babel}
\usepackage{xcolor}
\usepackage{changes}
\usepackage{tikz} 
\usetikzlibrary{matrix,arrows,decorations.markings}
\pgfrealjobname{dynkingrassmann1}  
\tikzset{->-/.style={decoration={markings,mark=at position #1 with {\color{black}\arrow{>}}},postaction={decorate,very thick}},
         -<-/.style={decoration={markings,mark=at position #1 with {\color{black}\arrow{<}}},postaction={decorate,very thick}},
         -x-/.style={decoration={markings,mark=at position #1 with {\node[auto,black,font=\large] {$\times$};}},postaction={decorate,very thick}},
         const/.style={rectangular,draw=black,dotted,thick,text height=0.8ex, text depth=-0.2ex,font=\scriptsize},
         vertex/.style={rectangular,rounded corners=7pt,draw=black,text height=1.0ex, text depth=0.0ex,font=\footnotesize},
         triple/.style={circle,draw=black,text height=-0.5ex, text depth=-1.5ex,font=\footnotesize},
         pair/.style={circle,inner sep=2pt,fill=black,text height=-0.3ex, text depth=-1.3ex,font=\footnotesize},
         patchbox/.style={rectangular,draw=black,minimum width=6.0ex,minimum height=3.0ex,text height=1.0ex, text depth=0.0ex,font=\footnotesize},
         invisiblebox/.style={rectangular,minimum width=5.0ex,minimum height=5.0ex}               }
\newcommand{\vertex}[2]{\node[vertex](#1){#2};}
\newcommand{\pair}[1]{\node[pair](#1){};}
\newcommand{\triple}[2]{\node[triple](#1){#2};}
\newcommand{\const}[2]{\node[const](#1){#2};}
\newcommand{\gbox}[2]{\node[patchbox](#1){\Xi_{#2}};}
\newcommand{\nobox}[1]{\node[invisiblebox](#1){};}

\usepackage{xparse}
\usetikzlibrary{backgrounds}
\pgfdeclarelayer{patchwork}
\pgfsetlayers{patchwork,background,main}
\tikzset{patchstyle/.style = {rectangular,rounded corners=7pt,draw,fill=#1}}
\NewDocumentCommand{\patch}{O{gray!20} m m}{\draw[patchstyle=#1] (#2.north west) rectangle (#3.south east);}

\usepackage{mathptmx}                                                              
\DeclareSymbolFont{cmletters}{OML}{cmm}{m}{it}                                     
\DeclareSymbolFont{cmsymbols}{OMS}{cmsy}{m}{n}
\DeclareSymbolFont{cmlargesymbols}{OMX}{cmex}{m}{n}
\DeclareMathSymbol{\myjmath}{\mathord}{cmletters}{"7C}     \let\jmath\myjmath 
\DeclareMathSymbol{\myamalg}{\mathbin}{cmsymbols}{"71}     \let\amalg\myamalg
\DeclareMathSymbol{\mycoprod}{\mathop}{cmlargesymbols}{"60}\let\coprod\mycoprod
\DeclareMathSymbol{\myalpha}{\mathord}{cmletters}{"0B}     \let\alpha\myalpha 
\DeclareMathSymbol{\mybeta}{\mathord}{cmletters}{"0C}      \let\beta\mybeta
\DeclareMathSymbol{\mygamma}{\mathord}{cmletters}{"0D}     \let\gamma\mygamma
\DeclareMathSymbol{\mydelta}{\mathord}{cmletters}{"0E}     \let\delta\mydelta
\DeclareMathSymbol{\myepsilon}{\mathord}{cmletters}{"0F}   \let\epsilon\myepsilon
\DeclareMathSymbol{\myzeta}{\mathord}{cmletters}{"10}      \let\zeta\myzeta
\DeclareMathSymbol{\myeta}{\mathord}{cmletters}{"11}       \let\eta\myeta
\DeclareMathSymbol{\mytheta}{\mathord}{cmletters}{"12}     \let\theta\mytheta
\DeclareMathSymbol{\myiota}{\mathord}{cmletters}{"13}      \let\iota\myiota
\DeclareMathSymbol{\mykappa}{\mathord}{cmletters}{"14}     \let\kappa\mykappa
\DeclareMathSymbol{\mylambda}{\mathord}{cmletters}{"15}    \let\lambda\mylambda
\DeclareMathSymbol{\mymu}{\mathord}{cmletters}{"16}        \let\mu\mymu
\DeclareMathSymbol{\mynu}{\mathord}{cmletters}{"17}        \let\nu\mynu
\DeclareMathSymbol{\myxi}{\mathord}{cmletters}{"18}        \let\xi\myxi
\DeclareMathSymbol{\mypi}{\mathord}{cmletters}{"19}        \let\pi\mypi
\DeclareMathSymbol{\myrho}{\mathord}{cmletters}{"1A}       \let\rho\myrho
\DeclareMathSymbol{\mysigma}{\mathord}{cmletters}{"1B}     \let\sigma\mysigma
\DeclareMathSymbol{\mytau}{\mathord}{cmletters}{"1C}       \let\tau\mytau
\DeclareMathSymbol{\myupsilon}{\mathord}{cmletters}{"1D}   \let\upsilon\myupsilon
\DeclareMathSymbol{\myphi}{\mathord}{cmletters}{"1E}       \let\phi\myphi
\DeclareMathSymbol{\mychi}{\mathord}{cmletters}{"1F}       \let\chi\mychi
\DeclareMathSymbol{\mypsi}{\mathord}{cmletters}{"20}       \let\psi\mypsi
\DeclareMathSymbol{\myomega}{\mathord}{cmletters}{"21}     \let\omega\myomega
\DeclareMathSymbol{\myvarepsilon}{\mathord}{cmletters}{"22}\let\varepsilon\myvarepsilon
\DeclareMathSymbol{\myvartheta}{\mathord}{cmletters}{"23}  \let\vartheta\myvartheta
\DeclareMathSymbol{\myvarpi}{\mathord}{cmletters}{"24}     \let\varpi\myvarpi
\DeclareMathSymbol{\myvarrho}{\mathord}{cmletters}{"25}    \let\varrho\myvarrho
\DeclareMathSymbol{\myvarsigma}{\mathord}{cmletters}{"26}  \let\varsigma\myvarsigma
\DeclareMathSymbol{\myvarphi}{\mathord}{cmletters}{"27}    \let\varphi\myvarphi

\newcommand{\Ext}{\mathrm{Ext}} 
\newcommand{\ext}{\mathrm{ext}} 
 
\newcommand{\Rep}{\mathrm{Rep}}

\theoremstyle{plain}
\newtheorem{thm}{Theorem}[section]
\newtheorem{cor}[thm]{Corollary}
\newtheorem{lemma}[thm]{Lemma}
\newtheorem{prop}[thm]{Proposition}

\newtheorem*{prop*}{Proposition}
\newtheorem*{thmA}{Theorem A}

\theoremstyle{definition}
\newtheorem{df}[thm]{Definition}
\newtheorem{ex}[thm]{Example}
\newtheorem{rem}[thm]{Remark}
\newtheorem*{rem*}{Remark}

\def\SS{\S}
\def\A{{\mathbb A}}

\def\C{{\mathbb C}}

\def\P{{\mathbb P}}
\def\Q{{\mathbb Q}}

\def\Z{{\mathbb Z}}
\def\cA{{\mathcal A}}
\def\cB{{\mathcal B}}

\def\ue{{\underline{e}}}
\def\udim{{\underline{\dim}\, }}

\def\ev{{\textup{ev}}}
\def\Ev{{\textup{Ev}}}
\def\min{{\textup{min}}}
\def\max{{\textup{max}}}
\def\op{{\textup{op}}}

\newcommand{\ocirc}{\begin{picture}(6,6) \put(0,0){\color{gray}{$\bullet$}}
                                          \put(0,0){$\circ$}                   \end{picture} }

\DeclareMathOperator{\Gr}{Gr}
\DeclareMathOperator{\Mat}{Mat}
\DeclareMathOperator{\Rel}{Rel}
\DeclareMathOperator{\Vertex}{Vert}
\DeclareMathOperator{\Edge}{Edge}
\DeclareMathOperator{\Link}{Link}
\DeclareMathOperator{\coker}{coker}
\DeclareMathOperator{\id}{id}

\begin{document}

\title[Quiver Grassmannians of type $\widetilde{D}_n$, Part 1]{Quiver Grassmannians of type $\widetilde{D}_n$\\[10pt] Part 1: Schubert systems and decompositions into affine spaces}

\author{Oliver Lorscheid}
\address{Instituto Nacional de Matem\'atica Pura e Aplicada, Estrada Dona Castorina 110, Rio de Janeiro, Brazil}
\email{oliver@impa.br}

\author{Thorsten Weist}
\address{Bergische Universit\"at Wuppertal, Gau\ss str.\ 20, 42097 Wuppertal, Germany}
\email{weist@uni-wuppertal.de}

\begin{abstract}
 Let $Q$ be a quiver of extended Dynkin type $\widetilde D_n$. In this first of two papers, we show that the quiver Grassmannian $\Gr_\ue(M)$ has a decomposition into affine spaces for every dimension vector $\ue$ and every indecomposable representation $M$ of defect $-1$ and defect $0$, with exception of the non-Schurian representations in homogeneous tubes. We characterize the affine spaces in terms of the combinatorics of a fixed coefficient quiver for $M$. The method of proof is to exhibit explicit equations for the Schubert cells of $\Gr_\ue(M)$ and to solve this system of equations successively in linear terms. This leads to an intricate combinatorial problem, for whose solution we develop the theory of Schubert systems.
 
 In the sequel \cite{LW15}, we extend the result of this paper to all indecomposable representations $M$ of $Q$ and determine explicit formulae for their $F$-polynomials.
\end{abstract}

\maketitle

\begin{small}  
 \tableofcontents 
\end{small}

\normalem


\section*{Introduction}

\noindent
Quiver Grassmannians are varieties that parametrize subrepresentations of quiver representations. They first appeared in literature nearly 25 years ago (cf.\ \cite{sch}), but the large attention they received in recent years is due to their relevance for cluster algebras as introduced by Fomin and Zelevinsky in the series of papers \cite{Fomin-Zelevinsky02}, \cite{Fomin-Zelevinsky03}, \cite{Fomin-Zelevinsky05} and \cite{Fomin-Zelevinsky07}. To explain the relevance of the results of the present paper and its sequel \cite{LW15}, we review parts of the developments in the theory of cluster algebras before we describe the result of this paper.

\subsection*{Cluster algebras}
There are many well-written introductions to cluster algebras, e.g.\ see \cite{Fomin-Zelevinsky02}, \cite{Schiffler09} or \cite{Keller10}. To simplify matters, we restrict ourselves in this exposition to cluster algebras $\cA$ that have trivial coefficients and come from a quiver $Q$ with $n$ vertices. The cluster algebra $\cA$ of $Q$ is defined as a subring of the field $\Q(x_1,\dotsc,x_n)$ of rational functions that is generated by the so-called \emph{cluster variables}. The cluster variable are obtained from $x_1,\dotsc,x_n$ by an iterative procedure of so-called \emph{mutations}, which are prescribed but the quiver $Q$ and its mutations.

Though it turns out, a posteriori, that some well-understood cluster algebras bear simple descriptions in terms of generators and relations (\cite{Geiss-Leclerc-Schroer13}), the recursive definition in terms of mutations makes it, in general, a difficult problem to find such descriptions. Much work from the past ten years has been devoted to find bases and relations for approachable classes of cluster algebras.

Similar to the sharp distinction between the representation theory of tame and wild quivers, cluster algebras are divided into mutation finite and mutation infinite algebras. While there is little known for mutation infinite cluster algebras, many results have been established for mutation finite cluster algebras. To wit, all cluster algebras coming from tame quivers are mutation finite, and we divide them into types $A$, $D$ and $E$, according to the type of the tame quiver. The only wild quivers leading to mutation finite cluster algebras are the generalized Kronecker quivers $K(m)$ for $m>2$ and five other quivers with oriented cycles (\cite{Buan-Reiten06}, \cite{Fomin-Shapiro-Thurston08}, \cite{Felikson-Shapiro-Tumarkin12}). 

There have been two major approaches to study mutation finite cluster algebras. One of them applies to mutation finite cluster algebras of types $A$ and $D$ and makes use of triangulations of marked surfaces (\cite{Musiker-Schiffler-Williams11}, \cite{Musiker-Schiffler-Williams13}). For type $A$, mutations are understood by geometric means, which yields bases for these cluster algebras (\cite{Schiffler-Thomas09},\cite{Dupont-Thomas13}). There are some partial results for type $D$ (\cite{Qiu-Zhou13}, \cite{Gunawan-Musiker14}) and a general way to compute Euler characteristics in terms of triangulations (\cite{Musiker-Schiffler-Williams11}), but this has not been carried out yet for cluster algebras that come from quivers of extended Dynkin type $D$.

An alternative approach is due to cluster characters (\cite{Fomin-Zelevinsky07}, \cite{Palu08}, \cite{Dupont11}, \cite{Dupont12a}, \cite{Dupont12b}). A cluster character associates with each representation $M$ of $Q$ an element $X_M$ of the cluster algebra $\cA$. This defines a map from the indecomposable representations of $Q$ whose image is the set of cluster variables (\cite{Caldero-Keller06}).

The cluster variable $X_M$ is given by an explicit expression in the Euler characteristics of the quiver Grassmannians associated with $M$ (\cite{cc}, \cite{dwz} and \cite{dwz2}). These Euler characteristics have been calculated for the indecomposable representations of quivers of (extended) Dynkin type $A$ in \cite{Cerulli11} and \cite{Haupt12} and for the Kronecker quiver $K(2)$ in \cite{cr} and \cite{Cerulli-Esposito11}. There have been partial results for representations of low rank for other types in \cite{Dupont10} and \cite{Cerulli-Dupont-Esposito13}.

\subsection*{Aim of this text} In this paper and its sequel \cite{LW15}, we will work out explicit formulae for the Euler characteristics of quiver Grassmannians of extended Dynkin type $D$, which yields a description of the cluster character for mutation finite cluster algebras of type $D$. Our technique is based on Schubert decompositions as introduced in \cite{L13} and \cite{L14}. In the following, we will explain this approach and our results in more detail.


\subsection*{Schubert decompositions of quiver Grassmannians}

Let $Q$ be a quiver, $M$ a complex representation and $\ue=(e_p)$ a dimension vector for $Q$. The quiver Grassmannian $\Gr_\ue(M)$ is defined as the set of all subrepresentations $N$ of $M$ with dimension vector $\udim N=\ue$. An \emph{ordered basis for $M$} is the union $\cB=\bigcup_{p\in Q_0} \cB_p$ of bases $\cB_p$ of $M_p$ whose elements are linearly ordered. The choice of an ordered basis $\cB$ of $M$ identifies $\Gr_\ue(M)$ with a closed subvariety of the usual Grassmannian $\Gr(e,d)$ by considering $N$ as a subvector space of $M\simeq \C^d$, where $d=\dim M$ and $e=\sum e_p$. This endows $\Gr_\ue(M)$ with the structure of a complex variety, and this structure does not depend on the choice of the ordered basis $\cB$.

The Schubert decomposition $\Gr(e,d)=\coprod C_\beta(d)$ of the usual Grassmannian, where $\beta$ ranges through all subsets of $\cB$ of cardinality $e$, restricts to the Schubert decomposition 
\[
 \Gr_\ue(M) \ = \ \coprod \ C_\beta^M
\]
of the quiver Grassmannian where $C_\beta^M=C_\beta(d)\cap\Gr_\ue(M)$. Note that the Schubert cells $C_\beta^M$ are affine varieties, but that they are, in general, not affine spaces. In particular, a Schubert cell $C_\beta^M$ might be empty. See section \ref{subsection: schubert decompositions in detail} for more details.

We say that $\Gr_\ue(M) = \coprod \ C_\beta^M$ is a \emph{decomposition into affine spaces} if every Schubert cell $C_\beta^M$ is either an affine space or empty. In this case, the classes of the closures of the non-empty Schubert cells in $\Gr_\ue(M)$ form an additive basis for the singular cohomology ring of $\Gr_\ue(M)$, and the cohomology is concentrated in even degree, cf.\ \cite[Cor.\ 6.2]{L14}. In particular, we derive the following characterization of the Euler characteristic of $\Gr_\ue(M)$.

\begin{prop*}
 If $\Gr_\ue(M)=\coprod C_\beta^M$ is a decomposition into affine spaces, then the Euler characteristic of $\Gr_\ue(M)$ is equal to the number of non-empty cells $C_\beta^M$. 
\end{prop*}


\subsection*{Representations of extended Dynkin type $D$}

Let $Q$ be a quiver of extended Dynkin type $\widetilde D_n$, i.e.\ of the form
\[
   \beginpgfgraphicnamed{fig58}
   \begin{tikzpicture}[>=latex]
  \matrix (m) [matrix of math nodes, row sep=-0.2em, column sep=2.5em, text height=1ex, text depth=0ex]
   {      & q_b &     &     &      &         &         &  q_c & \\
          &     & q_0 & q_1 & {\dotsb } & q_{n-5} & q_{n-4} &      & \\
      q_a &     &     &     &      &         &         &      &  q_d \\};
   \path[-,font=\scriptsize]
   (m-3-1) edge node[auto,swap] {$a$} (m-2-3)
   (m-1-2) edge node[auto] {$b$} (m-2-3)
   (m-2-3) edge node[auto] {$v_0$} (m-2-4)
   (m-2-6) edge node[auto] {$v_{n-5}$} (m-2-7)
   (m-2-7) edge node[auto] {$c$} (m-1-8)
   (m-2-7) edge node[auto,swap] {$d$} (m-3-9)
   (m-2-4) edge node[auto] {} (m-2-5)
   (m-2-5) edge node[auto] {} (m-2-6);
  \end{tikzpicture}
\endpgfgraphicnamed
\]
where the arrows are allowed to assume any orientation. The representation theory of $Q$ is well-understood; in particular, there are descriptions of all indecomposable representations of $Q$ in \cite{Crawley-Boevey92} or \cite{Dlab-Ringel76}. We summarize this theory in Appendix \ref{appendix: indecomposable representations of type D_n-tilde}.

With this knowledge about the representation theory of $Q$, it is easy to describe the matrix coefficients of the linear maps $M_v$ w.r.t.\ to a chosen basis $\cB$ for a representation $M$ of $Q$ where $v$ is an arrow of $Q$. For the purpose of this paper, we exhibit for every indecomposable representation $M$ a particular ordered basis $\cB$ and present the matrix coefficients in terms of the coefficient quiver $\Gamma$ of $M$ w.r.t.\ $\cB$. We explain the construction of these bases in Appendix \ref{appendix: bases for representations of type D_n-tilde} and restrict ourselves in section \ref{subsection: bases for indecomposables of small defect} to a short description of the cases that we treat in the present Part 1 of the paper.


\subsection*{The main result}

Let $M$ be an indecomposable representation of $Q$ and $\cB$ the ordered basis as constructed in Appendix \ref{appendix: bases for representations of type D_n-tilde}. Then the following holds true.

\begin{thmA}\label{thmA}
 The Schubert decomposition $\Gr_\ue(M)=\coprod C_\beta^M$ is a decomposition into affine spaces for every dimension vector $\ue$.
\end{thmA}

For indecomposable representations of defect $-1$ and $0$, with the exception of non-Schurian representations in a homogeneous tube, this is Theorem \ref{thm: the main theorem}. Moreover, Theorem \ref{thm: the main theorem} contains an explicit description of the empty Schubert cells in terms of the combinatorics of the coefficient quiver $\Gamma$ of $M$ w.r.t.\ $\cB$.


\subsection*{Schubert systems}

The technique of the proof of Theorem A is to exhibit a particular system of equations for each non-empty Schubert cell $C_\beta^M$ and solve this system of equations successively in linear terms.

More precisely, the results of \cite{L13} and \cite{L14} provide us with certain polynomials $E(v,t,s)$ in $\C[w_{i,j}|i,j\in\cB]$. If we denote the vanishing set of these polynomials by $V(M,\cB)$, then the various Schubert cells $C_\beta^M$ result as the intersection of $V(M,\cB)$ with an appropriate affine subspace of $\{(w_{i,j})\}_{i,j\in\cB}$ that is defined by equations of the form $w_{i,j}=0$ and $w_{i,j}=1$.

The difficulty of the proof lies in the intricate combinatorics of the system of equations $E(v,t,s)$. We organize this combinatorial data in terms of the Schubert system $\Sigma$ of $M$ w.r.t.\ $\cB$, which is a graph whose vertices are the indices $(v,t,s)$ of the polynomials $E(v,t,s)$ and the indices $(i,j)$ of the variables $w_{i,j}$ and whose edges indicate that $w_{i,j}$ appears non-trivially in $E(v,t,s)$. The Schubert system comes with additional data that remembers the terms of an equation $E(v,t,s)$ and their coefficients.

For a given $\beta$, we can evaluate a part of the variables $w_{i,j}$ according to the equations $w_{i,j}=0$ and $w_{i,j}=1$, which leads us to the $\beta$-state $\Sigma_\beta$ of the Schubert system. Certain combinatorial conditions on $\Sigma_\beta$ imply that the Schubert cell is empty, and others that the Schubert cell is an affine space. The proof of Theorem A consists in a verification of these conditions for all subsets $\beta$ of $\cB$.

\begin{rem*}
 One might wonder if there is no shorter proof of Theorem A that uses arguments from representation theory instead of the lengthy combinatorial proof in terms of Schubert systems. We were not able to find such a proof for the cases considered in this paper. However, the reader will find representation theoretic arguments in the sequel \cite{LW15}, which use the results from this paper to complete the proof of Theorem A. Though in principle, all the methods in the proof of Theorem A extend to other types of quiver Grassmannians, shorter and more systematic arguments would be of advantage for a generalization of the present results.
\end{rem*}


\subsection*{Side results}

As a first application of the theory of Schubert systems, we (re-)establish decompositions into affine spaces of the quiver Grassmannians $\Gr_\ue(M)$ for all exceptional representations of the Kronecker quiver and for all indecomposable representations of quivers of Dynkin types $A$ and $D$.


\subsection*{Results from Part 2}

In the sequel \cite{LW15} to this paper, we complete the proof of Theorem A for all indecomposable representations $M$ of $Q$, together with a characterization of the non-empty Schubert cells. This reduces the calculation of the Euler characteristic $\chi_\ue(M)$ of $\Gr_\ue(M)$ to a counting problem of non-empty Schubert cells. The cluster variables for $Q$ are given by the relation
\[
 X_M \quad = \quad \prod_{q\in Q_0} x_q^{\sum_{p\in Q_0}a(p,q)\dim M_p-\dim M_q} \ \cdot \ F_M \ \Bigl(\bigl(\prod_{p\in Q_0}x_p^{a(q,p)-a(p,q)}\bigr)_{q\in Q_0}\Bigr) 
\]
where $a(p,q)$ is the number of arrow from $p$ to $q$ and $F_M=\sum_\ue \chi_\ue(M)\underline x^\ue$ is the $F$-polynomial of $M$. The explicit formulae for the $F$-polynomial for indecomposable representations are described below. Via $F_{M\oplus N}=F_MF_N$, we obtain the cluster variables for all representations of $Q$.


\medskip
\noindent\textit{Homogeneous tubes.} Let $\delta$ be the unique imaginary Schur root. The $F$-polynomial $F_\delta$ of a Schur representation in a homogeneous tube is easily determined from Theorem \ref{thm: the main theorem}. It depends on the orientation of $Q$, but not on the homogeneous tube. The $F$-polynomial $F_{r\delta}$ of an indecomposable representation with dimension vector $r\delta$ in a homogeneous tube is given by
\[
 F_{r\delta} \quad = \quad \frac{1}{2z}(\lambda_+^{r+1}-\lambda_-^{r+1}) \qquad \text{where} \qquad z=\frac{1}{2}\sqrt{F_{\delta}^2-4\underline x^{\delta}} \quad \text{and} \quad\lambda_{\pm}=\frac{F_{\delta}}{2}\pm z.
\]

\medskip
\noindent\textit{Exceptional tubes.} Every indecomposable representation in an exceptional tube of rank $m$ is contained in a uniquely determined sequence
\[
 M_{0,0} \quad \longrightarrow \quad M_{0,1} \quad \longrightarrow \quad \dotsb \quad \longrightarrow \quad M_{0,m-1} \quad \longrightarrow \quad M_{1,0} \quad \longrightarrow \quad M_{1,1} \dotsb  
\]
of irreducible injective homomorphisms where $M_{0,0}=0$ is the trivial representation and all other representations $M_{i,j}$ in this sequence are indecomposables of the same tube. The dimension vector $\alpha(r,i)$ of $M_{r,i}$ is a real root for $i=1,\dotsc,m-1$ and $\alpha(r,0)=r\delta$ is an imaginary root. We denote the $F$-polynomial of $M_{r,i}$ by $F_{r,i}$. If $\alpha$ is a real root, then we denote the $F$-polynomial of the unique indecomposable representation $M$ with $\udim M=\alpha$ by $F_\alpha$. We define $F_\alpha=0$ if a component of $\alpha$ has a negative coefficient and set $\alpha(r,m)=(r+1)\delta$.

The polynomials $F_{0,i}$ are easily determined from Theorem \ref{thm: the main theorem}. The $F$-polynomials for $r\geq 1$ are given by
\[
 F_{r,i} \quad = \quad F_{0,i} F_{r\delta} \ + \ \underline x^{\alpha(0,i+1)} F_{\alpha(0,m-1)-\alpha(0,i+1)} F_{(r-1)\delta}.
\]

\medskip
\noindent\textit{Preprojective component.} Let $M$ be an indecomposable preprojective representation of defect $-1$ and $r\geq 0$ such that $0\leq \epsilon_M=\udim M-r\delta\leq\delta$. The $F$-polynomial $F_{\epsilon_M}$ of the real root $\epsilon_M$ is easily determined from Theorem \ref{thm: the main theorem}. Let $\tau$ be the Auslander-Reiten-translation. If $\delta-\epsilon_M$ is injective, then we have
\[
 F_M \quad = \quad F_{\epsilon_M}F_{r\delta} \ - \ \underline x^{\delta}F_{(r-1)\delta}.
\]
If $\delta-\epsilon_M$ is not injective, then we have
\[
 F_M \quad = \quad F_{\epsilon_M}F_{r\delta} \ - \ \underline x^{\tau^{-1}\epsilon_M}F_{\delta-\tau^{-1}\epsilon_M}F_{(r-1)\delta}.
\]
Every indecomposable representation $B$ of defect $-2$ is the extension of two particular indecomposable representations $M$ and $N$ of defect $-1$ such that
\[
 F_B \quad = \quad  F_NF_M \ - \ \underline x^{\udim\tau^{-1}M}F_{N/\tau^{-1}M}.
\]

\medskip
\noindent\textit{Preinjective component.} Every preinjective representation is the dual $M^\ast$ of a preprojective representation $M$ with $F$-polynomial $F_M=\sum c_\ue \underline x^\ue$. The $F$-polynomial of $M^\ast$ is 
\[
 F_{M^\ast} \quad = \quad \sum c^\ast_{\ue} \ \underline x^\ue \qquad \text{with} \qquad c^\ast_\ue=c_{\udim M-\ue}.
\]


\subsection*{Content overview}

In section \ref{section: background}, we review some background material on coefficient quivers and Schubert cells. In section \ref{section: schubert systems}, we introduce Schubert systems and their $\beta$-states. We develop ways to compute $\beta$-states efficiently and describe combinatorial criteria under which the Schubert cell $C_\beta^M$ is empty or an affine space. In section \ref{section: first applications}, we apply the theory of Schubert systems to the Kronecker quiver and Dynkin quivers of types $A$ and $D$.

In section \ref{section: Schubert decompositions for type D_n}, we formulate the main result Theorem \ref{thm: the main theorem} of this paper. For this, we describe a basis for all indecomposable representations that are considered in Theorem \ref{thm: the main theorem}, and we characterize those $\beta$ for which the Schubert cell $C_\beta^M$ is empty in terms of the combinatorics of the coefficient quiver $\Gamma$. The proof of Theorem \ref{thm: the main theorem} is the content of section \ref{section: proof of theorem 4.4}.

In Appendix \ref{appendix: indecomposable representations of type D_n-tilde}, we summarize the representation theory of quivers of extended Dynkin type $D$. In Appendix \ref{appendix: bases for representations of type D_n-tilde}, we explain how to construct the coefficient quivers that we use in this text.

\subsection*{Acknowledgements}

We would like to thank Jan Schr\"oer for calling our attention to type $D$ cluster algebras and for his help with this project. We would like to thank Giovanni Cerulli Irelli for his explanations about cluster algebras. We would like to thank Markus Reineke and Christof Gei\ss\ for their interest and their assistance of this project.


\section{Background}
\label{section: background}

\noindent
Let $Q$ be a quiver with vertex set $Q_0$ and arrow set $Q_1$. A \emph{(complex finite dimensional) representation} of $Q$ is a collection of finite dimensional complex vector spaces $M_p$ indexed by the vertices $p$ of $Q$ together with a collection of linear maps $M_a:M_p\to M_q$ indexed by the arrows $v:p\to q$ of $Q$. A representation $M$ is \emph{thin} if $\dim M_p\leq 1$ for all $p\in Q_0$, and $M$ is \emph{sincere} if $\dim M_p\geq 1$ for all $p\in Q_0$.

\begin{rem}
 Though many results of this text are valid for arbitrary base rings $k$ instead of $\C$, we restrict ourselves to $k=\C$, which is sufficient for applications to cluster algebras. We note that the representation theory of extended Dynkin quiver is essentially the same for all algebraic closed fields $k$, but it differs for other rings $k$. However, the proof of Theorem \ref{thm: the main theorem} is valid for any base ring $k$.
\end{rem}

\subsection{Coefficient quiver}
\label{subsection: coefficient quiver}
Let $M$ be a representation of $Q$ with basis $\cB=\bigcup_{p\in Q_0}\cB_p$. Let $v:p\to q$ be an arrow of $Q$ and $i\in\cB_p$. Then we have the equations 
\[
 M_v(i) \quad = \quad \sum_{j\in\cB_q} \mu_{v,i,j} \, j 
\]
with uniquely determined coefficients $\mu_{v,i,j}\in\C$. The \emph{coefficient quiver of $M$ w.r.t.\ $\cB$} is the quiver $\Gamma=\Gamma(M,\cB)$ with vertex set $\Gamma_0=\cB$ and arrow set 
\[
 \Gamma_1 \quad = \quad \bigl\{ \ (v:p\to q,i,j)\in Q_1\times\cB\times\cB \ \bigl| \ i\in\cB_{p}, j\in\cB_{q}\text{ and }\mu_{v,i,j}\neq 0 \ \bigr\}.
\]
It comes together with a morphism $F:\Gamma\to Q$ that sends $i\in\cB_p$ to $p$ and $(v,i,j)$ to $v$, and with a thin sincere representation $N=N(M,\cB)$ of $\Gamma$ with basis $\cB$ and $1\times 1$-matrices $N_{(v,i,j)}=(\mu_{v,i,j})$. Note that $M$ is canonically isomorphic to the push-forward $F_\ast N$. Therefore the morphism $F:\Gamma\to Q$ of quivers together with the weights $\mu_{v,i,j}$ of the arrows of $\Gamma$ determine $M$ and $\cB$. This allows us to describe a basis of $M$ in terms of the coefficient quiver $\Gamma=\Gamma(M,\cB)$ and the weights $\mu_{v,i,j}$.

\subsection*{Remark on illustrations} Typically, we identify $\Gamma_0$ with a set of natural numbers. We will illustrate coefficient quivers $\Gamma=\Gamma(M,\cB)$ and the associated morphism $F:\Gamma\to Q$ in the following way: we label the arrows $(v,i,j)$ of $\Gamma$ with their image arrows $v$ in $Q$, together with the non-zero weight $\mu_{v,i,j}$ unless it is $1$. We draw the vertices $i$ and arrows $v:i\to j$ of $\Gamma$ above their respective image $F(i)$ and $v$ in $Q$. An example of a coefficient quiver of an indecomposable module is the following.
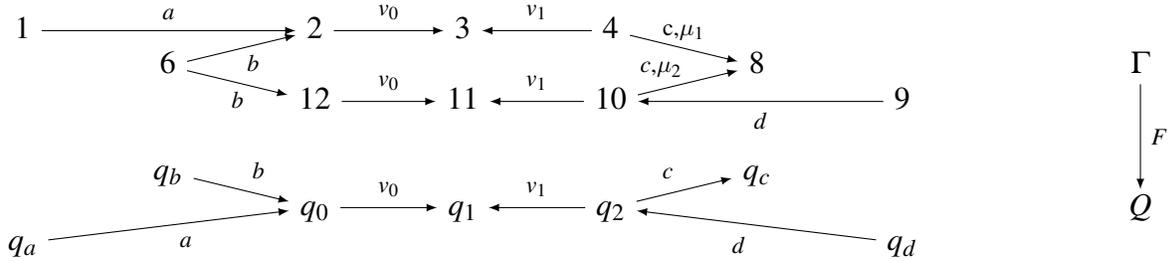
\[
   \beginpgfgraphicnamed{fig1}
   \begin{tikzpicture}[>=latex]
  \matrix (m) [matrix of math nodes, row sep=0em, column sep=3em, text height=1ex, text depth=0ex]
   {   1 &   & 2 & 3 & 4 &   &               \\   
         & 6 &   &   &   & 8 &    && \Gamma  \\   
         &   &12 &11 &10 &   & 9             \\   
      \  &   &   &   &   &   &               \\   
         &q_b&   &   &   &q_c&               \\   
         &   &q_0&q_1&q_2&   &    && Q       \\   
     q_a &   &   &   &   &   & q_d           \\}; 
   \path[->,font=\scriptsize]
   (m-2-2) edge node[below right=-2pt] {$b$} (m-1-3)
   (m-2-2) edge node[below=-1pt] {$b$} (m-3-3)
   (m-1-3) edge node[auto] {$v_0$} (m-1-4)
   (m-1-5) edge node[auto,swap] {$v_1$} (m-1-4)
   (m-3-7) edge node[auto] {$d$} (m-3-5)
   (m-1-5) edge node[above=-1pt] {c,$\mu_1$} (m-2-6)
   (m-3-5) edge node[above left=-2pt] {$c$,$\mu_2$} (m-2-6)
   (m-3-5) edge node[auto,swap] {$v_1$} (m-3-4)
   (m-3-3) edge node[auto] {$v_0$} (m-3-4)
   (m-1-1) edge node[auto] {$a$} (m-1-3)
   (m-7-1) edge node[auto,swap] {$a$} (m-6-3)
   (m-5-2) edge node[auto] {$b$} (m-6-3)
   (m-6-3) edge node[auto] {$v_0$} (m-6-4)
   (m-6-5) edge node[auto,swap] {$v_1$} (m-6-4)
   (m-6-5) edge node[auto] {$c$} (m-5-6)
   (m-7-7) edge node[auto] {$d$} (m-6-5)
   (m-2-9)  edge node[auto] {$F$} (m-6-9);
  \end{tikzpicture}
\endpgfgraphicnamed
\]
If there is no danger of confusion, we will omit the quiver $Q$ and the map $F:\Gamma\to Q$ from the illustration. Since it is often possible to renormalize the weights $\mu_{v,i,j}$ to $1$ without changing the isomorphism class of $M$, we are able to omit the weights for most arrows. To be precise, we can renormalize the weight $\mu_{v,i,j}$ to $1$ whenever the arrow $(v,i,j)$ is not contained in a cycle. This is achieved by scaling all the basis elements with a path to the tail of $(v,i,j)$ which  does not contain $(v,i,j)$ by the factor $\mu_{v,i,j}$.


\subsection{Schubert decompositions}\label{subsection: schubert decompositions in detail}

We review the definition of the Schubert cells $C_\beta^M$ and the Schubert decomposition of $\Gr_\ue(M)$ in more detail. A point of the Grassmannian $\Gr(e,d)$ is an $e$-dimensional subspace $V$ of $\C^{d}$. Let $V$ be spanned by vectors $w_1,\dotsc,w_e\in\C^d$. We write $w=(w_{i,j})_{i=1\dotsc d,j=1\dotsc, e}$ for the matrix of all coordinates of $w_1,\dotsc,w_e$. The Pl\"ucker coordinates 
\[
 \Delta_\beta(V) \quad = \quad \det(w_{i,j})_{i\in \beta,j=1\dotsc, e}
\]
(where $\beta$ is a subset of $\{1,\dotsc,d\}$ of cardinality $e$) define a point $(\Delta_\beta(V))_\beta$ in $\P\bigl(\Lambda^{e}\C^{d}\bigr)$. For two ordered subsets $\beta=\{i_1,\dotsc,i_e\}$ and $\beta'=\{j_1,\dotsc,j_e\}$ of $\{1,\dotsc,d\}$, we define $\beta\leq \beta'$ if $i_l\leq j_l$ for all $l=1\dotsc,e$. The Schubert cell $C_\beta(d)$ of $\Gr(e,d)$ is defined as the locally closed subvariety of all subspaces $V$ such that $\Delta_\beta(V)\neq 0$ and $\Delta_{\beta'}(V)=0$ for all $\beta'>\beta$.

Let $M$ be a representation of $Q$ with ordered basis $\cB$, $d$ the dimension of $M$ as complex vector space, $\ue=(e_p)$ a dimension vector of $Q$ and $e=\sum e_p$ the sum of its coordinates. The ordered basis $\cB$ identifies $M$ with $\C^d$ and a subrepresentation $N$ with $\udim N=\ue$ with an $e$-dimensional subvector space of $\C^d$. This association identifies $\Gr_\ue(M)$ with a closed subvariety of $\Gr(e,d)$. We define the Schubert cell $C_\beta^M$ as the intersection $C_\beta(d)\cap\Gr_\ue(M)$.

A subset $\beta$ of $\cB$ is \emph{of type $\ue$} if $\beta_p=\beta\cap\cB_p$ is of cardinality $e_p$ for every vertex $p$ of $Q$. Since the embedding into $\Gr(e,d)$ factors into
\[
 \Gr_\ue(M) \quad \longrightarrow \quad \prod_{p\in Q_0} \ \Gr(e_p,d_p) \quad \longrightarrow \quad \Gr(e,d)
\]
where $d_p=\dim M_p$, we see that $C_\beta(d)\cap\Gr_\ue(M)$ is empty if $\beta\subset\cB$ is a subset of cardinality $e$ that is not of type $\ue$. Therefore the cells in the Schubert decomposition
\[
 \Gr_\ue(M) \quad = \quad \coprod_{\substack{\beta\subset\cB\\\text{of type }\ue}} \ C_\beta^M
\]
can be indexed by subsets $\beta\subset\cB$ of type $\ue$.


\subsection{Representations of Schubert cells}\label{subsection: representation of schubert cells}

We describe explicit equations for the Schubert cells $C^M_\beta$, which realize them as closed subvarieties of an ambient affine space; cf.\ \cite{L13}.

For a subset $\beta$ of $\cB$, let $N$ be a point of $C_\beta^M$. If $\beta$ is of type $\ue$, then this means that $N_p$ is a $e_p$-dimensional subspace of $M_p$ for every $p\in Q_0$ such that $M_v(N_p)\subset N_q$ for every arrow $v:p\to q$ in $Q$. For every $p\in Q_0$, the vector space $N_p$ has a basis $(w_j)_{j\in\beta_p}$ where $w_j=(w_{i,j})_{i\in \cB_p}$ are column vectors in $M_p$ w.r.t.\ the coordinates given by $\cB_p$. If we define $w_{i,j}=0$ for $i,j\in\cB$ whenever $j\notin\beta$, or $i\in\cB_p$ and $j\in\cB_q$ with $p\neq q$, then we obtain a matrix $w=(w_{i,j})_{i,j\in\cB}$. We call such a matrix $w$ a \emph{matrix representation of $N$}. 

A matrix $w\in\Mat_{\cB\times\cB}$ is in \emph{$\beta$-normal form}, if it satisfies 
\begin{enumerate}
 \item[(NF1)]\label{NF1} $w_{i,i}=1$ for all $i\in\beta$,
 \item[(NF2)]\label{NF2} $w_{i,j}=0$ for all $i,j\in\beta$ with $j\neq i$,
 \item[(NF3)]\label{NF3} $w_{i,j}=0$ for all $i\in\cB$ and $j\in\beta$ with $j<i$,
 \item[(NF4)]\label{NF4} $w_{i,j}=0$ for all $i\in\cB$ and $j\in\cB-\beta$, and
 \item[(NF5)]\label{NF5} $w_{i,j}=0$ for all $i\in\cB_p$ and $j\in\beta_q$ with $p\neq q$.
\end{enumerate}

We say that $w_{i,j}$ is a \emph{constant coefficient (w.r.t.\ $\beta$)} if it appears in \textup{(NF1)}--\textup{(NF5)}; otherwise we say that $w_{i,j}$ is a \emph{free coefficient (w.r.t.\ $\beta$)}, which is the case if and only if there is a $p\in Q_0$ such that $i\in\cB_p-\beta_p$, $j\in\beta_p$ and $i<j$. 

\begin{lemma}[{\cite[Lemma 2.1]{L13}}] \label{lemma: normal form}
 Every $N\in C_\beta^M$ has a unique matrix representation $w=(w_{i,j})_{i,j\in\cB}$ in $\beta$-normal form.
\end{lemma}

It follows from the proof of this lemma that associating with $N_p$ its $\beta$-normal form $w$ defines an injective morphism $\iota_\beta: C_\beta^M \to \Mat_{\cB\times\cB}$ from the Schubert cell into an affine matrix space. Its image can be described as follows. Let $\Gamma$ be the coefficient quiver of $M$ w.r.t.\ $\cB$ with matrix coefficients $\mu_{v,i,j}$ and let $F:\Gamma\to Q$ be the associated morphism. Let $V(M,\cB)$ be the vanishing set of the polynomials
\[         \label{key formula}
 E(v,t,s) \quad = \quad \sum_{(v,s',t')\in \Gamma_1} \mu_{v,s',t'}  w_{t,t'} w_{s',s} \quad - \quad \sum_{(v,s',t)\in \Gamma_1}   \mu_{v,s',t} w_{s',s} 
\]
for all arrows $v:p\to q$ in $Q_1$ and all vertices $s\in F^{-1}(p)$ and $t\in F^{-1}(q)$. Note that the variety $V(M,\cB)$ does not depend on $\beta$. 

The following result is proven in a more general situation in \cite[section 4.1]{L14}; also cf.\ \cite[Lemma 2.2]{L13} for the case of coefficients $\mu_{v,s,t}=1$.

\begin{lemma}\label{lemma: defining equations of a schubert cell}
 Let $\beta\subset\cB$. Then the image of $\iota_\beta:C_\beta^M\to \Mat_{\cB\times\cB}$ is the intersection of $V(M,\cB)$ with the solution set of \textup{(NF1)}--\textup{(NF5)}.
\end{lemma}


\section{Schubert systems}
\label{section: schubert systems}

\noindent
In this section we introduce Schubert systems, which are the main tool for proving \mbox{Theorem A}. Let $M$ be a representation of a quiver $Q$ with ordered basis $\cB$. In a nutshell, the Schubert system $\Sigma=\Sigma(M,\cB)$ of $M$ w.r.t.\ $\cB$ is a graph with certain additional data that keeps track of the defining equations of the Schubert cells $C_\beta^M$. 

The essential information encoded by the Schubert system are the coordinates of a large affine space and the defining equations of a subvariety in this affine space whose intersection with certain affine subspaces yields the various Schubert cells $C_\beta^M$ of the Schubert decomposition $\Gr_\ue(M)=\coprod C_\beta^M$. The defining equations for these intersections can be visualized by certain subgraphs $\Sigma_\beta$ of $\Sigma$, called the $\beta$-states of the Schubert system.

Certain combinatorial conditions on $\Sigma_\beta$ imply that the Schubert cell $C_\beta^M$ is empty or, in other cases, an affine space. In the proof of Theorem \ref{thm: the main theorem}, we will verify one of these conditions for every $\beta$-state $\Sigma_\beta$ of the Schubert systems in question.

\subsection{The complete Schubert system} \label{subsection: the complete schubert system}

Let $M$ be a representation of $Q$ with ordered basis $\cB$. Let $\Gamma$ be the coefficient quiver of $M$ w.r.t.\ $\cB$ with matrix coefficients $\mu_{v,i,j}$ and let $F:\Gamma\to Q$ be the associated morphism. 

A \emph{relevant pair} is an element of the set
\[
 \Rel^2 \quad = \quad \{ \ (i,j) \in \Gamma_0\times \Gamma_0 \ | \ F(i)=F(j)\text{ and }i\leq j \ \}
\]
of indices $(i,j)$ for which there is a subset $\beta$ of $\cB$ such that $w_{i,j}$ is not constant $0$. A \emph{relevant triple} is an element of the set
\[
 \Rel^3 \quad = \quad \left\{ \ (v,t,s) \in Q_1\times \Gamma_0\times \Gamma_0 \ \left|  \begin{array}{c} \text{There is an }(v,s',t')\in\Gamma_1\\ \text{with } s\geq s'\text{ and }t\leq t' \end{array} \right. \right\}.
\]
of triples $(v,t,s)$ for which there is a subset $\beta$ of $\cB$ with $t\notin\beta$ and $s\in\beta$ such that $E(v,t,s)$ is not the zero polynomial. 

In this text, a \emph{graph} is undirected and simple, i.e.\ without multiple edges or loops. Therefore, we can identify an edge of a graph with the set $\{p,q\}$ of its end vertices. A vertex $q$ is \emph{adjacent} to a vertex $p$ if $\{p,q\}$ is an edge. A \emph{link} is a pair $(p,S)$ of a vertex $p$ and a set $S$ of vertices adjacent to $p$. We call $p$ the \emph{tip}, an edge $\{p,q\}$ with $q\in S$ a \emph{leg} and $q$ a \emph{base vertex} of the link. If $S$ is empty, then we call the link \emph{constant}, if $S$ consists of one element, then we call the link \emph{linear}, and if $S$ consists of two elements, then we call the link \emph{quadratic}.

\begin{rem}
 Note that our definition of a link deviates from the usual concept of the link at a vertex of a simplicial complex. This difference could be resolved by adjoining higher-dimensional simplices to the graph such that the links with tip $p$ in our sense appear naturally as simplices of the link at $p$. For our purposes it is, however, more adequate to use graphs with additional information instead of higher-dimensional structures.
\end{rem}

\begin{df}
 The \emph{complete Schubert system of $M$ w.r.t.\ $\cB$} is the graph $\Sigma=\Sigma(M,\cB)$ with vertex set
 \[
  \Vertex \Sigma \ = \ \Rel^2 \ \amalg \ \Rel^3
 \]
 and edge set
 \[
  \Edge \Sigma \ = \ \bigl\{ \, \{(i,j),(v,t,s)\} \, \bigl| \, w_{i,j}\text{ appears in }E(v,t,s) \, \bigr\}
 \]
 together with the set $\Link\Sigma$ of links $\lambda=\bigl((v,t,s),S\bigr)$ with $(v,t,s)\in\Rel^3$ and $S\subset \Rel^2$ such that $E(v,t,s)$ contains the term $\mu_\lambda\prod_{(i,j)\in S}w_{i,j}$ with some coefficient $\mu_\lambda\in\C^\times$. We call $\mu_\lambda$ the \emph{weight of the link $\lambda$}.
\end{df}

\begin{rem}
 Whenever there is no danger of confusion with the reduced Schubert system (as defined in section \ref{subsection: the reduced schubert system}), we refer to the complete Schubert system as the \emph{Schubert system}. The Schubert system $\Sigma(M,\cB)$ encodes the defining equations $E(v,t,s)$ of the affine variety $V(M,\cB)$: for a relevant triple $(v,t,s)\in\Vertex\Sigma$, we have
 \[
  E(v,t,s) \ = \ \sum_{\lambda=\bigl((v,t,s),S\bigr)\in\Link\Sigma} \mu_\lambda \cdot \prod_{(i,j)\in S} w_{i,j}.
 \]
 Thus $\mu_\lambda=-\mu_{v,t,s'}$ if $\lambda=\bigl((v,t,s),\{(s',s)\}\bigr)$ is linear and $\mu_\lambda=\mu_{v,s',t'}$ if $\lambda=\bigl((v,t,s),\{(s',s),(t,t')\}\bigr)$ is quadratic.
\end{rem}

\subsection*{Remark on illustrations}

For illustrations of Schubert systems we adopt the following conventions. To clearly indicate the links, we draw an edge multiple times if it is the leg of several links. If the link is linear, we annotate its weight next to its unique leg. If the link is quadratic, we draw a dotted arc between its two legs and annotate the weight of the link next to the arc. Often, we omit the weight if it is $1$.

In many situations a relevant triple $(v,t,s)$ is uniquely determined by its last two coefficients $t$ and $s$. Moreover, if $Q$ does not contain a loop, then $(t,s)$ does not occur as a relevant pair. In this situation, we often use the convention to place the vertices $(i,j)$ and $(v,t,s)$ in a matrix, according to their last two coefficients, and to suppress these last two coefficients. This means that we illustrate a relevant pair $(i,j)$ as a dot, and a relevant triple $(v,t,s)$ as $v$. The last two coefficients can be reconstructed from a coordinate system at the sides of the Schubert system.

\begin{ex}\label{ex: complete schubert system for A_3}
 Let $Q$ be an equioriented quiver of type $A_3$ and $M$ an indecomposable thin sincere representation. Then we can choose an ordered basis $\cB$ that corresponds to the following map of the associated coefficient quiver $\Gamma$ to $Q$:
 \[
 \beginpgfgraphicnamed{fig43}
  \begin{tikzpicture}[>=latex]
   \matrix (m) [matrix of math nodes, row sep=0.0em, column sep=3.0em, text height=1ex, text depth=0ex]
    { 1 & 2  & 3  &   & \Gamma  \\                 
        &    & \  &   &         \\                 
     q_1& q_2& q_3&   &   Q     \\ };              
    \path[->,font=\scriptsize]
    (m-1-1) edge node[auto] {$a$} (m-1-2)
    (m-1-2) edge node[auto] {$b$} (m-1-3)
    (m-3-1) edge node[auto] {$a$} (m-3-2)
    (m-3-2) edge node[auto] {$b$} (m-3-3)
    (m-1-5) edge node[auto] {$F$} (m-3-5);
  \end{tikzpicture}
 \endpgfgraphicnamed
 \]
 Then $\Rel^2=\{(1,1),(2,2),(3,3)\}$ and $\Rel^3=\{(a,2,1),(b,3,2)\}$. The corresponding polynomials are
 \begin{align*}
  E(a,2,1) \ &= \ w_{2,2}w_{1,1}-w_{1,1},    &  E(b,3,2) \ &= \ w_{3,3}w_{2,2}-w_{2,2}.    
 \end{align*}
 The Schubert system $\Sigma(M,\cB)$ looks as follows.
 \[
 \beginpgfgraphicnamed{fig44}
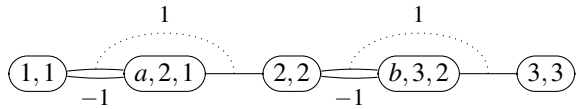
  \begin{tikzpicture}[>=latex]
   \matrix (m) [matrix of math nodes, row sep=-0.5em, column sep=0.5em, text height=1.5ex, text depth=0.5ex]
  { \node[vertex]{1,1};   &\,&\node[vertex]{a,2,1};&\,&\node[vertex]{2,2};  &\,&\node[vertex]{b,3,2};&\,&\node[vertex]{3,3};   \\}; 
    \path[-,font=\scriptsize]
    (m-1-3) edge[bend left=5] node[auto] {$-1$} (m-1-1)
    (m-1-3) edge[bend right=5] node[auto,swap] {} (m-1-1)
    (m-1-3) edge node[auto] {} (m-1-5)
    (m-1-7) edge[bend left=5] node[auto] {$-1$} (m-1-5)
    (m-1-7) edge[bend right=5] node[auto,swap] {} (m-1-5)
    (m-1-7) edge node[auto] {} (m-1-9);
    \path[dotted,font=\scriptsize]
    (m-1-4.center) edge[bend right=90] node[above] {$1$} (m-1-2.center)
    (m-1-8.center) edge[bend right=90] node[above] {$1$} (m-1-6.center);
  \end{tikzpicture}
 \endpgfgraphicnamed
 \]
\end{ex}

\begin{ex}\label{ex: complete schubert system for D_4}
 Let $Q$ be the quiver of type $D_4$ with subspace orientation and $M$ the representation with ordered basis $\cB$ that results from the map
 \[
 \beginpgfgraphicnamed{fig41}
  \begin{tikzpicture}[>=latex]
   \matrix (m) [matrix of math nodes, row sep=-0.5em, column sep=1.5em, text height=1ex, text depth=0ex]
    {   &   &   &   &   &   &   &   &   \\                  
        &   &   &   &   &   &   &   &   \\                  
      1 &   &   &   & 2 &   &   &   &    \\                 
        &   &   &   &   &   & 3 &   &   & \Gamma \\         
        &   & 5 &   & 4 &   &   &   &    \\                 
      \ &   &   &   &   &   &   &   &   \\                  
        &   &   &   &   &   &   &   &    \\                 
        &   &   &   &   &   &   &   & \ \\                  
     q_a&   &   &   &   &   &   &   &  \\                   
      \ &   &   &   &   &   &   &   &  \\                   
        &   &   &   &q_0&   &q_c&   &       & Q  \\         
      \ &   &   &   &   &   &   &   &   \\                  
        &   &q_b&   &   &   &   &   &   \\                  
        &   &   &   &   &   &   &   &   \\ };               
    \path[->,font=\scriptsize]
    (m-4-10) edge node[auto] {$F$} (m-11-10)
    (m-3-1) edge node[auto] {$a$} (m-3-5)
    (m-4-7) edge node[auto,swap] {$c$} (m-3-5)
    (m-4-7) edge node[auto] {$c$} (m-5-5)
    (m-5-3) edge node[auto,swap] {$b$} (m-5-5)
    (m-9-1) edge node[auto] {$a$} (m-11-5)
    (m-11-7) edge node[auto,swap] {$c$} (m-11-5)
    (m-13-3) edge node[auto,swap] {$b$} (m-11-5);
  \end{tikzpicture}
 \endpgfgraphicnamed
 \]
 from the coefficient quiver $\Gamma$ to $Q$. Then 
 \[
  \Rel^2=\{(1,1),(2,2),(2,4),(3,3),(4,4),(5,5)\},\]\[\Rel^3=\{(a,2,1),(c,2,3),(c,4,3),(b,2,5),(b,4,5)\},
 \]
 and
 \begin{align*}
  E(a,2,1) \ &= \ w_{2,2}w_{1,1}-w_{1,1},   \\
  E(c,2,3) \ &= \ w_{2,4}w_{3,3}+w_{2,2}w_{3,3}-w_{3,3},   &  E(c,4,3) \ &= \ w_{4,4}w_{3,3}-w_{3,3},   \\
  E(b,2,5) \ &= \ w_{2,4}w_{5,5},                          &  E(b,4,5) \ &= \ w_{4,4}w_{5,5}-w_{5,5}. 
 \end{align*}
 With the help of a coordinate system, we illustrate the Schubert system $\Sigma(M,\cB)$ as follows.
 \[
 \beginpgfgraphicnamed{fig42}
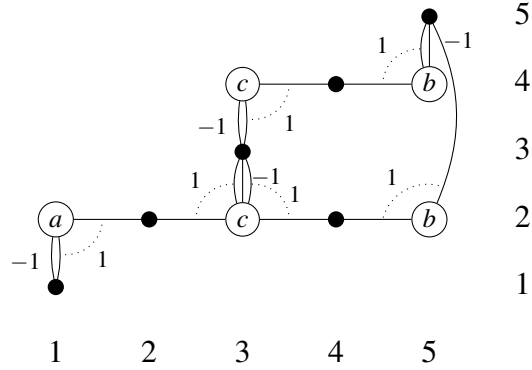
  \begin{tikzpicture}[>=latex]
   \matrix (m) [matrix of math nodes, row sep=-0.5em, column sep=0.5em, text height=1.5ex, text depth=0.5ex]
    {                   &\,&                 &\,&                  &\,&                 &\,&  \node[pair]{}; &\,&       5         \\   
              \!        &\!&       \!        &\!&        \!        &\!&       \!        &\!&       \!        &\!&                 \\   
                        &\!&                 &\!&\node[triple]{c}; &\!&   \node[pair]{};&\!&\node[triple]{b};&\!&       4         \\   
              \!        &\!&       \!        &\!&        \!        &\!&       \!        &\!&       \!        &\!&                 \\   
                        &\!&                 &\!& \node[pair]{};   &\!&                 &\!&                 &\!&       3         \\   
              \!        &\!&       \!        &\!&        \!        &\!&       \!        &\!&       \!        &\!&                 \\   
      \node[triple]{a}; &\!&  \node[pair]{}; &\!&\node[triple]{c}; &\!&   \node[pair]{};&\!&\node[triple]{b};&\!&       2         \\   
              \!        &\!&       \!        &\!&        \!        &\!&       \!        &\!&       \!        &\!&                 \\   
       \node[pair]{};   &\!&                 &\!&                  &\!&                 &\!&                 &\!&       1         \\   
              \!        &\!&       \!        &\!&        \!        &\!&       \!        &\!&       \!        &\!&                 \\   
               1        &  &        2        &  &         3        &  &       4         &  &        5        &\!&                 \\}; 
    \path[-,font=\scriptsize]
    (m-7-1) edge[bend left=10] node[right] {} (m-9-1.center)
    (m-7-1) edge[bend right=10] node[left] {$-1$} (m-9-1.center)
    (m-7-1) edge node[auto] {} (m-7-3.center)
    (m-7-5) edge node[auto] {} (m-7-3.center)
    (m-7-5) edge node[auto] {} (m-7-7.center)
    (m-7-9) edge node[auto] {} (m-7-7.center)
    (m-7-5) edge[bend left=20] node[left] {} (m-5-5.center)
    (m-7-5) edge node[right=-1pt] {$-1$} (m-5-5.center)
    (m-7-5) edge[bend right=20] node[auto] {} (m-5-5.center)
    (m-3-5) edge[bend left=10] node[right] {} (m-5-5.center)
    (m-3-5) edge[bend right=10] node[left] {$-1$} (m-5-5.center)
    (m-3-5) edge node[auto] {} (m-3-7.center)
    (m-3-9) edge node[auto] {} (m-3-7.center)
    (m-3-9) edge[bend left=20] node[right] {} (m-1-9.center)
    (m-3-9) edge[bend left=0] node[right=1pt] {$-1$} (m-1-9.center)
    (m-7-9) edge[bend right=25] node[auto] {} (m-1-9.center);
    \path[dotted,font=\scriptsize]
    (m-8-1) edge[bend right=45] node[below right=-2pt] {$1$} (m-7-2.center)
    (m-6-5) edge[bend right=45] node[above left=-2pt] {$1$} (m-7-4.center)
    (m-6-5) edge[bend left=45] node[right=0pt] {$1$} (m-7-6.center)
    (m-6-9.east) edge[bend right=45] node[above left=-2pt] {$1$} (m-7-8.center)
    (m-4-5) edge[bend right=45] node[below right=-2pt] {$1$} (m-3-6.center)
    (m-2-9) edge[bend right=45] node[above left=-2pt] {$1$} (m-3-8.center);
  \end{tikzpicture}
 \endpgfgraphicnamed
 \]
\end{ex}

\subsection{Partial Evaluations}\label{subsection: partial evaluations}
 
A \emph{partial evaluation of $\Sigma$} is a partial function $\ev:\Rel^2\dashrightarrow \C$ that satisfies
\begin{itemize}
 \item[(EV)] if $(k,l)$ is a neighbour of $(v,t,s)$ and if all other neighbours $(i,j)$ of $(v,t,s)$ are in the domain of $\ev$, then $(k,l)$ is in the domain of $\ev$ and
              \[
               \sum_{\lambda=\bigl((v,t,s),S\bigr)\in\Link\Sigma} \mu_\lambda \cdot \prod_{(i,j)\in S} \ev(i,j) \ = \ 0.
              \]
\end{itemize}
A partial valuation is \emph{total} if its domain  is $\Rel^2$.

We also write $\ev(i,j)\in\C$ to say that $(i,j)$ is in the domain of $\ev$ and $\ev(i,j)=\eta$ to say that $\ev$ is not defined in $(i,j)$. The symbol $\eta$ stays for a non-specified value of $w_{i,j}$, and it can be thought of as the generic point of the affine line over $\C$. The condition (EV) expresses the fact that the equation $E(v,t,s)=0$ has a unique solution in a variable $w_{i,j}$ if all the other variables are fixed. This is due to the property that $E(v,t,s)$ is linear in each of its variables.

The set of all partial evaluations of $\Sigma$ comes with a partial ordering. For two partial evaluations $\ev$ and $\ev'$, we define $\ev\leq\ev'$ if $\ev(i,j)\in\C$ implies $\ev'(i,j)=\ev(i,j)$. 

\begin{lemma}\label{lemma: minimal partial evaluation}
 Let $S$ be a subset of $\Rel^2$ and $f:S\to \C$ a function. Then the set
 \[
  \Ev(f) \ = \ \bigl\{ \, \ev:\Rel^2\dashrightarrow\C \, \bigl| \, \ev\text{ is a partial evaluation with }\ev(i,j)=f(i,j)\text{ for }(i,j)\in S \, \bigl\}
 \]
 has a unique minimal element unless it is empty.
\end{lemma}

\begin{proof}
 This can be seen by executing the following procedure. Define $\ev_0=f$ as a partial function $\Rel^2\dashrightarrow \C$. For $n\geq0$, assume that we have defined $\ev_n:\Rel^2\dashrightarrow \C$. 

 If there is a relevant triple $(v,t,s)$ that has precisely one neighbour $(k,l)$ that is not in the domain of $\ev_n$, then we define $\ev_{n+1}(i,j)=\ev_n(i,j)$ for all $(i,j)$ in the domain of $\ev_n$, and we define $\ev_{n+1}(k,l)$ as the unique solution of $E(v,t,s)=0$ in $w_{k,l}$ if all other variables $w_{i,j}$ of $E(v,t,s)=0$ have been substituted by the value $\ev_{n+1}(i,j)$. Note that this definition is dictated by property (EV) of an evaluation.

 If there is no relevant triple such that precisely one variable $w_{k,l}$ of $E(v,t,s)$ is not in the domain of $\ev_n$, then we face the following two cases.
 \begin{enumerate}
  \item\label{item1} There is a relevant triple $(v,t,s)$ such that all its neighbours are in the domain of $\ev_n$ and such that $E(v,t,s)$ evaluated in $w_{i,j}=\ev_n(i,j)$ is not $0$, then we derive a contradiction to (EV), which means that $\Ev(f)$ is empty.
 \item\label{item2} If there is none such $(v,t,s)$ as in \eqref{item1}, then $\ev_n$ satisfies (EV) and $\ev_n$ is a partial evaluation. By construction, it is the unique minimal element of $\Ev(f)$.
 \end{enumerate}
 In both cases, the claim of the lemma follows.
\end{proof}

\subsection{Contradictory $\beta$-states}\label{subsection: contradictory beta-states}

Let $\Sigma$ be the Schubert system of $M$ w.r.t.\ $\cB$. Let $\beta$ be a subset of $\cB$. Define the partial function $f_\beta:\Rel^2\dashrightarrow\C$ with
\[
 f_\beta(i,j) \ = \ \left\{\begin{array}{ll} 1 &\text{if }i=j\in\beta, \\ 0 &\text{if }i\in\beta\text{ and }i\neq j, \\ 0 & \text{if }j\notin\beta. \end{array} \right.
\]
If $\Ev(f_\beta)$ is empty, then we say that \emph{the $\beta$-state $\Sigma_\beta$ is contradictory}, without defining $\Sigma_\beta$. 

\begin{lemma}\label{lemma: contradictory beta have empty schubert cells}
 The points of the Schubert cell $C_\beta^M$ correspond to the total valuations in $\Ev(f_\beta)$. If $\Sigma_\beta$ is contradictory, then $C_\beta^M$ is empty. 
\end{lemma}

\begin{proof}
 Assume that $\Ev(f_\beta)$ is not empty. The points of $C_\beta^M$ correspond to the common solutions $w=(w_{i,j})$ of \textup{(NF1)}--\textup{(NF5)} and $E(v,t,s)=0$ for all relevant triples $(v,t,s)$, which can be seen as total evaluations $(i,j)\mapsto w_{i,j}$. These total evaluations are precisely the total evaluations in $\Ev(f_\beta)$.

 If $\Sigma_\beta$ is contradictory, i.e.\ $\Ev(f_\beta)$ is empty, there exists no total evaluation in $\Ev(f_\beta)$. Thus $C_\beta^M$ is empty. 
\end{proof}

An arrow $(v,s,t)$ of $\Gamma$ is \emph{extremal} if for all arrows $(v,s',t')\in\Gamma_1$ either $s<s'$ or $t'<t$. A subset $\beta$ of $\cB=\Gamma_0$ is \emph{extremal successor closed} if for all extremal arrows $(v,s,t)\in\Gamma_1$, $s\in\beta$ implies $t\in\beta$.

\begin{lemma}\label{lemma: non-extremal successor closed beta have contradictory beta-states}
 If $\beta$ is not extremal successor closed, then $\Sigma_\beta$ is contradictory.
\end{lemma}

\begin{proof}
 If $\beta$ is not extremal successor closed, then there exists an extremal arrow $(v,s,t)$ in $\Gamma$ such that $s\in\beta$ and $t\notin\beta$. Since $(v,t,s)$ is extremal, we have $E(v,t,s) = \mu_{v,s,t}w_{t,t}w_{s,s} - \mu_{v,s,t}w_{s,s}$. The equations $w_{s,s}=1$, by \textup{(NF1)}, and $w_{t,t}=0$, by \textup{(NF4)}, do not have a common solution with $E(v,t,s)=0$. Therefore (EV) cannot be satisfied by any partial function $\ev:\Rel^2\dashrightarrow\C$ with $\ev(s,s)=1$ and $\ev(t,t)=0$. This shows that $\Ev(f_\beta)$ is empty, i.e.\ $\Sigma_\beta$ is contradictory.
\end{proof}

\subsection{Definition of $\beta$-states}\label{subsection: definition of beta-states}

Assume $\Ev(f_\beta)$ is not empty and let $\ev_\beta$ be its minimal partial evaluation. Roughly speaking, the $\beta$-state is obtained from the Schubert system $\Sigma$ as follows: we replace the variable $w_{i,j}$ by $\ev_\beta(i,j)$ in all equations $E(v,t,s)$ whenever $\ev_\beta$ is defined in $(i,j)$; we adjust the links and their weights correspondingly; we remove all relevant pairs that are in the domain of $\ev_\beta$ from $\Sigma$ and all adjacent edges; we further remove all relevant triples $(v,t,s)$ without a neighbour from $\Sigma$. The $\beta$-state $\Sigma_\beta$ is the leftover graph with links.

The precise definition of the $\beta$-state is as follows. For a set $S$ of relevant pairs, define 
\[
 S_{\beta,\C} \ = \ \{ \, (i,j)\in S \, | \, \ev_\beta(i,j)\in\C \, \} \qquad \text{and} \qquad S_{\beta,\eta} \ = \ \{ \, (i,j)\in S \, | \, \ev_\beta(i,j)=\eta \, \}.
\]
Let $\Link\Sigma_\beta$ be the set of links $\lambda=\bigr((v,t,s),S\bigl)$ such that
\[
 \mu_{\beta,\lambda} \ = \ \sum_{\substack{\lambda'=\bigl((v,t,s),S'\bigr)\in\Link\Sigma\\\text{such that }S=S'_{\beta,\eta}}} \mu_{\lambda'} \cdot \prod_{(i,j)\in S'_{\beta,\C}} \ev_\beta(i,j)
\]
is not zero. Note that $\mu_{\beta,\lambda}$ occurs as a coefficient in $E(v,t,s)$ if we substitute $w_{i,j}$ by $\ev_\beta(i,j)$ if $\ev_\beta$ is defined in $(i,j)$.

Note that a link $\lambda\in\Link\Sigma_\beta$ can be constant, linear or quadratic. If $\lambda$ is quadratic, then it is also a link of $\Sigma$ with the same weight $\mu_\lambda=\mu_{\beta,\lambda}$. If $\lambda$ is linear, then it might be a link of $\Sigma$ or not, and if it is, then the weights $\mu_\lambda$ and $\mu_{\beta,\lambda}$ might be different. Since all links of $\Sigma$ are linear or quadratic, a constant link of $\Sigma_\beta$ cannot be a link of $\Sigma$.

\begin{df}
 The \emph{$\beta$-state $\Sigma_\beta$ of $\Sigma$} is the subgraph with vertex set
 \begin{align*}
  \Vertex \Sigma_\beta \ &=     \ \{\, (i,j)\in\Rel^2 \,|\, \ev_\beta(i,j)=\eta \,\} \\ 
                         &\quad \ \amalg \ \{\, (v,t,s)\in\Rel^3 \,|\, (v,t,s)\text{ is the tip of a link }\lambda\in\Link\Sigma_\beta \,\}
 \end{align*}
 and edge set
 \[
    \Edge \Sigma_\beta \ = \ \{\, x\in\Edge\Sigma \,|\, x\text{ is the leg of a link }\lambda\in\Link\Sigma_\beta \,\}
 \]
 together with the set $\Link\Sigma_\beta$ of links $\lambda$ and their associated weights $\mu_{\beta,\lambda}$. We say that a vertex or an edge of $\Sigma$ is \emph{$\beta$-relevant} if it is contained in $\Sigma_\beta$. Otherwise, we say that it is \emph{$\beta$-trivial}. If all of these sets are empty, we call $\Sigma_\beta$ \emph{trivial}.
\end{df}

For a $\beta$-relevant triple $(v,t,s)$, we define the \emph{$\beta$-reduced form of $E(v,t,s)$} as
\begin{equation*}\label{eq: beta-reduced E}
 E_\beta(v,t,s) \quad = \quad \sum_{\lambda=\bigl((v,t,s),S\bigr)\in\Link\Sigma_\beta} \mu_{\beta,\lambda} \cdot \prod_{(i,j)\in S} w_{i,j}. 
\end{equation*}
Since the equations $E_\beta(v,t,s)$ are derived from $E(v,t,s)$ by evaluating free coefficient $w_{i,j}$ w.r.t.\ $\beta$ according to \textup{(NF1)}--\textup{(NF5)} and by substituting variables by its unique solution according to (EV), we see that the image of $\iota_\beta:C_\beta\to\Mat_{\cB\times\cB}$ equals the zero set of the equations $E_\beta(v,t,s)$ (with $t\notin\beta$, $s\in\beta$) inside the affine subspace of $\Mat_{\cB\times \cB}$ described by $w_{i,j}=\ev_\beta(i,j)$ for $\beta$-trivial $(i,j)$. In other words:

\begin{prop}
 The Schubert cell $C_\beta^M$ is isomorphic to the zero set of the polynomials $E_\beta(v,t,s)$ for $(v,t,s)\in\Vertex\Sigma_\beta$ inside the affine space spanned by $w_{i,j}$ with $(i,j)\in\Vertex\Sigma_\beta$.\qed
\end{prop}

\begin{rem}
 Note that the definition of the $\beta$-reduced form of $E(v,t,s)$ deviates from the definition used in \cite{L13}. Namely, we obtain the $\beta$-reduced forms $E_\beta(v,t,s)$ (in the sense of this text) from the system of $\beta$-reduced equations (in the sense of \cite{L13}) if we solve successively all linear equations in one variable that appear in this system.
\end{rem}

Since we have not yet developed the tools to calculate the $\beta$-states efficiently, we postpone examples to a later point of the paper. We encourage the reader, however, to have a look at Examples \ref{ex: beta-states for D_4} and \ref{ex: Schubert system with extremal solution}.

\subsection{The reduced Schubert system}\label{subsection: the reduced schubert system}

For extremal successor closed subsets $\beta$ of $\cB$, it is possible to simplify the system of equations $E(v,t,s)$, which leads to the reduced Schubert system. These simplification are due to the substitution of $w_{i,i}$ by $1$ or $0$, depending on whether $i$ is in $\beta$ or not. We can apply this in the following two situations.

\begin{enumerate}
 \item\label{esc1} For an extremal arrow $(v,s,t)$ of $\Gamma$, the polynomial $E(v,t,s)=\mu_{v,s,t}w_{t,t}w_{s,s}-\mu_{v,s,t}w_{s,s}$ is constant zero if $\beta$ is extremal successor closed. 
 \item\label{esc2} Since $E_\beta(v,t,s)$ becomes trivial if $t\in\beta$ or $s\notin\beta$ (cf.\ \cite[Lemma 2.2]{L13}), we can assume that $t\notin\beta$ and $s\in\beta$. Substituting $w_{s,s}$ by $1$ in quadratic terms and $w_{t,t}$ by $0$ yields the \emph{reduced form}
       \begin{equation*}\label{eq: reduced form}
        \overline{E}(v,t,s) \quad = \hspace{-0pt} 
        \sum_{\substack{(v,s,t')\in \Gamma_1\\ t<t'}} \hspace{-5pt} \mu_{v,s,t'} w_{t,t'} \ \ + \hspace{-5pt}
        \sum_{\substack{(v,s',t')\in\Gamma_1\\ t<t'\text{ and }s'<s}}  \hspace{-5pt} \mu_{v,s',t'} w_{t,t'} w_{s',s} \ \ - \hspace{-5pt}
        \sum_{\substack{(v,s',t)\in \Gamma_1\\ s'< s}}   \mu_{v,s',t}w_{s',s} \ \ - \ \ 
        \mu_{v,s,t}
       \end{equation*}
       (where $\mu_{v,s,t}=0$ if $\Gamma$ does not contain the arrow $(v,s,t)$).
\end{enumerate}

By \eqref{esc1}, we can omit all relevant triples $(v,t,s)$ for which $(v,s,t)$ is an extremal arrow of $\Gamma$ from the Schubert system $\Sigma$. By \eqref{esc2}, we can omit all relevant pairs of the form $(i,i)$ from the Schubert system.

For a link $\lambda=\bigl((v,t,s),S\bigr)\in\Link\Sigma$, we define $\overline\lambda=\bigl((v,t,s),S-\{(s,s)\}\bigr)$.  Note that if $\overline\lambda=\overline{\lambda'}$, then $\lambda=\lambda'$. We have that either $\overline\lambda=\lambda$ or $\overline\lambda\notin\Link\Sigma$ since $E(v,t,s)$ contains neither a linear term in $w_{t,t'}$ nor a constant term. Therefore, we can use the same symbol $\mu$ for the weights of the links of the reduced Schubert system without causing ambiguity.

\begin{df}
 The \emph{reduced Schubert system of $M$ w.r.t.\ $\cB$} is the graph $\overline\Sigma=\overline\Sigma(M,\cB)$ with vertex set
 \begin{align*}
  \Vertex\overline\Sigma \ &= \     \bigl\{\, (i,j)\in\Rel^2 \,\bigl|\, i<j\,\bigl\}  \ \amalg \ \bigl\{\, (v,t,s)\in\Rel^3 \,\bigl|\, (v,s,t)\text{ is not an extremal arrow of }\Gamma \,\bigl\}
 \end{align*}
 and edge set
 \[
  \Edge\overline\Sigma \ = \ \bigl\{\, \{(v,t,s),(i,j)\}\in\Edge\Sigma \,\bigl|\, i<j \,\bigl\}
 \]
 together with the set
 \[
  \Link\overline\Sigma \ = \ \bigl\{\, \overline\lambda \,\bigl|\, \lambda\in\Link\Sigma\text{ with tip in }\Vertex\overline\Sigma \,\bigl\}
 \]
 of links $\overline\lambda$ with weights $\mu_{\overline\lambda}=\mu_{\lambda}$. 
\end{df}

In practice, we will always assume that $\beta$ is extremal successor closed, which allows us to work with the reduced Schubert system $\overline\Sigma(M,\cB)$.

The reduced Schubert system can be derived from the complete Schubert system by deleting
\begin{enumerate}
 \item all vertices of the form $(i,i)$, together with all connecting edges,
 \item all vertices of the form $(v,t,s)$ that correspond to an extremal edge $(v,s,t)$ of $\Gamma$, together with all connecting links,
 \item all legs from links with base vertices of the form $(i,i)$.
\end{enumerate}
Alternatively, the reduced Schubert system is determined by the reduced polynomials $\overline E(v,t,s)$ for relevant triples $(v,t,s)$ where $(v,s,t)$ is not an extremal arrow in $\Gamma$ in the same way as the complete Schubert system is derived from the polynomials $E(v,t,s)$ for all relevant triples $(v,t,s)$.

\begin{rem}
 The significance of the reduced Schubert system is that it carries less information, but still suffices to compute the Schubert cells $C_\beta^M$. We have
 \[
  \overline E(v,t,s) \ = \ \sum_{\lambda=\bigl((v,t,s),S\bigr)\in\Link\overline\Sigma} \mu_\lambda \cdot \prod_{(i,j)\in S} w_{i,j}.
 \]
 If $\beta$ is not extremal successor closed, then $C_\beta^M$ is empty. If $\beta$ is extremal successor closed, then $C_\beta^M$ is isomorphic to the subvariety of the affine space $\{(w_{i,j})|(i,j)\in\Vertex\overline\Sigma\}$ that is defined by the equations $\overline E(v,t,s)=0$ for $(v,t,s)\in\Vertex\overline\Sigma$ with $t\notin\beta$ and $s\in\beta$.
\end{rem}

\subsection*{Remark on illustrations} We illustrate constant links $\lambda=\bigl((v,t,s),\emptyset\bigr)$ as a dotted edge that connects the vertex $(v,t,s)$ to the weight $\mu_\lambda$, displayed inside a dotted box.

\begin{ex}\label{ex: reduced schubert system for D_4}
 Let $Q$, $M$, $\cB$ and $\Gamma$ be as in Example \ref{ex: complete schubert system for D_4}. The extremal edges of $\Gamma$ are $(a,1,2)$, $(c,3,4)$ and $(b,5,4)$. Therefore we have to consider only the reduced forms of the polynomials
 \begin{align*}
  E(c,2,3) \ &= \ w_{2,4}w_{3,3}+w_{2,2}w_{3,3}-w_{3,3},   &  E(b,2,5) \ &= \ w_{2,4}w_{5,5}.
 \end{align*}
 Substituting $w_{2,2}=0$ and $w_{3,3}=1$ in $E(c,2,3)$ and $w_{5,5}=1$ in $E(b,2,5)$ yields
 \begin{align*}
    \overline E(c,2,3) \ &= \ w_{2,4}-1,   &  \overline E(b,2,5) \ &= \ w_{2,4}.
 \end{align*}
 We conclude that $\overline\Sigma$ looks as follows.
 \[
 \beginpgfgraphicnamed{fig45}
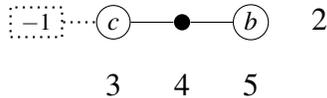
  \begin{tikzpicture}[>=latex]
   \matrix (m) [matrix of math nodes, row sep=0.5em, column sep=1.0em, text height=1.5ex, text depth=0.5ex]
    {  \node[const](33){-1}; &\node[triple](c23){c}; &\node[pair](24){}; &\node[triple](b25){b}; &       2         \\   
                             &            3          &           4       &          5            &                 \\}; 
    \path[-,font=\scriptsize]
    (c23) edge node[auto] {} (24)
    (b25) edge node[auto] {} (24);
    \path[dotted,-,thick,font=\scriptsize]
    (33) edge node[right=-1pt] {} (c23);
  \end{tikzpicture}
 \endpgfgraphicnamed
 \]
\end{ex}

\begin{ex}
 Let $Q$ be a quiver of type $A_2$ with arrow $a$. Let $M$ be the representation with dimension vector $(2,2)$ and matrix $M_a=\binom{\mu\ 1}{0\ \mu}$ where $\mu\in\C^\times$. Let $\cB$ be the basis with associated coefficient quiver $\Gamma$
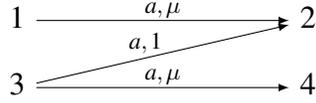
 \[
 \beginpgfgraphicnamed{fig46}
  \begin{tikzpicture}[>=latex]
   \matrix (m) [matrix of math nodes, row sep=1.0em, column sep=8.0em, text height=1ex, text depth=0ex]
    { 1 & 2  \\                 
      3 & 4  \\   };            
    \path[->,font=\scriptsize]
    (m-1-1) edge node[above=-3pt] {$a,\mu$} (m-1-2)
    (m-2-1) edge node[above left=-4pt] {$a,1$} (m-1-2)
    (m-2-1) edge node[above=-3pt] {$a,\mu$} (m-2-2)
;
  \end{tikzpicture}
 \endpgfgraphicnamed
 \]
 with the respective weights $\mu$ and $1$. Then
 \[
  \Rel^2 \ = \ \big\{\, (1,1),(1,3),(3,3),(2,2),(2,4),(4,4) \,\bigr\}, \qquad \Rel^3 \ = \  \big\{\, (a,2,1), (a,4,3), (a,2,3) \,\bigr\}.
 \]
 The defining polynomials of $V(M,\cB)$ are
 \begin{align*}
  E(a,2,1) \ &= \ \mu w_{2,2}w_{1,1}-w_{1,1},   &\hspace{-4cm}   E(a,4,3) \ &= \ \mu w_{4,4}w_{3,3}-w_{3,3},   \\   
  E(a,2,3) \ &= \ \mu w_{2,2}w_{1,3}+w_{2,2}w_{3,3}+\mu w_{2,4}w_{3,3}-\mu w_{1,3}-w_{3,3}.
 \end{align*}
 The complete Schubert system can be deduced from these equations. Since $(a,1,2)$ and $(a,3,4)$ are extremal arrows in $\Gamma$, the reduced Schubert system corresponds to the reduced form 
 \[
  \overline E(a,2,3)=\mu w_{2,4}-\mu w_{1,3}-1.
 \]
 The complete and the reduced Schubert systems of $M$ w.r.t.\ $\cB$ are, respectively, as follows:
 \[
 \beginpgfgraphicnamed{fig47}
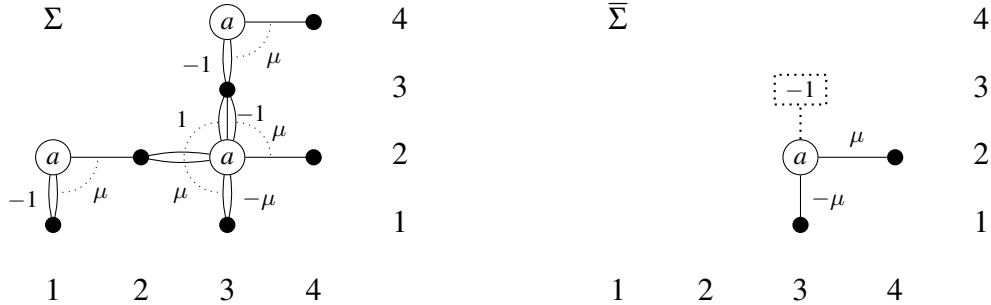
  \begin{tikzpicture}[>=latex]
   \matrix (m) [matrix of math nodes, row sep=-0.5em, column sep=0.5em, text height=1.5ex, text depth=0.5ex]
    {         \Sigma    &\!&                 &\!&\node[triple]{a}; &\!&   \node[pair]{};&\!&       4         \\   
              \!        &\!&       \!        &\!&        \!        &\!&       \!        &\!&                 \\   
                        &\!&                 &\!& \node[pair]{};   &\!&                 &\!&       3         \\   
              \!        &\!&       \!        &\!&        \!        &\!&       \!        &\!&                 \\   
      \node[triple]{a}; &\!&  \node[pair]{}; &\!&\node[triple]{a}; &\!&   \node[pair]{};&\!&       2         \\   
              \!        &\!&       \!        &\!&        \!        &\!&       \!        &\!&                 \\   
       \node[pair]{};   &\!&                 &\!&  \node[pair]{};  &\!&                 &\!&       1         \\   
              \!        &\!&       \!        &\!&        \!        &\!&       \!        &\!&                 \\   
               1        &  &        2        &  &         3        &  &       4         &  &                 \\}; 
    \path[-,font=\scriptsize]
    (m-5-1) edge[bend left=10] node[right] {} (m-7-1.center)
    (m-5-1) edge[bend right=10] node[left] {$-1$} (m-7-1.center)
    (m-5-1) edge node[auto] {} (m-5-3.center)
    (m-5-5) edge[bend left=10] node[auto] {} (m-5-3.center)
    (m-5-5) edge[bend right=10] node[auto] {} (m-5-3.center)
    (m-5-5) edge node[auto] {} (m-5-7.center)
    (m-5-5) edge[bend left=20] node[left] {} (m-3-5.center)
    (m-5-5) edge node[right=-1pt] {$-1$} (m-3-5.center)
    (m-5-5) edge[bend right=20] node[auto] {} (m-3-5.center)
    (m-5-5) edge[bend left=10] node[right] {$-\mu$} (m-7-5.center)
    (m-5-5) edge[bend right=10] node[left] {} (m-7-5.center)
    (m-1-5) edge[bend left=10] node[right] {} (m-3-5.center)
    (m-1-5) edge[bend right=10] node[left] {$-1$} (m-3-5.center)
    (m-1-5) edge node[auto] {} (m-1-7.center)
;
    \path[dotted,font=\scriptsize]
    (m-6-1) edge[bend right=45] node[below right=-2pt] {$\mu$} (m-5-2.center)
    (m-6-5) edge[bend left=45] node[below left=-2pt] {$\mu$} (m-5-4.center)
    (m-4-5) edge[bend right=45] node[above left=-2pt] {$1$} (m-5-4.center)
    (m-4-5) edge[bend left=45] node[right=0pt] {$\mu$} (m-5-6.center)
    (m-2-5) edge[bend right=45] node[below right=-2pt] {$\mu$} (m-1-6.center)
;
  \end{tikzpicture}
 \endpgfgraphicnamed
 \hspace{2cm}
 \beginpgfgraphicnamed{fig48}
  \begin{tikzpicture}[>=latex]
   \matrix (m) [matrix of math nodes, row sep=-0.5em, column sep=0.5em, text height=1.5ex, text depth=0.5ex]
    { \overline\Sigma   &\!&                 &\!&                  &\!&                 &\!&       4         \\   
              \!        &\!&       \!        &\!&        \!        &\!&       \!        &\!&                 \\   
                        &\!&                 &\!&\node[const]{-1}; &\!&                 &\!&       3         \\   
              \!        &\!&       \!        &\!&        \!        &\!&       \!        &\!&                 \\   
                        &\!&                 &\!&\node[triple]{a}; &\!&   \node[pair]{};&\!&       2         \\   
              \!        &\!&       \!        &\!&        \!        &\!&       \!        &\!&                 \\   
                        &\!&                 &\!&  \node[pair]{};  &\!&                 &\!&       1         \\   
              \!        &\!&       \!        &\!&        \!        &\!&       \!        &\!&                 \\   
               1        &  &        2        &  &         3        &  &       4         &  &                 \\}; 
    \path[-,font=\scriptsize]
    (m-5-5) edge node[auto] {$\mu$} (m-5-7.center)
    (m-5-5) edge node[auto] {$-\mu$} (m-7-5.center)
;
    \path[dotted,thick,-,font=\scriptsize]
    (m-5-5) edge node[auto] {} (m-3-5);
  \end{tikzpicture}
 \endpgfgraphicnamed
\] 
\end{ex}

\subsection{Computing $\beta$-states}\label{subsection: computing beta-states}

If $\beta$ is not extremal successor closed, then $\Sigma_\beta$ is contradictory by Lemma \ref{lemma: non-extremal successor closed beta have contradictory beta-states}. Whether $\beta$ is extremal successor closed can be easily verified with the help of the coefficient quiver. For extremal successor closed $\beta$, we determine whether $\Sigma_\beta$ is contradictory, and if not, compute $\Sigma_\beta$ and $\ev_\beta$ with the following algorithm.

\subsection*{The initial steps.} Apply \eqref{step1}--\eqref{step6} to all vertices, edges and links of $\overline\Sigma$.
\begin{enumerate}
 \item\label{step1} For a link $\lambda\in\Link\overline\Sigma$, set $\mu_{\beta,\lambda}=\mu_\lambda$. For all other links $\lambda$, set $\mu_{\beta,\lambda}=0$. These values might change while proceeding with the algorithm.
 \item\label{step2} If $(i,j)$ is a vertex with $i=j$ or $i\in\beta$ or $j\notin\beta$, then mark it as $\beta$-trivial. If $i=j\in\beta$, then define $\ev_\beta(i,j)=1$; otherwise define $\ev_\beta(i,j)=0$.
 \item\label{step3} If $(v,t,s)$ is a vertex with $t\in\beta$ or $s\notin\beta$, then mark it as $\beta$-trivial.
 \item\label{step4} If the tip of a link $\lambda$ is $\beta$-trivial, then set $\mu_{\beta,\lambda}=0$.
 \item\label{step5} If a link $\lambda=\bigl((v,t,s),S\bigr)$ has a $\beta$-trivial base vertex, then set $S'=S-\{\beta\text{-trivial }(i,j)\}$ and $\lambda'=\bigl((v,t,s),S'\bigr)$, replace $\mu_{\beta,\lambda'}$ by 
                    \[
                     \mu_{\beta,\lambda'} \ + \ \mu_{\beta,\lambda}\prod_{(i,j)\in S-S'}\ev_\beta(i,j)
                    \]
                    and set $\mu_{\beta,\lambda}=0$.
 \item\label{step6} If $\{(v,t,s),(i,j)\}$ is not the leg of a link $\lambda$ with $\mu_{\beta,\lambda}\neq0$, then mark it as $\beta$-trivial.
\end{enumerate}

\subsection*{The loop} Repeat steps \eqref{step7} and \eqref{step8} in arbitrary order until $\Sigma_\beta$ is declared contradictory or none of \eqref{step7} and \eqref{step8} applies anymore.
\begin{enumerate}\setcounter{enumi}{6}
 \item\label{step7} If $(v,t,s)$ is not $\beta$-trivial and there is precisely one edge $\{(v,t,s),(i,j)\}$ that is not $\beta$-trivial and $\mu_{\beta,\lambda}\neq0$ for $\lambda=\bigl((v,t,s),\{(i,j)\}\bigr)$, then 
       \begin{enumerate}
        \item define $\ev_\beta(i,j)=-\mu_{\beta,\lambda_0}/\mu_{\beta,\lambda}$ where $\lambda_0=\bigl((v,t,s),\emptyset\bigr)$ is the constant link;
        \item set $\mu_{\beta,\lambda}=\mu_{\beta,\lambda_0}=0$;
        \item mark $(v,t,s)$ and $(i,j)$ as $\beta$-trivial;
        \item apply step \eqref{step5} to all links $\lambda$ with base vertex $(i,j)$ and $\mu_{\beta,\lambda}\neq0$;
        \item mark all edges with end vertex $(i,j)$ as $\beta$-trivial.
       \end{enumerate}
 \item\label{step8} If $(v,t,s)$ is not $\beta$-trivial and there is no edge $\{(v,t,s),(i,j)\}$ that is not $\beta$-trivial, then mark $(v,t,s)$ as $\beta$-trivial if $\mu_{\beta,\lambda_0}=0$ for $\lambda_0=\bigl((v,t,s),\emptyset\bigr)$; if $\mu_{\beta,\lambda_0}\neq0$, then declare $\Sigma_\beta$ as contradictory and stop the algorithm.
\end{enumerate}

\subsection*{The outcome} Once the algorithm stops, we either know that $\Sigma_\beta$ is contradictory or we have calculated $\ev_\beta$. In the latter case, we mark all vertices and edges of $\overline\Sigma$ that are not $\beta$-trivial as $\beta$-relevant and obtain
\begin{align*}
  \Vertex \Sigma_\beta \ &= \ \big\{\, \beta\text{-relevant vertices in }\Vertex\overline\Sigma \,\bigl\},     \\
  \Edge \Sigma_\beta   \ &= \ \big\{\, \beta\text{-relevant edges in }\Vertex\overline\Sigma \,\bigl\},        \\
  \Link \Sigma_\beta   \ &= \ \big\{\, \text{links }\lambda\text{ with weight }\mu_{\beta,\lambda}\neq 0 \,\}.
\end{align*}

\begin{rem}\label{rem: effect of the initial steps of the algorithm}
 The initial steps \eqref{step1}--\eqref{step6} reduce the information to calculate the $\beta$-state to the full subgraph of $\overline\Sigma$ with vertices $(i,j)$ and $(v,t,s)$ with $i,t\notin\beta$ and $j,s\in\beta$. However, the links of this subgraph differ in general from the links of $\overline\Sigma$ by the substitutions made in step \eqref{step5}.
\end{rem}

\begin{ex}\label{ex: beta-states for D_4}
 Let $Q$, $M$, $\cB$ and $\Gamma$ be as in Example \ref{ex: complete schubert system for D_4}. Recall from Example \ref{ex: reduced schubert system for D_4} that the reduced Schubert system $\overline\Sigma$ is as follows:
 \[
 \beginpgfgraphicnamed{fig49}
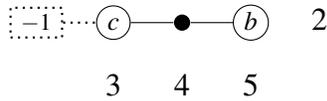
  \begin{tikzpicture}[>=latex]
   \matrix (m) [matrix of math nodes, row sep=0.5em, column sep=1.0em, text height=1.5ex, text depth=0.5ex]
    {  \node[const](33){-1}; &\node[triple](c23){c}; &\node[pair](24){}; &\node[triple](b25){b}; &       2         \\   
                             &            3          &           4       &          5            &                 \\}; 
    \path[-,font=\scriptsize]
    (c23) edge node[auto] {} (24)
    (b25) edge node[auto] {} (24);
    \path[dotted,-,thick,font=\scriptsize]
    (33) edge node[right=-1pt] {} (c23);
  \end{tikzpicture}
 \endpgfgraphicnamed
 \]

We calculate the $\{3,4,5\}$-state. With step \eqref{step1}, we define the weights $\mu_{\beta,\lambda}$. None of the steps \eqref{step2}--\eqref{step6} applies. We proceed to apply step \eqref{step7} to $(c,2,3)$, which yields \\
\indent (a) $\ev_\beta(2,4)=-(-1)/1=1$; \\ 
\indent (b) $\mu_\lambda=\mu_{\lambda_0}=0$ where $\lambda=\bigl((c,2,3),(2,4)\bigr)$ and $\lambda_0=\bigl((c,2,3),\emptyset\bigr)$; \\ 
\indent (c) $(c,2,3)$ and $(2,4)$ are $\beta$-trivial; \\
\indent (d) $\mu_{\lambda'}=0+1\cdot 1=1$ and $\mu_\lambda=0$ where $\lambda=\bigl((b,2,5),(2,4)\bigr)$ and $\lambda'=\bigl((b,2,5),\emptyset\bigr)$; \\
\indent (e) $\{(c,2,3),(2,4)\}$ and $\{(b,2,5),(2,4)\}$ are $\beta$-trivial. \\
We can illustrate this as follows:
 \[
 \beginpgfgraphicnamed{fig51}
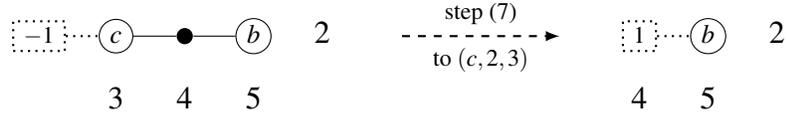
  \begin{tikzpicture}[>=latex]
   \matrix (m) [matrix of math nodes, row sep=0.5em, column sep=1.0em, text height=1.5ex, text depth=0.5ex]
    {  \node[const](33){-1}; &\node[triple](c23){c}; &\node[pair](24){}; &\node[triple](b25){b}; &       2   & \ &&&&& \ &\node[const](24'){1}; &\node[triple](b25'){b}; &       2    \\   
                             &            3          &           4       &          5            &           &   &&&&&   &         4           &           5            &            \\}; 
    \path[-,font=\scriptsize]
    (c23) edge node[auto] {} (24)
    (b25) edge node[auto] {} (24);
    \path[dotted,-,thick,font=\scriptsize]
    (33) edge node[right=-1pt] {} (c23)
    (24') edge node[right=-1pt] {} (b25');
    \path[dashed,->,thick,font=\scriptsize]
    (m-1-6) edge node[above] {step \eqref{step7}} node[below] {to $(c,2,3)$} (m-1-11);
  \end{tikzpicture}
 \endpgfgraphicnamed
 \]
 Next, we apply step \eqref{step8} to $(b,2,5)$, which declares the $\{3,4,5\}$-state as contradictory since $\mu_{\beta,\lambda}=1$ for $\lambda=\bigl((b,2,5),\emptyset\bigr)$.

 We calculate the $\{3,4\}$-state. The initial steps reduce $\overline\Sigma$ to the full subgraph with vertices $(c,2,3)$ and $(2,4)$, without changing the weights of the links with tip $(c,2,3)$. We apply step \eqref{step7} to $(c,2,3)$
 \[
 \beginpgfgraphicnamed{fig53}
  \begin{tikzpicture}[>=latex]
   \matrix (m) [matrix of math nodes, row sep=0.5em, column sep=1.0em, text height=1.5ex, text depth=0.5ex]
    {  \node[const](33){-1}; &\node[triple](c23){c}; &\node[pair](24){}; &   &    2   & \ &&&&& \ &\emptyset    \\   
                             &            3          &           4       & 5 &        &   &&&&&   &         \\}; 
    \path[-,font=\scriptsize]
    (c23) edge node[auto] {} (24);
    \path[dotted,-,thick,font=\scriptsize]
    (33) edge node[right=-1pt] {} (c23);
    \path[dashed,->,thick,font=\scriptsize]
    (m-1-6) edge node[above] {step \eqref{step7}} node[below] {to $(c,2,3)$} (m-1-11);
  \end{tikzpicture}
 \endpgfgraphicnamed
 \]
 and we see that the $\{3,4\}$-state is trivial.

 As a last example, we calculate the $\{4\}$-state. As explained in Remark \ref{rem: effect of the initial steps of the algorithm}, the initial steps reduce the information to calculate $\Sigma_{\{4\}}$ to the subgraph that consists of the vertex $(2,4)$. Since all triples are $\beta$-trivial, steps \eqref{step7} and \eqref{step8} do not apply anymore, and we see that $\Sigma_{\{4\}}=\{(2,4)\}$.

 For an example of a more complex $\beta$-state, see Example \ref{ex: Schubert system with extremal solution}.
\end{ex}

\subsection{Solvable $\beta$-states}\label{subsection: solvable beta-states}

In this section, we formulate a combinatorial condition on a $\beta$-state that implies that the Schubert cell is an affine space. Since we will apply this condition also to subgraphs of a $\beta$-state, we will formulate the results in this section for a broader class of graphs with weighted links.

A \emph{system} is a bipartite graph $\Xi$ whose vertex colours are `pairs' $\pi$ and `triples' $\tau$, together with a set $\Link\Xi$ of links whose tips are triples and whose base vertices are pairs, and a weight function $\mu:\Link\Xi\to\C^\times$ that associates a weight $\mu_\lambda$ with each link $\lambda$. 

An edge $\{\tau,\pi\}$ of a system $\Xi$ is \emph{simply linked} if $(\tau,\{\pi\})$ is a link of $\Xi$ and if $\{\tau,\pi\}$ is not the leg of any other link of $\Xi$.

\begin{df}\label{df: solvable beta state}
 A \emph{solution for a system $\Xi$} is an orientation of the edges of $\Xi$ such that
 \begin{enumerate}
  \item[(S1)]\label{solv1} for every triple $\tau$, there exists precisely one edge $\{\tau,\pi\}$ that is oriented away from $\tau$ and this edge is simply linked;
  \item[(S2)]\label{solv2} for every pair $\pi$, there exists at most one edge that is oriented towards $\pi$;
  \item[(S3)]\label{solv3} $\Xi$ is without oriented cycles.
 \end{enumerate}
 The system $\Xi$ is \emph{solvable} if it has a solution.
\end{df}

Given a solution for $\Xi$, an edge $\{\tau,\pi\}$ that is oriented towards the pair $\pi$ is simple by (S1). We can relax (S1) by allowing several edges oriented away from a given triple $\tau$, one of them simply linked. If this orientation satisfies (S2) and (S3), then inverting the orientation of all edges pointing away from a triple $\tau$ except for one simply linked edge yields a solution for $\Xi$.

\begin{thm}\label{thm: solvable beta-state implies affine schubert cell}
 Let $\Sigma_\beta$ be a $\beta$-state that is not contradictory. If $\Sigma_\beta$ is solvable, then $C_\beta^M$ is an affine space of dimension
 \[
  \dim C_\beta^M \ = \ \#\{\,\beta\text{-relevant pairs}\,\} \ - \ \# \{\,\beta\text{-relevant triples}\,\}.
 \]
\end{thm}

\begin{proof}
 Fix a solution for $\Sigma_\beta$. We can consider $\Sigma_\beta$ as a quiver and adopt the terminology of quivers. By (EV), every $\beta$-relevant triple connects to at least two $\beta$-relevant pairs. Since there is precisely one arrow pointing away from each $\beta$-relevant triple, all sinks and sources of $\Sigma_\beta$ are relevant pairs $(i,j)$. 

 Let $\Delta$ be the set of all full subquivers $\Gamma$ of $\Sigma_\beta$ whose sinks are relative pairs and that contain all \emph{predecessors}, i.e.\ all vertices of $\Sigma_\beta$ that are the start of an arrow in $\Sigma_\beta$ with target in $\Gamma$. The set $\Delta$ is partially ordered by inclusion.

 A subquiver $\Gamma$ in $\Delta$ satisfies the following properties. All sources of $\Gamma$ are also sources in $\Sigma_\beta$ and thus, in particular, relevant pairs. Every neighbour $(i,j)$ of a relevant triple $(v,t,s)$ in $\Gamma$ is also in $\Gamma$ since otherwise $(v,t,s)$ would be a sink or it would miss a predecessor. If all neighbours of a relevant triple $(v,t,s)$ are in $\Gamma$, then $(v,t,s)$ is the predecessor of some neighbour $(i,j)$ by (S1), which implies that $(v,t,s)$ is also in $\Gamma$. 

 For a non-empty $\Gamma\in\Delta$, we define $C(\Gamma)$ as the affine variety with coordinates $w_{i,j}$ for relevant pairs $(i,j)$ in $\Gamma$ and defining equations $E_\beta(v,t,s)=0$ for relevant triples $(v,t,s)$ in $\Gamma$. For the empty subgraph $\Gamma_0$, we define $C(\Gamma_0)$ as a point. Note that all variables of $E_\beta(v,t,s)$ are indexed by $\beta$-relevant pairs in $\Gamma$ because $\Gamma$ contains all neighbours of $(v,t,s)$.

 We will show by induction over $\Delta$ that $C(\Gamma)$ is an affine space of dimension
 \[
  \delta(\Gamma) \ := 
   \ \#\{\,\beta\text{-relevant pairs in }\Gamma\,\} \ - \ \#\{\,\beta\text{-relevant triples in }\Gamma\,\}.
 \]
 Since $\Sigma_\beta$ is the maximal element of $\Delta$ and $C_\beta^M=C(\Sigma_\beta)$, this induction implies the claim of the theorem. 

 The minimal element of $\Delta$ is empty subquiver $\Gamma_0$ for which our claim is trivially satisfied. This establishes the base case of the induction.

 Let $\Gamma$ be a non-empty subquiver in $\Delta$. Since $\Sigma_\beta$ does not contain any oriented cycle, $\Gamma$ must have a sink $(i,j)$. Let $\Gamma'$ be the maximal subgraph of $\Gamma$ in $\Delta$ that does not contain $(i,j)$. By the inductive hypothesis, $C(\Gamma')$ is an affine space of dimension $\delta(\Gamma')$.

 If $(i,j)$ does not have a neighbour in $\Gamma$, then $\Vertex\Gamma=\Vertex\Gamma'\cup\{(i,j)\}$ and $w_{i,j}$ occurs in none of the polynomials $E_\beta(v,t,s)$ with $(v,t,s)$ in $\Gamma$. Therefore $C(\Gamma)\simeq C(\Gamma')\times\A^1$ is an affine space of dimension
 \[
  \dim C(\Gamma) \ = \ \dim C(\Gamma') +1 \ = \ \delta(\Gamma')+1 \ = \ \delta(\Gamma),
 \]
 which establishes the induction step in case that $(i,j)$ is an isolated vertex in $\Gamma$.

 If $(i,j)$ has a neighbour $(v,t,s)$ in $\Gamma$, which is a predecessor of $(i,j)$ by our assumption that $(i,j)$ is a sink, then this neighbour is unique by (S2) of Definition \ref{df: solvable beta state}. Therefore $\Vertex\Gamma=\Vertex\Gamma'\cup\{(i,j),(v,t,s)\}$. This means that $w_{i,j}$ is uniquely determined by $E_\beta(v,t,s)=0$ and the values of all $w_{k,l}$ with $(k,l)$ in $\Gamma'$. Therefore $C(\Gamma)=C(\Gamma')$ is an affine space of dimension $\delta(\Gamma')$. Since $\Gamma$ differs from $\Gamma'$ by exactly one pair and one triple, we have $\delta(\Gamma)=\delta(\Gamma')$.

 This establishes the induction step in case that $(i,j)$ has a neighbour in $\Gamma$ and finishes the proof of the theorem.
\end{proof}

The important consequence for Schubert decompositions is the following.

\begin{cor} \label{cor: decomposition into affine spaces if all beta-states are contradictory or solvable}
 Let $\ue$ be a dimension vector. If for all $\beta\subset\cB$ of type $\ue$, $\Sigma_\beta$ is either contradictory or solvable, then $\Gr_\ue(M)=\coprod C_\beta^M$ is a decomposition into affine spaces, and $C_\beta^M$ is empty if and only if $\Sigma_\beta$ is contradictory.
\end{cor}

\begin{proof}
 This follows immediately from Lemma \ref{lemma: contradictory beta have empty schubert cells} and Theorem \ref{thm: solvable beta-state implies affine schubert cell}.
\end{proof}

\begin{cor}\label{cor: trivial beta-states are solvable and have one point schubert cells}
 If $\Sigma_\beta$ is trivial, then it is solvable and $\C_\beta^M$ is a point.
\end{cor}

\begin{proof}
 A trivial $\beta$-state is obviously solvable. Theorem \ref{thm: solvable beta-state implies affine schubert cell} yields that $C_\beta^M$ is an affine space of dimension $0$, i.e.\ a point.
\end{proof}

\begin{cor}\label{cor: schubert decompositions for thin representations}
 If $M$ is a thin representation of $Q$, then the reduced Schubert system is trivial. Thus $\Gr_\ue(M)=\coprod C_\beta^M$ is a decomposition into affine spaces for every dimension vector $\ue$. More precisely, $\Gr_\ue(M)$ is a point if there exists an extremal successor closed $\beta$ of type $\ue$, and $\Gr_\ue(M)$ is empty otherwise.
\end{cor}

\begin{proof}
 If $M$ is thin, then $(v,s,t)$ is an extremal arrow for every relevant triple $(v,t,s)$. Therefore the reduced Schubert system is empty. If $\beta$ is extremal successor closed, then $\Sigma_\beta$ is trivial and, in particular, not contradictory. This means that $\Gr_\ue(M)=\coprod C_\beta^M$ is a decomposition into affine spaces of dimension $0$. Since there is a unique subset $\beta$ of $\cB$ for every type $\ue$, the last claim follows.
\end{proof}

\begin{ex}\label{ex: solvable beta-states for  D_4}
 Let $Q$, $M$, $\cB$ and $\Gamma$ be as in Example \ref{ex: complete schubert system for D_4}. In Example \ref{ex: beta-states for D_4}, we calculated some $\beta$-states of the Schubert system. The $\{4\}$-state consists of the relevant pair $(2,4)$ and no edges. This means that the $\{4\}$ is solvable and $C_{\{4\}}^M\simeq\A^1$. The $\{3,4\}$-state is trivial and thus solvable with $C^M_{\{3,4\}}=\A^0$.
\end{ex}

\subsection{Extremal edges}\label{subsection: extremal edges}

We say that $\overline\Sigma$ is \emph{totally solvable} if every $\beta$-state is either contradictory or solvable. By Corollary \ref{cor: decomposition into affine spaces if all beta-states are contradictory or solvable}, this implies that the Schubert decomposition $\Gr_\ue(M)=\coprod C_\beta^M$ is a decomposition into affine spaces for every dimension vector $\ue$, and that $C_\beta^M$ is empty if and only if $\Sigma_\beta$ is contradictory.

In general, a solution for the (reduced) Schubert system does not restrict to a solution for all of its non-contradictory $\beta$-states. This means that the (reduced) Schubert system can be solvable, but not totally solvable. It also happens that the (reduced) Schubert system is not solvable, but totally solvable. This is, for instance, the case for preprojective representations of type $\widetilde D_n$, see section \ref{section: Schubert decompositions for type D_n}.

In this section, we present a condition that implies that a solution of the (reduced) Schubert system yields a solution for every non-contradictory $\beta$-state.

\begin{df}
 An edge $\{(v,t,s),(i,j)\}$ of $\overline\Sigma$ is \emph{extremal} if it is simply linked and if either $j=s$ and $(v,i,t)$ is an extremal arrow of $\Gamma$ or $i=t$ and $(v,s,j)$ is an extremal arrow of $\Gamma$.
\end{df}
 
The following pictures illustrate the two situation of an extremal edge $\{(v,t,s),(i,j)\}$ of $\overline\Sigma$ where we draw, as usual, vertices in the same fibre of $F:\Gamma\to Q$ on top of each other:
 \[
 \beginpgfgraphicnamed{fig55}
  \begin{tikzpicture}[>=latex,baseline=(current bounding box.center)]
   \matrix (m) [matrix of math nodes, row sep=1.0em, column sep=4.0em, text height=1ex, text depth=0ex]
    { i & t  \\                 
     j=s&    \\   };            
    \path[->,font=\scriptsize]
    (m-1-1) edge node[above] {$v$} node[below] {extremal} (m-1-2);
  \end{tikzpicture}
 \endpgfgraphicnamed
\hspace{2cm} \text{or} \hspace{2cm}
 \beginpgfgraphicnamed{fig56}
  \begin{tikzpicture}[>=latex,baseline=(current bounding box.center)]
   \matrix (m) [matrix of math nodes, row sep=1.0em, column sep=4.0em, text height=1ex, text depth=0ex]
    {   & t=i  \\                 
      s &  j   \\   };            
    \path[->,font=\scriptsize]
    (m-2-1) edge node[above] {$v$} node[below] {extremal} (m-2-2);
  \end{tikzpicture}
 \endpgfgraphicnamed
 \] 
\begin{lemma}\label{lemma: properties of extremal edges}
 The extremal edges of $\overline\Sigma$ satisfy the following properties.
 \begin{enumerate}
  \item\label{extr1} Any vertex $(v,t,s)$ of $\overline\Sigma$ is the end vertex of at most two extremal edges.
  \item\label{extr2} Assume that $\Sigma_\beta$ is not contradictory. If $\{(v,t,s),(i,j)\}$ is an extremal edge of $\overline\Sigma$ and $(v,t,s)$ is $\beta$-relevant, then $i\notin\beta$ and $j\in\beta$.
 \end{enumerate}
\end{lemma}

\begin{proof}
 For a relevant triple $(v,t,s)$, there is at most one $i\in\Gamma_0$ such that $(v,i,t)$ is an extremal arrow of $\Gamma$, and there is at most one $j\in\Gamma_0$ such that $(v,s,j)$ is an extremal arrow in $\Gamma$. Hence claim \eqref{extr1} of the lemma.

 If $\Sigma_\beta$ is not contradictory, then $\beta$ is extremal successor closed. If $(v,t,s)$ is $\beta$-relevant, then $t\notin\beta$ and $s\in\beta$. In either case, $j=s$ and $(v,i,t)$ is extremal or $i=t$ and $(v,s,j)$ is extremal, we have $i\notin\beta$ and $j\in\beta$. Hence claim \eqref{extr2} of the lemma.
\end{proof}

\begin{rem}

 It is in general not true that $(i,j)$ is $\beta$-relevant if $(v,t,s)$ is $\beta$-relevant and $\{(v,t,s),(i,j)\}$ is an extremal edge. In Example \ref{ex: Schubert system with extremal solution}, we provide a counterexample.
\end{rem}

\begin{df}
 Let $\Gamma$ be a full subsystem of $\overline\Sigma$. An \emph{extremal solution for $\Gamma$} is an orientation of the edges of $\Gamma$ such that 
 \begin{enumerate}
  \item[(ES1)] for every relevant triple $(v,t,s)$ in $\Gamma$, there exists precisely one edge $\{(v,t,s),(i,j)\}$ in $\Gamma$ that is oriented away from $(v,t,s)$ and this edge is extremal;
  \item[(ES2)] for every relevant pair $(i,j)$ in $\Gamma$, there exists at most one edge $\{(v,t,s),(i,j)\}$ in $\Gamma$ that is oriented towards $(i,j)$;
  \item[(ES3)] $\Gamma$ is without oriented cycles.
 \end{enumerate}
\end{df}

\begin{prop}\label{prop: extremal solutions of the reduced schubert system}
 If $\overline\Sigma$ has an extremal solution, then $\overline\Sigma$ is totally solvable.
\end{prop}

\begin{proof}
 This proof follows a similar strategy as the proof of Theorem \ref{thm: solvable beta-state implies affine schubert cell}. We will make use of the same conclusions without repeating them in detail.
 Fix an extremal solution of $\overline\Sigma$. Then all sinks and sources of $\overline\Sigma$ are relevant pairs.

 Let $\beta$ be a subset of $\cB$ such that $\Sigma_\beta$ is not contradictory. Let $\Delta$ be the partially ordered set of all full subgraphs $\Gamma$ of $\Sigma_\beta$ that are closed under predecessors and whose sinks are relevant pairs. We will prove by induction on $\Delta$ that every $\Gamma\in\Delta$ satisfies the properties (S1)--(S3) of a solvable $\beta$-state w.r.t.\ the orientation coming from $\overline\Sigma$.

 The empty subgraph is the minimal element of $\Delta$ and trivially satisfies (S1)--(S3).

 Let $\Gamma\in\Delta$ be a non-empty subgraph of $\Sigma_\beta$. Properties (S2) and (S3) clearly hold for $\Gamma$. To verify (S1), observe that $\Gamma$ contains a sink $(i,j)$ by (S3). Let $\Gamma'$ be a maximal element of $\Delta$ that is contained in $\Gamma-\{(i,j)\}$. By the induction hypothesis, $\Gamma'$ satisfies (S1). If $\{i,j\}$ is an isolated vertex of $\Gamma$, then $\Gamma'=\Gamma-\{(i,j)\}$ and (S1) is satisfied by $\Gamma$ as well. 

 If $\{i,j\}$ is not an isolated vertex of $\Gamma$, then there exists a unique edge $\{(v,t,s),(i,j)\}$ with end vertex $(i,j)$ in $\Gamma$ since $(i,j)$ is a sink and by (S2). By our assumptions, $\{(v,t,s),(i,j)\}$ is the unique edge pointing away from $(v,t,s)$, which is an extremal edge. This means that $(v,t,s)$ is not the predecessor of any vertex other than $(i,j)$. Therefore $\Gamma'$ is the full subgraph on the vertex set $\Vertex\Gamma-\{(v,t,s),(i,j)\}$. Also in this case, it follows that $\Gamma$ satisfies (S1).

 Thus every element $\Gamma$ of $\Delta$ satisfies (S1)--(S3). Once we have shown that $\Sigma_\beta$ is an element of $\Delta$, the proposition follows. Since $\Sigma_\beta$ is closed under predecessors, we are left with showing that all sinks of $\Sigma_\beta$ are relevant pairs.
 
 This can be verified along the algorithm computing $\Sigma_\beta$, cf.\ section \ref{subsection: computing beta-states}. The basic observation is that one can apply property (EV) only if there is a relevant triple $(v,t,s)$ with a unique neighbour $(i,j)$ for which $\ev_\beta(i,j)$ is not yet defined. With the help of Lemma \ref{lemma: properties of extremal edges} \eqref{extr2} and an inductive argument, it can be seen that the edge $\{(v,t,s),(i,j)\}$ is an extremal edge oriented towards $(i,j)$. We forgo to explain this induction in detail.

 Given such an edge, $(i,j)$ is declared $\beta$-trivial at some point of the algorithm (by step \eqref{step2} or step \eqref{step7}), only if the triple $(v,t,s)$ is also declared $\beta$-trivial (by step \eqref{step3} or step \eqref{step7}, respectively). This shows that we encounter at no time of the algorithm a sink that is a triple. Therefore all sinks of $\Sigma_\beta$ are relevant pairs, which concludes the proof.
\end{proof}

\begin{rem}
 With these results at hand, we can understand the approach of \cite{L13} as follows. Under certain conditions on the coefficient quiver, one can define a function that associates with every relevant triple a relevant pair that is adjacent to this triple, and this functions indicates a way to solve the defining equations of the Schubert system successively in linear terms. 

 It turns out that for the chosen function in \cite{L13}, every edge between a relevant triple and the associated relevant pair is extremal. The orientation that is given by orientating these extremal edges from the relevant triple to its associated relevant pair, and all other edges towards the relevant triple is an extremal solution for $\overline\Sigma$. Therefore Proposition \ref{prop: extremal solutions of the reduced schubert system} reproduces the main result of \cite{L13}.
\end{rem}

\begin{ex}\label{ex: Schubert system with extremal solution}
 The following is an example of a reduced Schubert system with an extremal solution. Let $Q$ be the equioriented quiver $\bullet\stackrel{a}\to \bullet \stackrel{b}\to \bullet$ of type $A_3$, and $M$ and $\cB$ be given by the coefficient quiver
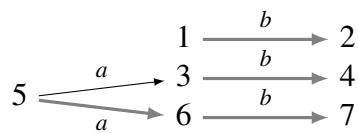
\[
 \beginpgfgraphicnamed{fig59}
  \begin{tikzpicture}[>=latex,baseline=(current bounding box.center)]
   \matrix (m) [matrix of math nodes, row sep=-0.5em, column sep=4.0em, text height=1ex, text depth=0ex]
    {        &  1  &  2    \\    
             &  \  &       \\    
             &  3  &  4    \\    
          5  &     &       \\    
             &  6  &  7    \\    
};
    \path[->,font=\scriptsize]
    (m-4-1) edge node[above] {$a$} (m-3-2);
    \path[->,very thick,gray,font=\scriptsize]
    (m-4-1) edge node[below,black] {$a$} (m-5-2)
    (m-1-2) edge node[above,black] {$b$} (m-1-3)
    (m-3-2) edge node[above,black] {$b$} (m-3-3)
    (m-5-2) edge node[above,black] {$b$} (m-5-3);
  \end{tikzpicture}
 \endpgfgraphicnamed
\]
whose extremal edges are illustrated bold and grey. The defining polynomials of the reduced Schubert system are
\begin{align*}
 \overline E(a,1,5) \ &= \ w_{1,3}+w_{1,6},   & \overline E(a,3,5) \ &= \ w_{3,6}-1,                       \\
 \overline E(b,2,3) \ &= \ w_{1,3}-w_{2,4},   & \overline E(b,2,6) \ &= \ w_{2,7}-w_{1,6}+w_{2,4}w_{3,6},  \\
 \overline E(b,4,6) \ &= \ w_{4,7}-w_{3,6}.  \\
\end{align*}
With the only exception of the non-extremal edge $\{(a,1,5),(1,3)\}$, the extremal edges of the reduced Schubert system correspond to the linear terms in the above polynomials, i.e.\ if $w_{i,j}$ appears as a linear term in $\overline E(v,t,s)$, then $\{(v,t,s),(i,j)\}$ is an extremal edge of $\overline\Sigma$. 

In the following illustrations of the reduced Schubert system $\overline\Sigma$, we draw extremal edges bold and grey. On the right hand side, we indicate an extremal solution by arrow symbols attached to the edges.
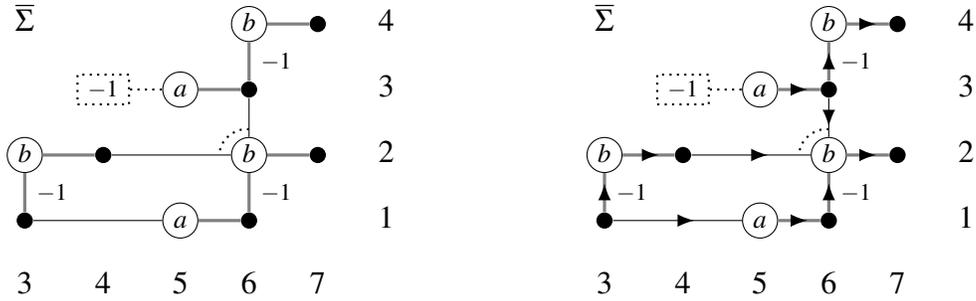
\[
 \beginpgfgraphicnamed{fig57}
  \begin{tikzpicture}[>=latex]
   \matrix (m) [matrix of math nodes, row sep=0.5em, column sep=1.0em, text height=1.5ex, text depth=0.5ex]
    {   \overline\Sigma     &        \              &                      &\node[triple](b46){b};&\node[pair](47){};    &  4  \\   
             \              &\node[const](55){-1};  &\node[triple](a35){a};&\node[pair](36){};    &       \              &  3  \\   
      \node[triple](b23){b};&\node[pair](24){};     &       \              &\node[triple](b26){b};&\node[pair](27){};    &  2  \\   
      \node[pair](13){};    &        \              &\node[triple](a15){a};&\node[pair](16){};    &       \              &  1  \\   
             3              &        4              &       5              &       6              &       7              &     \\   
};
    \path[-,very thick,font=\scriptsize,gray]
    (b23) edge node[auto,black] {$-1$} (13)
    (b23) edge (24)
    (a35) edge (36)
    (b46) edge node[auto,black] {$-1$} (36)
    (b46) edge (47)
    (a15) edge (16)
    (b26) edge node[auto,black] {$-1$} (16)
    (b26) edge (27)
;
    \path[-,font=\scriptsize]
    (a15) edge node[auto] {} node[pos=0.2] (b26-36) {} (13)
    (b26) edge node[auto] {} node[pos=0.2] (b26-36) {} (36)
    (b26) edge node[auto] {} node[pos=0.1] (b26-24) {} (24)
;
    \path[dotted,thick,font=\scriptsize]
    (b26-36.center) edge[bend right=45] node[above left=-2pt] {} (b26-24.center)
;
    \path[dotted,-,thick,font=\scriptsize]
    (55) edge node[right=-1pt] {} (a35)
;
  \end{tikzpicture}
 \endpgfgraphicnamed
\hspace{2cm}
 \beginpgfgraphicnamed{fig62}
  \begin{tikzpicture}[>=latex]
   \matrix (m) [matrix of math nodes, row sep=0.5em, column sep=1.0em, text height=1.5ex, text depth=0.5ex]
    {   \overline\Sigma     &        \              &                      &\node[triple](b46){b};&\node[pair](47){};    &  4  \\   
             \              &\node[const](55){-1};  &\node[triple](a35){a};&\node[pair](36){};    &       \              &  3  \\   
      \node[triple](b23){b};&\node[pair](24){};     &       \              &\node[triple](b26){b};&\node[pair](27){};    &  2  \\   
      \node[pair](13){};    &        \              &\node[triple](a15){a};&\node[pair](16){};    &       \              &  1  \\   
             3              &        4              &       5              &       6              &       7              &     \\   
};
    \path[-,very thick,font=\scriptsize,gray]
    (b23) edge[-<-=.7] node[auto,black] {$-1$} (13)
    (b23) edge[->-=.7] (24)
    (a35) edge[->-=.7] (36)
    (b46) edge[-<-=.7] node[auto,black] {$-1$} (36)
    (b46) edge[->-=.7] (47)
    (a15) edge[->-=.7] (16)
    (b26) edge[-<-=.7] node[auto,black] {$-1$} (16)
    (b26) edge[->-=.7] (27)
;
    \path[-,font=\scriptsize]
    (a15) edge[-<-=.5] node[auto] {} node[pos=0.2] (b26-36) {} (13)
    (b26) edge[-<-=.7] node[auto] {} node[pos=0.2] (b26-36) {} (36)
    (b26) edge[-<-=.5] node[auto] {} node[pos=0.1] (b26-24) {} (24)
;
    \path[dotted,thick,font=\scriptsize]
    (b26-36.center) edge[bend right=45] node[above left=-2pt] {} (b26-24.center)
;
    \path[dotted,-,thick,font=\scriptsize]
    (55) edge node[right=-1pt] {} (a35)
;
  \end{tikzpicture}
 \endpgfgraphicnamed
\]

We compute $\Sigma_\beta$ for $\beta=\{4,5,6,7\}$. After applying the initial steps \eqref{step1}--\eqref{step6} from section \ref{subsection: computing beta-states}, we are left with the full subsystem of $\overline\Sigma$ with horizontal coordinates in $\beta$ and vertical coordinates in $\cB-\beta$. This subsystem is illustrated below on the left hand side. After applying step \eqref{step7} to the triples $(a,1,5)$ and $(a,3,5)$, we obtain $\Sigma_\beta$, as illustrated on the right hand side.
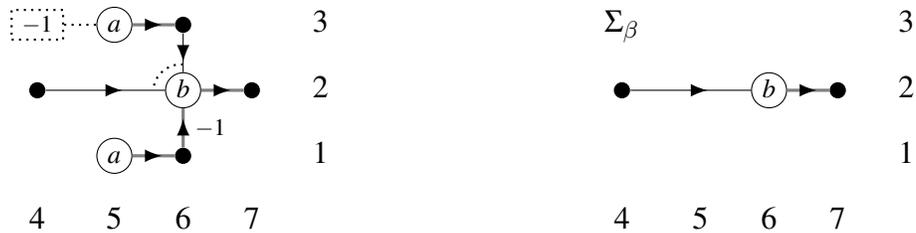
\[
 \beginpgfgraphicnamed{fig60}
  \begin{tikzpicture}[>=latex]
   \matrix (m) [matrix of math nodes, row sep=0.5em, column sep=1.0em, text height=1.5ex, text depth=0.5ex]
    { \node[const](55){-1};  &\node[triple](a35){a};&\node[pair](36){};    &       \              &  3  \\   
      \node[pair](24){};     &       \              &\node[triple](b26){b};&\node[pair](27){};    &  2  \\   
              \              &\node[triple](a15){a};&\node[pair](16){};    &       \              &  1  \\   
              4              &       5              &       6              &       7              &     \\   
};
    \path[-,very thick,font=\scriptsize,gray]
    (a35) edge[->-=.7] (36)
    (a15) edge[->-=.7] (16)
    (b26) edge[-<-=.7] node[auto,black] {$-1$} (16)
    (b26) edge[->-=.7] (27)
;
    \path[-,font=\scriptsize]
    (b26) edge[-<-=.7] node[auto] {} node[pos=0.2] (b26-36) {} (36)
    (b26) edge[-<-=.5] node[auto] {} node[pos=0.1] (b26-24) {} (24)
;
    \path[dotted,thick,font=\scriptsize]
    (b26-36.center) edge[bend right=45] node[above left=-2pt] {} (b26-24.center)
;
    \path[dotted,-,thick,font=\scriptsize]
    (55) edge node[right=-1pt] {} (a35)
;
  \end{tikzpicture}
 \endpgfgraphicnamed
\hspace{3cm}
 \beginpgfgraphicnamed{fig61}
  \begin{tikzpicture}[>=latex]
   \matrix (m) [matrix of math nodes, row sep=0.5em, column sep=1.0em, text height=1.5ex, text depth=0.5ex]
    {     \Sigma_\beta       &        \             &        \             &       \              &  3  \\   
      \node[pair](24){};     &       \              &\node[triple](b26){b};& \node[pair](27){};   &  2  \\   
              \              &        \             &        \             &       \              &  1  \\   
              4              &       5              &       6              &       7              &     \\   
};
    \path[-,very thick,font=\scriptsize,gray]
    (b26) edge[->-=.7] (27)
;
    \path[-,font=\scriptsize]
    (b26) edge[-<-=.5] node[auto] {} node[pos=0.1] (b26-24) {} (24)
;
  \end{tikzpicture}
 \endpgfgraphicnamed
\]
 In accordance with Proposition \ref{prop: extremal solutions of the reduced schubert system}, the restriction of the extremal solution to $\Sigma_\beta$ yields a solution for $\Sigma_\beta$, which shows that $C_\beta^M\simeq\A^1$. Note, however, that the extremal edge $\{(b,2,6),(1,6)\}$ is $\beta$-trivial though $(b,2,6)$ is $\beta$-relevant.
\end{ex}

\subsection{Patchwork solutions}\label{subsection: patchwork solutions}

In this section, we present a simple, but effective method to reduce the complexity of the Schubert system and its $\beta$-states to smaller parts, or patches, in order to find a solution.

Let $\Xi$ be a system and $\lambda=\{\tau,S\}$ a link of $\Xi$. The \emph{support of $\lambda$} is the subgraph of $\Xi$ with vertex set $S\cup\{\tau\}$ and edge set $\{\{\tau,\pi\}|\pi\in S\}$. A \emph{subsystem of $\Xi$} is a subgraph $\Gamma$ together with a subset of links with support in $\Gamma$. A link $\lambda$ of the subsystem has the same weight $\mu_\lambda$ as in $\Xi$. A subsystem $\Gamma$ of $\Xi$ is \emph{full} if it contains all edges with end vertices in $\Gamma$ and all links supported in $\Gamma$. Note that a full subsystem of $\Xi$ is determined by its vertex set.

\begin{df}
 A \emph{patch of $\Xi$} is a full subsystem $\Xi'$ of $\Xi$ that contains the base vertices of all links of $\Xi$ with a leg in $\Xi'$. A \emph{patchwork for $\Xi$} is a family $\{\Xi_k\}_{k\in I}$ of patches $\Xi_k$ of $\Xi$, indexed by a partially ordered set $I$, that satisfies the following properties:
 \begin{enumerate}
  \item[(P1)]\label{patch1} $\Vertex \Xi=\bigcup_{k\in I} \Vertex \Xi_k$;
  \item[(P2)]\label{patch2} $\Vertex\Xi_k\cap\Vertex \Xi_l$ does not contain any triple for $k\neq l$;
  \item[(P3)]\label{patch3} if $\tau$ is a triple in $\Xi_l$ and $\{\tau,\pi\}$ is an edge of $\Xi$, then there exists a $k\leq l$ such that $\pi$ is a pair in $\Xi_k$.
 \end{enumerate}
 A \emph{patchwork solution for $\Xi$} is a patchwork $\{\Xi_k\}$ for $\Xi$ together with a solution for each patch $\Xi_k$ that satisfies the following condition.
 \begin{enumerate}
  \item[(PS)]\label{patchworksolution} If $\pi$ is a pair in $\Xi_k\cap\Xi_l$ for some $l\in I$ and $\{\tau,\pi\}$ is an edge in $\Xi_k$ that is oriented towards $\pi$, then $k\leq l$.
 \end{enumerate}
\end{df}

In many cases, we can work with patches with empty intersections, i.e.\ $\Vertex\Xi=\coprod \Vertex\Xi_k$. In the proof of Theorem \ref{thm: the main theorem}, we make, however, use of patches that have certain relevant pairs in common (see Example \ref{ex: reduced schubert system of type d_n-tilde and its patchwork} and Figure \ref{fig34} for a concrete example). Note that two distinct patches do never have any edge in common since the intersection of their vertex sets contains only pairs. In other words, the canonical map $\coprod\Edge\Xi_k\to\Edge\Xi$ is an inclusion. 

\begin{prop}\label{prop: patchwork solutions}
 If $\Xi$ has a patchwork solution, then it is solvable.
\end{prop}

\begin{proof}
 Given a patchwork $\{\Xi_k\}$ for $\Xi$ together with solutions for each patch $\Xi_k$, we extend the orientation of the edges in $\coprod\Edge\Xi_k$ to $\Edge\Xi$ as follows. If $\{\tau,\pi\}$ is an edge of $\Xi$ that is not contained in any patch, then we orientate this edge towards the triple $\tau$.
 
 Since every triple is contained in some patch $\Xi_k$, property (S1) for the solution of $\Xi_k$ implies (S1) for the orientation of $\Xi$. 
 
 Axiom (PS) and the chosen orientation for all edges that are not contained in some patch imply the following property: if an edge $\{\tau,\pi\}$ of $\Xi$ is oriented towards $\pi$ where $\tau$ is in $\Xi_k$ and $\pi$ is in $\Xi_l$, then $k\leq l$. Therefore, there is only one patch $\Xi_k$ for a given pair $\pi$ that contains an edge oriented towards $\pi$. Since (S2) holds for each patch, (S2) holds for $\Xi$.
  
 We show property (S3) by contradiction. Assume $\Xi$ has an oriented cycle. By (S3) for the solution of $\Xi_k$, this cycle cannot be contained in a patch $\Xi_k$. Therefore the cycle contains an edge $\{\tau,\pi\}$ that is not contained in any patch. By the definition of the orientation, this edge is oriented towards the triple $\tau$, which is contained in some patch $\Xi_{k_0}$. By property (P3) of a patchwork, there is a $k_1\leq k_0$ such that $\pi$ is a vertex in $\Xi_{k_1}$. Since the edge in question is not contained in a $\Xi_{k_0}$, we see that $k_1<k_0$. Using this argument for every edge of the cycle that is not contained in any patch, we yield a sequence $k_0<k_n<\dotsb<k_1<k_0$, which is not possible.

 This shows that the defined orientation satisfies property (S3) of a solution and that $\Xi$ is solvable.
\end{proof}

\begin{cor} \label{cor: patchwork solution implies affine space}
 Assume $\Sigma_\beta$ is not contradictory. If $\Sigma_\beta$ has a patchwork solution, then $C_\beta^M$ is an affine space.
\end{cor}

\begin{proof}
 This follows immediately from Proposition \ref{prop: patchwork solutions} and Theorem \ref{thm: solvable beta-state implies affine schubert cell}.
\end{proof}

\begin{cor}\label{cor: extremal patchwork solutions for the reduced schubert system}
 Let $\{\Xi_k\}$ be a patchwork for $\overline\Sigma$. If every patch $\Xi_k$ has an extremal solution, then $\overline\Sigma$ is totally solvable.
\end{cor}

\begin{proof}
 The same proof as for Proposition \ref{prop: patchwork solutions} shows that $\overline\Sigma$ has an extremal solution. By Proposition \ref{prop: extremal solutions of the reduced schubert system}, $\overline\Sigma$ is totally solvable.
\end{proof}

\subsection{Extremal paths}\label{subsection: extremal paths}

Let $\Xi$ be a system.

\begin{df}
 An \emph{extremal path in $\Xi$} is a patch $\Pi$ in $\Xi$ of the form
 \[
   \beginpgfgraphicnamed{fig65}
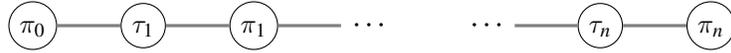
  \begin{tikzpicture}[>=latex]
   \matrix (m) [matrix of math nodes, row sep=0.5em, column sep=2.0em, text height=1.5ex, text depth=0.5ex]
    { \node[triple](p0){\pi_0}; &\node[triple](t1){\tau_1}; & \node[triple](p1){\pi_1};   & \node(dots1){\dotsb};  & \node(dots2){\dotsb};  & \node[triple](tn){\tau_n}; & \node[triple](pn){\pi_n};   \\   
};
    \path[-,very thick,font=\scriptsize,gray]
    (p0) edge (t1)
    (p1) edge (t1)
    (p1) edge (dots1)
    (tn) edge (dots2)
    (tn) edge (pn)
;
  \end{tikzpicture}
 \endpgfgraphicnamed
 \]
 and satisfies the following properties.
 \begin{enumerate}
  \item[(EP1)] The vertices $\pi_0,\dotsc,\pi_n$ are pairs and $\tau_1,\dotsc,\tau_n$ are triples.
  \item[(EP2)] All edges $\{\pi_i,\tau_j\}$ of $\Pi$ are extremal.
 \end{enumerate}
 An extremal path $\Pi$ in $\Xi$ is \emph{pure} if it satisfies the following additional property.
 \begin{enumerate}
  \item[(EP3)] All edges of $\Xi$ that connect to one of the vertices $\tau_1,\dotsc,\tau_n$ are those contained in $\Pi$.
 \end{enumerate}
 The \emph{contraction $\Xi/\Pi$ of $\Xi$ along an extremal path $\Pi$} is the system that results from identifying the vertices $\pi_0,\dotsc\pi_n$ and removing $\tau_1,\dotsc,\tau_n$ and all edges and links of $\Pi$ from $\Xi$. In particular, we identify $\pi_0,\dotsc,\pi_n$ in all the leftover edges and links of $\Xi$ for $i=1,\dotsc,n$.
\end{df}

Let $M$ be a representation of $Q$ with ordered basis $\cB$ and $\overline\Sigma$ its reduced Schubert system. Let $\Xi$ be a patch of $\overline\Sigma$. Note that an extremal path $\Pi$ in $\Xi$ is also an extremal path in $\overline\Sigma$, though $\Pi$ might be pure in $\Xi$, but not pure in $\overline\Sigma$.

For $\beta\subset\cB$, we define the \emph{$\beta$-state $\Xi_{\beta}$ of $\Xi$} to be the full subsystem of the $\beta$-state $\Sigma_\beta$ of $\overline \Sigma$ whose vertex set consists of all $\beta$-relevant vertices of $\Xi$. 

The \emph{contraction $\Xi_{\beta}/\Pi$ of $\Xi_{\beta}$ along $\Pi$} is the system that results from removing the $\beta$-relevant vertices among $\pi_1,\dotsc\pi_n,\tau_1,\dotsc,\tau_n$ and all $\beta$-relevant edges and links of $\Pi$ from $\Xi_{\beta}$ and replacing $\pi_i$ by $\pi_0$ in all the leftover edges and links of $\Xi_\beta$ for $i=1,\dotsc,n$. 

Let $\{\Xi_k\}$ be a patchwork for $\overline\Sigma$. Then $\{\Xi_{k,\beta}\}$ is a patchwork for $\Sigma_\beta$. For an index $l$ and an extremal path $\Pi$ in $\Xi_l$, we define $\Xi_k/\Pi$ as the contraction $\Xi_l/\Pi$ if $k=l$, and as $\Xi_k$ with all pairs in $\Xi_k\cap\Pi$ identified if $k\neq l$. Then $\{\Xi_k/\Pi\}$ is a patchwork for $\overline\Sigma/\Pi$.
 
\begin{prop}\label{prop:contracting extremal paths}
 Let $\{\Xi_k\}$ be a patchwork for $\overline\Sigma$ and $\Pi$ a pure extremal path in $\Xi_l$ for some index $l$.
 \begin{enumerate}
  \item\label{extpath1} The $\beta$-state $\Xi_{k,\beta}$ is solvable if $\Xi_{k,\beta}/\Pi$ is solvable. 
  \item\label{extpath2} The patch $\Xi_k$ has an extremal solution if $\Xi_k/\Pi$ has an extremal solution.
  \item\label{extpath3} The patch $\Xi_k$ is totally solvable if $\Xi_k/\Pi$ is totally solvable.
  \item\label{extpath4} If $\{\Xi_k/\Pi\}$ \big(or $\{\Xi_{k,\beta}/\Pi\}$\big) has a patchwork solution, then $\{\Xi_k\}$ \big(or $\{\Xi_{k,\beta}\}$, respectively\big) has a patchwork solution.
 \end{enumerate}
\end{prop}

\begin{proof}
 For the proof of \eqref{extpath1}--\eqref{extpath3}, we can assume that $k=l$ since the claim is trivial otherwise. A solution of $\Xi_{l,\beta}/\Pi$ can be extended to a solution of $\Xi_{l,\beta}$ in the following way. Assume there is an edge in $\Xi_{l,\beta}/\Pi$ that is oriented towards $\pi_0$ and this edge comes from an edge connecting to $\pi_i$ in $\Xi_{l,\beta}$. Then we orientate all $\beta$-relevant edges of $\Pi$ away from $\pi_i$. If there is no edge in $\Xi_{l,\beta}/\Pi$ that is oriented towards $\pi_0$, then we can choose an arbitrary vertex $\pi_i$ of $\Pi$ as a source and orientate all the other arrows away from $\pi_i$. Note that this orientation does not have cycles thanks to axiom (EP3) of a pure extremal path. Therefore this defines a solution for $\Xi_{l,\beta}$, which verifies \eqref{extpath1}.
 
 Since all edges of $\Pi$ are extremal, an extremal solution of $\Xi_{l}/\Pi$ extends to an extremal solution of $\Xi_l$ by the above argument, which shows \eqref{extpath2}. The same argument verifies \eqref{extpath3}. 
 
 A patchwork solution $\{\Xi_k/\Pi\}$ satisfies (PS) for every patch. If thus an edge $\{\tau,\pi_0\}$ in $\Xi_k$ is oriented towards $\pi_0$, then by (PS), $k\leq k'$ for $k'$ such that $\pi_0$ is in $\Xi_{k'}$. This means that the extension of the solution for $\Xi_l/\Pi$ to $\Xi_l$ satisfies (PS) as well (for all $k$ and $l$). The proof for $\{\Xi_{k,\beta}\}$ is analogous. This shows \eqref{extpath4}.
\end{proof}

\begin{cor} \label{cor: extremal solution for extremal paths}
 If a patch $\Xi$ is an extremal path, then it has an extremal solution.
\end{cor}

\begin{proof}
 If $\Xi$ is an extremal path, the contraction $\Xi/\Xi$ is relevant pair and does not have any edges. Therefore the empty orientation is an extremal solution for $\Xi/\Xi$, and by Proposition \ref{prop:contracting extremal paths} \eqref{extpath2} this extends to an extremal solution of $\Xi$.
\end{proof}


\section{First applications}
\label{section: first applications}

\noindent
In this section, we demonstrate the methods developed in section \ref{section: schubert systems} for the Kronecker quiver and quivers of Dynkin types $A_n$ and $D_n$.

\subsection{The Kronecker quiver}\label{subsection: kronecker quiver}

In this section, we reprove the following known result for the Kronecker quiver. For alternative proofs, see \cite{cr} and \cite{L13}.

\begin{prop}\label{prop: schubert decompositions for the kronecker quiver}
 Every exceptional representation $M$ of the Kronecker quiver $Q$ has an ordered basis $\cB$ such that the associated Schubert decomposition $\Gr_\ue(M)=\coprod C_\beta^M$ is a decomposition into affine spaces. 
\end{prop}

\begin{proof}
 An exceptional representation $M$ of the Kronecker quiver $Q$ is either preprojective or preinjective. Since both cases can be proven analogously, we restrict ourselves to a demonstration for preprojective representations.

 Let $a$ and $b$ be the arrows of $Q$. Then $M$ has an ordered basis $\cB=\{1,\dotsc,2n+1\}$ such that the associated coefficient quiver $\Gamma$ is 
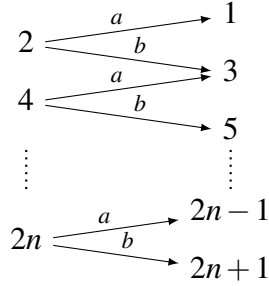
\[
 \beginpgfgraphicnamed{fig63}
  \begin{tikzpicture}[>=latex,baseline=(current bounding box.center)]
   \matrix (m) [matrix of math nodes, row sep=-0.2em, column sep=4.0em, text height=1ex, text depth=0ex]
    {        &  1   \\    
          2  &      \\    
             &  3   \\    
          4  &      \\    
          \  &  5   \\    
          \  &      \\    
             &      \\    
          \  &      \\    
          \  & 2n-1 \\    
         2n  &      \\    
             & 2n+1 \\    
};
    \path[->,font=\scriptsize]
    (m-2-1) edge node[above left=-2pt] {$a$} (m-1-2)
    (m-4-1) edge node[above left=-2pt] {$a$} (m-3-2)
    (m-10-1) edge node[above left=-2pt] {$a$} (m-9-2)
    (m-2-1) edge node[above right=-2pt] {$b$} (m-3-2)
    (m-4-1) edge node[above right=-2pt] {$b$} (m-5-2)
    (m-10-1) edge node[above right=-2pt] {$b$} (m-11-2)
;
    \path[dotted,thick]
    (m-5-1) edge (m-9-1)
    (m-5-2) edge (m-9-2)
;
  \end{tikzpicture}
 \endpgfgraphicnamed
\]
 The defining equations for the reduced Schubert system $\overline\Sigma$ are of the form
\begin{align*}
 \overline E(a,t,s) \ &= \ w_{t,s-1}-w_{t+1,s}+\sum_{\substack{i=t+3,\dotsc,s-2\text{ even}}} w_{t,i-1}w_{i,s} && (s\geq t+3)  \\
 \overline E(b,t,s) \ &= \ w_{t,s+1}\hspace{41pt}          +\sum_{\substack{i=t+1,\dotsc,s-2\text{ even}}} w_{t,i+1}w_{i,s} && (s\geq t+1\text{ and }t=1)  \\
 \overline E(b,t,s) \ &= \ w_{t,s+1}-w_{t-1,s}+\sum_{\substack{i=t+1,\dotsc,s-2\text{ even}}} w_{t,i+1}w_{i,s} && (s\geq t+1\text{ and }t\geq 3)
\end{align*}
 for $t\in\cB$ odd and $s\in\cB$ even. Note that all arrows of $\Gamma$ are extremal and therefore all linear terms that occur in the above equations correspond to an extremal edge of the Schubert system. For $I=\{2,4,\dotsc,2n\}$, we define the patchwork $\{\Xi_k\}_{k\in I}$ where $\Xi_k$ is the patch
\[
 \beginpgfgraphicnamed{fig64}
  \begin{tikzpicture}[>=latex]
   \matrix (m) [matrix of math nodes, row sep=1em, column sep=2em, text height=1.5ex, text depth=0.5ex]
    { \node[vertex](b-1-k){b,1,k}; & \node[vertex](1-k+1){1,k+1}; & \node[vertex](a-1-k+2){a,1,k+2}; & \node[vertex](2-k+2){2,k+2}; & \node[vertex](b-3-k+2){b,3,k+2}; & \dotsb & \node[vertex](2n+1-k-2n+1){2n+1-k,2n+1}; 
\\ };
    \path[-,very thick,font=\scriptsize,gray]
    (b-1-k) edge (1-k+1)
    (a-1-k+2) edge (1-k+1)
    (a-1-k+2) edge (2-k+2)
    (b-3-k+2) edge (2-k+2)
    (b-3-k+2) edge (m-1-6)
    (m-1-6) edge (2n+1-k-2n+1)
;
    \path[dotted,very thick,font=\scriptsize,gray]
;
  \end{tikzpicture}
 \endpgfgraphicnamed
\]

 The whole patch $\Xi_k$ is an extremal path, which we can contract to a point $\Xi_k/\Xi_k$. By Corollary \ref{cor: extremal solution for extremal paths}, the patches $\Xi_k$ have extremal solutions, and by Corollary \ref{cor: extremal patchwork solutions for the reduced schubert system}, the reduced Schubert system $\overline\Sigma$ is totally solvable. This shows that the Schubert decomposition $\Gr_\ue(M)=\coprod C_\beta^M$ is a decomposition into affine spaces for all $\ue$.
\end{proof}

\subsection{Dynkin quivers}\label{subsection: dynkin quivers}

Every indecomposable representation of a Dynkin quiver $Q$ of type $A_n$ is thin. Therefore Corollary \ref{cor: schubert decompositions for thin representations} implies that $\Gr_\ue(M)=\coprod C_\beta^M$ is a decomposition into affine spaces for any indecomposable representation $M$ of $Q$, any ordered basis $\cB$ of $M$ and any dimension vector $\ue$.

Let $Q$ be a Dynkin quiver of type $D_n$ and $M$ an exceptional representation of $Q$. If $M$ is thin, then we can apply Corollary \ref{cor: schubert decompositions for thin representations} to establish a Schubert decomposition of $\Gr_\ue(M)$. If $M$ is not thin, then it has an ordered basis $\cB$ such that the coefficient quiver, together with the canonical map $F:\Gamma\to Q$, looks like
 \[
   \beginpgfgraphicnamed{fig66}
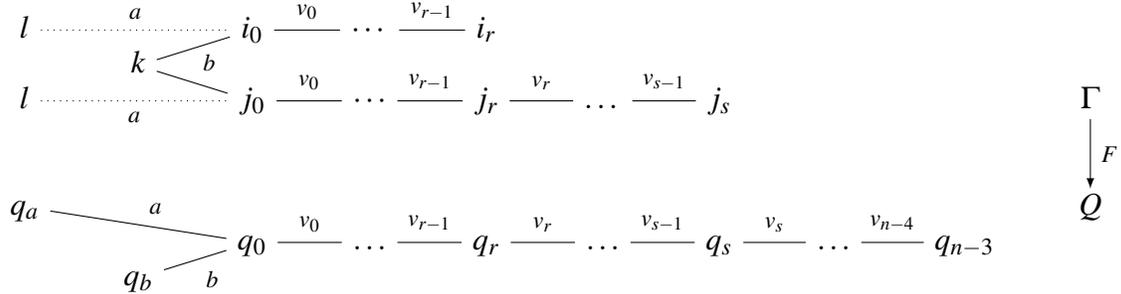
   \begin{tikzpicture}[>=latex]
  \matrix (m) [matrix of math nodes, row sep=0em, column sep=2em, text height=1ex, text depth=0ex]
   {   l &    & i_0 &\dotsb& i_r     \\
          & k   &     &      &     &       \\
       l &    & j_0 &\dotsb& j_r &  \dotsc & j_s  &        &         &  \Gamma \\
       \ & \\
       \ & \\
       q_a &  &     &        &     &        &     &        &         &  Q\\
         &    & q_0 & \dotsc & q_r & \dotsc & q_s & \dotsc & q_{n-3}  \\
         & q_b \\
       };
         \path[-,font=\scriptsize]
         (m-2-2) edge node[auto] {} (m-1-3)
         (m-2-2) edge node[auto] {$b$} (m-3-3)
         (m-1-3) edge node[auto] {$v_0$} (m-1-4)
         (m-3-3) edge node[auto] {$v_0$} (m-3-4)
         (m-1-4) edge node[auto] {$v_{r-1}$} (m-1-5)
         (m-3-4) edge node[auto] {$v_{r-1}$} (m-3-5)
         (m-3-5) edge node[auto] {$v_{r}$} (m-3-6)
         (m-3-6) edge node[auto] {$v_{s-1}$} (m-3-7)
         (m-6-1) edge node[auto] {$a$} (m-7-3)
         (m-8-2) edge node[auto,swap] {$b$} (m-7-3)
         (m-7-3) edge node[auto] {$v_0$} (m-7-4)
         (m-7-4) edge node[auto] {$v_{r-1}$} (m-7-5)
         (m-7-5) edge node[auto] {$v_{r}$} (m-7-6)
         (m-7-6) edge node[auto] {$v_{s-1}$} (m-7-7)
         (m-7-7) edge node[auto] {$v_{s}$} (m-7-8)
         (m-7-8) edge node[auto] {$v_{n-4}$} (m-7-9)
         ;
         \path[-,dotted,font=\scriptsize]
         (m-1-1) edge node[auto] {$a$} (m-1-3)
         (m-3-1) edge node[auto,swap] {$a$} (m-3-3)
         ;
         \path[->,font=\scriptsize]
         (m-3-10) edge node[auto] {$F$} (m-6-10)
         ;
        \end{tikzpicture}
        \endpgfgraphicnamed
 \]
 where $i_k<j_k$ for all $k=0,\dotsc,r$. The arrow of $\Gamma$ with label $a$ can either connect to $i_0$ or $j_0$, and each arrow of $Q$ can have either orientation. Note that not every coefficient quiver of this form defines an exceptional representation. For instance, the representation is decomposable if $r=s$. Moreover, indecomposability depends on the orientation of the edges. The reduced Schubert system $\overline \Sigma$ of $M$ w.r.t.\ $\cB$ is of the form
 \[
   \beginpgfgraphicnamed{fig67}
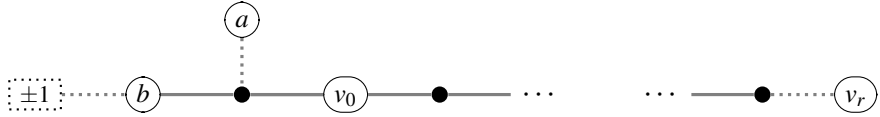
  \begin{tikzpicture}[>=latex]
   \matrix (m) [matrix of math nodes, row sep=1em, column sep=2em, text height=1.5ex, text depth=0.5ex]
    {                         &                      & \node[vertex](a){a}; \\ 
     \node[const](const){\pm 1}; & \node[vertex](b){b}; & \node[pair](p0){}; & \node[vertex](v0){v_0}; & \node[pair](p1){}; & \dotsb & \dotsb & \node[pair](pk){}; & \node[vertex](vr){v_{r}}; 
\\ };
    \path[-,very thick,font=\scriptsize,gray]
    (b) edge (p0)
    (v0) edge (p0)
    (v0) edge (p0)
    (v0) edge (p1)
    (m-2-6) edge (p1)
    (m-2-7) edge (pk)
;
    \path[dotted,very thick,font=\scriptsize,gray]
    (pk) edge (vr)
    (a) edge (p0)
    (const) edge (b)
;
  \end{tikzpicture}
 \endpgfgraphicnamed
 \]
 where we omit the coordinates and the weights $\pm 1$ of the edges, which both depend on the orientation of the arrows. Note that $\overline\Sigma$ contains the triples with labels $a$ and $v_{r}$ depending on the orientation of the arrows $a$ and $v_r$ in $Q$, and depending on whether the arrow of $\Gamma$ with label $a$ connects to $i_0$ or $j_0$. 
 
 We will show that each Schubert cell $C_\beta^M$ is either an affine space or empty. Since $\Sigma_\beta$ is contradictory if $\beta$ is not extremal successor closed, we can assume that $\beta$ is closed under extremal successors. If $\Sigma_\beta$ contains 
 \[
   \beginpgfgraphicnamed{fig68}
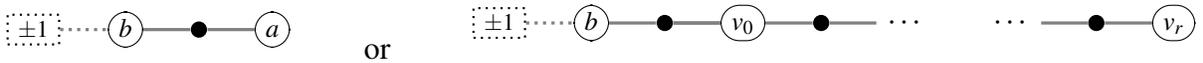
  \begin{tikzpicture}[>=latex]
   \matrix (m) [matrix of math nodes, row sep=1em, column sep=1.5em, text height=1.5ex, text depth=0.5ex]
    {  \node[const](const){\pm 1}; & \node[vertex](b){b}; & \node[pair](p0){}; & \node[vertex](a){a}; \\
};
    \path[-,very thick,font=\scriptsize,gray]
    (b) edge (p0)
    (a) edge (p0)
;
    \path[dotted,very thick,font=\scriptsize,gray]
    (const) edge (b)
;
  \end{tikzpicture}
 \endpgfgraphicnamed
 \qquad \text{or} \qquad
   \beginpgfgraphicnamed{fig69}
  \begin{tikzpicture}[>=latex]
   \matrix (m) [matrix of math nodes, row sep=1em, column sep=1.5em, text height=1.5ex, text depth=0.5ex]
    {\node[const](const){\pm 1}; & \node[vertex](b){b}; & \node[pair](p0){}; & \node[vertex](v0){v_0}; & \node[pair](p1){}; & \dotsb & \dotsb & \node[pair](pk){}; & \node[vertex](vr){v_{r}}; \\
 };
    \path[-,very thick,font=\scriptsize,gray]
    (b) edge (p0)
    (v0) edge (p0)
    (v0) edge (p0)
    (v0) edge (p1)
    (m-1-6) edge (p1)
    (m-1-7) edge (pk)
    (pk) edge (vr)
;
    \path[dotted,very thick,font=\scriptsize,gray]
    (const) edge (b)
;
  \end{tikzpicture}
 \endpgfgraphicnamed
\] 
 as a subsystem (or both), then we can apply step \eqref{step8} of the algorithm to compute $\Sigma_\beta$ to each $\beta$-relevant triple and see that $\Sigma_\beta$ is contradictory. If $\Sigma_\beta$ equals
\[ 
 \beginpgfgraphicnamed{fig70}
  \begin{tikzpicture}[>=latex]
   \matrix (m) [matrix of math nodes, row sep=1em, column sep=1.5em, text height=1.5ex, text depth=0.5ex]
    {   \node[vertex](a){a}; & \node[pair](p0){}; & \node[vertex](v0){v_0}; & \node[pair](p1){}; & \dotsb & \dotsb & \node[pair](pk){}; & \node[vertex](vr){v_{r}}; \\
 };
    \path[-,very thick,font=\scriptsize,gray]
    (a) edge (p0)
    (v0) edge (p0)
    (v0) edge (p0)
    (v0) edge (p1)
    (m-1-5) edge (p1)
    (m-1-6) edge (pk)
    (pk) edge (vr)
;
  \end{tikzpicture}
 \endpgfgraphicnamed
\] 
 then the corresponding Schubert cell $C_\beta^M$ consists of one point with coordinates $w_{i_0,j_0}=\dotsb=w_{i_r,j_r}=0$. Any other $\beta$-state $\Sigma_\beta$ does not contain a path connecting two of the outer triples $b$, $a$ and $v_r$. Since all edges of $\overline \Sigma$ are extremal, $\Sigma_\beta$ is solvable along extremal edges. This shows that $\Gr_\ue(M)=\coprod C_\beta^M$ is a decomposition into affine spaces.

\begin{rem}
 Note that in the cases of contradictory $\beta$-states (as illustrated above), $\beta$ is \emph{contradictory of the second kind}, cf.\ section \ref{subsection: contradictory beta}. In the last illustration of a $\beta$-state $\Sigma_\beta$, we see a first example of a $\beta$-state that is not solvable, but whose Schubert cell is an affine space. In the situation of this example, an inversion of the ordering (i.e. $i_k>j_k$ for $k=0,\dotsc,r$) makes this situation disappear, and $\overline \Sigma$ becomes totally solvable. Note that for representations of a quiver of extended Dynkin type $\widetilde D_n$, we will face similar situations, which we cannot avoid by a re-ordering of the basis elements.
\end{rem}


\section{Schubert decompositions for type $\widetilde D_n$}
\label{section: Schubert decompositions for type D_n}

\noindent
Let $Q$ be a quiver of extended Dynkin type $\widetilde D_n$ and $M$ an indecomposable representation that is of defect $-1$ or defect $0$, with exclusion of the non-Schurian representations in the homogeneous tubes. In this section, we will exhibit a certain ordered basis $\cB$ for which the Schubert decompositions $\Gr_\ue(M)=\coprod C_\beta^M$ is a decomposition into affine spaces, and we give a precise characterisation of the empty cells $C_\beta^M$ in terms of the combinatorics of the coefficient quiver $\Gamma$ of $M$ w.r.t.\ $\cB$.


\subsection{Contradictory $\beta$ of the first and of the second kind}\label{subsection: contradictory beta}
Let $M$ be a representation of $Q$ with ordered basis $\cB$. Let $F:\Gamma\to Q$ be the associated map from the coefficient quiver $\Gamma=\Gamma(M,\cB)$ to $Q$. We say that $\beta$ is \emph{contradictory of the first kind} if it is not extremal successor closed.

The subset $\beta$ of $\Gamma_0=\cB$ is \emph{contradictory of the second kind} if it satisfies the following conditions.
\begin{enumerate}
 \item\label{2nd1} $\beta$ is not contradictory of the first kind.
 \item\label{2nd2} There is a subgraph $\Gamma'$ of $\Gamma$ of the form
       \[
   \beginpgfgraphicnamed{fig71}
   \begin{tikzpicture}[>=latex]
  \matrix (m) [matrix of math nodes, row sep=0em, column sep=4em, text height=1ex, text depth=0ex]
   {      & i_0 &\dotsb& i_s &       \\
       k  &     &      &     &  l    \\
          & j_0 &\dotsb& j_s &       \\};
         \path[-,font=\scriptsize]
         (m-2-1) edge node[auto] {$x,\mu_0$} (m-1-2)
         (m-2-1) edge node[auto,swap] {$x,\mu_1$} (m-3-2)
         (m-1-2) edge node[auto] {$z_0$} (m-1-3)
         (m-3-2) edge node[auto] {$z_0$} (m-3-3)
         (m-1-3) edge node[auto] {$z_{s-1}$} (m-1-4)
         (m-3-3) edge node[auto] {$z_{s-1}$} (m-3-4)
         (m-2-5) edge node[auto,swap] {$y,\nu_0$} (m-1-4)
         (m-2-5) edge node[auto] {$y,\nu_1$} (m-3-4);
        \end{tikzpicture}
        \endpgfgraphicnamed
       \]
       where $i_e<j_e$ for $e=0,\dotsc,s$, the arrows $x,y,z_0,\dotsc,z_{s-1}\in Q_1$ are pairwise distinct and of arbitrary orientation, one of the weights $\mu_0,\mu_1,\nu_0,\nu_1\in\C$ is allowed to be zero, which means that the corresponding arrow is not part of $\Gamma'$.
 \item\label{2nd3} If both $k$ and $l$ are sinks or both are sources of $\Gamma'$, then $\mu_0\nu_1\neq\mu_1\nu_0$. If one of $k$ and $l$ is a sink and the other vertex is a source, then $\mu_0\nu_0\neq-\mu_1\nu_1$.
 \item\label{2nd4} $i_0,\dotsc,i_s\notin\beta$, $j_0,\dotsc,j_s\in\beta$; we have $k\in\beta$ if and only if $k$ is a source in $\Gamma'$; we have $l\in\beta$ if and only if $l$ is a source in $\Gamma'$.
 \item\label{2nd5} If $(v,s,t)$ is an arrow of $\Gamma$ that is not contained in $\Gamma'$ with $v\in F(\Gamma')_1$, $F(s)=F(j)$ and $F(t)=F(i)$ where $j\in\{k,j_0,\dotsc,j_r,l\}$ and $i\in\{k,i_0,\dotsc,i_r,l\}$, then $s<j$ or $i<t$.
\end{enumerate}

\begin{rem}\label{rem: second contradictory beta}
 Since $\beta$ is not contradictory of the first kind, $\mu_0\neq 0$ if $k$ is a source and $\mu_1\neq 0$ if $k$ is a sink. Similarly, $\nu_0\neq 0$ if $l$ is a source and $\nu_1\neq 0$ if $l$ is a sink. 
 
 With our convention of drawing the coefficient quiver according to the ordering of its vertices, \eqref{2nd5} says that there is no arrow of $\Gamma$ that lies between the arrows of $\Gamma'$. This ensures that the defining equations of the Schubert cell $C_\beta^M$ that come from $\Gamma'$ are completely determined by $\Gamma'$. 
\end{rem}

We encourage the reader to have a look into section \ref{subsection: coefficient quiver} where we describe the second kind contradictory subsets $\beta$ of $\cB$ in some examples.

\begin{prop}\label{prop: contradictory beta of first and second kind have contradictory beta-states}
 Let $\Sigma$ be the Schubert system of $M$ w.r.t.\ $\cB$. If $\beta$ is contradictory of the first or of the second kind, then the $\beta$-state $\Sigma_\beta$ is contradictory.
\end{prop}

\begin{proof}
 The claim for contradictory $\beta$ of the first kind is Lemma \ref{lemma: non-extremal successor closed beta have contradictory beta-states}.

 Assume that $\beta$ is contradictory of the second kind. Then there is a subgraph $\Gamma'$ of $\Gamma$ satisfying properties \eqref{2nd2}--\eqref{2nd5} above. For a relevant triple $(v,t,s)$ such that $(v,s,t)$ is not an extremal arrow of $\Gamma$ and such that $t\notin\beta$ and $s\in\beta$, we have
 \[
  \overline{E}(v,t,s) \quad = \hspace{-0pt} 
        \sum_{\substack{(v,s,t')\in \Gamma_1\\ t<t'}} \hspace{-5pt} \mu_{v,s,t'} w_{t,t'} \ \ + \hspace{-5pt}
        \sum_{\substack{(v,s',t')\in\Gamma_1\\ t<t'\text{ and }s'<s}}  \hspace{-5pt} \mu_{v,s',t'} w_{t,t'} w_{s',s} \ \ - \hspace{-5pt}
        \sum_{\substack{(v,s',t)\in \Gamma_1\\ s'< s}}   \mu_{v,s',t}w_{s',s} \ \ - \ \ 
        \mu_{v,s,t}
 \]
 where $\mu_{v,s,t}=0$ if $\Gamma$ does not contain the arrow $(v,s,t)$. If $v$ is an arrow in $F(\Gamma')$ and $s,t,$ are vertices of $\Gamma'$, then property \eqref{2nd5} guarantees that no quadratic term occurs in $\overline E(v,t,s)$ and that all indices of the non-trivial variables $w_{i,j}$ of $\overline E(v,t,s)$ are vertices of $\Gamma'$.
 
 Property \eqref{2nd4} determines which vertices of $\Gamma'$ are in $\beta$ and which not. This yields
 \begin{align*}
   \overline E(x,k,j_0)           \ &= \ -\mu_0 w_{i_0,j_0}-\mu_1,                  & \overline E(x,i_0,k)         \ &= \ \mu_1w_{i_0,j_0}-\mu_0,                                       \\
   \overline E(z_r,i_{r},j_{r+1}) \ &= \ w_{i_r,j_r} - w_{i_{r+1},j_{r+1}},         & \overline E(z_r,i_{r+1},j_r) \ &= \ w_{i_{r+1},j_{r+1}}-w_{i_r,j_r},       \\
   \overline E(y,i_s,l)           \ &= \ \nu_1 w_{i_s,j_s} -\nu_0,                  & \overline E(y,l,j_s)         \ &= \ -\nu_0 w_{i_s,j_s} -\nu_1                            
 \end{align*}
 (for $r=0,\dotsc,s-1$) where an equation in the left column (right column) holds if the corresponding arrow is oriented to the left (right).

 We lead the assumption that there is a partial evaluation $\ev_\beta$ in $\Ev(f_\beta)$ to a contradiction. If $k$ is a sink, then $\mu_0\neq 0$ (cf.\ Remark \ref{rem: second contradictory beta}) and $\overline E(x,k,j_0)=0$, which implies $w_{i_0,j_0}=-\mu_1/\mu_0$. The equations $\overline E(z_r,\dotsc)=0$ imply that $w_{i_s,j_s}=\dotsb=w_{i_0,j_0}=-\mu_1/\mu_0$. If $l$ is a sink, then $\overline E(y,l,j_s)=0$ can only be satisfied if $\mu_0\nu_1=\mu_1\nu_0$, and if $l$ is a source, then $\overline E(y,i_s,l)=0$ can only be satisfied if $\mu_0\nu_0=-\mu_1\nu_1$. Neither of these cases holds true by property \eqref{2nd3}, which shows that $\Sigma_\beta$ is contradictory if $k$ is a sink.
 
 If $l$ is a source, then we conclude similar to the above case that $\mu_1\neq 0$ and that $w_{i_s,j_s}=\dotsb=w_{i_0,j_0}=\mu_0/\mu_1$. If $l$ is a sink, then $\overline E(y,l,j_s)=0$ can only be satisfied if $\mu_0\nu_0=-\mu_1\nu_1$, and if $l$ is a source, then $\overline E(y,i_s,l)=0$ can only be satisfied if $\mu_0\nu_1=\mu_1\nu_0$. Neither of these cases holds true by property \eqref{2nd3}, which shows that $\Sigma_\beta$ is contradictory if $k$ is a source. This concludes the proof of the proposition.
\end{proof}


\subsection{Automorphisms of the Dynkin diagram}
\label{subsection: automorphisms of the dynkin diagram}

In this section, we investigate the effect of an automorphism of the underlying Dynkin diagram on the quiver Grassmannian $\Gr_\ue(M)$.

Let $\omega$ be the orientation of $Q$. Let $M$ be a representation with basis $\cB$ and dimension vector $\alpha$. Let $\Gamma=\Gamma(M,\cB)$ be the coefficient quiver of $M$ w.r.t.\ $\cB$ and $F:\Gamma\to Q$ the associated morphism of quivers. An automorphism $\sigma:\widetilde D_n\to \widetilde D_n$ of the Dynkin diagram $\widetilde D_n$ of $Q$ permutes the labels of the vertices and arrows of $Q$ and, similarly, of $M$, $\cB$ and $\Gamma$. This yields the quiver $Q'=\sigma Q$, the representation $M'=\sigma M$ with basis $\cB'=\sigma\cB$ and coefficient quiver $\Gamma'=\sigma\Gamma$. It comes together with a morphism $\sigma F:\Gamma'\to Q'$. The orientation $\omega'=\sigma(\omega)$ of $Q'$ and the dimension vector $\alpha'=\sigma(\alpha)$ of $M'$ result from permuting the respective coefficients of $\omega$ and $\alpha$.

Since an automorphism $\sigma$ of $\widetilde D_n$ permutes the arrows $\{a,b,c,d\}$ and is determined by this permutation, we will identify $\sigma$ with the corresponding element of the permutation group of $\{a,b,c,d\}$. In case $n=4$, the automorphism group of $\widetilde D_n$ is the full permutation group, in case $n\geq 5$, the automorphism group of $\widetilde D_n$ is the dihedral group
\[
 \{\id, (acbd), (ab)(cd), (adbc), (ab), (cd), (ac)(bd), (ad)(bc)\}.
\]

It is clear that an automorphism of the Dynkin diagram induces a change of coordinates for the quiver Grassmannian $\Gr_\ue(M)$, and therefore preserves the isomorphism type of the quiver Grassmannian and its Schubert decomposition. More precisely, we have the following.

\begin{prop}\label{prop: automorphism of the Dynkin diagram}
 Let $M$ be a representation of $Q$ with ordered basis $B$. Let $\sigma$ be an automorphism the underlying Dynkin diagram of $Q$. Then the association $N\mapsto \sigma N$ defines an isomorphism $\Gr_\ue(M)\to \Gr_{\sigma(\ue)}(\sigma M)$ that identifies the Schubert cell $C_\beta^M$ with $C_{\sigma(\beta)}^{\sigma M}$. In particular, $\Gr_{\sigma(\ue)}(\sigma M)=\coprod C_{\sigma(\beta)}^{\sigma M}$ is a decomposition into affine spaces if and only if $\Gr_\ue(M)=\coprod C_\beta^M$ is a decomposition into affine spaces. \qed
\end{prop}


\subsection{Bases for some indecomposable representations}\label{subsection: bases for indecomposables of small defect}

In this section, we describe bases for those indecomposable representations of $Q$ for which we will establish a decomposition of the associated quiver Grassmannians into affine spaces. Up to an automorphism of the underlying Dynkin diagram, this will exhaust all isomorphism classes of indecomposable representations of defect $-1$ and $0$, with exception of the non-Schurian representations in homogeneous tubes. See Appendix \ref{appendix: bases for representations of type D_n-tilde} for a construction of these bases.

\subsection*{Defect $-1$} Let $M$ be an indecomposable representation of $Q$ of defect $-1$. Up to an automorphism of $Q$, the representation $M$ has an ordered basis such that the coefficient quiver $\Gamma=\Gamma(M,\cB)$ is
\[
   \beginpgfgraphicnamed{fig30}
   \begin{tikzpicture}[>=latex]
  \matrix (m) [matrix of math nodes, row sep=0em, column sep=1.3em, text height=1ex, text depth=0ex]
   {        &  2  &     &     &      &     &     &     &     \\   
         3  &     &  4  &  5  &\dotsb& n-1 &  n  &     & n+1 \\   
            &     &     &     &      &     &     & n+2 &     \\   
       2n+1 &     & 2n  & 2n-1&\dotsb& n+5 & n+4 &     & n+3 \\   
            & 2n+2&     &     &      &     &     &     &     \\   
       2n+3 &     & 2n+4&2n+5 &\dotsb&3n-1 & 3n  &     &3n+1 \\   
            &     &     &     &      &     &     &3n+2 &     \\   
            &     &     &     &\dotsb&3n+5 &3n+4 &     &3n+3 \\   
};
   \path[-,font=\scriptsize]
   (m-2-3) edge node[auto] {$v_{0}$} (m-2-4)
   (m-2-6) edge node[auto] {$v_{n-5}$} (m-2-7)
   (m-2-7) edge node[auto,swap] {} (m-3-8)
   (m-3-8) edge node[auto,swap] {$c$} (m-4-7)
   (m-4-3) edge node[auto] {$v_0$} (m-4-4)
   (m-4-6) edge node[auto] {$v_{n-5}$} (m-4-7)
   (m-4-3) edge node[auto] {} (m-5-2)
   (m-5-2) edge node[auto] {$b$} (m-6-3)
   (m-6-3) edge node[auto] {$v_0$} (m-6-4)
   (m-6-6) edge node[auto] {$v_{n-5}$} (m-6-7)
   (m-6-7) edge node[auto,swap] {} (m-7-8)
   (m-7-8) edge node[auto,swap] {$c$} (m-8-7)
   (m-8-6) edge node[auto] {$v_{n-5}$} (m-8-7);
   \path[->,dashed,font=\scriptsize]
   (m-2-3) edge node[auto,swap] {$b$} (m-1-2)
   (m-2-1) edge node[auto,swap] {$a$} (m-2-3)
   (m-2-7) edge node[auto] {$d$} (m-2-9)
   (m-4-9) edge node[auto] {$d$} (m-4-7)
   (m-4-3) edge node[auto,swap] {$a$} (m-4-1)
   (m-6-1) edge node[auto,swap] {$a$} (m-6-3)
   (m-6-7) edge node[auto] {$d$} (m-6-9)
   (m-8-9) edge node[auto] {$d$} (m-8-7);
   \path[-,font=\scriptsize]
   (m-2-4) edge node[auto] {} (m-2-5)
   (m-2-5) edge node[auto,swap] {} (m-2-6)
   (m-4-4) edge node[auto] {} (m-4-5)
   (m-4-5) edge node[auto] {} (m-4-6)
   (m-6-4) edge node[auto,swap] {} (m-6-5)
   (m-6-5) edge node[auto] {} (m-6-6)
   (m-8-5) edge node[auto] {} (m-8-6);
  \end{tikzpicture}
\endpgfgraphicnamed
\]
where a dashed arrow, together with its isolated end vertex, is contained in $\Gamma$ if and only if the corresponding arrow of $Q$ is oriented in the indicated direction. Moreover, the following condition is satisfied:
\begin{enumerate}
 \item[(max)] Let $i=i_\max$ be the largest vertex of $\Gamma$ such that $F(i)\in\{q_0,\dotsc, q_{n-4}\}$. Let $v$ be an arrow of $Q$ that connects to $F(i)$. Then there is an arrow $(v,s,t)$ in $\Gamma$ that connects to $i$ if and only if $(v,s,t)$ connects to a vertex $j<i$ or if $v$ is oriented away from $F(i)$.
\end{enumerate}
 More explicitly, this means the following. If $F(i)=q_x$ with $x\in\{1,\dotsc,n-5\}$, then there is precisely one arrow $v$ in $Q$ that does not have a preimage in $\Gamma$ with end vertex $i$. This arrow $v$ has to be oriented towards $F(i)$, i.e.\ we have one of the following situations:
\[
   \beginpgfgraphicnamed{fig31}
   \begin{tikzpicture}[>=latex]
  \matrix (m) [matrix of math nodes, row sep=0em, column sep=2.5em, text height=1ex, text depth=0ex]
   {        &     &     &               &     &     &     \\   
            &  i  &     &               &    & i   &    \\   
            &     &     &               &     &     &     \\   
            &  \  &  \  &     \text{or} &  \  &  \  &     \\   
            &     &     &               &     &     &     \\   
            & q_x &q_{x+1}&            &q_{x-1}&q_{x}&            \\   
            &     &     &               &     &     &     \\}; 
   \path[->,font=\scriptsize]
   (m-6-3) edge node[auto,swap] {$v$} (m-6-2)
   (m-6-5) edge node[auto] {$v$} (m-6-6);
  \end{tikzpicture}
\endpgfgraphicnamed
\]
 If $F(i)=q_0$, then the arrows $a$ and $b$ have preimages in $\Gamma$ with end vertex $i$ if and only if they are oriented away from $q_0$. The analogous statement holds true if $F(i)=q_{n-4}$ with $a$ and $b$ replaced by $c$ and $d$. These situations can be illustrated as 
\[
   \beginpgfgraphicnamed{fig115}
   \begin{tikzpicture}[>=latex]
  \matrix (m) [matrix of math nodes, row sep=0em, column sep=2.5em, text height=1ex, text depth=0ex]
   {    i+1 &     &  i  &               &  i  &     & i+1 \\   
            & i+2 &     &   \text{or}   &     & i+2 &     \\   
            &  \  &  \  &               &  \  &  \  &     \\   
   };
   \path[->,dashed,font=\scriptsize]
   (m-1-3) edge node[auto,swap] {$a$} (m-1-1)
   (m-1-3) edge node[auto] {$b$} (m-2-2)
   (m-1-5) edge node[auto] {$d$} (m-1-7)
   (m-1-5) edge node[auto,swap] {$c$} (m-2-6)
   ;
  \end{tikzpicture}
\endpgfgraphicnamed
\]
 where a dashed arrow, together with its isolated end vertex, is contained in $\Gamma$ if and only if the corresponding arrow of $Q$ is oriented in the indicated direction.

 A subset $\beta$ of this basis is contradictory of the second kind if and only if $\beta$ is not contradictory of the first kind and if there is a $k\in\{1,\dotsc,r\}$ such that 
\begin{align*}
 & kn\notin\beta, && kn+1\notin\beta \ (\text{in case }kn+1\in\Gamma_0), && kn+2\notin\beta \text{ if it is a sink of $\Gamma$}, \\
 & kn+4\in\beta   && kn+3\in\beta    \ (\text{in case }kn+3\in\Gamma_0), && kn+2\in\beta \text{ if it is a source of $\Gamma$}.
\end{align*}
Note that all of the vertices in question are contained in a subgraph of $\Gamma$ of the form
\[
  \beginpgfgraphicnamed{fig90}
   \begin{tikzpicture}[>=latex]
  \matrix (m) [matrix of math nodes, row sep=0em, column sep=1.3em, text height=1ex, text depth=0ex]
   {   kn+1 &     & kn  & \ &           & \ & kn   &      & kn+1 \\   
            & kn+2&     & \ & \text{or} & \ &      & kn+2 &      \\   
       kn+3 &     & kn+4& \ &           & \ & kn+4 &      & kn+3 \\   
};
   \path[-,font=\scriptsize]
   (m-2-8) edge node[auto,swap] {} (m-1-7)
   (m-2-8) edge node[auto,swap] {$c$} (m-3-7)
   (m-2-2) edge node[auto] {} (m-1-3)
   (m-2-2) edge node[auto] {$b$} (m-3-3)
;
   \path[->,dashed,font=\scriptsize]
    (m-1-7) edge node[auto] {$d$} (m-1-9)
    (m-3-9) edge node[auto] {$d$} (m-3-7)
    (m-1-3) edge node[auto,swap] {$a$} (m-1-1)
    (m-3-1) edge node[auto,swap] {$a$} (m-3-3)
;
  \end{tikzpicture}
\endpgfgraphicnamed
\]

\subsection*{Tube of rank $n-2$}
In the following, we will describe an ordered basis for indecomposable representations $M$ of defect $0$. Up to an automorphism of the underlying Dynkin diagram, $M$ has an ordered basis $\cB$ such that the coefficient quiver $\Gamma$ takes the following shape.
\[
   \beginpgfgraphicnamed{fig50}
   \begin{tikzpicture}[>=latex]
  \matrix (m) [matrix of math nodes, row sep=0em, column sep=1.3em, text height=1ex, text depth=0ex]
   {        &     &     &     &      &     &     &     &     \\   
            &     &     &     &\dotsb& n-1 &  n  &     & n+1 \\   
            &     &     &     &      &     &     & n+2 &     \\   
       2n+1 &     & 2n  & 2n-1&\dotsb& n+5 & n+4 &     & n+3 \\   
            & 2n+2&     &     &      &     &     &     &     \\   
       2n+3 &     & 2n+4&2n+5 &\dotsb&     &     &     &     \\   
};
   \path[-,font=\scriptsize]
   (m-2-5) edge node[auto] {} (m-2-6)
   (m-2-6) edge node[auto] {$v_{n-5}$} (m-2-7)
   (m-2-7) edge node[auto,swap] {} (m-3-8)
   (m-3-8) edge node[auto,swap] {$c$} (m-4-7)
   (m-4-3) edge node[auto] {$v_0$} (m-4-4)
   (m-4-6) edge node[auto] {$v_{n-5}$} (m-4-7)
   (m-4-3) edge node[auto] {} (m-5-2)
   (m-5-2) edge node[auto] {$b$} (m-6-3)
   (m-6-3) edge node[auto] {$v_0$} (m-6-4)
;
   \path[->,dashed,font=\scriptsize]
   (m-2-7) edge node[auto] {$d$} (m-2-9)
   (m-4-9) edge node[auto] {$d$} (m-4-7)
   (m-4-3) edge node[auto,swap] {$a$} (m-4-1)
   (m-6-1) edge node[auto,swap] {$a$} (m-6-3)
   ;
   \path[-,font=\scriptsize]
   (m-4-4) edge node[auto] {} (m-4-5)
   (m-4-5) edge node[auto] {} (m-4-6)
   (m-6-4) edge node[auto,swap] {} (m-6-5)
   ;
  \end{tikzpicture}
\endpgfgraphicnamed
\]
where a dashed arrow, together with its isolated end vertex, is contained in $\Gamma$ if and only if the corresponding arrow of $Q$ is oriented in the indicated direction. Moreover, the following conditions are satisfied:
\begin{enumerate}
 \item[(min)] Let $i=i_\min$ be the smallest vertex of $\Gamma$ such that $F(i)\in\{q_0,\dotsc, q_{n-4}\}$. Let $v$ be an arrow of $Q$ that connects to $F(i)$. Then there is an arrow $(v,s,t)$ in $\Gamma$ that connects to $i$ if and only if $(v,s,t)$ connects to $i+1$ or if $v$ is oriented towards $F(i)$.
 \item[(max)] Let $i=i_\max$ be the largest vertex of $\Gamma$ such that $F(i)\in\{q_0,\dotsc, q_{n-4}\}$. Let $v$ be an arrow of $Q$ that connects to $F(i)$. Then there is an arrow $(v,s,t)$ in $\Gamma$ that connects to $i$ if and only if $(v,s,t)$ connects to $i-1$ or if $v$ is away from $F(i)$.
\end{enumerate}

We assume that $4\leq i_\min\leq n$. The second kind contradictory $\beta$ are characterized exactly as in the defect $-1$-case.

\subsection*{Tubes of rank $2$}
Let $M$ be an indecomposable representation of defect $0$ that is contained in a tube of rank $2$ of the Auslander-Reiten quiver. Up to an automorphism of $Q$, $M$ has an ordered basis $\cB$ such that the coefficient quiver $\Gamma$ is
\[
   \beginpgfgraphicnamed{fig52}
   \begin{tikzpicture}[>=latex]
  \matrix (m) [matrix of math nodes, row sep=0.3em, column sep=1.3em, text height=1ex, text depth=0ex]
   {        & n+1 &     &     &      &     &     &     &     \\   
         3  &     &  4  &  5  &\dotsb& n-1 &  n  &     &  2  \\   
            & n+3 &     &     &      &     &     & n+2 &     \\   
       2n+1 &     & n+4 & n+5 &\dotsb&2n-1 & 2n  &     &     \\   
            &3n+1 &     &     &      &     &     &     & 2n+2\\   
       2n+3 &     &2n+4 &2n+5 &\dotsb&3n-1 & 3n  &     &     \\   
            &3n+3 &     &     &  \   &     &     &3n+2 &     \\   
       4n+1 &     &3n+4 &3n+5 &\dotsb&4n-1 & 4n  &     &     \\   
            &     &     &     &      &     &     &     & 4n+2\\   
            &     &     &     &      &     &\dotsb&    &     \\   
};
   \path[-,font=\scriptsize]
   (m-2-3) edge node[auto] {$v_{0}$} (m-2-4)
   (m-2-6) edge node[auto] {$v_{n-5}$} (m-2-7)
   (m-2-7) edge node[auto,swap] {} (m-3-8)
   (m-3-8) edge node[above=3pt] {$c$} (m-4-7)
   (m-4-3) edge node[auto] {$v_0$} (m-4-4)
   (m-4-6) edge node[auto] {$v_{n-5}$} (m-4-7)
   (m-4-7) edge node[auto,swap] {} (m-5-9)
   (m-5-9) edge node[above=3pt] {$d$} (m-6-7)
   (m-6-3) edge node[auto] {$v_0$} (m-6-4)
   (m-6-6) edge node[auto] {$v_{n-5}$} (m-6-7)
   (m-6-7) edge node[auto,swap] {} (m-7-8)
   (m-7-8) edge node[above=3pt] {$c$} (m-8-7)
   (m-8-3) edge node[auto] {$v_0$} (m-8-4)
   (m-8-6) edge node[auto] {$v_{n-5}$} (m-8-7)
   (m-8-7) edge node[auto,swap] {} (m-9-9)
   (m-9-9) edge node[above=3pt] {$d$} (m-10-7)
;
   \path[->,dashed,font=\scriptsize]
   (m-2-3) edge node[auto,swap] {$b$} (m-1-2)
   (m-2-1) edge node[auto] {$a$} (m-2-3)
   (m-2-9) edge node[auto,swap] {$d$} (m-2-7)
   (m-4-3) edge node[auto,swap] {$a$} (m-4-1)
   (m-3-2) edge node[auto] {$b$} (m-4-3)
   (m-6-1) edge node[auto] {$a$} (m-6-3)
   (m-6-3) edge node[auto,swap] {$b$} (m-5-2)
   (m-8-3) edge node[auto,swap] {$a$} (m-8-1)
   (m-7-2) edge node[auto] {$b$} (m-8-3)
;
   \path[-,font=\scriptsize]
   (m-2-4) edge node[auto] {} (m-2-5)
   (m-2-5) edge node[auto,swap] {} (m-2-6)
   (m-4-4) edge node[auto] {} (m-4-5)
   (m-4-5) edge node[auto] {} (m-4-6)
   (m-6-4) edge node[auto] {} (m-6-5)
   (m-6-5) edge node[auto] {} (m-6-6)
   (m-8-4) edge node[auto] {} (m-8-5)
   (m-8-5) edge node[auto] {} (m-8-6)
;
  \end{tikzpicture}
\endpgfgraphicnamed
\]
whose lower end is 
\[
   \beginpgfgraphicnamed{fig91}
   \begin{tikzpicture}[>=latex]
  \matrix (m) [matrix of math nodes, row sep=0.0em, column sep=1.3em, text height=1ex, text depth=0ex]
   {        &     &     &       & \dotsb &      &     \\   
            &rn+3 &     &   \   &        & 3n+2 &     \\   
    (r+1)n+1&     &rn+4 &\dotsb & (r+1)n &      & (r+1)n+2 \\   
};
   \path[-,font=\scriptsize]
   (m-1-5) edge node[auto,swap] {} (m-2-6)
   (m-2-6) edge node[auto,swap] {$c$} (m-3-5)
;
   \path[->,dashed,font=\scriptsize]
   (m-2-2) edge node[auto] {$b$} (m-3-3)
   (m-3-3) edge node[auto] {$a$} (m-3-1)
   (m-3-5) edge node[auto,swap] {$d$} (m-3-7)
;
   \path[-,font=\scriptsize]
   (m-3-3) edge node[auto] {} (m-3-4)
   (m-3-4) edge node[auto,swap] {} (m-3-5)
;
  \end{tikzpicture}
\endpgfgraphicnamed
\]
if $r$ is odd, and 
\[
   \beginpgfgraphicnamed{fig92}
   \begin{tikzpicture}[>=latex]
  \matrix (m) [matrix of math nodes, row sep=0.2em, column sep=1.3em, text height=1ex, text depth=0ex]
   {        &     &     &       & \dotsb &      &      \\   
       & (r+1)n+3 &     &   \   &        &      & 3n+2 \\   
       rn+1 &     &rn+4 &\dotsb & (r+1)n & (r+1)n+2 \\   
};
   \path[-,font=\scriptsize]
   (m-1-5) edge node[auto,swap] {} (m-2-7)
   (m-2-7.west) edge node[auto,swap] {$d$} (m-3-5)
;
   \path[->,dashed,font=\scriptsize]
   (m-3-3) edge node[auto,swap] {$b$} (m-2-2)
   (m-3-1) edge node[auto,swap] {$a$} (m-3-3)
   (m-3-5) edge node[auto,swap] {$c$} (m-3-6)
;
   \path[-,font=\scriptsize]
   (m-3-3) edge node[auto] {} (m-3-4)
   (m-3-4) edge node[auto,swap] {} (m-3-5)
;
  \end{tikzpicture}
\endpgfgraphicnamed
\]
if $r$ is even. As usual, the dashed arrows together with their isolated end vertices are part of $\Gamma$ if and only if the corresponding arrow of $Q$ is oriented in the indicated direction. 

A subset $\beta$ of this basis $\cB$ is contradictory of the second kind if and only if $\beta$ is not contradictory of the first kind and if there is a $k\in\{1,\dotsc,r\}$ such that 
\begin{align*}
 & (k-1)n+4,\dotsc,kn\notin\beta, && kn+1\notin\beta \ (\text{in case }kn+1\in\Gamma_0), && kn+2\notin\beta \text{ if it is a sink of $\Gamma$}, \\
 & kn+4,\dotsc,(k+1)n\in\beta     && kn+3\in\beta    \ (\text{in case }kn+3\in\Gamma_0), && kn+2\in\beta \text{ if it is a source of $\Gamma$}.
\end{align*}
Note that all of the vertices in question are contained in a subgraph of $\Gamma$ of the form
\[
   \beginpgfgraphicnamed{fig97}
   \begin{tikzpicture}[>=latex]
  \matrix (m) [matrix of math nodes, row sep=0em, column sep=1.3em, text height=1ex, text depth=0ex]
   {        &      &          &        &        &     &     \\   
        \   & kn+1 & (k-1)n+4 & \dotsb & kn     &     &     \\   
\hspace{1.6cm}&    &          &        &   &\hspace{1.6cm}& kn+2\\   
        \   & kn+3 & kn+4     & \dotsb & (k+1)n &     &     \\   
};
   \path[-,font=\scriptsize]
   (m-2-5) edge node[auto,swap] {} (m-3-7)
   (m-3-7) edge node[auto,swap] {$d$} (m-4-5)
;
   \path[->,dashed,font=\scriptsize]
   (m-2-3) edge node[auto,swap] {$b$} (m-2-2)
   (m-4-2) edge node[auto] {$b$} (m-4-3)
;
   \path[-,font=\scriptsize]
   (m-2-3) edge node[auto] {} (m-2-4)
   (m-2-4) edge node[auto,swap] {} (m-2-5)
   (m-4-3) edge node[auto] {} (m-4-4)
   (m-4-4) edge node[auto,swap] {} (m-4-5)
;
  \end{tikzpicture}
\endpgfgraphicnamed
\]
if $k$ is even, and
\[
   \beginpgfgraphicnamed{fig98}
   \begin{tikzpicture}[>=latex]
  \matrix (m) [matrix of math nodes, row sep=0em, column sep=1.3em, text height=1ex, text depth=0ex]
   {        &      &          &        &        &     &     \\   
       kn+1 &   \  & (k-1)n+4 & \dotsb & kn     &     &     \\   
       &\hspace{1.6cm}&       &        &        & kn+2& \hspace{1.6cm} \\   
       kn+3 &   \  & kn+4     & \dotsb & (k+1)n &     &     \\   
};
   \path[-,font=\scriptsize]
   (m-2-5) edge node[auto,swap] {} (m-3-6)
   (m-3-6) edge node[auto,swap] {$c$} (m-4-5)
;
   \path[->,dashed,font=\scriptsize]
   (m-2-3) edge node[auto,swap] {$a$} (m-2-1)
   (m-4-1) edge node[auto] {$a$} (m-4-3)
;
   \path[-,font=\scriptsize]
   (m-2-3) edge node[auto] {} (m-2-4)
   (m-2-4) edge node[auto,swap] {} (m-2-5)
   (m-4-3) edge node[auto] {} (m-4-4)
   (m-4-4) edge node[auto,swap] {} (m-4-5)
;
  \end{tikzpicture}
\endpgfgraphicnamed
\]
if $k$ is odd.

\subsection*{Homogeneous tubes}
Let $M$ be a Schurian representation of defect $0$ whose isomorphism class is contained in a homogeneous tube of of the Auslander-Reiten quiver. Then there exists an ordered basis $\cB$ of $M$ such that the coefficient quiver $\Gamma$ is
\[
   \beginpgfgraphicnamed{fig93}
   \begin{tikzpicture}[>=latex]
  \matrix (m) [matrix of math nodes, row sep=0em, column sep=3em, text height=1ex, text depth=0ex]
   {     1   &         &  2      & \dotsb  &  n-2    &         & n-1     \\   
             &   n     &         &         &         & n+2     &         \\   
       2n+1  &         & 2n      & \dotsb  & n+4     &         & n+3     \\   
};
   \path[-,font=\tiny]
   (m-1-3) edge node[auto] {} (m-1-4)
   (m-1-4) edge node[auto,swap] {} (m-1-5)
   (m-3-3) edge node[auto] {} (m-3-4)
   (m-3-4) edge node[auto,swap] {} (m-3-5)
   (m-1-5) edge node[below left=-4pt] {$c,\mu_0$} (m-2-6)
   (m-2-6) edge node[above left=-4pt] {$c,\mu_1$} (m-3-5)
   (m-2-2) edge node[auto,swap] {$b$} (m-1-3)
   (m-2-2) edge node[auto] {} (m-3-3)
   ;
   \path[->,dashed,font=\tiny]
   (m-1-1) edge node[auto] {$a$} (m-1-3)
   (m-1-5) edge node[auto] {$d$} (m-1-7)
   (m-3-3) edge node[auto] {$a$} (m-3-1)
   (m-3-7) edge node[auto] {$d$} (m-3-5)
   ;
  \end{tikzpicture}
\endpgfgraphicnamed
\]
where a dashed arrow, together with its isolated end vertex, is contained in $\Gamma$ if and only if the corresponding arrow of $Q$ is oriented in the indicated direction. The weights $\mu_0,\mu_1\in\C^\times$ satisfy that $\mu_0\neq\mu_1$ if both $q_b$ and $q_c$ are a sink or both are a source of $Q$, and $\mu_0\neq-\mu_1$ otherwise.

For this coefficient quiver, we encounter second kind contradictory $\beta$-stated for the following subquivers of $\Gamma$.
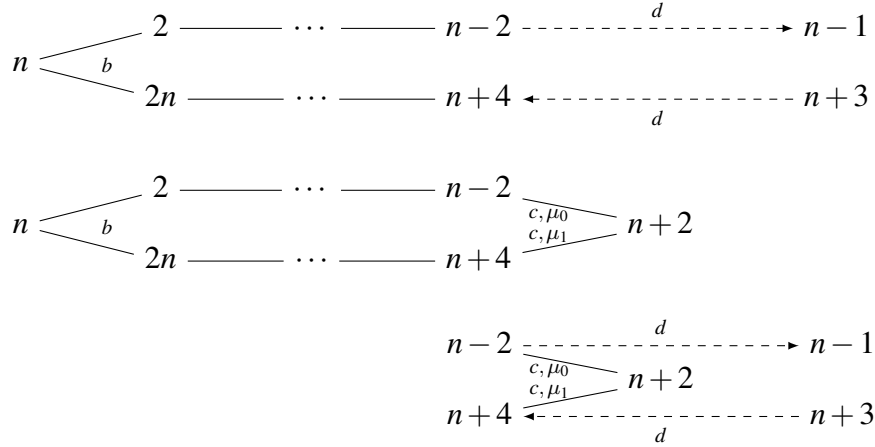
\[ \hspace{2.8cm}
  \beginpgfgraphicnamed{fig94}
   \begin{tikzpicture}[>=latex]
  \matrix (m) [matrix of math nodes, row sep=0em, column sep=3em, text height=1ex, text depth=0ex]
   {              & 2  & \dotsb  &  n-2  &              & n-1  \\   
              n   &    &         &       &\hspace{0.8cm}&           \\   
                  & 2n & \dotsb  &  n+4  &              & n+3  \\   
};
   \path[-,font=\tiny]
   (m-2-1) edge node[auto,swap] {$b$} (m-1-2)
   (m-2-1) edge node[auto] {} (m-3-2)
;
   \path[->,dashed,font=\tiny]
   (m-1-4) edge node[auto] {$d$} (m-1-6)
   (m-3-6) edge node[auto] {$d$} (m-3-4)
;
   \path[-,font=\tiny]
   (m-1-2) edge node[auto] {} (m-1-3)
   (m-1-3) edge node[auto,swap] {} (m-1-4)
   (m-3-2) edge node[auto] {} (m-3-3)
   (m-3-3) edge node[auto] {} (m-3-4)
;
  \end{tikzpicture}
\endpgfgraphicnamed
\]
\[ \hspace{2.8cm}
   \beginpgfgraphicnamed{fig95}
   \begin{tikzpicture}[>=latex]
  \matrix (m) [matrix of math nodes, row sep=0em, column sep=3em, text height=1ex, text depth=0ex]
   {              & 2   & \dotsb  &  n-2  &     & \hspace{0.8cm}  \\   
              n   &     &         &       & n+2 &  \\   
                  & 2n  & \dotsb  &  n+4  &     &  \\   
};
   \path[-,font=\tiny]
   (m-1-4) edge node[below left=-4pt] {$c,\mu_0$} (m-2-5)
   (m-2-5) edge node[above left=-4pt] {$c,\mu_1$} (m-3-4)
   (m-2-1) edge node[auto,swap] {$b$} (m-1-2)
   (m-2-1) edge node[auto] {} (m-3-2)
;
   \path[-,font=\tiny]
   (m-1-2) edge node[auto] {} (m-1-3)
   (m-1-3) edge node[auto,swap] {} (m-1-4)
   (m-3-2) edge node[auto] {} (m-3-3)
   (m-3-3) edge node[auto] {} (m-3-4)
;
  \end{tikzpicture}
  \endpgfgraphicnamed
\]
\[\hspace{8.65cm}
  \beginpgfgraphicnamed{fig96}
   \begin{tikzpicture}[>=latex]
  \matrix (m) [matrix of math nodes, row sep=0em, column sep=3em, text height=1ex, text depth=0ex]
   {          n-2  &     & n-1  \\   
                   & n+2 &           \\   
              n+4  &     & n+3  \\   
};
   \path[-,font=\tiny]
   (m-1-1) edge node[below left=-4pt] {$c,\mu_0$} (m-2-2)
   (m-2-2) edge node[above left=-4pt] {$c,\mu_1$} (m-3-1)
;
   \path[->,dashed,font=\tiny]
   (m-1-1) edge node[auto] {$d$} (m-1-3)
   (m-3-3) edge node[auto] {$d$} (m-3-1)
;
  \end{tikzpicture}
\endpgfgraphicnamed
\]
The set $\beta$ is contradictory if there is a subquiver of the above shape such that the following conditions are satisfied for the vertices of this subquiver.
\begin{enumerate}
 \item $2,\dotsc,n-2\notin\beta$;
 \item $n+3,\dotsc,2n\in\beta$;
 \item $n\in\beta$ if and only if it is a source of $\Gamma$;
 \item $n+2\in\beta$ if and only if it is a source of $\Gamma$.
\end{enumerate}


\subsection{The main theorem}\label{subsection: the main theorem}
Let $M$ be one indecomposable representation of $Q$ that we considered in the previous section, and $\cB$ its ordered basis.

\begin{thm}\label{thm: the main theorem}
 Let $\ue$ be a dimension vector for $Q$. Then the Schubert decomposition
 \[
  \Gr_\ue(M) \quad = \quad \coprod_{\substack{\beta\subset\cB\\\text{of type }\ue}} \ C_\beta^M
 \]
 w.r.t.\ $\cB$ is a decomposition into affine spaces. A Schubert cell $C_\beta^M$ is empty if and only if $\beta$ is contradictory of the first or of the second kind.
\end{thm}

This has the following immediate consequence.

\begin{cor}
The Euler characteristic of $\Gr_\ue(M)$ is
 \[
  \chi\bigl(\,\Gr_\ue(M)\,\bigr) \ = \ \# \, \biggl\{ \, \beta\subset\cB  \, \biggl| \, \begin{array}{c} \beta\text{ of type }\ue\text{ and not contradictory}\\ \text{of the first or of the second kind} \end{array}\, \biggr\}.
 \]
\end{cor}

\begin{rem}
 In the sequel \cite{LW15} of this paper, we will extend Theorem \ref{thm: the main theorem} to all indecomposable representations of $Q$. This is done by other means than our treatment of the present cases in terms of Schubert systems. 
 
 Tentative calculations indicate that Schubert systems can also be used to handle the other cases, but these calculations also show that the combinatorics of the Schubert system gets too rich to present such a proof in a reasonable way. 
 
 Furthermore, it seems to us that for other representations $M$ than considered in Theorem \ref{thm: the main theorem}, there exists no ordered basis $\cB$ such that the empty Schubert cells $C_\beta^M$ are characterized by contradictory subsets $\beta$ of the first or second kind. For all the bases that we consider in Appendix \ref{appendix: bases for representations of type D_n-tilde}, additional $\beta$ with empty cells $C_\beta^M$ occur.
\end{rem}


\section{Proof of Theorem \ref{thm: the main theorem}}
\label{section: proof of theorem 4.4}

\subsection{Defect $-1$} We begin the proof with a preprojective representation $M$ of defect $-1$ with basis $\cB$ and associated coefficient quiver $\Gamma$ of the form
\[
   \beginpgfgraphicnamed{fig32}
   \begin{tikzpicture}[>=latex]
  \matrix (m) [matrix of math nodes, row sep=0em, column sep=1.3em, text height=1ex, text depth=0ex]
   {        &  2  &     &     &      &     &     &     &     \\   
         3  &     &  4  &  5  &\dotsb& n-1 &  n  &     & n+1 \\   
            &     &     &     &      &     &     & n+2 &     \\   
       2n+1 &     & 2n  & 2n-1&\dotsb& n+5 & n+4 &     & n+3 \\   
            & 2n+2&     &     &      &     &     &     &     \\   
       2n+3 &     & 2n+4&2n+5 &\dotsb&3n-1 & 3n  &     &3n+1 \\   
            &     &     &     &      &     &     &3n+2 &     \\   
            &     &     &     &\dotsb&3n+5 &3n+4 &     &3n+3 \\   
   };
   \path[-,font=\scriptsize]
   (m-2-3) edge node[auto] {$v_{0}$} (m-2-4)
   (m-2-6) edge node[auto] {$v_{n-5}$} (m-2-7)
   (m-2-7) edge node[auto,swap] {} (m-3-8)
   (m-3-8) edge node[auto,swap] {$c$} (m-4-7)
   (m-4-3) edge node[auto] {$v_0$} (m-4-4)
   (m-4-6) edge node[auto] {$v_{n-5}$} (m-4-7)
   (m-4-3) edge node[auto] {} (m-5-2)
   (m-5-2) edge node[auto] {$b$} (m-6-3)
   (m-6-3) edge node[auto] {$v_0$} (m-6-4)
   (m-6-6) edge node[auto] {$v_{n-5}$} (m-6-7)
   (m-6-7) edge node[auto,swap] {} (m-7-8)
   (m-7-8) edge node[auto,swap] {$c$} (m-8-7)
   (m-8-6) edge node[auto] {$v_{n-5}$} (m-8-7);
   \path[->,dashed,font=\scriptsize]
   (m-2-3) edge node[auto,swap] {$b$} (m-1-2)
   (m-2-1) edge node[auto,swap] {$a$} (m-2-3)
   (m-2-7) edge node[auto] {$d$} (m-2-9)
   (m-4-9) edge node[auto] {$d$} (m-4-7)
   (m-4-3) edge node[auto,swap] {$a$} (m-4-1)
   (m-6-1) edge node[auto,swap] {$a$} (m-6-3)
   (m-6-7) edge node[auto] {$d$} (m-6-9)
   (m-8-9) edge node[auto] {$d$} (m-8-7);
   \path[-,font=\scriptsize]
   (m-2-4) edge node[auto] {} (m-2-5)
   (m-2-5) edge node[auto,swap] {} (m-2-6)
   (m-4-4) edge node[auto] {} (m-4-5)
   (m-4-5) edge node[auto] {} (m-4-6)
   (m-6-4) edge node[auto,swap] {} (m-6-5)
   (m-6-5) edge node[auto] {} (m-6-6)
   (m-8-5) edge node[auto] {} (m-8-6);
  \end{tikzpicture}
\endpgfgraphicnamed
\]
where the dashed arrows and their isolated ends are contained in $\Gamma$ if and only if the corresponding arrow of $Q$ is oriented in the indicated direction. Recall further from section \ref{subsection: bases for indecomposables of small defect} that the largest vertex of $\Gamma$ satisfies property (max), which will be important in the proof and to which we refer at the appropriate place in the proof.

Let $r$ be the number of non-extremal edges of $\Gamma$. Then the largest vertex $i$ of $\Gamma$ is between $rn+4$ and $(r+1)n+2$ unless $i=2$ and $r=0$.

\subsection*{The patchwork} We define the following patchwork for $\overline\Sigma$. The index set $I$ consists of all triples $(k,l,\epsilon)$ where $k,l\in\Z$ and $\epsilon\in\{+,-\}$ with the further restrictions that $0\leq k \leq l \leq r+1$ and $l-k$ is even if $\epsilon=+$, and $1\leq k < l \leq r+1$ if $\epsilon=-$. We order $I$ by the rule 
\[
 (k,l,\epsilon)<(k',l',\epsilon') \qquad \text{if and only if} \qquad \left\{ \begin{array}{l}
                                                                               l-k < l'-k', \text{ or} \\
                                                                               l-k = l'-k'\text{ and }l<l', \text{ or} \\
                                                                               l-k = l'-k', l=l', \epsilon=-\text{ and }\epsilon'=+. \\
                                                                              \end{array}\right.
\]

For $(k,l,-)\in I$, we define the patch $\Pi_{k,l}=\Xi_{(k,l,-)}$ as the full subsystem of $\overline\Sigma$ whose vertices are the relevant pairs $(i,j)$ and relevant triples $(v,t,s)$ with $i,t\in\{(k-1)n+4,\dotsc,kn\}$ and $j,s\in\{(l-1)n+4,\dotsc,ln\}$.

For $(k,l,+)\in I$, we define the patch $\Xi_{k,l}=\Xi_{(k,l,+)}$ as the full subsystem of $\overline\Sigma$ whose vertices are the relevant pairs $(i,j)$ and relevant triples $(v,t,s)$ with $i,t\in\{kn,\dotsc,kn+4\}$ and $j,s\in\{ln,\dotsc,ln+4\}$, together with the relevant pairs $(kn,kn+4)$ and $(ln,ln+4)$, with the exception of the following cases: if $k=0$ or if $(kn,kn+4)$ would be an isolated vertex of $\Xi_{k,l}$, we omit $(kn,kn+4)$ from $\Xi_{k,l}$; if $l=0$ or if $(ln,ln+4)$ would be an isolated vertex of $\Xi_{k,l}$, we omit $(ln,ln+4)$ from $\Xi_{k,l}$. In other words, $(kn,kn+4)$ is contained in $\Xi_{k,l}$ if and only if it is the base vector of a quadratic link whose other base vector and whose tip are contained in $\Xi_{k,l}$; and analogous for $(ln,ln+4)$.

Note that the patches $\Xi_{0,0}$ and $\Xi_{r+1,r+1}$ are always empty, and the patches $\Xi_{k,r+1}$ are empty unless $\Gamma$ contains $(kn,(r+1)n)$. It is easily verified that $\{\Xi_{(k,l,\epsilon)}\}$ defines indeed a patchwork for $\overline\Sigma$, but we forgo to spell out the details. We encourage the reader to convince himself along the following example.

\begin{figure}[tb]
\[
 \beginpgfgraphicnamed{fig34}
  \begin{tikzpicture}[>=latex]
   \matrix (m) [matrix of math nodes, row sep=0.5em, column sep=0.6em, text height=1.5ex, text depth=0.5ex]
{           &        \          &         \        &         \        &         \        &         \        &         \        &        \         &        \         &        \         & \const{2020}{-1} &\triple{2022}{c}  &         \        &  20 \\
            &        \          &         \        &         \        &         \        &         \        &         \        &        \         &        \         &        \         &         \        &\triple{1922}{d}  &         \        &  19 \\
            &        \          &         \        &         \        &         \        &         \        &         \        &        \         &        \         &        \         &         \        &  \pair{1822}     &\vertex{1823}{v_1}&  18 \\
            &        \          &         \        &         \        & \const{1714}{-1} &         \        &         \        &        \         &        \         &        \         &         \        &         \        &  \pair{1723}     &  17 \\
            &        \          &         \        &         \        &\triple{1214}{b}  &\triple{1215}{a}  &  \pair{1216}     &        \         &        \         &        \         &         \        &         \        &         \        &  12 \\
            &        \          &         \        &         \        &         \        &         \        &\vertex{1116}{v_0}&  \pair{1117}     &        \         &        \         &         \        &         \        &  \pair{1123}     &  11 \\
            &        \          &         \        &         \        &         \        &         \        &         \        &\vertex{1017}{v_1}&  \pair{1018}     &        \         &         \        &  \pair{1022}     &\vertex{1023}{v_1}&  10 \\
\const{1010}{-1}&\triple{0810}{c}&        \        &         \        &         \        &         \        &         \        &        \         &\triple{0818}{c}  &        \         &  \pair{0820}     &\triple{0822}{c}  &         \        &   8 \\
            & \triple{0710}{d}  &         \        &         \        &         \        &         \        &         \        &        \         &\triple{0718}{d}  &  \pair{0719}     &         \        &\triple{0722}{d}  &         \        &   7 \\
            &   \pair{0610}     &\vertex{0611}{v_1}&         \        &         \        &         \        &         \        &\vertex{0617}{v_1}&  \pair{0618}     &        \         &         \        &  \pair{0622}     &\vertex{0623}{v_1}&   6 \\
            &        \          &  \pair{0511}     &\vertex{0512}{v_0}&         \        &         \        &\vertex{0516}{v_0}&  \pair{0517}     &        \         &        \         &         \        &         \        &  \pair{0523}     &   5 \\
            &        \          &         \        &  \pair{0412}     &\triple{0414}{b}  &\triple{0415}{a}  &  \pair{0416}     &        \         &        \         &        \         &         \        &         \        &         \        &   4 \\
            &        \          &         \        &         \        &         \        &  \pair{0315}     &         \        &        \         &        \         &        \         &         \        &         \        &         \        &   3 \\
            &         10        &        11        &        12        &        14        &        15        &        16        &       17         &       18         &       19         &        20        &        22        &        23        &     \\
};
     \path[-,very thick,font=\scriptsize,gray]
      (0810) edge[bend left=35] (0610)
      (0710) edge (0610)
      (0610) edge (0611)
      (0511) edge (0611)
      (0511) edge (0512)
      (0412) edge (0512)
      (0414) edge[bend left=35] (0416)
      (0415) edge (0315)
      (0415) edge (0416)
      (0416) edge (0516)
      (0516) edge (0517)
      (0517) edge (0617)
      (0617) edge (0618)
      (0718) edge (0618)
      (0718) edge (0719)
      (0818) edge (0820)
      (0818) edge[bend right=35] (0618)
      (1214) edge[bend right=35] (1216)
      (1215) edge (1216)
      (1116) edge (1216)
      (1116) edge (1117)
      (1017) edge (1117)
      (1017) edge (1018)
      (0822) edge[bend left=35] (0622)
      (0722) edge (0622)
      (0622) edge (0623)
      (0523) edge (0623)
      (1022) edge (1023)
      (1123) edge (1023)
      (2022) edge[bend left=35] (1822)
      (1922) edge (1822)
      (1822) edge (1823)
      (1723) edge (1823)
;
     \path[-,font=\scriptsize]
      (0412) edge (0414)
      (0818) edge (1018)
      (0822) edge (1022)
      (0822) edge[bend left=10] (0820)
      (0822) edge[bend right=10] node[pos=0.2] (0822-0820) {} (0820)
      (0822) edge[bend left=30] node[pos=0.03] (0822-1822) {} (1822)
      (0722) edge node[pos=0.2] (0722-0719) {} (0719)
      (0722) edge[bend left=30] node[pos=0.05] (0722-1822) {} (1822)
;
    \path[dotted,-,thick,font=\scriptsize]
      (0810) edge (1010)
      (1214) edge (1714)
      (2022) edge (2020)
      (0722-0719.center) edge[bend left=45] node[above left=-2pt] {} (0722-1822.center)
      (0822-0820.center) edge[bend left=45] node[above left=-2pt] {} (0822-1822.center)
;
\begin{pgfonlayer}{patchwork}
 \filldraw[fill=gray!40,rounded corners] 
   (-5.7,-1.8) rectangle (-2.9,-4.3) 
   (1.9,0.1) rectangle (-0.9,2.6)    
   (4.0,4.3) rectangle (5.7,2.7)     
   (1.9,-4.3) rectangle (-0.9,-1.8)  
   (4.0,1.7) rectangle (5.7,0.1)     
   (4.0,-1.8) rectangle (5.7,-3.4)   
;
 \filldraw[fill=gray!20,rounded corners] 
   (1.0,0.8) -- (1.0,-2.5) -- (4.8,-2.5) -- (4.8,0.8) -- (4.1,0.8) -- (4.1,3.6) -- (4.8,3.6) -- (4.8,4.2) -- (3.4,4.2) -- (3.4,0.8) -- cycle 
;
 \filldraw[fill=gray!20,rounded corners] 
   (-6.8,0.0) rectangle (-4.8,-2.5)  
   (-0.1,1.7) rectangle (-2.9,3.5)   
   (2.8,6.1) rectangle (4.9,3.5)     
   (-3.7,-3.5) rectangle (-0.1,-5.1) 
;
 \draw[dashed,rounded corners] 
   (1.0,0.8) -- (1.0,-2.5) -- (4.8,-2.5) -- (4.8,0.8) -- (4.1,0.8) -- (4.1,3.6) -- (4.8,3.6) -- (4.8,4.2) -- (3.4,4.2) -- (3.4,0.8) -- cycle 
   (-5.7,-1.8) rectangle (-2.9,-4.3) 
   (1.9,0.1) rectangle (-0.9,2.6) 
   (1.9,-4.3) rectangle (-0.9,-1.8) 
   (4.0,4.3) rectangle (5.7,2.7) 
   (4.0,-1.8) rectangle (5.7,-3.4) 
   (4.0,1.7) rectangle (5.7,0.1) 
;
  \node at (-4.8,-3.7) {$\Pi_{1,2}$};
  \node at (1.0,1.9) {$\Pi_{2,3}$};
  \node at (4.7,3.1) {$\Pi_{3,4}$};
  \node at (1.0,-3.7) {$\Pi_{1,3}$};
  \node at (4.7,1.2) {$\Pi_{2,4}$};
  \node at (4.7,-3.0) {$\Pi_{1,4}$};
  \node at (-6.1,-1.3) {$\Xi_{1,1}$};
  \node at (-1.4,3.0) {$\Xi_{2,2}$};
  \node at (3.5,4.8) {$\Xi_{3,3}$};
  \node at (-2.3,-4.7) {$\Xi_{0,2}$};
  \node at (2.7,0.2) {$\Xi_{1,3}$};
 \end{pgfonlayer}
\end{tikzpicture}
 \endpgfgraphicnamed
\] 
 \caption{$\overline\Sigma$ together with its patchwork for $M$ and $\cB$ as in Example \ref{ex: reduced schubert system of type d_n-tilde and its patchwork}}
 \label{fig34}
\end{figure}

\begin{ex}\label{ex: reduced schubert system of type d_n-tilde and its patchwork}
 To get an idea of how the reduced Schubert system $\overline\Sigma$ of a defect $-1$ indecomposable and its patchwork looks like, consider the following example. Let $\Gamma$ be the coefficient quiver
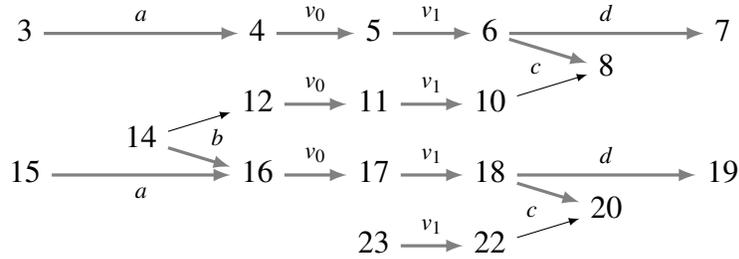
\[
\beginpgfgraphicnamed{fig33}
   \begin{tikzpicture}[>=latex]
  \matrix (m) [matrix of math nodes, row sep=0em, column sep=2em, text height=1ex, text depth=0ex]
   {     3  &     &  4  &  5  &  6  &     &  7  \\   
            &     &     &     &     &  8  &     \\   
            &     & 12  & 11  & 10  &     &     \\   
            & 14  &     &     &     &     &     \\   
        15  &     & 16  & 17  & 18  &     & 19  \\   
            &     &     &     &     & 20  &     \\   
            &     &     & 23  & 22  &     &     \\   
};
   \path[->,very thick,font=\scriptsize,gray]
   (m-1-1) edge node[auto,black] {$a$} (m-1-3)
   (m-1-3) edge node[auto,black] {$v_{0}$} (m-1-4)
   (m-1-4) edge node[auto,black] {$v_{1}$} (m-1-5)
   (m-1-5) edge node[auto,black] {} (m-2-6)
   (m-1-5) edge node[auto,black] {$d$} (m-1-7)
   (m-3-3) edge node[auto,black] {$v_0$} (m-3-4)
   (m-3-4) edge node[auto,black] {$v_{1}$} (m-3-5)
   (m-5-1) edge node[auto,black,swap] {$a$} (m-5-3)
   (m-4-2) edge node[auto,black] {$b$} (m-5-3)
   (m-5-3) edge node[auto,black] {$v_0$} (m-5-4)
   (m-5-4) edge node[auto,black] {$v_{1}$} (m-5-5)
   (m-5-5) edge node[auto,black] {} (m-6-6)
   (m-5-5) edge node[auto,black] {$d$} (m-5-7)
   (m-7-4) edge node[auto,black] {$v_{1}$} (m-7-5)
   ;
   \path[->,font=\scriptsize]
   (m-3-5) edge node[auto] {$c$} (m-2-6)
   (m-4-2) edge node[auto] {} (m-3-3)
   (m-7-5) edge node[auto] {$c$} (m-6-6)
   ;
  \end{tikzpicture}
\endpgfgraphicnamed
\]
 whose extremal arrows are illustrated bold and grey. The number of non-extremal arrows is $r=3$.

 The reduced Schubert system of $M$ w.r.t.\ to $\cB$ and the patchwork $\{\Xi_{(k,l,\epsilon)}\}=\{\Xi_{k,l}\}\cup\{\Pi_{k,l}\}$ is illustrated in Figure \ref{fig34}. Note that we omit the edges and links of $\overline\Sigma$ that are not contained in any patch in this illustration. As usual, we draw extremal edges grey and thick.

 The ambitious reader might verify that $\overline\Sigma$ is totally solvable, but that there is no orientation for $\overline\Sigma$ that restricts to a solution of each non-contradictory $\beta$-state. Namely, the possible solutions of the patch $\Xi_{1,3}$ depend on which of the vertices $(6,22)$, $(d,7,22)$, $(c,8,22)$ and $(10,22)$ are in $\Sigma_\beta$. This effect is studied systematically in the following proof and occurs for patches of the form $\Xi_{k,l}$ with $0<k<l<r+1$.
\end{ex}

\subsection*{Orientation of the patches $\Pi_{k,l}$} The patches $\Pi_{k,l}$ are extremal paths and therefore have an extremal solution. In order to satisfy (PS) for all patches, all arrows of $\Pi_{k,l}$ have to be oriented away from the patch $\Xi_{k',l'}$ with $(k',l',+)<(k,l,-)$. In Figure \ref{fig35}, we indicate this orientation with arrows where the patches $\Pi_{k,l}$ are symbolized as lines and the patches $\Xi_{k,l}$ are symbolized as boxes. 

By property (max) of the coefficient quiver $\Gamma$ from section \ref{subsection: bases for indecomposables of small defect}, the paths $\Pi_{k,r+1}$ either connect non-trivially to the patch $\Xi_{k',r+1}$ (where $k'=k$ or $k'=k-1$, depending on the parity of $r-k$) or the last vertex of $\Pi_{k,r+1}$ is a relevant pair. This means that the chosen orientation is indeed a solution for $\Pi_{k,r+1}$.

Note that the shape of the lower right corner depends on the parity of $r$. The illustration is adequate if r is odd. If $r$ is even, there is no patch in the lower right corner, but one patch $\Xi_{0,r}$ slightly to the left and one patch $\Xi_{1,r+1}$ slightly on top of the corner.

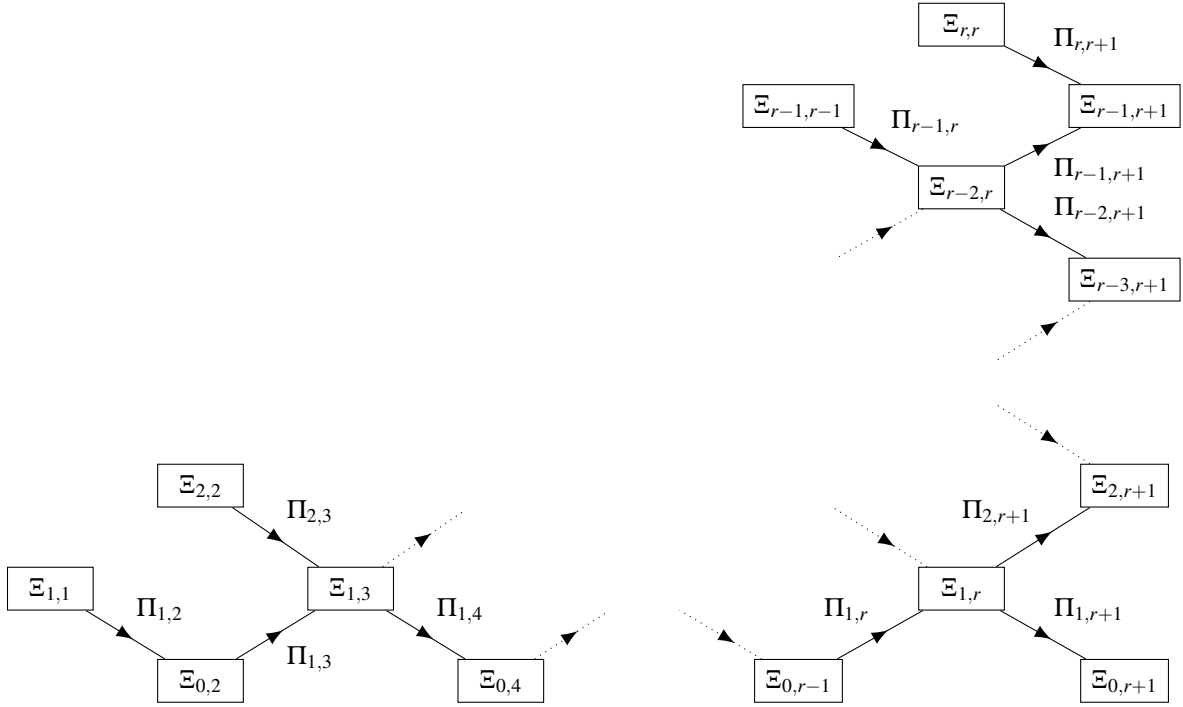
\begin{figure}[tb]
\[
 \beginpgfgraphicnamed{fig35}
  \begin{tikzpicture}[>=latex]
   \matrix (m) [matrix of math nodes, row sep=1.0em, column sep=2em, text height=1.5ex, text depth=0.5ex]
{                &        \        &         \       &         \       & \ &         \               & \gbox{r-r}{r,r}     &        \          \\  
                 &        \        &         \       &         \       & \ & \gbox{r-1-r-1}{r-1,r-1} &         \           & \gbox{r-1-r+1}{r-1,r+1}   \\  
                 &        \        &         \       &         \       & \ &         \               & \gbox{r-2-r}{r-2,r} &        \          \\  
                 &        \        &         \       &         \       & \ &    \nobox{r-3-r-1}      &         \           & \gbox{r-3-r+1}{r-3,r+1}   \\  
                 &        \        &         \       &         \       & \ &         \               & \nobox{-r}          &        \          \\  %
                 & \gbox{2-2}{2,2} &         \       & \nobox{2-4}     & \ &    \nobox{2-r-1}        &         \           & \gbox{2-r+1}{2,r+1}   \\  
 \gbox{1-1}{1,1} &        \        & \gbox{1-3}{1,3} &         \    &\nobox{1-}&     \               & \gbox{1-r}{1,r}     &        \          \\  
                 & \gbox{0-2}{0,2} &         \       & \gbox{0-4}{0,4} & \ & \gbox{0-r-1}{0,r-1}     &         \           & \gbox{0-r+1}{0,r+1}   \\  
};
     \path[-,font=\footnotesize]
      (1-1)     edge[->-=0.6] node[auto] {$\Pi_{1,2}$} (0-2)
      (2-2)     edge[->-=0.6] node[auto] {$\Pi_{2,3}$} (1-3)
      (0-2)     edge[->-=0.6] node[auto,swap] {$\Pi_{1,3}$} (1-3)
      (1-3)     edge[->-=0.6] node[auto] {$\Pi_{1,4}$} (0-4)
      (r-1-r-1) edge[->-=0.6] node[auto] {$\Pi_{r-1,r}$} (r-2-r)
      (r-r)     edge[->-=0.6] node[auto] {$\Pi_{r,r+1}$} (r-1-r+1)
      (r-2-r)   edge[->-=0.6] node[auto,swap] {$\Pi_{r-1,r+1}$} (r-1-r+1)
      (r-2-r)   edge[->-=0.6] node[auto] {$\Pi_{r-2,r+1}$} (r-3-r+1)
      (0-r-1)   edge[->-=0.6] node[auto] {$\Pi_{1,r}$} (1-r)
      (1-r)     edge[->-=0.6] node[auto] {$\Pi_{2,r+1}$} (2-r+1)
      (1-r)     edge[->-=0.6] node[auto] {$\Pi_{1,r+1}$} (0-r+1)
;
     \path[dotted,font=\footnotesize]
      (1-3)     edge[->-=0.6] (2-4)
      (0-4)     edge[->-=0.6] (1-)
      (0-r-1)   edge[-<-=0.5] (1-)
      (1-r)     edge[-<-=0.5] (2-r-1)
      (2-r+1)   edge[-<-=0.5] (-r)
      (r-3-r+1) edge[-<-=0.5] (-r)
      (r-2-r)   edge[-<-=0.5] (r-3-r-1)
;
\end{tikzpicture}
 \endpgfgraphicnamed
\] 
 \caption{The patchwork of $\overline\Sigma$ with an orientation for $\Pi_{k,l}$}
 \label{fig35}
\end{figure}

\subsection*{Strategy of the proof}

In the following, we will investigate the patches $\Xi_{k,l}$, which is a case by case study. Given a non-contradictory $\beta$, we will show that the $\beta$-state $\Xi_{k,l,\beta}$ of $\Xi_{k,l}$ has a solution. Similarly, the chosen extremal orientation of $\Pi_{k,l}$ induces a solution of the $\beta$-states $\Pi_{k,l,\beta}$. All these solutions together form a patchwork solution for $\overline\Sigma$, which implies that $C_\beta^M$ is an affine space (cf.\ Corollary \ref{cor: patchwork solution implies affine space}). We will also see that if $\Sigma_\beta$ is contradictory, then $\beta$ is contradictory of the first or second kind. This will establish the Theorem \ref{thm: the main theorem} for defect $-1$. 

In fact, we will consider for a given patch $\Xi_{k,l}$ a class $C$ of subsets $\beta$ of $\cB$ at once, and---after a suitable variable transformation, which is necessary in a certain case---, we will describe an orientation for $\Xi_{k,l}$ that restricts to a solution of the $\beta$-state $\Xi_{k,l,\beta}$ for each $\beta$ in the class $C$. Unless $k=0$ or $l=r+1$ where we can use the same orientation of $\Xi_{k,l}$ for all $\beta$, we will make use of the following arguments.
\begin{enumerate}
 \item\label{strategy1} We apply certain steps of the algorithm in section \ref{subsection: computing beta-states} that apply to all $\beta$ in $C$. In particular, we will apply the initial steps to identify $\beta$-trivial vertices, edges and links, which simplifies the patch $\Xi_{k,l}$ to a system $\Xi_{k,l,C}$.
 \item\label{strategy2}  A solution of a patch $\Xi_{k,l,\beta}$ satisfies (PS) if and only if there is no edge oriented towards $(kn,ln)$ or $(kn+4,ln)$---provided these pairs are vertices of $\Xi_{k,l,\beta}$. This fact is apparent from the illustrations for each case below.
 \item\label{strategy3} We describe an orientation of $\Xi_{k,l,C}$ with the property that for every $\beta$ in $C$ and for every edge in $\Xi_{k,l,C}$ that is oriented from a triple to a pair, this edge is $\beta$-relevant if the triple is so. Most of the edges in question will be extremal, thus they satisfy this property since $\beta$ is not contradictory of the first kind. For the other edges, we will reason this property in detail.
\end{enumerate}
The argument described in \eqref{strategy3} is rigorous for the following reason. Property (PS) is satisfied for the solution of each $\beta$-state of each patch. This implies that for none of the edges in $\Xi_{k,l}$ that are oriented from a triple to a pair, the pair is contained in a patch $\Xi_{k',l',\epsilon}$ with $(k',l',\epsilon)<(k,l,+)$. Therefore, a $\beta$-state $\Sigma_\beta$ can be computed ``patch-wise'', i.e.\ we can compute $\Xi_{k,l,\epsilon,\beta}$ recursively over the index set $I$. This implies that step \eqref{step7} applied to patches $\Xi_{k',l',\epsilon}$ with $(k',l',\epsilon)>(k,l,+)$ does not have an effect on the edges of $\Xi_{k,l}$.

\subsection*{A remark on the notation}
A patch $\Xi_{k,l}$ contains only triples whose first coordinate is either in $\{a,b\}$ (if $k$ and $l$ are even) of in $\{c,d\}$ (if $k$ and $l$ are odd). Since these two cases behave symmetrically, we investigate $\Xi_{k,l}$ for even $k$ and $l$. The proof for odd $k$ and $l$ is literally the same if $a$ is replaced by $d$ and $b$ is replaced by $c$. There are four possible orientations of the arrows $a$ and $b$, which we will study one by one.

We refer to the last coordinate of a vertex as its \emph{horizontal coordinate} and to the one but the last coordinate as its \emph{vertical coordinate}, which refers to our way of illustrating the Schubert system.

\bigskip

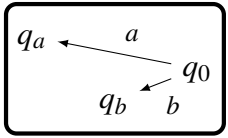
\setlength\intextsep{0pt}
\begin{wrapfigure}{l}{3.0cm}
\vspace{-0,9cm}  
 \beginpgfgraphicnamed{fig36}
  \begin{tikzpicture}[>=latex]
   \matrix (m) [matrix of math nodes, row sep=-1.0em, column sep=1em, text height=2.0ex, text depth=1.0ex]
{ q_a &     &  \   \\
      &     & q_0  \\
   \  & q_b &  \quad   \\
};
     \path[->,font=\footnotesize]
      (m-2-3) edge node[auto,swap] {$a$} (m-1-1)
      (m-2-3) edge node[auto] {$b$} (m-3-2)
;
     \draw[inner sep=10pt,ultra thick,rounded corners]
     (m-1-1.north west) rectangle (m-3-3.south east);
 \end{tikzpicture}
 \endpgfgraphicnamed
\vspace{0,0cm}  
\end{wrapfigure}
\noindent
\begin{minipage}{\minipagewidth}
 \textbf{The patches $\Xi_{k,k}$.} Since the last two coordinates of every vertex in $\Xi_{k,k}$ vary by definition through the same set of vertices $\{kn,\dotsc,kn+4\}$, the patches $\Xi_{k,k}$ are the reduced Schubert systems of the full subgraph of $\Gamma$ with the same set of vertices. This subgraph and $\Xi_{k,k}$ take the following shape.
\end{minipage}
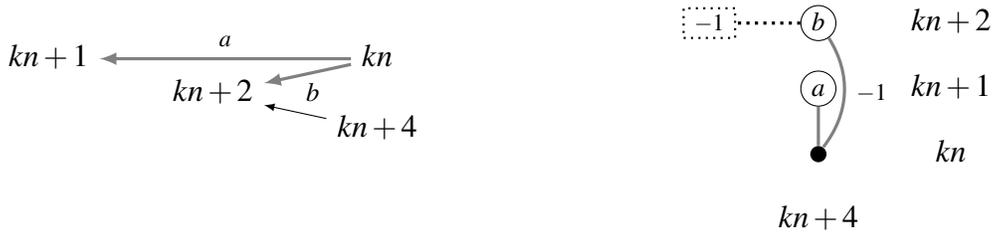
\[
 \beginpgfgraphicnamed{fig37}
  \begin{tikzpicture}[>=latex]
   \matrix (m) [matrix of math nodes, row sep=0em, column sep=2em, text height=1ex, text depth=0ex]
{  kn+1 &      & kn    \\
        & kn+2 &     \\
        &      & kn+4    \\
};
     \path[->,font=\footnotesize]
      (m-3-3) edge node[auto,swap] {$b$} (m-2-2)
;
    \path[->,very thick,font=\scriptsize,gray]
      (m-1-3) edge node[auto,swap,black] {$a$} (m-1-1)
      (m-1-3) edge node[auto] {} (m-2-2)
;
\node at (0,-2.0) {};
 \end{tikzpicture}
 \endpgfgraphicnamed
\hspace{3cm} 
  \beginpgfgraphicnamed{fig38}
  \begin{tikzpicture}[>=latex]
   \matrix (m) [matrix of math nodes, row sep=0.5em, column sep=1em, text height=1.5ex, text depth=0.5ex]
    {   \node[const](22){-1}; & \node[triple](b24){b}; & kn+2  \\
                              & \node[triple](a14){a}; & kn+1  \\
                              & \node[pair](04){};     & kn    \\
                              &       kn+4             &       \\
 };
    \path[-,very thick,font=\scriptsize,gray]
    (a14) edge (04)
    (b24) edge[bend left=35] node[auto,black] {$-1$} (04)
;
    \path[dotted,very thick]
    (b24) edge (22)
;
   \end{tikzpicture}
 \endpgfgraphicnamed
\]

From this it is visible that $\beta$ is contradictory of the first kind if $kn\in\beta$ and $kn+1\notin\beta$, or if $kn\in\beta$ and $kn+2\notin\beta$; $\beta$ is contradictory of the second kind if it is not contradictory of the first kind and if $kn,kn+1,kn+2\notin\beta$ and $kn+4\in\beta$.

If $\Sigma_\beta$ is not contradictory and if $kn\in\beta$ (and thus $kn+1,kn+2\in\beta$), or if $kn+4\notin\beta$, then $\Xi_{k,k,\beta}$ is empty and $\ev_\beta(kn,kn+4)=0$. 

The same is true if $\Sigma_\beta$ is not contradictory and if $kn+2,kn+4\in\beta$ and $kn,kn+1\notin\beta$ or if $kn+1,kn+4\in\beta$ and $kn,kn+2\notin\beta$, as we can apply step \eqref{step7} from section \ref{subsection: computing beta-states} to $(a,kn+1,kn+4)$ in the former case and to $(b,kn+2,kn+4)$ in the latter case. In all cases, the trivial solution for $\Xi_{k,k,\beta}$ satisfies (PS) since (PS) is an empty condition if $\Xi_{k,k,\beta}$ does not contain any edge.

If $\Sigma_\beta$ is not contradictory and if $kn+1,kn+2,kn+4\in\beta$ and $kn\notin\beta$, then $\Xi_{k,k,\beta}$ consists of the vertex $(kn,kn+4)$. Thus the trivial solution for $\Xi_{k,k,\beta}$ satisfies (PS). This exhausts all possibilities for $\{kn,\dotsc,kn+4\}\cap\beta$. 

We will see that the above cases of contradictory $\beta$-states are (up to a different orientation of $a$ and $b$) the only cases of contradictory $\beta$-states that occur. Therefore $\Sigma_\beta$ is contradictory if and only if $\beta$ is contradictory of the first or second kind.

\subsection*{The patches $\Xi_{k,l}$} We turn to the general case $0<k<l<r+1$. In this situation, the two last coordinates of the vertices of $\Xi_{k,l}$ vary in $\{kn,\dotsc,kn+4\}\times\{ln,\dotsc,ln+4\}$, and $\Xi_{k,l}$ looks as follows. Note that the vertex $(kn,kn+4)$ is not the base vertex of a quadratic link with tip and the other base vertex in $\Xi_{k,l}$; therefore $(kn,kn+4)$ is not a vertex of $\Xi_{k,l}$.
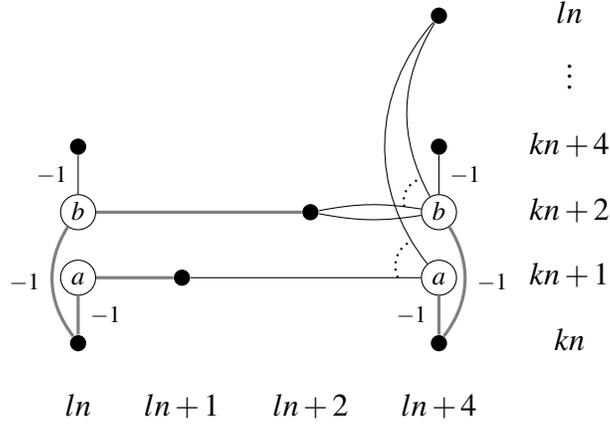
\[
  \beginpgfgraphicnamed{fig39}
  \begin{tikzpicture}[>=latex]
   \matrix (m) [matrix of math nodes, row sep=0.5em, column sep=1em, text height=1.5ex, text depth=0.5ex]
    {                         &                        &                        & \node[pair](ll){};     & ln   \\
                              &                        &                        &                        & \vdots   \\
       \node[pair](40){};     &                        &                        & \node[pair](44){};     & kn+4 \\
       \node[triple](b20){b}; &                        & \node[pair](22){};     & \node[triple](b24){b}; & kn+2 \\
       \node[triple](a10){a}; & \node[pair](11){};     &                        & \node[triple](a14){a}; & kn+1 \\
       \node[pair](00){};     &                        &                        & \node[pair](04){};     & kn   \\
               ln             &       ln+1             &        ln+2            &        ln+4            &      \\
 };
    \path[-,very thick,font=\scriptsize,gray]
    (a10) edge node[auto,black] {$-1$} (00)
    (a10) edge (11)
    (b20) edge[bend right=35] node[auto,swap,black] {$-1$} (00)
    (b20) edge (22)
    (a14) edge node[auto,black,swap] {$-1$} (04)
    (b24) edge[bend left=35] node[auto,black] {$-1$} (04)
;
   \path[-,font=\scriptsize]
    (b20) edge node[auto,black] {$-1$} (40)
    (b24) edge node[auto,black,swap] {$-1$} (44)
    (b24) edge[bend left=10] node[auto] {} (22)
    (b24) edge[bend right=10] node[pos=0.15] (b24-22) {} (22)
    (b24) edge[bend left=30]  node[pos=0.1] (b24-ll) {} (ll)
    (a14) edge node[pos=0.1] (a14-11) {} (11)
    (a14) edge[bend left=40] node[pos=0.1] (a14-ll) {} (ll)
;
    \path[dotted,-,thick,font=\scriptsize]
    (b24-22.center) edge[bend left=45] node[above left=-2pt] {} (b24-ll.center)
    (a14-11.center) edge[bend left=45] node[above left=-2pt] {} (a14-ll.center)
;
   \end{tikzpicture}
 \endpgfgraphicnamed
\]

We consider an arbitrary subset $\beta$ of $\cB$ such that $\Sigma_\beta$ is not contradictory and inspect solutions for the patch $\Xi_{k,l}$, depending on $\{kn,\dotsc,kn+4\}\cap\beta$. Since we will orientate in all cases the edges $\bigr\{(a,kn+1,ln+4),(kn+1,ln+1)\bigl\}$ and $\bigr\{(b,kn+2,ln+4),(kn+2,ln+2)\bigl\}$ towards the triples $(a,kn+1,ln+4)$ and $(b,kn+2,ln+4)$---as far as they are $\beta$-relevant---, we can disregard the quadratic links and the vertex $(ln,ln+4)$ of $\Xi_{k,l}$; for simplicity, we will omit them from the following illustrations.

Assume that $kn+1\in\beta$. Then all vertices of $\Xi_{k,l}$ with vertical coordinate $kn+1$ are $\beta$-trivial, and the following extremal solution of the resulting subsystem of $\Xi_{k,l}$ restricts to a solution of $\Sigma_\beta$
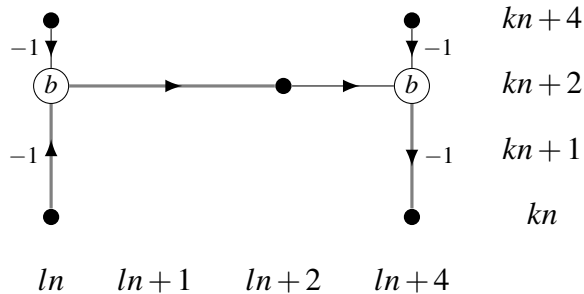
\[
  \beginpgfgraphicnamed{fig73}
  \begin{tikzpicture}[>=latex]
   \matrix (m) [matrix of math nodes, row sep=0.5em, column sep=1em, text height=1.5ex, text depth=0.5ex]
    {  \node[pair](40){};     &                        &                        & \node[pair](44){};     & kn+4 \\
       \node[triple](b20){b}; &                        & \node[pair](22){};     & \node[triple](b24){b}; & kn+2 \\
                              &                        &                        &                        & kn+1 \\
       \node[pair](00){};     &                        &                        & \node[pair](04){};     & kn   \\
               ln             &       ln+1             &        ln+2            &        ln+4            &      \\
 };
    \path[-,very thick,font=\scriptsize,gray]
    (b20) edge[-<-=0.5] node[auto,swap,black] {$-1$} (00)
    (b20) edge[->-=0.55] (22)
    (b24) edge[->-=0.6] node[auto,black] {$-1$} (04)
;
    \path[-,font=\scriptsize]
    (b20) edge[-<-=0.7] node[auto,black] {$-1$} (40)
    (b24) edge[-<-=0.7] node[auto,black,swap] {$-1$} (44)
    (b24) edge[-<-=0.5] (22)
;
   \end{tikzpicture}
 \endpgfgraphicnamed
\]

Assume that $kn+2\in\beta$. Then all vertices of $\Xi_{k,l}$ with vertical coordinate $kn+2$ are $\beta$-trivial, and the following extremal solution of the resulting subsystem of $\Xi_{k,l}$ restricts to a solution of $\Sigma_\beta$.
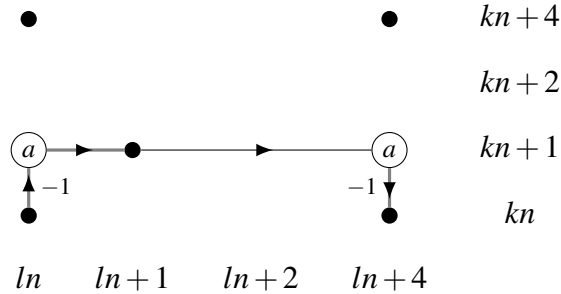
\[
  \beginpgfgraphicnamed{fig72}
  \begin{tikzpicture}[>=latex]
   \matrix (m) [matrix of math nodes, row sep=0.5em, column sep=1em, text height=1.5ex, text depth=0.5ex]
    {  \node[pair](40){};     &                        &                        & \node[pair](44){};     & kn+4 \\
                              &                        &                        &                        & kn+2 \\
       \node[triple](a10){a}; & \node[pair](11){};     &                        & \node[triple](a14){a}; & kn+1 \\
       \node[pair](00){};     &                        &                        & \node[pair](04){};     & kn   \\
               ln             &       ln+1             &        ln+2            &        ln+4            &      \\
 };
    \path[-,very thick,font=\scriptsize,gray]
    (a10) edge[-<-=0.6] node[auto,black] {$-1$} (00)
    (a10) edge[->-=0.6] (11)
    (a14) edge[->-=0.8] node[auto,black,swap] {$-1$} (04)
;
    \path[-,font=\scriptsize]
    (a14) edge[-<-=0.5] (11)
;
   \end{tikzpicture}
 \endpgfgraphicnamed
\]

Assume that $kn+1,kn+2\notin\beta$. From our study of the patch $\Xi_{k,k}$, we know that since $\beta$ is not contradictory of the first kind, $kn\notin\beta$, and since $\beta$ is not contradictory of the second kind, $kn+4\notin\beta$. Thus the following orientation for $\Xi_{k,l}$ restricts to a solution of $\Sigma_\beta$ for any value of $\{ln,\dotsc,ln+4\}\cap\beta$.
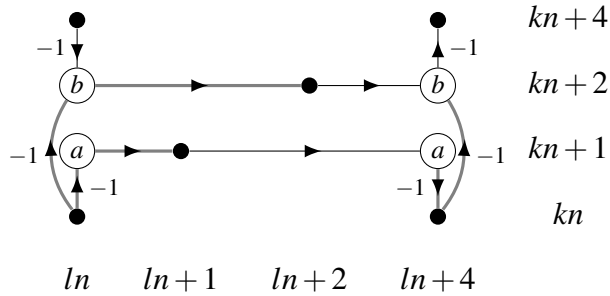
\[
  \beginpgfgraphicnamed{fig74}
  \begin{tikzpicture}[>=latex]
   \matrix (m) [matrix of math nodes, row sep=0.5em, column sep=1em, text height=1.5ex, text depth=0.5ex]
    {  \node[pair](40){};     &                        &                        & \node[pair](44){};     & kn+4 \\
       \node[triple](b20){b}; &                        & \node[pair](22){};     & \node[triple](b24){b}; & kn+2 \\
       \node[triple](a10){a}; & \node[pair](11){};     &                        & \node[triple](a14){a}; & kn+1 \\
       \node[pair](00){};     &                        &                        & \node[pair](04){};     & kn   \\
               ln             &       ln+1             &        ln+2            &        ln+4            &      \\
 };
    \path[-,very thick,font=\scriptsize,gray]
    (a10) edge[-<-=0.6] node[auto,black] {$-1$} (00)
    (a10) edge[->-=0.6] (11)
    (b20) edge[-<-=0.5,bend right=35] node[auto,swap,black] {$-1$} (00)
    (b20) edge[->-=0.55] (22)
    (a14) edge[->-=0.7] node[auto,black,swap] {$-1$} (04)
    (b24) edge[-<-=0.5,bend left=35] node[auto,black] {$-1$} (04)
;
    \path[-,font=\scriptsize]
    (b20) edge[-<-=0.6] node[auto,black] {$-1$} (40)
    (b24) edge[->-=0.7] node[auto,black,swap] {$-1$} (44)
    (b24) edge[-<-=0.55] (22)
    (a14) edge[-<-=0.5] (11)
;
   \end{tikzpicture}
 \endpgfgraphicnamed
\]

In all cases, we oriented the edges connecting to $(ln,kn)$ and $(ln,kn+4)$ away from these relevant pairs. Therefore the constructed solution of $\Xi_{k,l}$ satisfies in all cases (PS).

\subsection*{The patches $\Xi_{k,r+1}$} We assume that $k>0$. The patch $\Xi_{k,r+1}$ is non-empty if and only if $(r+1)n\in\beta$. In this case, also $(r+1)n+1,(r+1)n+2\in\beta$ by property (max) from section \ref{subsection: bases for indecomposables of small defect}. However, $(r+1)n+4\notin\beta$ since $F:\Gamma\to Q$ does not ramify in $(r+1)n+2$. Therefore $\Xi_{k,r+1}$ has the following extremal solution that satisfies (PS).
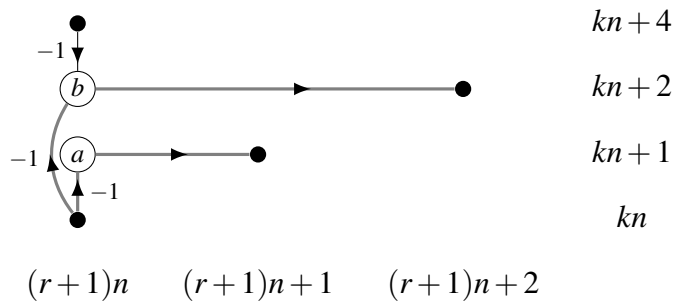
\[
  \beginpgfgraphicnamed{fig75}
  \begin{tikzpicture}[>=latex]
   \matrix (m) [matrix of math nodes, row sep=0.5em, column sep=1em, text height=1.5ex, text depth=0.5ex]
    {  \node[pair](40){};     &                        &                        & kn+4 \\
       \node[triple](b20){b}; &                        & \node[pair](22){};     & kn+2 \\
       \node[triple](a10){a}; & \node[pair](11){};     &                        & kn+1 \\
       \node[pair](00){};     &                        &                        & kn   \\
             (r+1)n           &     (r+1)n+1           &      (r+1)n+2          &      \\
 };
    \path[-,very thick,font=\scriptsize,gray]
    (a10) edge[-<-=0.6] node[auto,black] {$-1$} (00)
    (a10) edge[->-=0.6] (11)
    (b20) edge[-<-=0.6,bend right=35] node[auto,swap,black] {$-1$} (00)
    (b20) edge[->-=0.6] (22)
;
    \path[-,font=\scriptsize]
    (b20) edge[-<-=0.6] node[auto,black] {$-1$} (40)
;
   \end{tikzpicture}
 \endpgfgraphicnamed
\]

\subsection*{The patches $\Xi_{0,l}$} We assume that $l\leq r$. Note that in this case $l$ is even, i.e.\ this case does not occur for $a$ and $b$ replaced by $d$ and $c$. From the description of $\Gamma$ in section \ref{subsection: bases for indecomposables of small defect}, it is visible that $2,4\in\beta$, but $0,1\notin\beta$. Therefore $\Xi_{0,l}$ has the following extremal solution that satisfies (PS).
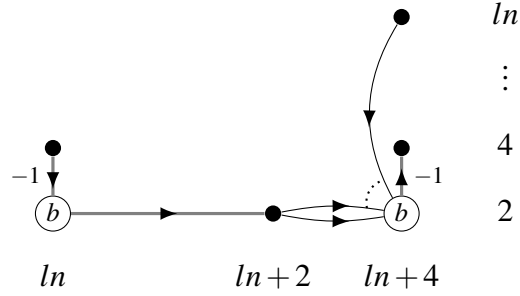
\[
  \beginpgfgraphicnamed{fig76}
  \begin{tikzpicture}[>=latex]
   \matrix (m) [matrix of math nodes, row sep=0.5em, column sep=1em, text height=1.5ex, text depth=0.5ex]
    {                         &                        &                        & \node[pair](ll){};     & ln   \\
                              &                        &                        &                        & \vdots   \\
       \node[pair](40){};     &                        &                        & \node[pair](44){};     & 4 \\
       \node[triple](b20){b}; &                        & \node[pair](22){};     & \node[triple](b24){b}; & 2 \\
               ln             &       \qquad           &        ln+2            &        ln+4            &   \\
 };
    \path[-,very thick,font=\scriptsize,gray]
    (b20) edge[->-=0.55] (22)
    (b20) edge[-<-=0.6] node[auto,black] {$-1$} (40)
    (b24) edge[->-=0.7] node[auto,black,swap] {$-1$} (44)
;
    \path[-,font=\scriptsize]
    (b24) edge[bend left=10,-<-=0.5] (22)
    (b24) edge[bend right=10,-<-=0.5] node[pos=0.15] (b24-22) {} (22)
    (b24) edge[bend left=30,-<-=0.5]  node[pos=0.1] (b24-ll) {} (ll)
;
    \path[dotted,-,thick,font=\scriptsize]
    (b24-22.center) edge[bend left=45] node[above left=-2pt] {} (b24-ll.center)
;
   \end{tikzpicture}
 \endpgfgraphicnamed
\]

\subsection*{The patches $\Xi_{0,r+1}$} If $r$ is odd, then the patch $\Xi_{0,r+1}$ is part of the patchwork, and it is non-empty if and only if $(r+1)n$ is a vertex of $\Gamma$. In this case, $\Xi_{0,r+1}$ has the following extremal solution that satisfies (PS).
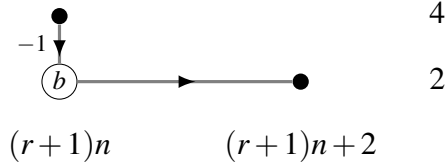
\[
  \beginpgfgraphicnamed{fig77}
  \begin{tikzpicture}[>=latex]
   \matrix (m) [matrix of math nodes, row sep=0.5em, column sep=1em, text height=1.5ex, text depth=0.5ex]
    {  \node[pair](40){};     &                        &                        & 4 \\
       \node[triple](b20){b}; &                        & \node[pair](22){};     & 2 \\
             (r+1)n           &     \           &      (r+1)n+2          &   \\
 };
    \path[-,very thick,font=\scriptsize,gray]
    (b20) edge[->-=0.55] (22)
    (b20) edge[-<-=0.6] node[auto,black] {$-1$} (40)
;
   \end{tikzpicture}
 \endpgfgraphicnamed
\]

This concludes the study of patches w.r.t.\ to the given orientation.

\bigskip

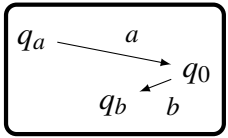
\begin{wrapfigure}{l}{3.0cm}
\vspace{-0,9cm}  
 \beginpgfgraphicnamed{fig78}
  \begin{tikzpicture}[>=latex]
   \matrix (m) [matrix of math nodes, row sep=-1.0em, column sep=1em, text height=2.0ex, text depth=1.0ex]
{ q_a &     &  \   \\
      &     & q_0  \\
   \  & q_b &  \quad   \\
};
     \path[->,font=\footnotesize]
      (m-1-1) edge node[auto] {$a$} (m-2-3)
      (m-2-3) edge node[auto] {$b$} (m-3-2)
;
     \draw[inner sep=10pt,ultra thick,rounded corners]
     (m-1-1.north west) rectangle (m-3-3.south east);
 \end{tikzpicture}
 \endpgfgraphicnamed
\vspace{0,0cm}  
\end{wrapfigure}
\noindent
\begin{minipage}{\minipagewidth}
 \textbf{The patches $\Xi_{k,k}$.} The patches $\Xi_{k,k}$ are the reduced Schubert systems of the full subgraph of $\Gamma$ whose vertices are $\{kn,kn+2,kn+3,kn+4\}$. This subgraph and $\Xi_{k,k}$ take the following shape.
\end{minipage}
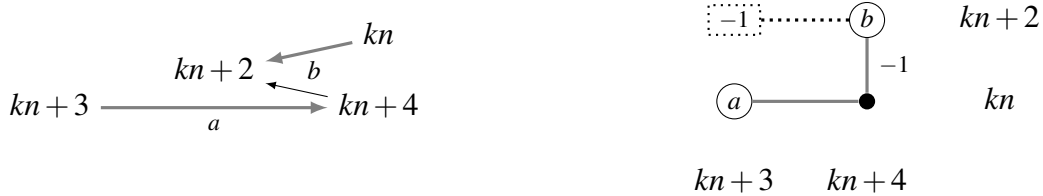
\[
 \beginpgfgraphicnamed{fig112}
  \begin{tikzpicture}[>=latex]
   \matrix (m) [matrix of math nodes, row sep=0em, column sep=2em, text height=1ex, text depth=0ex]
{       &      & kn    \\
        & kn+2 &     \\
   kn+3 &      & kn+4    \\
};
     \path[->,font=\footnotesize]
      (m-3-3) edge node[auto,swap] {$b$} (m-2-2)
;
    \path[->,very thick,font=\scriptsize,gray]
      (m-1-3) edge node[auto] {} (m-2-2)
      (m-3-1) edge node[auto,swap,black] {$a$} (m-3-3)
;
\node at (0,-1.8) {};
 \end{tikzpicture}
 \endpgfgraphicnamed
\hspace{3cm} 
  \beginpgfgraphicnamed{fig79}
  \begin{tikzpicture}[>=latex]
   \matrix (m) [matrix of math nodes, row sep=1.0em, column sep=1em, text height=1.5ex, text depth=0.5ex]
    {   \node[const](22){-1}; & \node[triple](b24){b}; & kn+2  \\
       \node[triple](a03){a}; & \node[pair](04){};     & kn    \\
               kn+3           &       kn+4             &       \\
 };
    \path[-,very thick,font=\scriptsize,gray]
    (a03) edge (04)
    (b24) edge node[auto,black] {$-1$} (04)
;
    \path[dotted,very thick]
    (b24) edge (22)
;
   \end{tikzpicture}
 \endpgfgraphicnamed
\]
A subset $\beta$ of $\cB$ is contradictory of the first kind if $kn\in\beta$ and $kn+2\notin\beta$ or if $kn+3\in\beta$ and $kn+4\notin\beta$. It is contradictory of the second kind if is not contradictory of the first kind and if $kn,kn+2\notin\beta$ and $kn+3,kn+4\in\beta$. If $\Sigma_\beta$ is not contradictory, one of the following cases occurs. The beginning of each case is marked by a box that contains the value of $\ev_\beta(kn,kn+4)$. We will refer to these cases in the study of $\Xi_{k,l,\beta}$ with $k<l$.

\noindent\fbox{$0$}\quad If $\Sigma_\beta$ is not contradictory and if $kn\in\beta$, or if $kn+4\notin\beta$, or if $kn\notin\beta$ and $kn+2,kn+3,kn+4\in\beta$, then $\Xi_{k,k,\beta}$ is empty and $\ev_\beta(kn,kn+4)=0$.
 
\noindent\fbox{$-1$}\quad If $\Sigma_\beta$ is not contradictory and if $kn,kn+2,kn+3\notin\beta$ and $kn+4\in\beta$, then $\Xi_{k,k,\beta}$ is empty and $\ev_\beta(kn,kn+4)=-1$.

\noindent\fbox{$\eta$}\quad If $\Sigma_\beta$ is not contradictory and if $kn,kn+3\notin\beta$ and $kn+2,kn+4\in\beta$, then $\Xi_{k,k,\beta}$ consists of the vertex $(kn,kn+4)$ and $\ev_\beta(kn,kn+4)=\eta$.

\subsection*{The patches $\Xi_{k,l}$} We turn to the general case $0<k<l<r+1$. In this situation, the two last coordinates of the vertices of $\Xi_{k,l}$ vary in $\{kn,\dotsc,kn+4\}\times\{ln,\dotsc,ln+4\}$, and $\Xi_{k,l}$ looks as follows. Since the vertex $(kn,kn+4)$ occurs in the quadratic link $\bigl((a,kn,ln+3),\{(kn,kn+4),(kn+3,ln+3)\}\bigr)$, it is a vertex of $\Xi_{k,l}$.
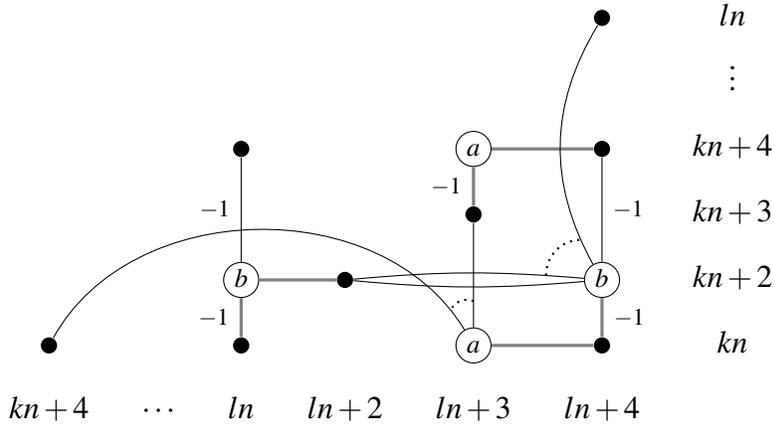
\begin{figure}[h]\label{fig80}
\[
  \beginpgfgraphicnamed{fig80}
  \begin{tikzpicture}[>=latex]
   \matrix (m) [matrix of math nodes, row sep=0.5em, column sep=1em, text height=1.5ex, text depth=0.5ex]
    {              &  &                        &                        &                        & \node[pair](ll){};     & ln   \\
                   &  &                        &                        &                        &                        & \vdots   \\
                   &  & \node[pair](40){};     &                        & \node[triple](a43){a}; & \node[pair](44){};     & kn+4 \\
                   &  &                        &                        & \node[pair](33){};     &                        & kn+3 \\
                   &  & \node[triple](b20){b}; & \node[pair](22){};     &                        & \node[triple](b24){b}; & kn+2 \\
\node[pair](kk){}; &  & \node[pair](00){};     &                        & \node[triple](a03){a}; & \node[pair](04){};     & kn   \\
     kn+4     &\dotsb &         ln             &       ln+2             &        ln+3            &        ln+4            &      \\
 };
    \path[-,very thick,font=\scriptsize,gray]
    (a03) edge (04)
    (a43) edge (44)
    (a43) edge node[auto,black,swap] {$-1$} (33)
    (b20) edge node[auto,swap,black] {$-1$} (00)
    (b20) edge (22)
    (b24) edge node[auto,black] {$-1$} (04)
;
    \path[-,font=\scriptsize]
    (b20) edge node[auto,black] {$-1$} (40)
    (b24) edge node[auto,black,swap] {$-1$} (44)
    (b24) edge[bend left=5] (22)
;
    \path[-,font=\scriptsize]
    (a03) edge node[auto] {} node[pos=0.25] (a03-33) {} (33)
    (a03) edge[bend right=60] node[auto] {} node[pos=0.05] (a03-kk) {} (kk)
    (b24) edge[bend right=5] node[pos=0.15] (b24-22) {} (22)
    (b24) edge[bend left=30]  node[pos=0.1] (b24-ll) {} (ll)
;
    \path[dotted,-,thick,font=\scriptsize]
    (b24-22.center) edge[bend left=45] node[above left=-2pt] {} (b24-ll.center)
    (a03-33.center) edge[bend right=45] node[above left=-2pt] {} (a03-kk.center)
;
   \end{tikzpicture}
 \endpgfgraphicnamed
\]
\end{figure}

Let $\beta$ be a subset of $\cB$ such that $\Sigma_\beta$ is not contradictory. In all cases, we will orientate the edge $\bigl\{(b,kn+2,ln+4),(kn+2,ln+2)\bigr\}$ towards the triple $(b,kn+2,ln+4)$ if $\beta$-relevant. As we did for the previous orientation of the arrows $a$ and $b$, we disregard the vertex $(ln, ln+4)$ and the edge and the quadratic link involving this vertex. The vertex $(kn,kn+4)$ and value $\ev_\beta(kn,kn+4)$ will play, however, an important role for the solutions of the $\beta$-state $\Xi_{k,l,\beta}$. We consider the different outcomes for $\ev_\beta(kn,kn+4)$ in the following.
\noindent\fbox{$0$}\quad In this case, $(kn,kn+4)$ and the quadratic link involving this vertex are $\beta$-trivial. If the triple $(b,kn+2,ln+4)$ is $\beta$-relevant, then $kn+2\notin\beta$. Since $\beta$ is not contradictory of the first kind, also $kn\notin\beta$, and since $\ev_\beta(kn,kn+4)=0$, we conclude that $kn+4\notin\beta$. This means that also the vertex $(kn+4,ln+4)$ and the edge $\big\{(b,kn+2,ln+4),(kn+4,ln+4)\bigl\}$ are $\beta$-relevant. Therefore the following solution, which is extremal except for the edge $\big\{(b,kn+2,ln+4),(kn+4,ln+4)\bigl\}$, will restrict to a solution of $\Xi_{k,l,\beta}$ that satisfies (PS) for all $\beta$ with $\ev_\beta(kn,kn+4)=0$.
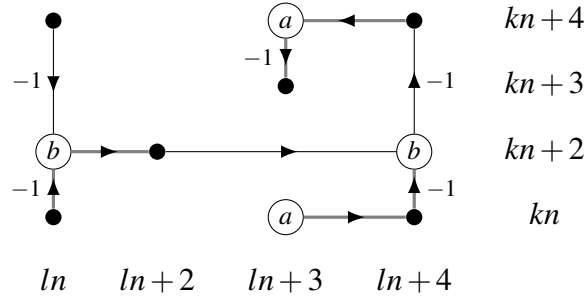
\[
  \beginpgfgraphicnamed{fig81}
  \begin{tikzpicture}[>=latex]
   \matrix (m) [matrix of math nodes, row sep=0.5em, column sep=1em, text height=1.5ex, text depth=0.5ex]
    {              &  & \node[pair](40){};     &                        & \node[triple](a43){a}; & \node[pair](44){};     & kn+4 \\
                   &  &                        &                        & \node[pair](33){};     &                        & kn+3 \\
                   &  & \node[triple](b20){b}; & \node[pair](22){};     &                        & \node[triple](b24){b}; & kn+2 \\
                   &  & \node[pair](00){};     &                        & \node[triple](a03){a}; & \node[pair](04){};     & kn   \\
                   &  &         ln             &       ln+2             &        ln+3            &        ln+4            &      \\
 };
    \path[-,very thick,font=\scriptsize,gray]
    (a03) edge[->-=0.6] (04)
    (a43) edge[-<-=0.5] (44)
    (a43) edge[->-=0.7] node[auto,black,swap] {$-1$} (33)
    (b20) edge[-<-=0.65] node[auto,swap,black] {$-1$} (00)
    (b20) edge[->-=0.6] (22)
    (b24) edge[-<-=0.6] node[auto,black] {$-1$} (04)
;
    \path[-,font=\scriptsize]
    (b20) edge[-<-=0.55] node[auto,black] {$-1$} (40)
    (b24) edge[->-=0.6] node[auto,black,swap] {$-1$} (44)
    (b24) edge[-<-=0.5] (22)
;
   \end{tikzpicture}
 \endpgfgraphicnamed
\]

\noindent\fbox{$-1$}\quad In this case, we substitute the quadratic link involving $(kn,kn+4)$ by the linear link $\bigl\{(a,kn,ln+3),(kn+3,ln+3)\bigr\}$ with weight $-1$. Furthermore, we have $kn+4\in\beta$, which implies that all vertices with vertical coordinate $kn+4$ and the connecting edges and links are $\beta$-trivial. Since $kn,kn+3\notin\beta$, the vertex $(kn+3,ln+3)$ and the edge $\bigl\{(a,kn,ln+3),(kn+3,ln+3)\bigr\}$ are $\beta$-relevant if $(a,kn,ln+3)$ is so. Therefore the following solution, which is extremal except for the edge $\bigl\{(a,kn,ln+3),(kn+3,ln+3)\bigr\}$, restricts to a solution of $\Xi_{k,l,\beta}$ that satisfies (PS) for all choices of $\beta$ with $\ev_\beta(kn,kn+4)=-1$.
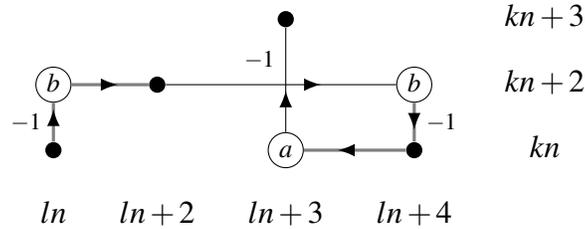
\[
  \beginpgfgraphicnamed{fig82}
  \begin{tikzpicture}[>=latex]
   \matrix (m) [matrix of math nodes, row sep=0.5em, column sep=1em, text height=1.5ex, text depth=0.5ex]
    {              &  &                        &                        & \node[pair](33){};     &                        & kn+3 \\
                   &  & \node[triple](b20){b}; & \node[pair](22){};     &                        & \node[triple](b24){b}; & kn+2 \\
                   &  & \node[pair](00){};     &                        & \node[triple](a03){a}; & \node[pair](04){};     & kn   \\
                   &  &         ln             &       ln+2             &        ln+3            &        ln+4            &      \\
 };
    \path[-,very thick,font=\scriptsize,gray]
    (a03) edge[-<-=0.5] (04)
    (b20) edge[-<-=0.6] node[auto,swap,black] {$-1$} (00)
    (b20) edge[->-=0.6] (22)
    (b24) edge[->-=0.7] node[auto,black] {$-1$} (04)
;
    \path[-,font=\scriptsize]
    (b24) edge[-<-=0.4] (22)
;
    \path[-,font=\scriptsize]
    (a03) edge[->-=0.4] node[above left] {$-1$} node[pos=0.25] (a03-33) {} (33)
;
   \end{tikzpicture}
 \endpgfgraphicnamed
\]

\noindent\fbox{$\eta$}\quad In this case, all vertices with vertical coordinate $kn+2$ or $kn+4$ are $\beta$-trivial. Therefore the following extremal solution restricts to a solution of $\Xi_{k,l,\beta}$ that satisfies (PS) for all $\beta$ with $\ev_\beta(kn,kn+4)=\eta$.
\[
  \beginpgfgraphicnamed{fig83}
  \begin{tikzpicture}[>=latex]
   \matrix (m) [matrix of math nodes, row sep=0.5em, column sep=1em, text height=1.5ex, text depth=0.5ex]
    {              &  &                        &                        & \node[pair](33){};     &                        & kn+3 \\
           \       &  &                        &                        &                        &                        & kn+2 \\
\node[pair](kk){}; &  & \node[pair](00){};     &                        & \node[triple](a03){a}; & \node[pair](04){};     & kn   \\
     kn+4     &\dotsb &         ln             &       ln+2             &        ln+3            &        ln+4            &      \\
 };
    \path[-,very thick,font=\scriptsize,gray]
    (a03) edge[->-=0.6] (04)
;
    \path[-,font=\scriptsize]
    (a03) edge[-<-=0.6] node[auto] {} node[pos=0.25] (a03-33) {} (33)
    (a03) edge[-<-=0.14,bend right=60] node[auto] {} node[pos=0.05] (a03-kk) {} (kk)
;
    \path[dotted,-,thick,font=\scriptsize]
    (a03-33.center) edge[bend right=45] node[above left=-2pt] {} (a03-kk.center)
;
   \end{tikzpicture}
 \endpgfgraphicnamed
\]

\subsection*{The patches $\Xi_{k,r+1}$} We assume that $k>0$. The patch $\Xi_{k,r+1}$ is non-empty if and only if $(r+1)n\in\beta$. In this case, $(r+1)n+2\in\beta$, but $(r+1)n+3\notin\beta$ by property (max) from section \ref{subsection: bases for indecomposables of small defect} and $(r+1)n+4\notin\beta$ since $F:\Gamma\to Q$ does not ramify in $(r+1)n+2$. Therefore $\Xi_{k,r+1}$ has the following extremal solution that satisfies (PS).
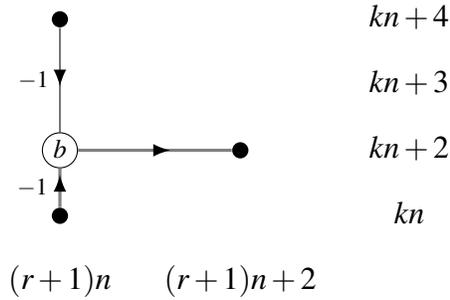
\[
  \beginpgfgraphicnamed{fig84}
  \begin{tikzpicture}[>=latex]
   \matrix (m) [matrix of math nodes, row sep=0.5em, column sep=1em, text height=1.5ex, text depth=0.5ex]
    {  \node[pair](40){};     &                        & kn+4 \\
                              &                        & kn+3 \\
       \node[triple](b20){b}; & \node[pair](22){};     & kn+2 \\
       \node[pair](00){};     &                        & kn   \\
             (r+1)n           &      (r+1)n+2          &      \\
 };
    \path[-,very thick,font=\scriptsize,gray]
    (b20) edge[-<-=0.6] node[auto,swap,black] {$-1$} (00)
    (b20) edge[->-=0.6] (22)
;
    \path[-,font=\scriptsize]
    (b20) edge[-<-=0.6] node[auto,black] {$-1$} (40)
;
   \end{tikzpicture}
 \endpgfgraphicnamed
\]

\subsection*{The patches $\Xi_{0,l}$} We assume that $l\leq r$. From the description of $\Gamma$ in section \ref{subsection: bases for indecomposables of small defect}, it is visible that $2,3,4\in\beta$, but $0\notin\beta$. Therefore $\Xi_{0,l}$ has the following extremal solution that satisfies (PS).
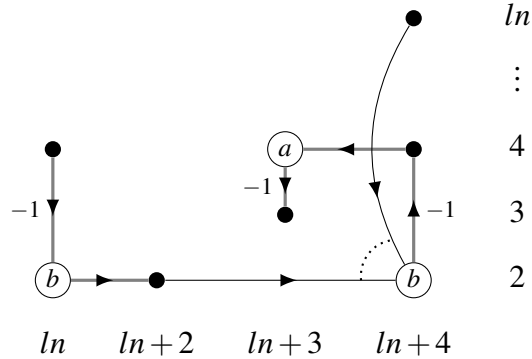
\[
  \beginpgfgraphicnamed{fig85}
  \begin{tikzpicture}[>=latex]
   \matrix (m) [matrix of math nodes, row sep=0.5em, column sep=1em, text height=1.5ex, text depth=0.5ex]
    {                         &                        &                        & \node[pair](ll){};     & ln   \\
                              &                        &                        &                        & \vdots   \\
       \node[pair](40){};     &                        & \node[triple](a43){a}; & \node[pair](44){};     & 4 \\
                              &                        & \node[pair](33){};     &                        & 3 \\
       \node[triple](b20){b}; & \node[pair](22){};     &                        & \node[triple](b24){b}; & 2 \\
               ln             &        ln+2            &         ln+3           &        ln+4            &   \\
 };
    \path[-,very thick,font=\scriptsize,gray]
    (b20) edge[->-=0.55] (22)
    (b20) edge[-<-=0.6] node[auto,black] {$-1$} (40)
    (b24) edge[->-=0.6] node[auto,black,swap] {$-1$} (44)
    (a43) edge[-<-=0.5] (44)
    (a43) edge[->-=0.7] node[auto,black,swap] {$-1$} (33)
;
    \path[-,font=\scriptsize]
    (b24) edge[-<-=0.5] node[pos=0.15] (b24-22) {} (22)
    (b24) edge[-<-=0.35,bend left=30]  node[pos=0.1] (b24-ll) {} (ll)
;
    \path[dotted,-,thick,font=\scriptsize]
    (b24-22.center) edge[bend left=45] node[above left=-2pt] {} (b24-ll.center)
;
   \end{tikzpicture}
 \endpgfgraphicnamed
\]

\subsection*{The patches $\Xi_{0,r+1}$} If $r$ is odd, then the patch $\Xi_{0,r+1}$ is part of the patchwork, and it is non-empty if and only if $(r+1)n$ is a vertex of $\Gamma$. In this case, $\Xi_{0,r+1}$ has the following extremal solution that satisfies (PS).
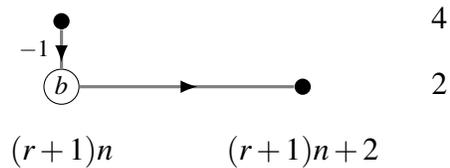
\[
  \beginpgfgraphicnamed{fig86}
  \begin{tikzpicture}[>=latex]
   \matrix (m) [matrix of math nodes, row sep=0.5em, column sep=1em, text height=1.5ex, text depth=0.5ex]
    {  \node[pair](40){};     &                        &                        & 4 \\
       \node[triple](b20){b}; &                        & \node[pair](22){};     & 2 \\
             (r+1)n           &     \           &      (r+1)n+2          &   \\
 };
    \path[-,very thick,font=\scriptsize,gray]
    (b20) edge[->-=0.55] (22)
    (b20) edge[-<-=0.6] node[auto,black] {$-1$} (40)
;
   \end{tikzpicture}
 \endpgfgraphicnamed
\]
This concludes the study of patches w.r.t.\ to the given orientation.

\bigskip

\bigskip

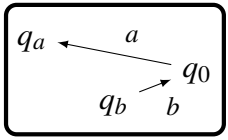
\begin{wrapfigure}{l}{3.0cm}
\vspace{-0,9cm}  
 \beginpgfgraphicnamed{fig87}
  \begin{tikzpicture}[>=latex]
   \matrix (m) [matrix of math nodes, row sep=-1.0em, column sep=1em, text height=2.0ex, text depth=1.0ex]
{ q_a &     &  \   \\
      &     & q_0  \\
   \  & q_b &  \quad   \\
};
     \path[->,font=\footnotesize]
      (m-2-3) edge node[auto,swap] {$a$} (m-1-1)
      (m-3-2) edge node[auto,swap] {$b$} (m-2-3)
;
     \draw[inner sep=10pt,ultra thick,rounded corners]
     (m-1-1.north west) rectangle (m-3-3.south east);
 \end{tikzpicture}
 \endpgfgraphicnamed
\vspace{-0,0cm}  
\end{wrapfigure}
\noindent 
\begin{minipage}{\minipagewidth}
 \textbf{The patches $\Xi_{k,k}$.} The patches $\Xi_{k,k}$ are the reduced Schubert systems of the full subgraph of $\Gamma$ whose vertices are $\{kn,kn+1,kn+2,kn+4\}$. This subgraph and $\Xi_{k,k}$ take the following shape.
\end{minipage}
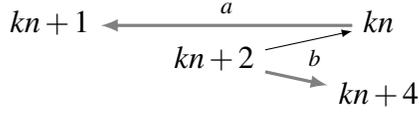
\[
 \beginpgfgraphicnamed{fig88}
  \begin{tikzpicture}[>=latex]
   \matrix (m) [matrix of math nodes, row sep=0em, column sep=2em, text height=1ex, text depth=0ex]
{  kn+1 &      & kn    \\
        & kn+2 &     \\
        &      & kn+4    \\
};
     \path[->,font=\footnotesize]
      (m-2-2) edge node[auto] {} (m-1-3)
;
    \path[->,very thick,font=\scriptsize,gray]
      (m-1-3) edge node[auto,swap,black] {$a$} (m-1-1)
      (m-2-2) edge node[auto,black] {$b$} (m-3-3)
;
\node at (0,-1.8) {};
 \end{tikzpicture}
 \endpgfgraphicnamed
\hspace{3cm} 
  \beginpgfgraphicnamed{fig89}
  
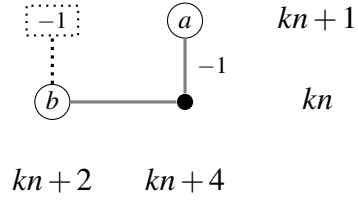
\begin{tikzpicture}[>=latex]
   \matrix (m) [matrix of math nodes, row sep=1.0em, column sep=1em, text height=1.5ex, text depth=0.5ex]
    {   \node[const](22){-1}; & \node[triple](a14){a}; & kn+1  \\
       \node[triple](b24){b}; & \node[pair](04){};     & kn    \\
               kn+2           &       kn+4             &       \\
 };
    \path[-,very thick,font=\scriptsize,gray]
    (a14) edge node[auto,black] {$-1$} (04)
    (b24) edge node[auto,black] {} (04)
;
    \path[dotted,very thick]
    (b24) edge (22)
;
   \end{tikzpicture}
 \endpgfgraphicnamed
\]
A subset $\beta$ of $\cB$ is contradictory of the first kind if $kn+2\in\beta$ and $kn+4\notin\beta$ or if $kn\in\beta$ and $kn+1\notin\beta$. It is contradictory of the second kind if is not contradictory of the first kind and if $kn,kn+1\notin\beta$ and $kn+2,kn+4\in\beta$. If $\Sigma_\beta$ is not contradictory, one of the following cases occurs. 

\noindent\fbox{$0$}\quad If $\Sigma_\beta$ is not contradictory and if $kn\in\beta$, or if $kn+4\notin\beta$, or if $kn,kn+1,kn+2\notin\beta$ and $kn+4\in\beta$, then $\Xi_{k,k,\beta}$ is empty and $\ev_\beta(kn,kn+4)=0$.
 
\noindent\fbox{$1$}\quad If $\Sigma_\beta$ is not contradictory and if $kn\notin\beta$ and $kn+1,kn+2,kn+4\in\beta$, then $\Xi_{k,k,\beta}$ is empty and $\ev_\beta(kn,kn+4)=1$.

\noindent\fbox{$\eta$}\quad If $\Sigma_\beta$ is not contradictory and if $kn,kn+2\notin\beta$ and $kn+1,kn+4\in\beta$, then $\Xi_{k,k,\beta}$ consists of the vertex $(kn,kn+4)$ and $\ev_\beta(kn,kn+4)=\eta$.

\subsection*{The patches $\Xi_{k,l}$} We turn to the general case $0<k<l<r+1$. In this situation, the two last coordinates of the vertices of $\Xi_{k,l}$ vary in $\{kn,\dotsc,kn+4\}\times\{ln,\dotsc,ln+4\}$, and $\Xi_{k,l}$ looks as follows. 
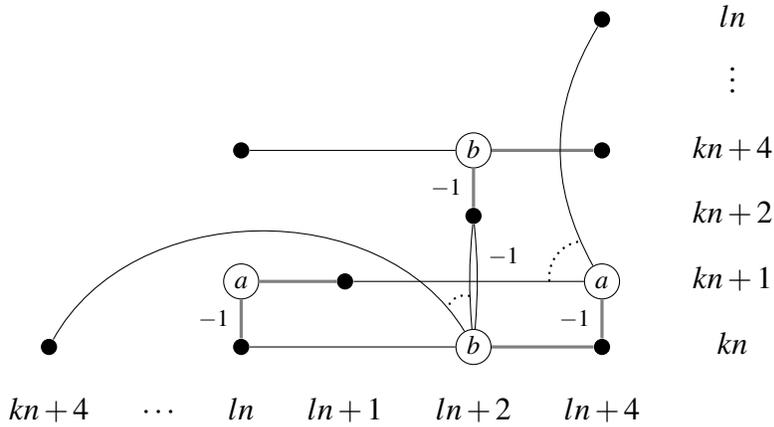
\[
  \beginpgfgraphicnamed{fig99}
  \begin{tikzpicture}[>=latex]
   \matrix (m) [matrix of math nodes, row sep=0.5em, column sep=1em, text height=1.5ex, text depth=0.5ex]
    {              &  &                        &                        &                        & \node[pair](ll){};     & ln   \\
                   &  &                        &                        &                        &                        & \vdots   \\
                   &  & \node[pair](40){};     &                        & \node[triple](b42){b}; & \node[pair](44){};     & kn+4 \\
                   &  &                        &                        & \node[pair](22){};     &                        & kn+2 \\
                   &  & \node[triple](a10){a}; & \node[pair](11){};     &                        & \node[triple](a14){a}; & kn+1 \\
\node[pair](kk){}; &  & \node[pair](00){};     &                        & \node[triple](b02){b}; & \node[pair](04){};     & kn   \\
     kn+4     &\dotsb &         ln             &       ln+1             &        ln+2            &        ln+4            &      \\
 };
    \path[-,very thick,font=\scriptsize,gray]
    (a10) edge (11)
    (a10) edge node[auto,black,swap] {$-1$} (00)
    (a14) edge node[auto,black,swap] {$-1$} (04)
    (b02) edge (04)
    (b42) edge (44)
    (b42) edge node[auto,swap,black] {$-1$} (22)
;
    \path[-,font=\scriptsize]
    (b02) edge (00)
    (b42) edge (40)
    (a14) edge node[auto] {} node[pos=0.15] (a14-11) {} (11)
    (a14) edge[bend left=30] node[auto] {} node[pos=0.1] (a14-ll) {} (ll)
    (b02) edge[bend right=5] node[above right,swap,black] {$-1$} (22)
    (b02) edge[bend left=5] node[pos=0.3] (b02-22) {} (22)
    (b02) edge[bend right=60]  node[pos=0.06] (b02-kk) {} (kk)
;
    \path[dotted,-,thick,font=\scriptsize]
    (b02-22.center) edge[bend right=45] node[above left=-2pt] {} (b02-kk.center)
    (a14-11.center) edge[bend left=45] node[above left=-2pt] {} (a14-ll.center)
;
   \end{tikzpicture}
 \endpgfgraphicnamed
\]

Let $\beta$ be a subset of $\cB$ such that $\Sigma_\beta$ is not contradictory. In all cases, we will orientate the edge $\bigl\{(a,kn+1,ln+4),(kn+1,ln+1)\bigr\}$ towards the triple $(a,kn+1,ln+4)$ if $\beta$-relevant. As we did for the previous orientations of the arrows $a$ and $b$, we disregard the vertex $(ln, ln+4)$ and the edge and the quadratic link involving this vertex. Again, the vertex $(kn,kn+4)$ is important, and we consider the different outcomes for $\ev_\beta(kn,kn+4)$ in the following.

\noindent\fbox{$0$}\quad In this case, $(kn,kn+4)$ and the quadratic link involving this vertex are $\beta$-trivial, and we can omit it from the illustration below. If the triple $(b,kn,ln+2)$ is $\beta$-relevant, then $kn\notin\beta$ and $ln+2\in\beta$. In case $kn+4\notin\beta$, then $kn+2\notin\beta$ since we assume that $\beta$ is not contradictory of the first kind. In case $kn+4\in\beta$, $\ev_\beta(kn,kn+4)=0$ implies that $kn+2\notin\beta$. We conclude that also the vertex $(kn+2,ln+2)$ and the edge $\bigl\{(b,kn,ln+2),(kn+2,ln+2)\bigr\}$ are $\beta$-relevant. Therefore the following orientation, which is extremal except for the edge $\bigl\{(b,kn,ln+2),(kn+2,ln+2)\bigr\}$, restricts to solution of $\Xi_{k,l,\beta}$ that satisfies (PS) for every $\beta$ with $\ev_\beta(kn,kn+4)=0$.
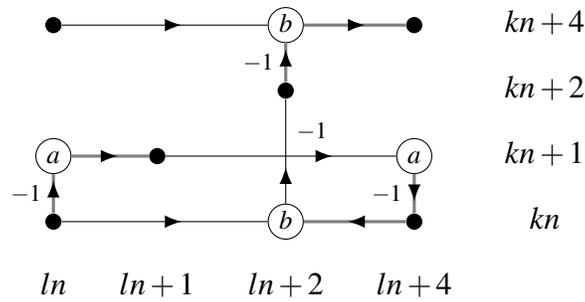
\[
  \beginpgfgraphicnamed{fig100}
  \begin{tikzpicture}[>=latex]
   \matrix (m) [matrix of math nodes, row sep=0.5em, column sep=1em, text height=1.5ex, text depth=0.5ex]
    {              &  & \node[pair](40){};     &                        & \node[triple](b42){b}; & \node[pair](44){};     & kn+4 \\
                   &  &                        &                        & \node[pair](22){};     &                        & kn+2 \\
                   &  & \node[triple](a10){a}; & \node[pair](11){};     &                        & \node[triple](a14){a}; & kn+1 \\
                   &  & \node[pair](00){};     &                        & \node[triple](b02){b}; & \node[pair](04){};     & kn   \\
                   &  &         ln             &       ln+1             &        ln+2            &        ln+4            &      \\
 };
    \path[-,very thick,font=\scriptsize,gray]
    (a10) edge[->-=0.6] (11)
    (a10) edge[-<-=0.6] node[auto,black,swap] {$-1$} (00)
    (a14) edge[->-=0.7] node[auto,black,swap] {$-1$} (04)
    (b02) edge[-<-=0.6] (04)
    (b42) edge[->-=0.6] (44)
    (b42) edge[-<-=0.6] node[auto,swap,black] {$-1$} (22)
;
    \path[-,font=\scriptsize]
    (b02) edge[-<-=0.5] (00)
    (b42) edge[-<-=0.5] (40)
    (a14) edge[-<-=0.35] (11)
    (b02) edge[->-=0.35] node[above right,swap,black] {$-1$} (22)
;
   \end{tikzpicture}
 \endpgfgraphicnamed
\]

\noindent\fbox{$1$}\quad In this case, $kn+1,kn+2,kn+4\in\beta$, and the following extremal solution restricts to a solution of $\Xi_{k,l,\beta}$ that satisfies (PS) for all choices of $\beta$ with $\ev_\beta(kn,kn+4)=1$.
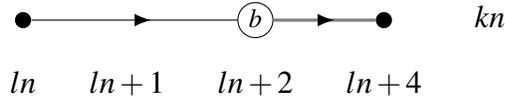
\[
  \beginpgfgraphicnamed{fig101}
  \begin{tikzpicture}[>=latex]
   \matrix (m) [matrix of math nodes, row sep=0.5em, column sep=1em, text height=1.5ex, text depth=0.5ex]
    {              &  & \node[pair](00){};     &                        & \node[triple](b02){b}; & \node[pair](04){};     & kn   \\
                   &  &         ln             &       ln+1             &        ln+2            &        ln+4            &      \\
 };
    \path[-,very thick,font=\scriptsize,gray]
    (b02) edge[->-=0.6] (04)
;
    \path[-,font=\scriptsize]
    (b02) edge[-<-=0.5] (00)
;
   \end{tikzpicture}
 \endpgfgraphicnamed
\]

\noindent\fbox{$\eta$}\quad In this case, $kn+1,kn+4\in\beta$, and the following extremal solution restricts to a solution of $\Xi_{k,l,\beta}$ that satisfies (PS) for all choices of $\beta$ with $\ev_\beta(kn,kn+4)=\eta$.
\[
  \beginpgfgraphicnamed{fig102}
  \begin{tikzpicture}[>=latex]
   \matrix (m) [matrix of math nodes, row sep=0.5em, column sep=1em, text height=1.5ex, text depth=0.5ex]
    {              &  &                        &                        & \node[pair](22){};     &                        & kn+2 \\
                   &  &                        &                        &                        &                        & kn+1 \\
\node[pair](kk){}; &  & \node[pair](00){};     &                        & \node[triple](b02){b}; & \node[pair](04){};     & kn   \\
       kn+4  & \dotsb &         ln             &       ln+1             &        ln+2            &        ln+4            &      \\
 };
    \path[-,very thick,font=\scriptsize,gray]
    (b02) edge[->-=0.6] (04)
;
    \path[-,font=\scriptsize]
    (b02) edge[-<-=0.5] (00)
    (b02) edge[-<-=0.6,bend right=10] node[above right,swap,black] {$-1$} (22)
    (b02) edge[-<-=0.6,bend left=10] node[pos=0.3] (b02-22) {} (22)
    (b02) edge[-<-=0.2,bend right=30]  node[pos=0.06] (b02-kk) {} (kk)
;
    \path[dotted,-,thick,font=\scriptsize]
    (b02-22.center) edge[bend right=45] node[above left=-2pt] {} (b02-kk.center)
;
   \end{tikzpicture}
 \endpgfgraphicnamed
\]

\subsection*{The patches $\Xi_{k,r+1}$} We assume that $k>0$. The patch $\Xi_{k,r+1}$ is non-empty if and only if $(r+1)n\in\beta$. In this case, $(r+1)n+1\in\beta$, but $(r+1)n+2\notin\beta$ by property (max) from section \ref{subsection: bases for indecomposables of small defect} and $(r+1)n+4\notin\beta$ since $F:\Gamma\to Q$ does not ramify in $(r+1)n+2$. Therefore $\Xi_{k,r+1}$ has the following extremal solution that satisfies (PS).
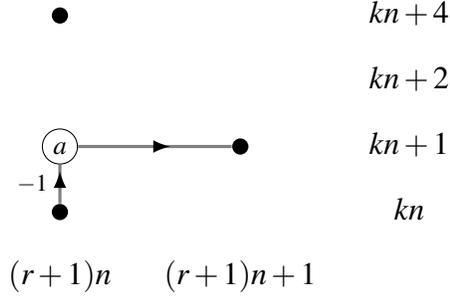
\[
  \beginpgfgraphicnamed{fig103}
  \begin{tikzpicture}[>=latex]
   \matrix (m) [matrix of math nodes, row sep=0.5em, column sep=1em, text height=1.5ex, text depth=0.5ex]
    {  \node[pair](40){};     &                        & kn+4 \\
                              &                        & kn+2 \\
       \node[triple](a10){a}; & \node[pair](11){};     & kn+1 \\
       \node[pair](00){};     &                        & kn   \\
             (r+1)n           &      (r+1)n+1          &      \\
 };
    \path[-,very thick,font=\scriptsize,gray]
    (a10) edge[-<-=0.6] node[auto,swap,black] {$-1$} (00)
    (a10) edge[->-=0.6] (11)
;
   \end{tikzpicture}
 \endpgfgraphicnamed
\]

\subsection*{The patches $\Xi_{0,l}$} We assume that $l\leq r$. In this case, $4\in\beta$, but $0,1,2\notin\beta$. Therefore $\Xi_{0,l}$ has the following extremal solution that satisfies (PS).
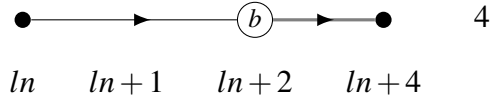
\[
  \beginpgfgraphicnamed{fig104}
  \begin{tikzpicture}[>=latex]
   \matrix (m) [matrix of math nodes, row sep=0.5em, column sep=1em, text height=1.5ex, text depth=0.5ex]
    {  \node[pair](40){};     &                        & \node[triple](b42){b}; & \node[pair](44){};     & 4 \\
               ln             &        ln+1            &         ln+2           &        ln+4            &   \\
 };
    \path[-,very thick,font=\scriptsize,gray]
    (b42) edge[->-=0.6] (44)
;
    \path[-,font=\scriptsize]
    (b42) edge[-<-=0.5] (40)
;
   \end{tikzpicture}
 \endpgfgraphicnamed
\]

\subsection*{The patches $\Xi_{0,r+1}$} If $r$ is odd, then the patch $\Xi_{0,r+1}$ is part of the patchwork, and it is non-empty if and only if $(r+1)n$ is a vertex of $\Gamma$. In this case, $\Xi_{0,r+1}$ consists of the vertex $(2,(r+1)n)$ and is therefore trivially solvable with (PS).

\bigskip

\noindent This concludes the study of patches w.r.t.\ to the given orientation.

\bigskip

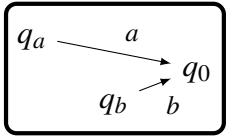
\begin{wrapfigure}{l}{3.0cm}
\vspace{-0,9cm}  
 \beginpgfgraphicnamed{fig105}
  \begin{tikzpicture}[>=latex]
   \matrix (m) [matrix of math nodes, row sep=-1.0em, column sep=1em, text height=2.0ex, text depth=1.0ex]
{ q_a &     &  \   \\
      &     & q_0  \\
   \  & q_b &  \quad   \\
};
     \path[->,font=\footnotesize]
      (m-1-1) edge node[auto] {$a$} (m-2-3)
      (m-3-2) edge node[auto,swap] {$b$} (m-2-3)
;
     \draw[inner sep=10pt,ultra thick,rounded corners]
     (m-1-1.north west) rectangle (m-3-3.south east);
 \end{tikzpicture}
 \endpgfgraphicnamed
\vspace{-0,0cm}  
\end{wrapfigure}
\noindent
\begin{minipage}{\minipagewidth}
 \textbf{The patches $\Xi_{k,k}$.} The patches $\Xi_{k,k}$ are the reduced Schubert systems of the full subgraph of $\Gamma$ whose vertices are $\{kn,kn+2,kn+3,kn+4\}$. This subgraph and $\Xi_{k,k}$ take the following shape.
\end{minipage}
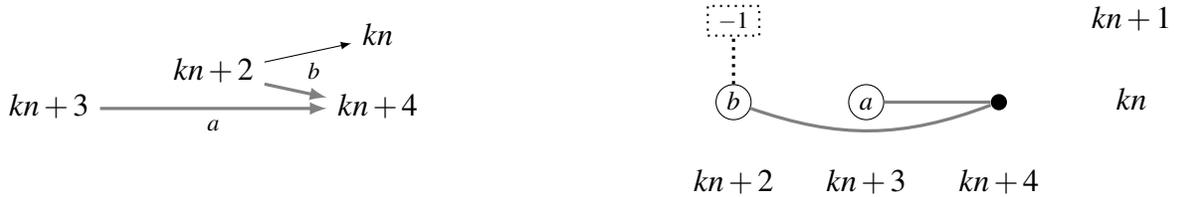
\[
 \beginpgfgraphicnamed{fig106}
  \begin{tikzpicture}[>=latex]
   \matrix (m) [matrix of math nodes, row sep=0em, column sep=2em, text height=1ex, text depth=0ex]
{       &      & kn    \\
        & kn+2 &     \\
   kn+3 &      & kn+4    \\
};
     \path[->,font=\footnotesize]
      (m-2-2) edge node[auto] {} (m-1-3)
;
    \path[->,very thick,font=\scriptsize,gray]
      (m-3-1) edge node[auto,swap,black] {$a$} (m-3-3)
      (m-2-2) edge node[auto,black] {$b$} (m-3-3)
;
\node at (0,-1.8) {};
 \end{tikzpicture}
 \endpgfgraphicnamed
\hspace{3cm} 
  \beginpgfgraphicnamed{fig107}
  \begin{tikzpicture}[>=latex]
   \matrix (m) [matrix of math nodes, row sep=1.0em, column sep=1em, text height=1.5ex, text depth=0.5ex]
    {   \node[const](22){-1}; &                        &                        & kn+1  \\
       \node[triple](b24){b}; & \node[triple](a03){a}; & \node[pair](04){};     & kn    \\
              kn+2            &         kn+3           &       kn+4             &       \\
 };
    \path[-,very thick,font=\scriptsize,gray]
    (a03) edge (04)
    (b24) edge[bend right=20] node[auto,black] {} (04)
;
    \path[dotted,very thick]
    (b24) edge (22)
;
   \end{tikzpicture}
 \endpgfgraphicnamed
\]
A subset $\beta$ of $\cB$ is contradictory of the first kind if $kn+2\in\beta$ and $kn+4\notin\beta$ or if $kn+3\in\beta$ and $kn+4\notin\beta$. It is contradictory of the second kind if is not contradictory of the first kind and if $kn\notin\beta$ and $kn+2,kn+3,kn+4\in\beta$. If $\Sigma_\beta$ is not contradictory, one of the following cases occurs.

\noindent\fbox{$0$}\quad If $\Sigma_\beta$ is not contradictory and if $kn\in\beta$, or if $kn+4\notin\beta$, or if $kn,kn+2\notin\beta$ and $kn+3,kn+4\in\beta$, then $\Xi_{k,k,\beta}$ is empty and $\ev_\beta(kn,kn+4)=0$.
 
\noindent\fbox{$1$}\quad If $\Sigma_\beta$ is not contradictory and if $kn,kn+3\notin\beta$ and $kn+2,kn+4\in\beta$, then $\Xi_{k,k,\beta}$ is empty and $\ev_\beta(kn,kn+4)=1$.

\noindent\fbox{$\eta$}\quad If $\Sigma_\beta$ is not contradictory and if $kn,kn+2,kn+3\notin\beta$ and $kn+4\in\beta$, then $\Xi_{k,k,\beta}$ consists of the vertex $(kn,kn+4)$ and $\ev_\beta(kn,kn+4)=\eta$.

\subsection*{The patches $\Xi_{k,l}$} We turn to the general case $0<k<l<r+1$. In this situation, the two last coordinates of the vertices of $\Xi_{k,l}$ vary in $\{kn,\dotsc,kn+4\}\times\{ln,\dotsc,ln+4\}$, and $\Xi_{k,l}$ looks as follows. Note that the vertex $(ln,ln+4)$ is not in $\Xi_{k,l}$ since it is not the base vertex of any quadratic link in this patch.
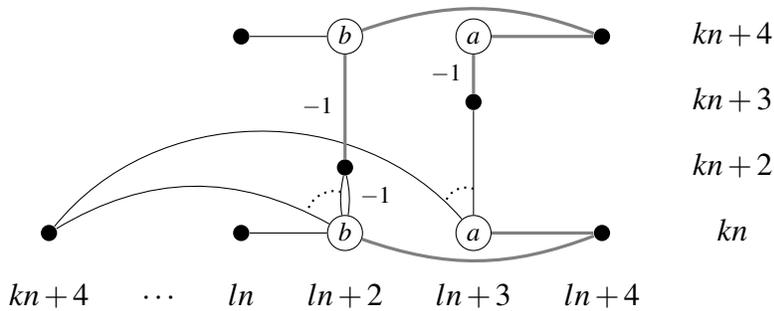
\[
  \beginpgfgraphicnamed{fig108}
  \begin{tikzpicture}[>=latex]
   \matrix (m) [matrix of math nodes, row sep=0.5em, column sep=1em, text height=1.5ex, text depth=0.5ex]
    {              &  & \node[pair](40){};     & \node[triple](b42){b}; & \node[triple](a43){a}; & \node[pair](44){};     & kn+4 \\
                   &  &                        &                        & \node[pair](33){};     &                        & kn+3 \\
                   &  &                        & \node[pair](22){};     &                        &                        & kn+2 \\
\node[pair](kk){}; &  & \node[pair](00){};     & \node[triple](b02){b}; & \node[triple](a03){a}; & \node[pair](04){};     & kn   \\
     kn+4     &\dotsb &         ln             &       ln+2             &        ln+3            &        ln+4            &      \\
 };
    \path[-,font=\scriptsize]
    (b02) edge (00)
    (b02) edge[bend right=10] node[auto,black,swap] {$-1$} (22)
    (b42) edge (40)
;
    \path[-,font=\scriptsize]
    (a03) edge node[auto] {} node[pos=0.25] (a03-33) {} (33)
    (a03) edge[bend right=50] node[auto] {} node[pos=0.05] (a03-kk) {} (kk)
    (b02) edge[bend left=10] node[pos=0.6] (b02-22) {} (22)
    (b02) edge[bend right=30]  node[pos=0.08] (b02-kk) {} (kk)
;
    \path[dotted,-,thick,font=\scriptsize]
    (b02-22.center) edge[bend right=45] node[above left=-2pt] {} (b02-kk.center)
    (a03-33.center) edge[bend right=45] node[above left=-2pt] {} (a03-kk.center)
;
    \path[-,very thick,font=\scriptsize,gray]
    (a03) edge (04)
    (a43) edge (44)
    (a43) edge node[auto,black,swap] {$-1$} (33)
    (b02) edge[bend left=-20] (04)
    (b42) edge node[auto,swap,black] {$-1$} (22)
    (b42) edge[bend right=-20] node[auto,black] {} (44)
;
   \end{tikzpicture}
 \endpgfgraphicnamed
\]

Let $\beta$ be a subset of $\cB$ such that $\Sigma_\beta$ is not contradictory. We consider the different outcomes for $\ev_\beta(kn,kn+4)$ in the following.

\noindent\fbox{$0$}\quad In this case, $(kn,kn+4)$ and the quadratic links involving this vertex are $\beta$-trivial, and we omit these from the following illustration. If the triple $(b,kn,ln+2)$ is $\beta$-relevant, then $kn\notin\beta$ and $ln+2\in\beta$. In case $kn+4\notin\beta$, then $kn+2\notin\beta$ since we assume that $\beta$ is not contradictory of the first kind. In case $kn+4\in\beta$, $\ev_\beta(kn,kn+4)=0$ implies that $kn+2\notin\beta$. We conclude that also the vertex $(kn+2,ln+2)$ and the edge $\bigl\{(b,kn,ln+2),(kn+2,ln+2)\bigr\}$ are $\beta$-relevant. Therefore the following orientation, which is extremal except for the edge $\bigl\{(b,kn,ln+2),(kn+2,ln+2)\bigr\}$, restricts to solution of $\Xi_{k,l,\beta}$ that satisfies (PS) for every $\beta$ with $\ev_\beta(kn,kn+4)=0$.
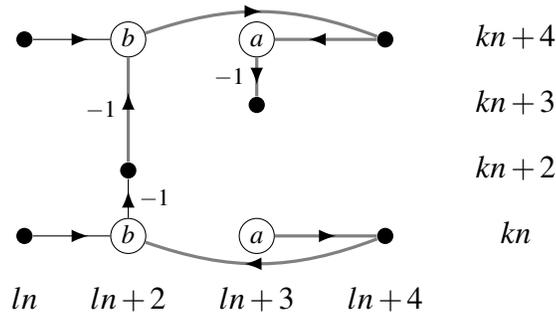
\[
  \beginpgfgraphicnamed{fig109}
  \begin{tikzpicture}[>=latex]
   \matrix (m) [matrix of math nodes, row sep=0.5em, column sep=1em, text height=1.5ex, text depth=0.5ex]
    {              &  & \node[pair](40){};     & \node[triple](b42){b}; & \node[triple](a43){a}; & \node[pair](44){};     & kn+4 \\
                   &  &                        &                        & \node[pair](33){};     &                        & kn+3 \\
                   &  &                        & \node[pair](22){};     &                        &                        & kn+2 \\
                   &  & \node[pair](00){};     & \node[triple](b02){b}; & \node[triple](a03){a}; & \node[pair](04){};     & kn   \\
                   &  &         ln             &       ln+2             &        ln+3            &        ln+4            &      \\
 };
    \path[-,font=\scriptsize]
    (b02) edge[-<-=0.5] (00)
    (b02) edge[->-=0.7] node[auto,black,swap] {$-1$} (22)
    (b42) edge[-<-=0.5] (40)
;
    \path[-,very thick,font=\scriptsize,gray]
    (a03) edge[->-=0.6] (04)
    (a43) edge[-<-=0.5] (44)
    (a43) edge[->-=0.7] node[auto,black,swap] {$-1$} (33)
    (b02) edge[-<-=0.5,bend left=-20] (04)
    (b42) edge[-<-=0.5] node[auto,swap,black] {$-1$} (22)
    (b42) edge[->-=0.5,bend right=-20] node[auto,black] {} (44)
;
   \end{tikzpicture}
 \endpgfgraphicnamed
\]

\noindent\fbox{$1$}\quad The vertices with vertical coordinate $kn+2$ or $kn+4$ are $\beta$-trivial. If the triple $(a,kn,ln+3)$ is $\beta$-relevant, then $ln+3\in\beta$ and thus the vertex $(kn+3,ln+3)$ and the edge $\bigl\{(a,kn,ln+3),(kn+3,ln+3)\bigr\}$ are also $\beta$-relevant since $kn+3\notin\beta$ if $\ev_\beta(kn,kn+4)=1$. Therefore the following solution restricts to a solution of $\Xi_{k,l,\beta}$ that satisfies (PS) for all $\beta$ with $\ev_\beta(kn,kn+4)=1$.
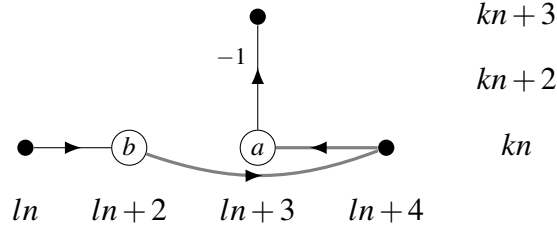
\[
  \beginpgfgraphicnamed{fig110}
  \begin{tikzpicture}[>=latex]
   \matrix (m) [matrix of math nodes, row sep=0.5em, column sep=1em, text height=1.5ex, text depth=0.5ex]
    {              &  &                        &                        & \node[pair](33){};     &                        & kn+3 \\
                   &  &                        &                        &                        &                        & kn+2 \\
                   &  & \node[pair](00){};     & \node[triple](b02){b}; & \node[triple](a03){a}; & \node[pair](04){};     & kn   \\
                   &  &         ln             &       ln+2             &        ln+3            &        ln+4            &      \\
 };
    \path[-,very thick,font=\scriptsize,gray]
    (a03) edge[-<-=0.5] (04)
    (b02) edge[->-=0.5,bend left=-20] (04)
;
    \path[-,font=\scriptsize]
    (b02) edge[-<-=0.6] (00)
    (a03) edge[->-=0.6] node[above left] {$-1$} node[pos=0.25] (a03-33) {} (33)
;
   \end{tikzpicture}
 \endpgfgraphicnamed
\]

\noindent\fbox{$\eta$}\quad In this case $kn+4\notin\beta$, and the $\beta$-state can be calculated from
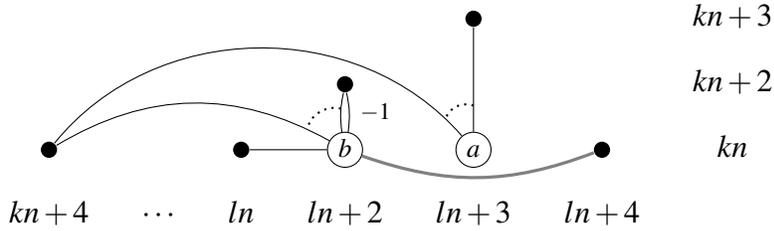
\[
  \beginpgfgraphicnamed{fig111}
  \begin{tikzpicture}[>=latex]
   \matrix (m) [matrix of math nodes, row sep=0.5em, column sep=1em, text height=1.5ex, text depth=0.5ex]
    {              &  &                        &                        & \node[pair](33){};     &                        & kn+3 \\
                   &  &                        & \node[pair](22){};     &                        &                        & kn+2 \\
\node[pair](kk){}; &  & \node[pair](00){};     & \node[triple](b02){b}; & \node[triple](a03){a}; & \node[pair](04){};     & kn   \\
     kn+4     &\dotsb &         ln             &       ln+2             &        ln+3            &        ln+4            &      \\
 };
    \path[-,font=\scriptsize]
    (b02) edge (00)
    (b02) edge[bend right=10] node[auto,black,swap] {$-1$} (22)
;
    \path[-,font=\scriptsize]
    (a03) edge node[auto] {} node[pos=0.25] (a03-33) {} (33)
    (a03) edge[bend right=50] node[auto] {} node[pos=0.05] (a03-kk) {} (kk)
    (b02) edge[bend left=10] node[pos=0.6] (b02-22) {} (22)
    (b02) edge[bend right=30]  node[pos=0.08] (b02-kk) {} (kk)
;
    \path[dotted,-,thick,font=\scriptsize]
    (b02-22.center) edge[bend right=45] node[above left=-2pt] {} (b02-kk.center)
    (a03-33.center) edge[bend right=45] node[above left=-2pt] {} (a03-kk.center)
;
    \path[-,very thick,font=\scriptsize,gray]
    (b02) edge[bend left=-20] (04)
;
   \end{tikzpicture}
 \endpgfgraphicnamed
\]
This system is not solvable, and it occurs indeed as a $\beta$-state $\Xi_{k,l,\beta}$ if $ln,ln+2,ln+3,ln+4\in\beta$. We use the following variable transformation to see that the Schubert cell $C_\beta^M$ is an affine space. If we replace 
\begin{align*}
 w_{kn+2,ln+2} && \text{by} && \tilde w_{kn+2,ln+2} &= w_{kn+2,ln+2} - w_{kn+2,kn+2}         \\
 w_{kn,ln+4}   && \text{by} && \tilde w_{kn,ln+4} &= w_{kn,ln+4} + w_{kn,kn+4}w_{kn+3,ln+3}    
\end{align*}
then the equations $\overline E(b,kn,ln+2)$ and $\overline E(a,kn,ln+3)$ become
\begin{align*}
 \widetilde E(b,kn,ln+2) &= \tilde w_{kn,ln+4} + \text{(quadratic terms)} \\
 \widetilde E(a,kn,ln+3) &= -w_{kn+2,ln+2} + \tilde w_{kn,ln} + \tilde w_{kn,ln+4} - w_{kn,kn+4} \tilde w_{kn+3,ln+3} + \text{(quadratic terms)} 
\end{align*}
where the ``quadratic terms'' refer to quadratic terms in variables $w_{i,j}$ whose index $(i,j)$ is not a vertex of $\Xi_{k,l}$. Therefore these terms do not contribute to the modification of the patch $\Xi_{k,l}$ that results from this variable transformation. Note that since the modified variables $w_{kn+2,ln+2}$ and $w_{kn,ln+4}$ do not occur as the end of an arrow of any solution of any other patch, this variable transformation does not have an influence on these solutions.

The modified patch $\widetilde \Xi_{k,l}$ that we obtain from this variable transformation (with the $\beta$-trivial vertices with vertical coordinate $kn+4$ omitted) is as follows.
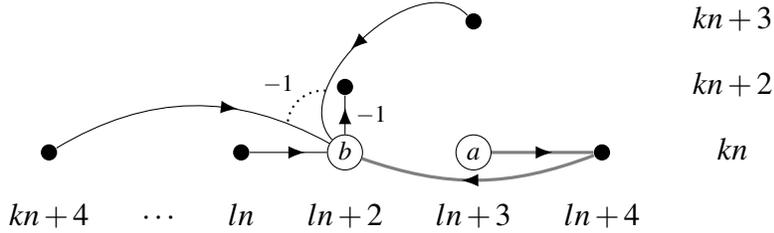
\[
  \beginpgfgraphicnamed{fig113}
  \begin{tikzpicture}[>=latex]
   \matrix (m) [matrix of math nodes, row sep=0.5em, column sep=1em, text height=1.5ex, text depth=0.5ex]
    {              &  &                        &                        & \node[pair](33){};     &                        & kn+3 \\
                   &  &                        & \node[pair](22){};     &                        &                        & kn+2 \\
\node[pair](kk){}; &  & \node[pair](00){};     & \node[triple](b02){b}; & \node[triple](a03){a}; & \node[pair](04){};     & kn   \\
     kn+4     &\dotsb &         ln             &       ln+2             &        ln+3            &        ln+4            &      \\
 };
    \path[-,font=\scriptsize]
    (b02) edge[-<-=0.5] (00)
    (b02) edge[->-=0.7] node[auto,black,swap] {$-1$} (22)
;
    \path[-,font=\scriptsize]
    (b02) edge[-<-=0.5,bend left=90] node[auto] {} node[pos=0.25] (b02-33) {} (33)
    (b02) edge[-<-=0.4,bend right=30]  node[pos=0.15] (b02-kk) {} (kk)
;
    \path[dotted,-,thick,font=\scriptsize]
    (b02-33.center) edge[bend right=45] node[above left=-2pt] {$-1$} (b02-kk.center)
;
    \path[-,very thick,font=\scriptsize,gray]
    (a03) edge[->-=0.6] (04)
    (b02) edge[-<-=0.5,bend left=-20] (04)
;
   \end{tikzpicture}
 \endpgfgraphicnamed
\]
The indicated orientation shows that we can solve the modifications equations $\widetilde E(a,kn,ln+3)$ and $\widetilde E(b,kn,ln+2)$ in the linear terms $\tilde w_{kn,ln+4}$ and $w_{kn+2,ln+2}$, respectively. Note that if $(a,kn,ln+3)$ is $\beta$-relevant, then $(kn,ln+4)$ is $\beta$-relevant since the connecting edge $\bigl\{(a,kn,ln+3),(kn,ln+4)\bigr\}$ is extremal. If $(b,kn,ln+2)$ is $\beta$-relevant, then $(kn+2,ln+2)$ is $\beta$-relevant since $kn+2\notin\beta$ if $\ev_\beta(kn,kn+4)=\eta$.

\subsection*{The patches $\Xi_{k,r+1}$} We assume that $k>0$. The patch $\Xi_{k,r+1}$ is non-empty if and only if $(r+1)n\in\beta$. In this case, $(r+1)n+2,(r+1)n+3\notin\beta$ by property (max) from section \ref{subsection: bases for indecomposables of small defect} and $(r+1)n+4\notin\beta$ since $F:\Gamma\to Q$ does not ramify in $(r+1)n+2$. Therefore $\Xi_{k,r+1}$ consists of the two pairs $(kn,ln)$ and $(kn+4,ln)$ and is trivially solvable with (PS).

\subsection*{The patches $\Xi_{0,l}$} We assume that $l\leq r$. We have $3,4\in\beta$, but $0,2\notin\beta$. Therefore $\Xi_{0,l}$ has the following extremal solution that satisfies (PS).
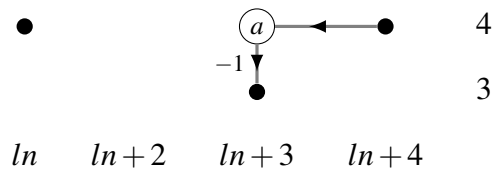
\[
  \beginpgfgraphicnamed{fig114}
  \begin{tikzpicture}[>=latex]
   \matrix (m) [matrix of math nodes, row sep=0.5em, column sep=1em, text height=1.5ex, text depth=0.5ex]
    {  \node[pair](40){};     &                        & \node[triple](a43){a}; & \node[pair](44){};     & 4 \\
                              &                        & \node[pair](33){};     &                        & 3 \\
               ln             &        ln+2            &         ln+3           &        ln+4            &   \\
 };
    \path[-,very thick,font=\scriptsize,gray]
    (a43) edge[-<-=0.5] (44)
    (a43) edge[->-=0.7] node[auto,black,swap] {$-1$} (33)
;
   \end{tikzpicture}
 \endpgfgraphicnamed
\]

\subsection*{The patches $\Xi_{0,r+1}$} If $r$ is odd, then the patch $\Xi_{0,r+1}$ is part of the patchwork, and it is non-empty if and only if $(r+1)n$ is a vertex of $\Gamma$. In this case, $\Xi_{0,r+1}$ consists of the pair $(kn+4,ln)$ and is therefore trivially solvable.

\bigskip

\noindent
This concludes the proof of Theorem \ref{thm: the main theorem} in the case of an indecomposable representation of defect $-1$.

\bigskip


\subsection*{Tube of rank $n-2$}

The Schubert systems of a representation $M$ in an exceptional tube of rank $n-2$ w.r.t. to the basis $\cB$ as described in section \ref{subsection: bases for indecomposables of small defect} behave very similar to the Schubert system in the defect $-1$ case. The coefficient quiver of $M$ deviates from the coefficient quiver for defect $-1$ only in the vertices $1,\dotsc,n$ and the arrows in between. Therefore we can define a patchwork $\{\Xi_{k,l}\}\cup\{\Pi_{k,l}\}$ analogous to the defect $-1$-case, and only the patches of the form $\Pi_{1,l}$ and $\Xi_{0,l}$ possibly differ. In the following, we will calculate these patches and show that they have extremal solutions satisfying (PS).

A patch $\Pi_{1,l}$ is equal to the corresponding patch for defect $-1$ if and only if $4$ is a vertex of $\Gamma$. Otherwise, $\Pi_{1,l}$ is still a path, and property (min) from section \ref{subsection: bases for indecomposables of small defect} guarantees that the end of this path $\Pi_{1,l}$, i.e.\ its $1$-valent vertex that is not part of any other patch, is a relevant pair and not a relevant triple. This means that $\Pi_{1,l}$ is an extremal path, which can be oriented in the same direction as in the proof for defect $-1$, cf.\ Figure \ref{fig35}.

The patches $\Xi_{0,l}$ are non-empty if and only if $4$ is a vertex of $\Gamma$. In this case $\Xi_{0,r+1}$ consists of the single vertex $(4,(r+1)n)$, provided $r+1$ is even and the patch $\Xi_{0,r+1}$ exists. Therefore $\Xi_{0,r+1}$ is trivially solvable.

We continue to inspect the patches $\Xi_{0,l}$ for $0<l<r+1$. We illustrate $\Xi_{0,l}$ together with an extremal solution satisfying (PS) to the right of a symbol indicating the orientation of $a$ and $b$. 

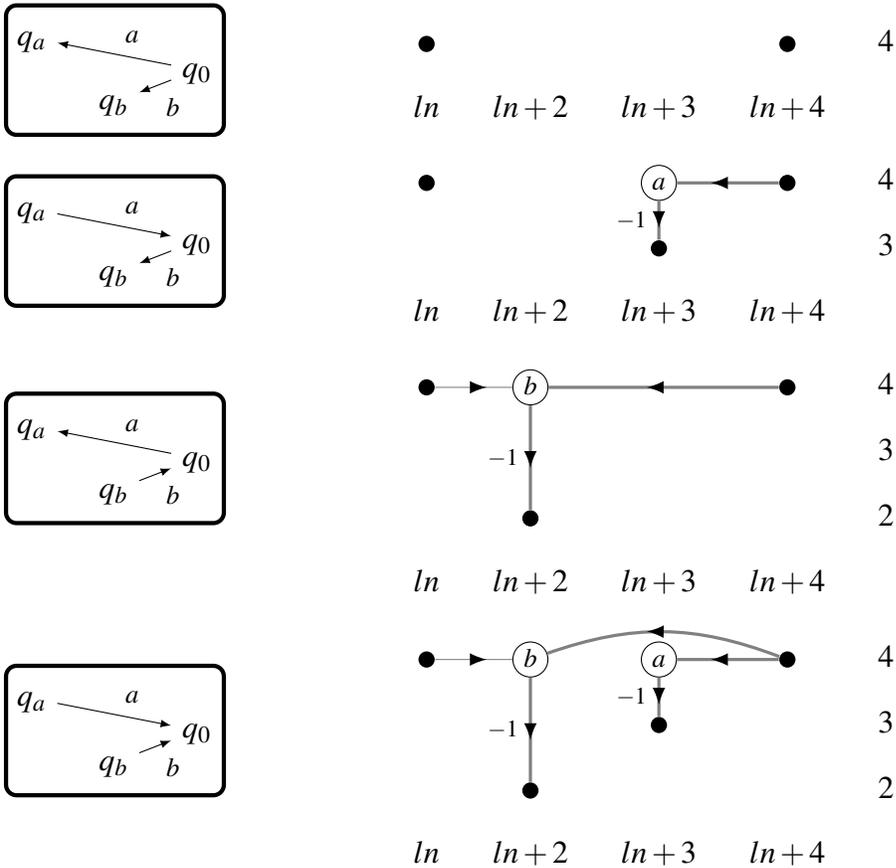
\enlargethispage{2\baselineskip}
\noindent
\(
\beginpgfgraphicnamed{fig116}
  \begin{tikzpicture}[>=latex]
   \matrix (m) [matrix of math nodes, row sep=-1.0em, column sep=1em, text height=2.0ex, text depth=1.0ex]
{ q_a &     &  \   \\
      &     & q_0  \\
   \  & q_b &  \quad   \\
};
     \path[->,font=\footnotesize]
      (m-2-3) edge node[auto,swap] {$a$} (m-1-1)
      (m-2-3) edge node[auto] {$b$} (m-3-2)
;
     \draw[inner sep=10pt,ultra thick,rounded corners]
     (m-1-1.north west) rectangle (m-3-3.south east);
     \node at (0,-0.0) {};
 \end{tikzpicture}
\endpgfgraphicnamed
\hspace{2cm}
  \beginpgfgraphicnamed{fig117}
  \begin{tikzpicture}[>=latex]
   \matrix (m) [matrix of math nodes, row sep=0.5em, column sep=1em, text height=1.5ex, text depth=0.5ex]
    {  \node[pair](40){};     &                        &                        & \node[pair](44){};     & 4 \\
               ln             &        ln+2            &         ln+3           &        ln+4            &   \\
 };
   \end{tikzpicture}
 \endpgfgraphicnamed
\)
\\
\(
\beginpgfgraphicnamed{fig118}
  \begin{tikzpicture}[>=latex]
   \matrix (m) [matrix of math nodes, row sep=-1.0em, column sep=1em, text height=2.0ex, text depth=1.0ex]
{ q_a &     &  \   \\
      &     & q_0  \\
   \  & q_b &  \quad   \\
};
     \path[->,font=\footnotesize]
      (m-1-1) edge node[auto] {$a$} (m-2-3)
      (m-2-3) edge node[auto] {$b$} (m-3-2)
;
     \draw[inner sep=10pt,ultra thick,rounded corners]
     (m-1-1.north west) rectangle (m-3-3.south east);
     \node at (0,-1.3) {};
 \end{tikzpicture}
\endpgfgraphicnamed
\hspace{2cm}
  \beginpgfgraphicnamed{fig119}
  \begin{tikzpicture}[>=latex]
   \matrix (m) [matrix of math nodes, row sep=0.5em, column sep=1em, text height=1.5ex, text depth=0.5ex]
    {  \node[pair](40){};     &                        & \node[triple](a43){a}; & \node[pair](44){};     & 4 \\
                              &                        & \node[pair](33){};     &                        & 3 \\
               ln             &        ln+2            &         ln+3           &        ln+4            &   \\
 };
    \path[-,very thick,font=\scriptsize,gray]
    (a43) edge[-<-=0.5] (44)
    (a43) edge[->-=0.7] node[auto,black,swap] {$-1$} (33)
;
   \end{tikzpicture}
 \endpgfgraphicnamed
\)
\\
\(
\beginpgfgraphicnamed{fig120}
  \begin{tikzpicture}[>=latex]
   \matrix (m) [matrix of math nodes, row sep=-1.0em, column sep=1em, text height=2.0ex, text depth=1.0ex]
{ q_a &     &  \   \\
      &     & q_0  \\
   \  & q_b &  \quad   \\
};
     \path[->,font=\footnotesize]
      (m-2-3) edge node[auto,swap] {$a$} (m-1-1)
      (m-3-2) edge node[auto,swap] {$b$} (m-2-3)
;
     \draw[inner sep=10pt,ultra thick,rounded corners]
     (m-1-1.north west) rectangle (m-3-3.south east);
     \node at (0,-2.0) {};
 \end{tikzpicture}
\endpgfgraphicnamed
\hspace{2cm}
  \beginpgfgraphicnamed{fig121}
  \begin{tikzpicture}[>=latex]
   \matrix (m) [matrix of math nodes, row sep=0.5em, column sep=1em, text height=1.5ex, text depth=0.5ex]
    {  \node[pair](40){};     & \node[triple](b42){b}; &                        & \node[pair](44){};     & 4 \\
                              &                        &                        &                        & 3 \\
                              & \node[pair](22){};     &                        &                        & 2 \\
               ln             &        ln+2            &         ln+3           &        ln+4            &   \\
 };
    \path[-,very thick,font=\scriptsize,gray]
    (b42) edge[-<-=0.5] (44)
    (b42) edge[->-=0.6] node[auto,black,swap] {$-1$} (22)
;
    \path[-,font=\scriptsize]
    (b42) edge[-<-=0.55] (40)
;
   \end{tikzpicture}
 \endpgfgraphicnamed
\)
\\
\(
\beginpgfgraphicnamed{fig122}
  \begin{tikzpicture}[>=latex]
   \matrix (m) [matrix of math nodes, row sep=-1.0em, column sep=1em, text height=2.0ex, text depth=1.0ex]
{ q_a &     &  \   \\
      &     & q_0  \\
   \  & q_b &  \quad   \\
};
     \path[->,font=\footnotesize]
      (m-1-1) edge node[auto] {$a$} (m-2-3)
      (m-3-2) edge node[auto,swap] {$b$} (m-2-3)
;
     \draw[inner sep=10pt,ultra thick,rounded corners]
     (m-1-1.north west) rectangle (m-3-3.south east);
     \node at (0,-2.0) {};
 \end{tikzpicture}
\endpgfgraphicnamed
\hspace{2cm}
  \beginpgfgraphicnamed{fig123}
  \begin{tikzpicture}[>=latex]
   \matrix (m) [matrix of math nodes, row sep=0.5em, column sep=1em, text height=1.5ex, text depth=0.5ex]
    {  \node[pair](40){};     & \node[triple](b42){b}; & \node[triple](a43){a}; & \node[pair](44){};     & 4 \\
                              &                        & \node[pair](33){};     &                        & 3 \\
                              & \node[pair](22){};     &                        &                        & 2 \\
               ln             &        ln+2            &         ln+3           &        ln+4            &   \\
 };
    \path[-,very thick,font=\scriptsize,gray]
    (a43) edge[-<-=0.5] (44)
    (a43) edge[->-=0.7] node[auto,black,swap] {$-1$} (33)
    (b42) edge[-<-=0.5,bend left=20] (44)
    (b42) edge[->-=0.6] node[auto,black,swap] {$-1$} (22)
;
    \path[-,font=\scriptsize]
    (b42) edge[-<-=0.55] (40)
;
   \end{tikzpicture}
 \endpgfgraphicnamed
\)
\\
This completes the proof for indecomposable representations in a tube of rank $n-2$.


\subsection*{Tubes of rank $2$}
Let $M$ be an indecomposable representation in an exceptional tube of rank $2$ with a basis $\cB$ as described in section \ref{subsection: bases for indecomposables of small defect}. Though the coefficient quiver $\Gamma$ looks different from the coefficient quiver of the previous representations that we considered, we are able to reduce the present case to the treatment of a representation in an exceptional tube of rank $n-2$ where $n=4$.

This is obvious for $n=4$ since $\Gamma$ coincides with the coefficient quiver of a representation in a rank $n-2$ tube after permuting the arrows $c$ and $d$, i.e.\ after applying an automorphism of the underlying Dynkin diagram of $Q$. Therefore Theorem \ref{thm: the main theorem} holds for $M$ in the case $n=4$.

The basic observation for $n>4$ is that the linear maps $M_{v_r}$ are identity matrices w.r.t.\ to the basis $\cB$ for $r=0,\dotsc,n-5$. The contribution of the corresponding arrows of $\Gamma$ to the Schubert system are essentially extremal paths, which we can contract. The resulting system is equal to a Schubert system for $n=4$, for which we have verified Theorem \ref{thm: the main theorem} already.

We elaborate this argument in more detail. In order to satisfy axiom (EP3), we will define the following patchwork $\{\Xi_{k,l}\}$ for the reduced Schubert system $\overline\Sigma$. The indices range over all $(k,l)$ with $0\leq k\leq l\leq r+1$ where $l-k$ is even and $r$ is the number of non-extremal arrows of $\Gamma$. The indices are ordered as before: $(k,l)<(k',l')$ if and only if $l-k<l'-k'$ or if $l-k=l'-k'$ and $l<l'$.

The patch $\Xi_{k,l}$ is defined as the full subsystem of the reduced Schubert system with vertices $(i,j)$ and $(v,t,s)$ with $i,t\in\{(k-1)n+4,\dotsc,(k+1)n\}$ and $j,s\in\{ln,\dotsc,(l+1)n\}$, together with the vertices $\bigl((k-1)n+4,kn+4\bigr)$ and $\bigl(ln,(l+1)n\bigr)$ provided they are involved in a quadratic link whose tip and whose other base vertex is in $\Xi_{k,l}$.

Let $\Pi_{k,l}$ be the full subsystem whose vertices are all pairs $(i,j)$ and triples $(v,t,s)$ with vertical coordinates $i,t\in\{(k-1)n+4,\dotsc,kn\}$ and $j,s\in\{(l-1)n+4,\dotsc,ln\}$. Then $\Pi_{k,l+1}$ and $\Pi_{k+1,l+1}$ are extremal paths in $\Xi_{k,l}$, provided they are not empty. If we contract these extremal paths, we obtain the system $\Xi'_{k,l}=\bigl(\Xi_{k,l}/\Pi_{k,l+1}\bigr)/\Pi_{k+1,l+1}$, and $\Xi'_{k,l}$ is equal to the corresponding patch of a representation $M'$ in a rank $2$ tube with $n=4$. This might be best understood in a concrete situation, see Example \ref{ex: contracted patch of a rank 2 tube} below.

If we add the trivial patches $\Pi'_{k,l}$ that consist only of the vertex $(kn,ln)$ to the contracted patches $\Xi'_{k,l}$, then we obtain the same patchwork as we considered in the case of a tube of rank $n-2$, where $n=4$. 

Therefore we can use the same case study as for defect $-1$ and the tube of rank $n-2$ to obtain solutions for all $\beta$-states of the contracted patchwork. By Proposition \ref{prop: patchwork solutions}, these solutions extend to solutions of the $\beta$-states of the patches $\Xi_{k,l}$. We note that the coordinate transformation that we considered in the defect $-1$ case behaves well w.r.t.\ to the contractions. This proves Theorem \ref{thm: the main theorem} for $M$ with basis $\cB$.

\begin{ex}\label{ex: contracted patch of a rank 2 tube}
We inspect the patch $\Xi_{k,l}$ and the contraction $\Xi'_{k,l}=\bigl(\Xi_{k,l}/\Pi_{k,l+1}\bigr)/\Pi_{k+1,l+1}$ in the following example. Let $n=6$ and let all arrows of $Q$ be oriented to the right. For $0<k<l<r+1$ with $k$ and $l$ even, the last two coordinates of vertices of $\Xi_{k,l}$ vary through the vertices of the part
 \[
   \beginpgfgraphicnamed{fig124}
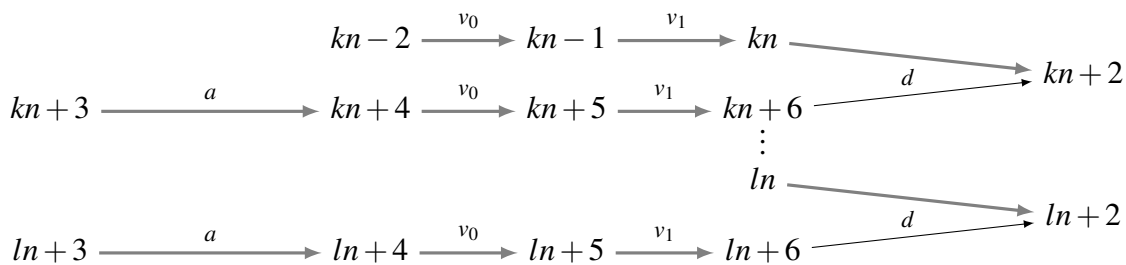
   \begin{tikzpicture}[>=latex]
  \matrix (m) [matrix of math nodes, row sep=0em, column sep=3em, text height=1ex, text depth=0ex]
   {      &   &   kn-2   &    kn-1  &    kn  &   &      \\ 
          &   &          &          &        &   & kn+2 \\ 
     kn+3 &   &  kn+4    &  kn+5    &  kn+6  &   &      \\ 
          &   &          &          &        &   &      \\ 
          & \ &          &          & \vdots & \ &      \\ 
          &   &          &          &        &   &      \\ 
          &   &          &          &    ln  &   &      \\ 
          &   &          &          &        &   & ln+2 \\ 
     ln+3 &   &  ln+4    &  ln+5    & ln+6   &   &      \\ 
   };
   \path[->,very thick,gray,font=\scriptsize]
   (m-1-3) edge node[auto,black] {$v_0$} (m-1-4)
   (m-1-4) edge node[auto,black] {$v_1$} (m-1-5)
   (m-1-5) edge node[auto,black] {} (m-2-7)
   (m-3-1) edge node[auto,black] {$a$} (m-3-3)
   (m-3-3) edge node[auto,black] {$v_0$} (m-3-4)
   (m-3-4) edge node[auto,black] {$v_1$} (m-3-5)
   (m-9-1) edge node[auto,black] {$a$} (m-9-3)
   (m-7-5) edge node[auto,black] {} (m-8-7)
   (m-9-3) edge node[auto,black] {$v_0$} (m-9-4)
   (m-9-4) edge node[auto,black] {$v_1$} (m-9-5)
   ;
   \path[->,font=\scriptsize]
   (m-3-5) edge node[above left=-1pt,black] {$d$} (m-2-7)
   (m-9-5) edge node[above left=-1pt,black] {$d$} (m-8-7)
   ;
   \end{tikzpicture}
\endpgfgraphicnamed
\]
of the coefficient quiver $\Gamma$ where the extremal arrows are illustrated bold and grey. The patch $\Xi_{k,l}$ looks as follows
 \[
  \beginpgfgraphicnamed{fig125}
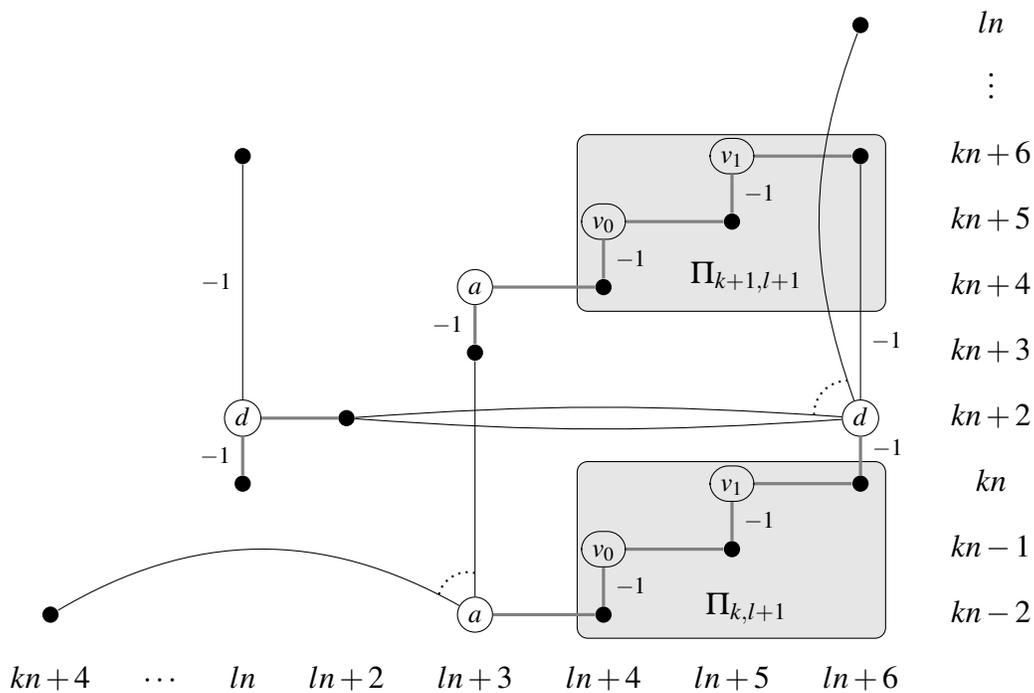
  \begin{tikzpicture}[>=latex]
   \matrix (m) [matrix of math nodes, row sep=0.5em, column sep=1em, text height=1.5ex, text depth=0.5ex]
   {               &  &                        &                        &                        &                        &                        & \node[pair](ll){};     & ln   \\
                   &  &                        &                        &                        &                        &                        &                        & \vdots   \\
                   &  & \node[pair](60){};     &                        &                        &                        &\node[vertex](v65){v_1};& \node[pair](66){};     & kn+6   \\
                   &  &                        &                        &                        &\node[vertex](v54){v_0};& \node[pair](55){};     &                        & kn+5   \\
                   &  &                        &                        & \node[triple](a43){a}; & \node[pair](44){};     &                        &                        & kn+4 \\
                   &  &                        &                        & \node[pair](33){};     &                        &                        &                        & kn+3 \\
                   &  & \node[triple](d20){d}; & \node[pair](22){};     &                        &                        &                        & \node[triple](d26){d}; & kn+2 \\
                   &  & \node[pair](00){};     &                        &                        &                        &\node[vertex](v05){v_1};& \node[pair](06){};     & kn   \\
                   &  &                        &                        &                       &\node[vertex](v-14){v_0};& \node[pair](-15){};    &                        & kn-1   \\
\node[pair](kk){}; &  &                        &                        &\node[triple](a-23){a}; & \node[pair](-24){};    &                        &                        & kn-2   \\
    kn+4      &\dotsb &         ln             &       ln+2             &        ln+3            &        ln+4            &         ln+5           &        ln+6            &      \\
 };
    \path[-,very thick,font=\scriptsize,gray]
    (a-23) edge (-24)
    (a43) edge (44)
    (a43) edge node[auto,swap,black] {$-1$} (33)
    (d20) edge node[auto,swap,black] {$-1$} (00)
    (d20) edge (22)
    (d26) edge node[pos=0.3,auto,black] {$-1$} (06)
    (v-14) edge (-15)
    (v-14) edge node[auto,black] {$-1$} (-24)
    (v05) edge (06)
    (v05) edge node[auto,black] {$-1$} (-15)
    (v54) edge (55)
    (v54) edge node[auto,black] {$-1$} (44)
    (v65) edge (66)
    (v65) edge node[auto,black] {$-1$} (55)
;
    \path[-,font=\scriptsize]
    (d20) edge node[auto,black] {$-1$} (60)
    (d26) edge node[pos=0.25,auto,swap,black] {$-1$} (66)
    (d26) edge[bend left=4] (22)
    (a-23) edge node[auto] {} node[pos=0.1] (a-23-33) {} (33)
    (a-23) edge[bend right=30] node[pos=0.05] (a-23-kk) {} (kk)
    (d26) edge[bend right=4] node[pos=0.05] (d26-22) {} (22)
    (d26) edge[bend left=20]  node[pos=0.05] (d26-ll) {} (ll)
;
    \path[dotted,-,thick,font=\scriptsize]
    (d26-22.center) edge[bend left=45] node[above left=-2pt] {} (d26-ll.center)
    (a-23-33.center) edge[bend right=45] node[above left=-2pt] {} (a-23-kk.center)
;
\begin{pgfonlayer}{patchwork}
 \filldraw[fill=gray!20,rounded corners] 
   (0.75,-1.45) rectangle (4.85,-3.8) 
   (0.75,0.55) rectangle (4.85,2.9) 
;
  \node at (3.0,-3.4) {$\Pi_{k,l+1}$};
  \node at (3.0,1.0) {$\Pi_{k+1,l+1}$};
 \end{pgfonlayer}
   \end{tikzpicture}
 \endpgfgraphicnamed
\] 
where we illustrate extremal edges bold and grey, with the extremal paths $\Pi_{k,l+1}$ and $\Pi_{k+1,l+1}$ highlighted. The following is an illustration of $\Xi'_{k,l}$ where we highlight the contractions $\Pi'_{k,l+1}$ and $\Pi'_{k+1,l+1}$ of the extremal paths $\Pi_{k,l+1}$ and $\Pi_{k+1,l+1}$, respectively. Up to exchanging $d$ by $b$, this is identical with the illustration of $\Xi_{k,l}$ on page~\pageref{fig80}.
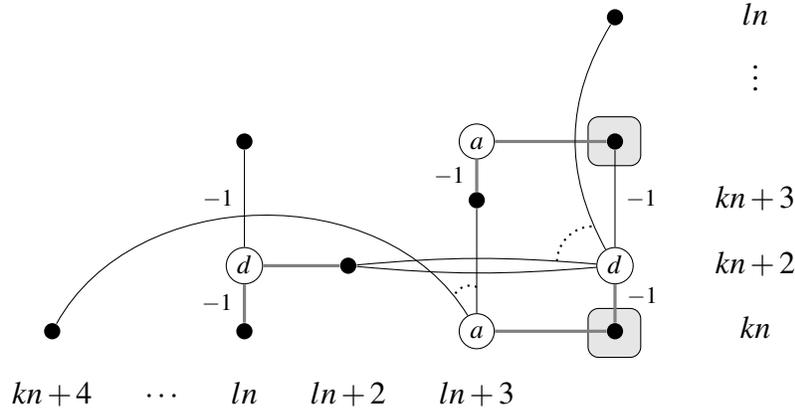
\[
  \beginpgfgraphicnamed{fig126}
  \begin{tikzpicture}[>=latex]
   \matrix (m) [matrix of math nodes, row sep=0.5em, column sep=1em, text height=1.5ex, text depth=0.5ex]
    {              &  &                        &                        &                        & \node[pair](ll){};     & ln   \\
                   &  &                        &                        &                        &                        & \vdots   \\
                   &  & \node[pair](40){};     &                        & \node[triple](a43){a}; & \node[pair](44){};     &      \\
                   &  &                        &                        & \node[pair](33){};     &                        & kn+3 \\
                   &  & \node[triple](b20){d}; & \node[pair](22){};     &                        & \node[triple](b24){d}; & kn+2 \\
\node[pair](kk){}; &  & \node[pair](00){};     &                        & \node[triple](a03){a}; & \node[pair](04){};     & kn   \\
     kn+4     &\dotsb &         ln             &       ln+2             &        ln+3            &    \hspace{3em}        &      \\
 };
    \path[-,very thick,font=\scriptsize,gray]
    (a03) edge (04)
    (a43) edge (44)
    (a43) edge node[auto,black,swap] {$-1$} (33)
    (b20) edge node[auto,swap,black] {$-1$} (00)
    (b20) edge (22)
    (b24) edge node[pos=0.3,auto,black] {$-1$} (04)
;
    \path[-,font=\scriptsize]
    (b20) edge node[auto,black] {$-1$} (40)
    (b24) edge node[auto,black,swap] {$-1$} (44)
    (b24) edge[bend left=5] (22)
;
    \path[-,font=\scriptsize]
    (a03) edge node[auto] {} node[pos=0.25] (a03-33) {} (33)
    (a03) edge[bend right=60] node[auto] {} node[pos=0.05] (a03-kk) {} (kk)
    (b24) edge[bend right=5] node[pos=0.15] (b24-22) {} (22)
    (b24) edge[bend left=30]  node[pos=0.1] (b24-ll) {} (ll)
;
    \path[dotted,-,thick,font=\scriptsize]
    (b24-22.center) edge[bend left=45] node[above left=-2pt] {} (b24-ll.center)
    (a03-33.center) edge[bend right=45] node[above left=-2pt] {} (a03-kk.center)
;
\begin{pgfonlayer}{patchwork}
 \filldraw[fill=gray!20,rounded corners] 
   (2.45,-1.35) rectangle (3.15,-2) 
   (2.45,0.55) rectangle (3.15,1.2) 
;
 \end{pgfonlayer}
   \end{tikzpicture}
 \endpgfgraphicnamed
\]

\end{ex}


\subsection*{Homogeneous tubes}

Let $M$ be a Schurian representation whose isomorphism class is contained in a homogeneous tube, and $\cB$ the basis from section \ref{subsection: bases for indecomposables of small defect}. Then the 
reduced Schubert system is of the following form
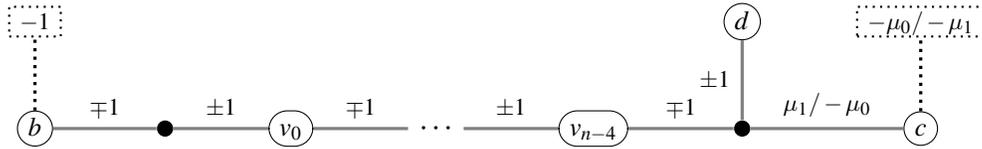
\[
  \beginpgfgraphicnamed{fig54}
  \begin{tikzpicture}[>=latex]
   \matrix (m) [matrix of math nodes, row sep=2em, column sep=3em, text height=1.5ex, text depth=0.5ex]
    { \node[const](b0){-1}; &                   &                         &        &                               & \node[triple](d){d}; & \node[const](c0){-\mu_0/-\mu_1};   \\
      \node[triple](b){b};  & \node[pair](0){}; & \node[vertex](v0){v_0}; & \dotsb & \node[vertex](vn-4){v_{n-4}}; & \node[pair](n-4){};  & \node[triple](c){c}; \\
 };
    \path[-,very thick,font=\scriptsize,gray]
    (b) edge node[auto,black] {$\mp 1$} (0)
    (v0) edge node[auto,swap,black] {$\pm 1$} (0)
    (v0) edge node[auto,black] {$\mp 1$} (m-2-4)
    (vn-4) edge node[auto,swap,black] {$\pm 1$} (m-2-4)
    (vn-4) edge node[auto,black] {$\mp 1$} (n-4)
    (d) edge node[auto,swap,black] {$\pm 1$} (n-4)
    (c) edge node[auto,swap,black] {$\mu_1/-\mu_0$} (n-4)
;
    \path[dotted,very thick,font=\scriptsize]
    (b) edge (b0)
    (c) edge (c0)
;
   \end{tikzpicture}
 \endpgfgraphicnamed
\]
where we omit the last two coordinates of the vertices, which depend on the orientation of the arrow, and where the weight have to be understood as follows: if the arrow that appears as the label of a triple is oriented to the left, then the weight of the adjacent edges is annotated to the left / has the top sign; if the arrow is oriented to the right, then the weight of the adjacent edges is annotated to the right / has the bottom sign. Note that all edges are extremal.

Consider a subset $\beta$ of $\cB$ that is not contradictory of the first or second kind. The following happens if we calculate the $\beta$-state. In the initial step, certain vertices, edges and links are marked as $\beta$-trivial. We denote the system of all vertices, edges, links that are not marked as $\beta$-trivial by $\Sigma_\beta'$. We see from the characterization of the second kind contradictory subsets of $\cB$ in section \ref{subsection: bases for indecomposables of small defect} that for each path from (the triple with first coordinate) $b$ to $c$ and from $b$ to $d$ and from $c$ to $d$, there must be a vertex on this path that is marked as $\beta$-trivial. 

Since all edges of $\Sigma_\beta'$ are extremal, a path originating from one of the triples $b$, $c$ or $d$ ends in a pair. Therefore, we can apply step \ref{step7} successively to all triples of such a path and see that all vertices, edges and links of this path are marked as $\beta$-trivial without deriving a contradiction. If we have done this for all paths originating from the triples $b$, $c$ and $d$---as far as these triples are not $\beta$-trivial---, we have calculated the $\beta$-state $\Sigma_\beta$.

We conclude that $\Sigma_\beta$ consists of a disjoint union of extremal paths, which are solvable by Corollary \ref{cor: extremal solution for extremal paths}. 

\enlargethispage{2\baselineskip}
\bigskip

\noindent This concludes the proof of Theorem \ref{thm: the main theorem}.\qed

\appendix


\section{Representations for quivers of type $\widetilde D_n$}
\label{appendix: indecomposable representations of type D_n-tilde}

\noindent
In this section, we recall facts from the representation theory for quivers $Q$ of extended Dynkin type $\widetilde D_n$, as illustrated below. For more details on the following exposition, see \cite{Crawley-Boevey92} or \cite{Dlab-Ringel76}.
\[
   \beginpgfgraphicnamed{fig2}
   \begin{tikzpicture}[>=latex]
  \matrix (m) [matrix of math nodes, row sep=-0.2em, column sep=2.5em, text height=1ex, text depth=0ex]
   {      & q_b &     &     &      &         &         &  q_c & \\
          &     & q_0 & q_1 & {\dotsb } & q_{n-5} & q_{n-4} &      & \\
      q_a &     &     &     &      &         &         &      &  q_d \\};
   \path[-,font=\scriptsize]
   (m-3-1) edge node[auto,swap] {$a$} (m-2-3)
   (m-1-2) edge node[auto] {$b$} (m-2-3)
   (m-2-3) edge node[auto] {$v_0$} (m-2-4)
   (m-2-6) edge node[auto] {$v_{n-5}$} (m-2-7)
   (m-2-7) edge node[auto] {$c$} (m-1-8)
   (m-2-7) edge node[auto,swap] {$d$} (m-3-9)
   (m-2-4) edge node[auto] {} (m-2-5)
   (m-2-5) edge node[auto] {} (m-2-6);
  \end{tikzpicture}
\endpgfgraphicnamed
\]

\subsection{Reflections and Auslander-Reiten translates} 
\label{subsection: reflections and auslander-reiten translates}

Let $Q$ be a quiver of extended Dynkin type $\widetilde D_n$. In order to describe a complete list of indecomposable representations of $Q$, we recall some bits of the representation theory of $Q$. The \emph{reflection $s_p$ at a sink $p$} of $Q$ is defined as follows. The quiver $Q'=s_pQ$ is obtained from $Q$ by inverting all arrows pointing towards $p$. Therefore $p$ is a source of $Q'$. For a representation $M$ of $Q$, we define $M'=s_pM$ as the representation of $Q'$ with $M'_q=M_q$ for $q\neq p$, and
\[
 M'_p \ = \ \ker\big( \, \bigoplus_{v:q\to p} M_q \longrightarrow M_p \, \bigr),
\]
together with the natural maps $M'_v:M'_p\to M'_q$ for $v:p\to q$ in $Q'$. All other maps $M_v:M_q\to M_{q'}$, which do not involve $p$, are unchanged. Note that $s_pS_p=0$ for the simple representation $S_p$ with support $p$.

We define the \emph{reflection $s_p$ at a source $p$} of $Q$ in analogy: $Q'=s_pQ$ is obtained from $Q$ by inverting all arrows pointing away from $p$. Thus $p$ is a sink of $Q'$. For a representation $M$ of $Q$, we define $M'=s_pM$ as the representation of $Q'$ such that $M'_q=M_q$ for $q\neq p$, and
\[
 M'_p \ = \ \coker\big( \, M_p \longrightarrow \bigoplus_{v:p\to q} M_q  \, \bigr),
\]
together with the natural maps $M'_v:M'_q\to M'_p$ for $v:q\to p$ in $Q'$. Again, we have $s_p S_p=0$.

In fact, the reflections at sinks and sources define functors $s_p:\Rep Q\to \Rep s_pQ$. The Auslander-Reiten translate $\tau:\Rep Q\to \Rep Q$ is equal to the product $s_{p_{n+1}}\dotsb s_{p_{1}}$ of reflections where $p_1,\dotsc,p_{n+1}$ range through the $n+1$ vertices of $Q$ in an admissible order, i.e. when reflecting at $p_i$, the vertex $p_i$ is a sink.  The inverse $\tau^{-1}$ is obtained when reflecting only at sources.

See appendix \ref{appendix: bases for representations of type D_n-tilde} for an explicit application of reflection functors and Auslander-Reiten translates.

\subsection{Indecomposable and exceptional representations}

The \emph{support of $M$} is the full subgraph $Q'$ of $Q$ with vertices $p$ such that the dimension of $M_p$ is positive. For a vertex $p$ in $Q$, we define the representation $P=P(p)$ as follows: $P_q=\C$ if there is an oriented path from $p$ to $q$ and $P_q=0$ otherwise; if $v:q\to q'$ is in the support of $P$, then $P_v:P_q=\C\to \C=P_{q'}$ is the identity map; all other maps are trivial. A representation $M$ is indecomposable projective if and only if it is isomorphic to $P(p)$ for some vertex $p$ of $Q$. 

Analogously, we define the representation $I=I(p)$ as the thin representation whose support consists of all vertices $q$ such that there is an oriented path from $q$ to $p$, and such that for an arrow $v:q\to q'$ in the support of $I$, the map $I_v:I_q=\C\to \C=I_{q'}$ is the identity map. A representation is indecomposable injective if and only if it is isomorphic to $I(p)$ for some vertex $p$ of $Q$.

The indecomposable representations of $Q$ are either \emph{preprojective}, which are isomorphic to Auslander-Reiten translates of indecomposable projective representations, \emph{preinjective}, which are isomorphic to Auslander-Reiten translates of indecomposable injective representations, or \emph{regular}, which are the indecomposables $M$ such that $\tau^i(M)\neq0$ for all $i\in \Z$.

We write the coefficients of a dimension vector $\alpha=(\alpha_a,\alpha_b,\alpha_0,\dotsc,\alpha_{n-4},\alpha_c,\alpha_d)$ from left to right w.r.t.\ to the illustration of $Q$ as at the beginning of section \ref{appendix: indecomposable representations of type D_n-tilde}.

A representation $M$ of $Q$ is \emph{exceptional} if it is indecomposable and $\Ext^1(M,M)=0$. Up to the $(n-2)(n-3)+4$ regular representations $M$ with $\udim M\leq\delta$, the exceptional representations of $Q$ are precisely the preprojective and the preinjective representations.

In appendix \ref{appendix: bases for representations of type D_n-tilde}, we will give explicit descriptions of the indecomposable representations of $Q$ in terms of their coefficient quivers (w.r.t.\ to a certain basis).

\subsection{The Auslander-Reiten quiver}
\label{subsection: the auslander-reiten quiver}

Let $M$ and $N$ be two representations of $Q$. A homomorphism $f:M\to N$ is \emph{irreducible} if it not an isomorphism and if it does not factor into two homomorphisms that are not isomorphisms. The Auslander-Reiten quiver $Q$ is the quiver whose vertices are the isomorphism classes $[M]$ of indecomposable representations $M$ of $Q$ and which has an arrow $[M]\to[N]$ if there exists an irreducible homomorphism $f:M\to N$.

The connected components of the Auslander-Reiten quiver of $Q$ are as follows. The \emph{preprojective component} consists of the isomorphism classes of all preprojective representations of $Q$. The \emph{preinjective component} consists of the isomorphism classes of all preinjective representations of $Q$. The regular indecomposable representations form the \emph{tubes}, which are infinitely many connected components, parametrized by $\P^1(\C)$. This parametrizations is explicitly visible from the bases that we describe in Appendix \ref{appendix: bases for representations of type D_n-tilde}.

\subsection{The tubes}
\label{subsection: the tubes}

For a regular indecomposable representation $M$, there is a positive number $n$ such that $M$ is isomorphic to its $n$-th Auslander-Reiten translate. The smallest $k>0$ with $\tau^k M\simeq M$ is equal for all representations $M$ whose isomorphism class is contained in a fixed tube, and this number $k$ is called the \emph{rank of the tube}. Except for three tubes, all tubes have rank $1$ and are called \emph{homogeneous tubes}. There are two \emph{exceptional tubes} of rank $2$ and one exceptional tube of rank $n-2$. 

Note that in the case $n=4$, all three exceptional tubes have rank $n-2=2$. Also confer section \ref{subsection: automorphisms of the dynkin diagram} where we discuss that the Dynkin diagram of $Q$ has more automorphisms if $n=4$. In Part 2 of this series of papers, we will describe the connected components of the Auslander-Reiten quiver in more detail.

\subsection{Roots}
\label{subsection: roots}

Let $\Lambda=\Z^{Q_0}$ be the root lattice of $Q$. An element $\alpha=(\alpha_p)$ of $\Lambda$ is \emph{positive} if it is non-zero and all its coefficients $\alpha_p$ are non-negative. It is \emph{negative} if $-\alpha=(-\alpha_p)$ is positive. 

The \emph{Euler form} $\langle -,-\rangle:\Z Q_0\times\Z Q_0\to\Z$ is defined as
\[
 \langle \alpha,\alpha' \rangle \ = \ \sum_{p\in Q_0} \alpha_p\alpha'_p - \sum_{(v:p\to q)\in Q_1} \alpha_p\alpha'_q.
\]
The function $q(\alpha)=\langle \alpha,\alpha\rangle$ defines a positive semi-definite quadratic form $q:\Z Q_0\to \Z$.

A \emph{root of $Q$} is a non-zero element $\alpha\in\Lambda$ with $q(\alpha)\leq 1$. A root $\alpha$ is \emph{real} if $q(\alpha)=1$ and it is \emph{imaginary} is $q(\alpha)=0$. The imaginary roots together with the zero element of $\Lambda$ form the subgroup $\{l\delta\}_{l\in\Z}$ of $\Lambda$ where $\delta=(1,1,2,\dotsc,2,1,1)$ is the minimal positive imaginary root of $Q$.

Every root is either positive or negative, and the positive roots are precisely the dimension vectors of the indecomposable representations of $Q$. We say that a representation or its isomorphism class is a \emph{real (imaginary) root representation} if its dimension vector is a real (imaginary) root. The isomorphism class of an indecomposable real root representation $M$ is uniquely determined by its dimension vector $\udim M$. All preprojective and preinjective representations are real root representations and the regular real root representations are contained in the three exceptional tubes. For a fixed $l>0$, every tube contains $k$ isomorphism classes of imaginary root representations $M$ with $\udim M=l\delta$ where $k$ is the rank of the tube.

A regular indecomposable representation is called \emph{quasi-simple} if it does not contain any proper regular subrepresentation. In particular, each homogeneous tube contains a unique isomorphism class of quasi-simple representations, which has dimension vector $\delta$.

A representation is called \emph{Schurian} if its endomorphism ring is $\C$. The associated dimension vector is called a \emph{Schur root}.

\subsection{The defect}
\label{subsection: the defect}
The \emph{defect} of a representation $M$ of $Q$ is defined as the Euler form $\langle\delta,\udim M\rangle$. We have the following classification of the indecomposable representations $M$ of $Q$ (\cite[$\SS$7, Lemma 2]{Crawley-Boevey92} or \cite{Dlab-Ringel76}):
\begin{enumerate}
 \item $M$ is preprojective if and only if its defect is negative;
 \item $M$ is regular if and only if its defect is $0$;
 \item $M$ is preinjective if and only if its defect is positive;
\end{enumerate}

By the definition of a preprojective representation, we obtain every preprojective representation of $Q$ by starting with a simple projective representation $S_p$ of $Q'$, which is a quiver of type $\widetilde D_n$ with a possibly different orientation than $Q$, and applying a finite number of reflections at sources to $M$. The same is true for a preinjective representation $M$: it is obtained from a simple injective $S_p$ by applying a finite number of reflections at sinks. Since reflections leave the defect invariant, the defect of $M$ is equal to the defect of $S_p$. Calculating the defect for simple projectives and simple injectives yields the following list.
\begin{enumerate}
 \item[\textbf{Defect  2:}\hspace{-1cm}] \hspace{1cm} $S_p$ is simple injective and $p=q_r$ for $r\in\{0,\dotsc,n-4\}$.
 \item[\textbf{Defect  1:}\hspace{-1cm}] \hspace{1cm} $S_p$ is simple injective and $p\in\{q_a,q_b,q_c,q_d\}$.
 \item[\textbf{Defect -1:}\hspace{-1cm}] \hspace{1cm} $S_p$ is simple projective and $p\in\{q_a,q_b,q_c,q_d\}$.
 \item[\textbf{Defect -2:}\hspace{-1cm}] \hspace{1cm} $S_p$ is simple projective and $p=q_r$ for $r\in\{0,\dotsc,n-4\}$.
\end{enumerate}


\section{Bases for representations of type $\widetilde D_n$}
\label{appendix: bases for representations of type D_n-tilde}

\noindent
Let $Q$ be a quiver of extended Dynkin type $\widetilde D_n$. In this section, we construct bases for indecomposable representations of $Q$. We provide explicit descriptions of the coefficient quivers that we need in this paper, and indicate how to derive bases for all other isomorphism classes of indecomposable representations of $Q$.

Note that Kussin and Meltzer construct bases for preprojective and preinjective indecomposable representations of $Q$ in \cite{Kussin-Meltzer06}, which differ from our choice of bases. Examples make clear that the property that the Schubert decomposition $\Gr_\ue(M)=\coprod C_\beta^M$ is a decomposition into affine spaces is not stable under disturbing the chosen ordered basis $\cB$. Therefore it is unlikely that the bases of \cite{Kussin-Meltzer06} yield decompositions into affine spaces.

\subsection*{Remark on illustrations:} In the illustrations of coefficient quivers $\Gamma$, we will use the following conventions. We will draw the vertices of $\Gamma$ as bullets. If an illustration is valid for either orientation of an arrow, we omit the head of the arrow and just draw a line. A dashed arrow together with its grey end vertex is part of $\Gamma$ if and only if the corresponding arrow of $Q$ is oriented in the indicated direction. A crossed dotted arrow indicates the orientation of the corresponding arrow of $Q$, but this arrow does not belong to $\Gamma$. Circles at the ends of crossed arrows refer to the corresponding vertices of $Q$, but they are not part of $\Gamma$. 

The ordering of the basis $\cB$ that we use in section \ref{section: Schubert decompositions for type D_n} can be constructed by the following rule: if a vertex $i$ is drawn on top of another vertex $j$, then $i<j$. Since the Schubert decomposition $\Gr_\ue(M)=\coprod C_\beta^M$ depends only on the ordering of each fibre of $F:\Gamma\to Q$, we can extend this partial order arbitrarily to a linear order of $\cB$.


\subsection{Defect $-1$}
\label{subsection: bases for defect -1}

Let $M$ be an indecomposable representation of $Q$ of defect $-1$. Recall the notion of an automorphism of the underlying Dynkin diagram $\widetilde D_n$ of $Q$ from section \ref{subsection: automorphisms of the dynkin diagram}.

\begin{thm}
 Up to an automorphism of $\widetilde D_n$, the representation $M$ has an ordered basis such that the coefficient quiver $\Gamma=\Gamma(M,\cB)$ is
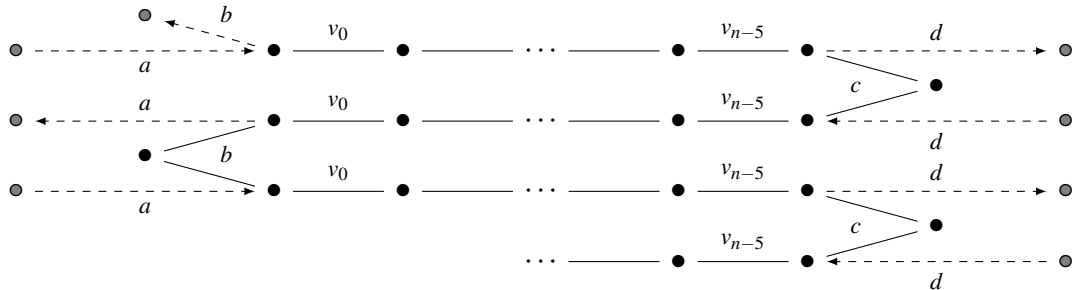
\[
   \beginpgfgraphicnamed{fig19}
   \begin{tikzpicture}[>=latex]
  \matrix (m) [matrix of math nodes, row sep=0em, column sep=2.9em, text height=1ex, text depth=0ex]
   {          &\ocirc &       &       &        &       &       &       &       \\   
      \ocirc  &       &\bullet&\bullet&\dotsb  &\bullet&\bullet&       &\ocirc \\   
              &       &       &       &        &       &       &\bullet&       \\   
      \ocirc  &       &\bullet&\bullet&\dotsb  &\bullet&\bullet&       &\ocirc \\   
              &\bullet&       &       &        &       &       &       &       \\   
      \ocirc  &       &\bullet&\bullet&\dotsb  &\bullet&\bullet&       &\ocirc \\   
              &       &       &       &        &       &       &\bullet&       \\   
              &       &       &       &\dotsb  &\bullet&\bullet&       &\ocirc \\   
};
   \path[-,font=\scriptsize]
   (m-2-3) edge node[auto] {$v_{0}$} (m-2-4)
   (m-2-6) edge node[auto] {$v_{n-5}$} (m-2-7)
   (m-2-7) edge node[auto,swap] {} (m-3-8)
   (m-3-8) edge node[auto,swap] {$c$} (m-4-7)
   (m-4-3) edge node[auto] {$v_0$} (m-4-4)
   (m-4-6) edge node[auto] {$v_{n-5}$} (m-4-7)
   (m-4-3) edge node[auto] {} (m-5-2)
   (m-5-2) edge node[auto] {$b$} (m-6-3)
   (m-6-3) edge node[auto] {$v_0$} (m-6-4)
   (m-6-6) edge node[auto] {$v_{n-5}$} (m-6-7)
   (m-6-7) edge node[auto,swap] {} (m-7-8)
   (m-7-8) edge node[auto,swap] {$c$} (m-8-7)
   (m-8-6) edge node[auto] {$v_{n-5}$} (m-8-7);
   \path[->,dashed,font=\scriptsize]
   (m-2-3) edge node[auto,swap] {$b$} (m-1-2)
   (m-2-1) edge node[auto,swap] {$a$} (m-2-3)
   (m-2-7) edge node[auto] {$d$} (m-2-9)
   (m-4-9) edge node[auto] {$d$} (m-4-7)
   (m-4-3) edge node[auto,swap] {$a$} (m-4-1)
   (m-6-1) edge node[auto,swap] {$a$} (m-6-3)
   (m-6-7) edge node[auto] {$d$} (m-6-9)
   (m-8-9) edge node[auto] {$d$} (m-8-7);
   \path[-,font=\scriptsize]
   (m-2-4) edge node[auto] {} (m-2-5)
   (m-2-5) edge node[auto,swap] {} (m-2-6)
   (m-4-4) edge node[auto] {} (m-4-5)
   (m-4-5) edge node[auto] {} (m-4-6)
   (m-6-4) edge node[auto,swap] {} (m-6-5)
   (m-6-5) edge node[auto] {} (m-6-6)
   (m-8-5) edge node[auto] {} (m-8-6);
  \end{tikzpicture}
\endpgfgraphicnamed
\]
where the bottom row ends in one of the situations
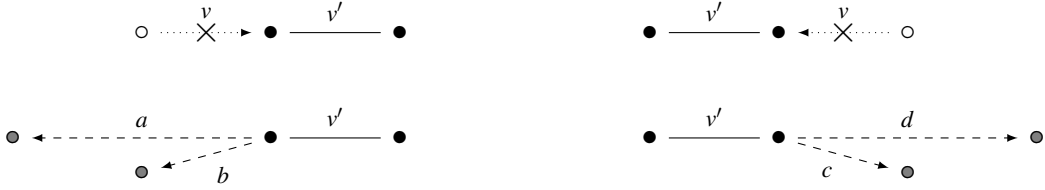
\[
\beginpgfgraphicnamed{fig21}
   \begin{tikzpicture}[>=latex]
  \matrix (m) [matrix of math nodes, row sep=0em, column sep=2.9em, text height=1ex, text depth=0ex]
   {          &\circ  &\bullet&\bullet&         &\bullet&\bullet& \circ &       \\   
              &  \    &  \    &       &         &       &  \    &  \    &       \\   
              &  \    &  \    &       &         &       &  \    &  \    &       \\   
     \ocirc   &       &\bullet&\bullet&         &\bullet&\bullet&       &\ocirc \\   
              &\ocirc &       &       &   \     &       &       &\ocirc &       \\   
   };
   \path[->,dotted,font=\scriptsize]
   (m-1-2) edge[-x-=0.5] node[above=1pt] {$v$} (m-1-3)
   (m-1-8) edge[-x-=0.5] node[above=1pt] {$v$}  (m-1-7)
   ;
   \path[-,font=\scriptsize]
   (m-1-3) edge node[auto] {$v'$} (m-1-4)
   (m-1-6) edge node[auto] {$v'$} (m-1-7)
   ;
   \path[->,dashed,font=\scriptsize]
   (m-4-3) edge node[auto,swap] {$a$} (m-4-1)
   (m-4-3) edge node[auto] {$b$} (m-5-2)
   (m-4-7) edge node[auto] {$d$} (m-4-9)
   (m-4-7) edge node[auto,swap] {$c$} (m-5-8)
   ;
   \path[-,font=\scriptsize]
   (m-4-3) edge node[auto] {$v'$} (m-4-4)
   (m-4-6) edge node[auto] {$v'$} (m-4-7)
   ;
  \end{tikzpicture}
\endpgfgraphicnamed
\]
where $v\in\{v_0,\dotsc,v_{n-5}\}$ and $v'\in\{b,v_0,\dotsc,v_{n-5},c\}$.
\end{thm}

\begin{proof}
The strategy of the proof is to exhibit bases for preprojective representations of defect $-1$ step by step along the successive application of reflections at sources to a simple projective representation. In the following, we consider all possible reflections and their effect on the coefficient quiver. It will be clear from these considerations that we can compute the effect of reflections locally in the coefficient quiver. 

Preprojective representations of defect $-1$ result from applying reflections to the simple projectives $S_p$ where $p$ is one of the vertices $q_a$, $q_b$, $q_c$ or $q_d$. Since the automorphism group of the underlying Dynkin diagram acts transitively on $q_a$, $q_b$, $q_c$ and $q_d$, we can assume that $p=q_a$.

The simple projective $S_{q_a}$ consists of a $1$-dimensional vector space at $q_a$ and trivial vector spaces at all other vertices of $Q$. Therefore the coefficient quiver of $S_{q_a}$ consists of one vertex and no arrows (w.r.t.\ to any basis $\cB$ of $S_{q_a}$). The arrow $a$ of $Q$ is oriented towards $q_a$, and all other arrows can have an arbitrary orientation.

Since the reflection $s_p$ affects a representation $M$ of $Q$ only at the vertex $p$ and at the adjacent arrows, it is enough to describe it in the \emph{neighbourhood of $p$}, by which we mean the subgraph $U_p$ of $Q$ that consists of $p$ and all its adjacent vertices together with the connecting arrows. If the restriction $M\vert_{U_p}$ to $U_p$ decomposes into a direct sum $N\oplus N'$, then we have $s_p(N\oplus N')=s_pN\oplus s_pN'$. This means that it suffices to describe the effect of an reflection $s_p$ \emph{locally in the coefficient quiver} to understand the global effect of $s_p$ on $M$. Therefore we can restrict ourselves to study the effect of $s_p$ on connected components of the fibres $F^{-1}(U_p)$ where $F:\Gamma\to Q$ is the canonical map.

In the following, we will describe all situations that occur if we apply the reflections $s_p$ successively to $S_{q_a}$. To be precise, we will choose bases for $\coker\big( M_p \to \bigoplus_{v:p\to q} M_q \bigr)$ when reflecting at a source $p$.

Though we will keep track of the values $\mu_{v,i,j}$ for our basis choices (as long as they differ from $1$), we remark already that none of the below steps creates a cycle in the coefficient quiver. This means that the constructed coefficient quivers $\Gamma$ will be trees for all successive reflections of $S_{q_a}$ w.r.t.\ to the chosen bases. Therefore, we can scale the basis elements after each reflection to renormalize all $\mu_{v,i,j}$ to $1$.

Note further that $s_p^2M=M$ unless $M$ contains $S_p$ as a direct summand. Therefore all the below steps are reversible. For brevity, we will write $s_x=s_{q_x}$ for the reflection at $q_x\in Q_0$ where $x\in\{a,b,c,d,0,\dotsc,n-4\}$.

In the following illustrations, we will assume that $n\geq 5$, i.e.\ that $Q$ contains vertices $q_0$ and $q_1$ and a connecting arrow $v_0$. It is easily seen that the case $n=4$ is leading to analogous cases as below, but with different labels for vertices and arrows.

We begin with the description of some initial reflections applied to $S_{q_a}$. 

\subsection*{Reflection at $q_0$}
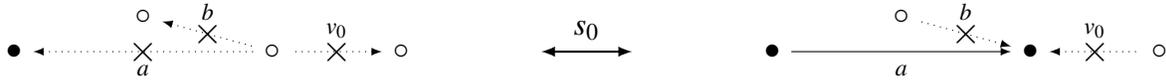
\[
   \beginpgfgraphicnamed{fig4}
   \begin{tikzpicture}[>=latex]
  \matrix (m) [matrix of math nodes, row sep=0em, column sep=2.9em, text height=1ex, text depth=0ex]
   {        &\circ&     &     &          &     &     &\circ \\   
     \bullet&  \  &\circ&\circ&  \       &  \  &\bullet&    &\bullet&\circ  \\}; 
   \path[thick,<->]
   (m-2-5) edge node[auto] {$s_0$} (m-2-6);
   \path[dotted,<-,font=\scriptsize]
   (m-2-1) edge[-x-=0.5] node[below=1pt] {$a$} (m-2-3)
   (m-1-2) edge[-x-=0.5] node[above=1pt] {$b$} (m-2-3)
   (m-2-4) edge[-x-=0.5] node[above=1pt] {$v_0$} (m-2-3)
   (m-2-9) edge[-x-=0.5] node[above=1pt] {$b$} (m-1-8)
   (m-2-9) edge[-x-=0.5] node[above=1pt] {$v_0$} (m-2-10);
   \path[->,font=\scriptsize]
   (m-2-7) edge node[below=1pt] {$a$} (m-2-9);
  \end{tikzpicture}
\endpgfgraphicnamed
\]
Note that after the reflection, the arrows $b$ and $v_0$ are pointing towards $q_0$, in agreement with the coefficient quiver of the theorem.

\subsection*{Reflection at $q_b$}
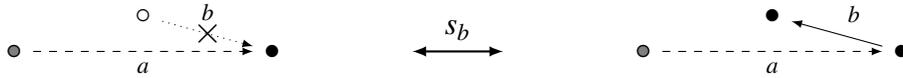
\[
   \beginpgfgraphicnamed{fig6}
   \begin{tikzpicture}[>=latex]
  \matrix (m) [matrix of math nodes, row sep=0em, column sep=2.9em, text height=1ex, text depth=0ex]
   {        &\circ&     &     &     &     &\bullet&   \\   
     \ocirc &  \  &\bullet&\  &  \  &\ocirc&   &\bullet\\}; 
   \path[thick,<->]
   (m-2-4) edge node[auto] {$s_b$} (m-2-5);
   \path[->,font=\scriptsize]
   (m-2-8) edge node[auto,swap] {$b$} (m-1-7);
   \path[dotted,->,font=\scriptsize]
   (m-1-2) edge[-x-=0.5] node[above=1pt] {$b$} (m-2-3);
   \path[dashed,->,font=\scriptsize]
   (m-2-1) edge node[auto,swap] {$a$} (m-2-3)
   (m-2-6) edge node[auto,swap] {$a$} (m-2-8);
  \end{tikzpicture}
\endpgfgraphicnamed
\]

\subsection*{Reflection at $q_a$} 
\[
   \beginpgfgraphicnamed{fig5}
   
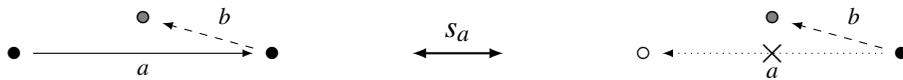
\begin{tikzpicture}[>=latex]
  \matrix (m) [matrix of math nodes, row sep=0em, column sep=2.9em, text height=1ex, text depth=0ex]
   {        &\ocirc&    &     &     &     &\ocirc&     \\   
     \bullet&    &\bullet&\  &  \   &\circ&     &\bullet  \\}; 
   \path[thick,<->]
   (m-2-4) edge node[auto] {$s_a$} (m-2-5);
   \path[dotted,->,font=\scriptsize]
   (m-2-8) edge[-x-=0.5] node[below=1pt] {$a$} (m-2-6);
   \path[->,font=\scriptsize]
   (m-2-1) edge node[auto,swap] {$a$} (m-2-3);
   \path[dashed,->,font=\scriptsize]
   (m-2-3) edge node[auto,swap] {$b$} (m-1-2)
   (m-2-8) edge node[auto,swap] {$b$} (m-1-7);
  \end{tikzpicture}
\endpgfgraphicnamed
\]

\subsection*{Reflection at $q_0$} After reflecting at $q_1$ (see below), we are led to the following situation.
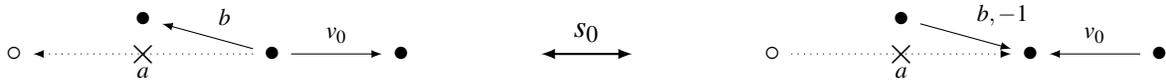
\[
   \beginpgfgraphicnamed{fig8}
   \begin{tikzpicture}[>=latex]
  \matrix (m) [matrix of math nodes, row sep=0em, column sep=2.9em, text height=1ex, text depth=0ex]
   {       &\bullet&    &     &     &     &     &\bullet&   \\   
      \circ&      &\bullet&\bullet&\ & \  &\circ&     &\bullet&\bullet& \\}; 
   \path[thick,<->]
   (m-2-5) edge node[auto] {$s_0$} (m-2-6);
   \path[dotted,->,font=\scriptsize]
   (m-2-3) edge[-x-=0.5] node[below=1pt] {$a$} (m-2-1)
   (m-2-7) edge[-x-=0.5] node[below=1pt] {$a$} (m-2-9);
   \path[<-,font=\scriptsize]
   (m-1-2) edge node[auto] {$b$} (m-2-3)
   (m-2-4) edge node[auto,swap] {$v_0$} (m-2-3)
   (m-2-9) edge node[auto,swap] {$b,-1$} (m-1-8)
   (m-2-9) edge node[auto] {$v_0$} (m-2-10);
  \end{tikzpicture}
\endpgfgraphicnamed
\]
We explain this in detail: let $M$ be the representation whose coefficient quiver is illustrated on the left hand side. Then $M_{q_0}\simeq M_{q_1}\simeq M_{q_b}\simeq\C$ and therefore $(s_{0}M)_{q_0}=\coker(M_{q_o}\to M_{q_1}\oplus M_{q_b})\simeq\C$. The coefficient $=-1$ appears w.r.t.\ to the following choice of a matrix representation for the projection to the cokernel:
\[
 \xymatrix@C=6pc{ M_{q_0} \ar[r]^(0.4){\left[\begin{smallmatrix}1\\1\end{smallmatrix}\right]} & M_{q_1}\oplus M_{q_b} \ar[r]^(0.52){\left[\begin{smallmatrix}1&-1\end{smallmatrix}\right]} & (s_{0}M)_{q_0} }.   
\]

After rescaling the basis elements, we can assume that all coefficients are $1$. Since $Q$ is symmetric in the vertices $q_a$ and $q_b$, we can exchange a posteriori these vertices and obtain
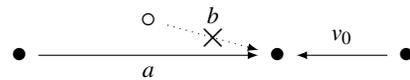
\[
   \beginpgfgraphicnamed{fig9}
   \begin{tikzpicture}[>=latex]
  \matrix (m) [matrix of math nodes, row sep=0em, column sep=2.9em, text height=1ex, text depth=0ex]
   {       &     &     &     &     &     &     &\circ&     &        \\
      \    &\    &\    &\    &\    &\quad&\bullet&   &\bullet&\bullet\\
           &     &     &     &     &     &     &      &     &       \\};
   \path[->,font=\scriptsize]
   (m-2-7) edge node[auto,swap] {$a$} (m-2-9)
   (m-2-10) edge node[auto,swap] {$v_0$} (m-2-9);
   \path[dotted,->,font=\scriptsize]
   (m-1-8) edge[-x-=0.5] node[above=1pt] {$b$} (m-2-9);
  \end{tikzpicture}
\endpgfgraphicnamed
\]
which agrees with the coefficient quiver of the theorem.

\subsection*{Reflection at $q_{x}$ for $x=1,\dotsc,n-5$} The following situations occur if we continue to apply reflections at the vertices $q_1,\dotsc,q_{n-5}$.
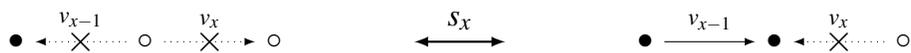
\[
   \beginpgfgraphicnamed{fig10}
   \begin{tikzpicture}[>=latex]
  \matrix (m) [matrix of math nodes, row sep=0em, column sep=2.9em, text height=1ex, text depth=0ex]
   {        &     &     &     &     &     &     &          \\  
     \bullet&\circ&\circ&  \  & \ &\bullet&\bullet&\circ   \\}; 
   \path[thick,<->]
   (m-2-4) edge node[auto] {$s_{x}$} (m-2-5);
   \path[->,font=\scriptsize]
   (m-2-6) edge node[auto] {$v_{x-1}$} (m-2-7);
   \path[dotted,->,font=\scriptsize]
   (m-2-2) edge[-x-=0.5] node[above=1pt] {$v_{x-1}$} (m-2-1)
   (m-2-2) edge[-x-=0.5] node[above=1pt] {$v_{x}$} (m-2-3)
   (m-2-8) edge[-x-=0.5] node[above=1pt] {$v_{x}$} (m-2-7);
  \end{tikzpicture}
\endpgfgraphicnamed
\]
Note that after the reflection, the arrow $v_x$ is pointing towards $q_x$, in agreement with the coefficient quiver of the theorem.

\subsection*{Reflection at $q_{x}$ for $x=1,\dotsc,n-5$}
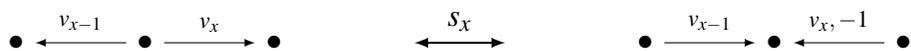
\[
   \beginpgfgraphicnamed{fig11}
   \begin{tikzpicture}[>=latex]
  \matrix (m) [matrix of math nodes, row sep=0em, column sep=2.9em, text height=1ex, text depth=0ex]
   {        &     &     &     &     &     &     &          \\   
     \bullet&\bullet&\bullet& \ & \ &\bullet&\bullet&\bullet \\}; 
   \path[thick,<->]
   (m-2-4) edge node[auto] {$s_{x}$} (m-2-5);
   \path[->,font=\scriptsize]
   (m-2-2) edge node[auto,swap] {$v_{x-1}$} (m-2-1)
   (m-2-2) edge node[auto] {$v_{x}$} (m-2-3)
   (m-2-6) edge node[auto] {$v_{x-1}$} (m-2-7)
   (m-2-8) edge node[auto,swap] {$v_{x},-1$} (m-2-7);
  \end{tikzpicture}
\endpgfgraphicnamed
\]
Note that we can scale the basis elements to renormalize $\mu_{v_r,k,j}$ to $1$.

\subsection*{Reflection at $q_{x}$ for $x=1,\dotsc,n-5$}
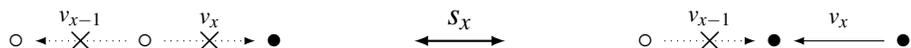
\[
   \beginpgfgraphicnamed{fig20}
   \begin{tikzpicture}[>=latex]
  \matrix (m) [matrix of math nodes, row sep=0em, column sep=2.9em, text height=1ex, text depth=0ex]
   {        &     &     &     &     &     &     &          \\  
     \circ&\circ&\bullet&  \  & \ &\circ&\bullet&\bullet   \\}; 
   \path[thick,<->]
   (m-2-4) edge node[auto] {$s_{x}$} (m-2-5);
   \path[->,font=\scriptsize]
   (m-2-8) edge node[above=1pt] {$v_{x}$} (m-2-7);
   \path[dotted,->,font=\scriptsize]
   (m-2-2) edge[-x-=0.5] node[above=1pt] {$v_{x-1}$} (m-2-1)
   (m-2-2) edge[-x-=0.5] node[above=1pt] {$v_{x}$} (m-2-3)
   (m-2-6) edge[-x-=0.5] node[above=1pt] {$v_{x-1}$} (m-2-7);
  \end{tikzpicture}
\endpgfgraphicnamed
\]
Note that after the reflection, the arrow $v_{x-1}$ is pointing towards $q_x$, in agreement with the coefficient quiver of the theorem.

\subsection*{Reflection at $q_{n-4}$} The following reflections occur at the vertices $q_{n-4}$, $q_c$ and $q_d$.
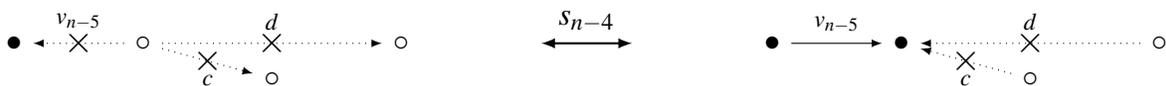
\[
   \beginpgfgraphicnamed{fig12}
   \begin{tikzpicture}[>=latex]
  \matrix (m) [matrix of math nodes, row sep=0em, column sep=2.9em, text height=1ex, text depth=0ex]
   {        &     &     &     &     &     &     &     &    &      \\   
     \bullet&\circ&     &\circ&  \  &  \ &\bullet&\bullet&  &\circ\\   
            &     &\circ&     &     &     &     &     &\circ&     \\};
   \path[thick,<->]
   (m-2-5) edge node[auto] {$s_{n-4}$} (m-2-6);
   \path[->,font=\scriptsize]
   (m-2-7) edge node[auto] {$v_{n-5}$} (m-2-8);
   \path[dotted,->,font=\scriptsize]
   (m-2-2) edge[-x-=0.5] node[above=1pt] {$v_{n-5}$} (m-2-1)
   (m-2-2) edge[-x-=0.5] node[below=1pt] {$c$} (m-3-3)
   (m-2-2) edge[-x-=0.5] node[above=1pt] {$d$} (m-2-4)
   (m-3-9) edge[-x-=0.5] node[below=1pt] {$c$} (m-2-8)
   (m-2-10) edge[-x-=0.5] node[above=1pt] {$d$} (m-2-8);
  \end{tikzpicture}
\endpgfgraphicnamed
\]

\subsection*{Reflection at $q_{c}$}
\[
   \beginpgfgraphicnamed{fig13}
   
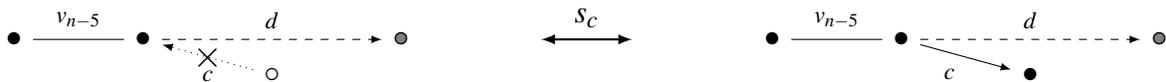
\begin{tikzpicture}[>=latex]
  \matrix (m) [matrix of math nodes, row sep=0em, column sep=2.9em, text height=1ex, text depth=0ex]
   { \bullet&\bullet&  &\ocirc& \ &  \ &\bullet&\bullet&  &\ocirc&  \\
            &     &\circ&     &     &     &     &     &\bullet&     \\};
   \path[thick,<->]
   (m-1-5) edge node[auto] {$s_{c}$} (m-1-6);
   \path[->,font=\scriptsize]
   (m-1-8) edge node[auto,swap] {$c$} (m-2-9);
   \path[dotted,->,font=\scriptsize]
   (m-2-3) edge[-x-=0.5] node[below=1pt] {$c$} (m-1-2);
   \path[-,font=\scriptsize]
   (m-1-1) edge node[auto] {$v_{n-5}$} (m-1-2)
   (m-1-7) edge node[auto] {$v_{n-5}$} (m-1-8);
   \path[->,dashed,font=\scriptsize]
   (m-1-8) edge node[auto] {$d$} (m-1-10)
   (m-1-2) edge node[auto] {$d$} (m-1-4);
  \end{tikzpicture}
\endpgfgraphicnamed
\]

\subsection*{Reflection at $q_{d}$}
\[
   \beginpgfgraphicnamed{fig14}
   
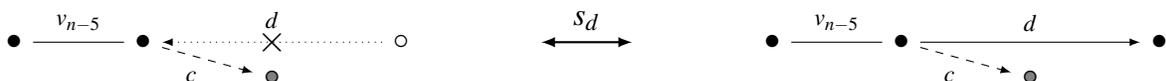
\begin{tikzpicture}[>=latex]
  \matrix (m) [matrix of math nodes, row sep=0em, column sep=2.9em, text height=1ex, text depth=0ex]
   { \bullet&\bullet&  &\circ& \ & \ &\bullet&\bullet& &\bullet \\   
            &    &\ocirc&   &   &     &     &     &\ocirc&      \\}; 
   \path[thick,<->]
   (m-1-5) edge node[auto] {$s_{d}$} (m-1-6);
   \path[->,font=\scriptsize]
   (m-1-8) edge node[auto] {$d$} (m-1-10);
   \path[dotted,->,font=\scriptsize]
   (m-1-4) edge[-x-=0.5] node[above=1pt] {$d$} (m-1-2);
   \path[-,font=\scriptsize]
   (m-1-1) edge node[auto] {$v_{n-5}$} (m-1-2)
   (m-1-7) edge node[auto] {$v_{n-5}$} (m-1-8);
   \path[->,dashed,font=\scriptsize]
   (m-1-8) edge node[auto,swap] {$c$} (m-2-9)
   (m-1-2) edge node[auto,swap] {$c$} (m-2-3);
  \end{tikzpicture}
\endpgfgraphicnamed
\]

\subsection*{Reflection at $q_{n-4}$}
\[
   \beginpgfgraphicnamed{fig15}
   
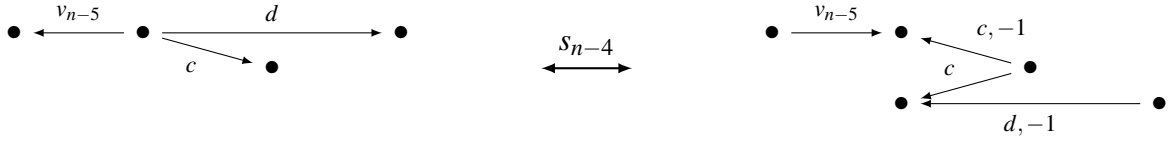
\begin{tikzpicture}[>=latex]
  \matrix (m) [matrix of math nodes, row sep=0em, column sep=2.9em, text height=1ex, text depth=0ex]
   { \bullet&\bullet& &\bullet&     &    &\bullet&\bullet&  &     \\   
            &    &\bullet&    &  \  &  \  &     &    &\bullet&    \\   
            &     &     &     &     &    &     &\bullet&   &\bullet \\}; 
   \path[thick,<->]
   (m-2-5) edge node[auto] {$s_{n-4}$} (m-2-6);
   \path[->,font=\scriptsize]
   (m-1-2) edge node[auto,swap] {$v_{n-5}$} (m-1-1)
   (m-1-2) edge node[auto] {$d$} (m-1-4)
   (m-1-2) edge node[auto,swap] {$c$} (m-2-3)
   (m-1-7) edge node[auto] {$v_{n-5}$} (m-1-8)
   (m-2-9) edge node[auto,swap] {$c,-1$} (m-1-8)
   (m-2-9) edge node[auto,swap] {$c$} (m-3-8)
   (m-3-10) edge node[auto] {$d,-1$} (m-3-8);
  \end{tikzpicture}
\endpgfgraphicnamed
\]
We explain this in detail: let $M$ be the representation on the left hand side in the above illustration. In contrast to all the previous cases, the cokernel of $M_{n-4}\to M_{n-5}\oplus M_c\oplus M_d$ is $2$-dimensional since $M_{n-4}= M_{n-5}=M_c=M_d=\C$. The following choice of a matrix representation for the projection to the cokernel yields the coefficient quiver on the right hand side of the above illustration:
\[
 \xymatrix@C=6pc{ M_{n-4} \ar[r]^(0.4){\left[\begin{smallmatrix}1\\1\\1\end{smallmatrix}\right]} & M_{n-5}\oplus M_c\oplus M_d \ar[r]^(0.52){\left[\begin{smallmatrix}1&-1&\ \ 0\\0&\ \ 1&-1\end{smallmatrix}\right]} & (s_{n-4}M)_{n-4} }.   
\]
Note that after the reflection, $v_{n-5}$ points towards $q_{n-4}$, in agreement with the coefficient quiver of the theorem.

\subsection*{Reflection at $q_d$} After the last reflection at $q_{n-4}$, reflecting at $q_d$ has the following effect where $c$ can have either orientation.
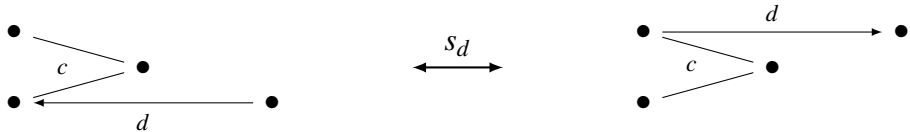
\[
   \beginpgfgraphicnamed{fig24}
   \begin{tikzpicture}[>=latex]
  \matrix (m) [matrix of math nodes, row sep=0em, column sep=2.9em, text height=1ex, text depth=0ex]
   { \bullet&     &     &    &    &\bullet&   &\bullet\\  
           &\bullet&    &  \ &  \  &    &\bullet&    \\   
     \bullet&    &\bullet&   &    &\bullet&    &     \\}; 
   \path[thick,<->]
   (m-2-4) edge node[auto] {$s_d$} (m-2-5);
   \path[-,font=\scriptsize]
   (m-1-1) edge node[auto] {} (m-2-2)
   (m-3-1) edge node[auto] {$c$} (m-2-2)
   (m-1-6) edge node[auto,swap] {$c$} (m-2-7)
   (m-3-6) edge node[auto,swap] {} (m-2-7);
   \path[->,font=\scriptsize]
   (m-3-3) edge node[auto] {$d$} (m-3-1)
   (m-1-6) edge node[auto] {$d$} (m-1-8);
  \end{tikzpicture}
\endpgfgraphicnamed
\]

\subsection*{Reflection at $q_c$} 
\[
   \beginpgfgraphicnamed{fig25}
   
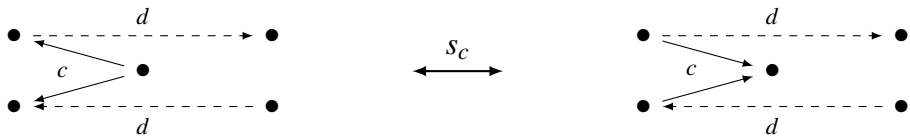
\begin{tikzpicture}[>=latex]
  \matrix (m) [matrix of math nodes, row sep=0em, column sep=2.9em, text height=1ex, text depth=0ex]
   { \bullet&    &\bullet&   &   &\bullet&   &\bullet\\    
           &\bullet&    &  \ & \  &    &\bullet&     \\   
     \bullet&    &\bullet&   &   &\bullet&   &\bullet\\}; 
   \path[thick,<->]
   (m-2-4) edge node[auto] {$s_c$} (m-2-5);
   \path[->,font=\scriptsize]
   (m-2-2) edge node[auto,swap] {} (m-1-1)
   (m-2-2) edge node[auto,swap] {$c$} (m-3-1)
   (m-1-6) edge node[auto,swap] {$c$} (m-2-7)
   (m-3-6) edge node[auto,swap] {} (m-2-7);
   \path[dashed,->,font=\scriptsize]
   (m-3-3) edge node[auto] {$d$} (m-3-1)
   (m-1-1) edge node[auto] {$d$} (m-1-3)
   (m-3-8) edge node[auto] {$d$} (m-3-6)
   (m-1-6) edge node[auto] {$d$} (m-1-8);
  \end{tikzpicture}
\endpgfgraphicnamed
\]
Note that, a priori, the reflection at $q_c$ causes a coefficient that is not equal to one. However, by scaling the basis elements we can renormalize all coefficients to $1$, which justifies the illustration as above.

\subsection*{Reflection at $q_{n-4}$} After a reflection at $q_{n-5}$ , we are led to the following situation.
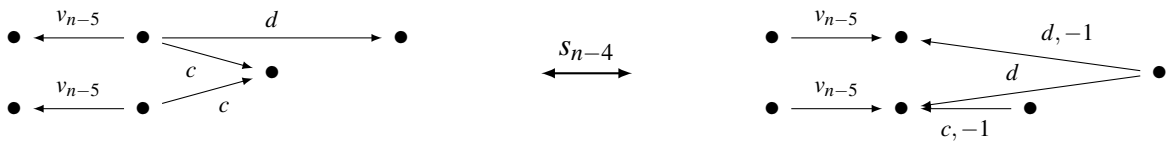
\[
   \beginpgfgraphicnamed{fig16}
   \begin{tikzpicture}[>=latex]
  \matrix (m) [matrix of math nodes, row sep=0em, column sep=2.9em, text height=1ex, text depth=0ex]
   { \bullet&\bullet&       &\bullet&   &   &\bullet&\bullet&       &       \\   
            &       &\bullet&       & \ & \ &       &       &       &\bullet\\
     \bullet&\bullet&       &       &   &   &\bullet&\bullet&\bullet&       \\}; 
   \path[thick,<->]
   (m-2-5) edge node[auto] {$s_{n-4}$} (m-2-6);
   \path[->,font=\scriptsize]
   (m-1-2) edge node[auto,swap] {$v_{n-5}$} (m-1-1)
   (m-1-2) edge node[auto] {$d$} (m-1-4)
   (m-1-2) edge node[auto,swap] {$c$} (m-2-3)
   (m-3-2) edge node[auto,swap] {$c$} (m-2-3)
   (m-3-2) edge node[auto,swap] {$v_{n-5}$} (m-3-1)
   (m-1-7) edge node[auto] {$v_{n-5}$} (m-1-8)
   (m-2-10) edge node[auto,swap] {$d,-1$} (m-1-8)
   (m-2-10) edge node[auto,swap] {$d$} (m-3-8)
   (m-3-7) edge node[auto] {$v_{n-5}$} (m-3-8)
   (m-3-9) edge node[auto] {$c,-1$} (m-3-8);
  \end{tikzpicture}
\endpgfgraphicnamed
\]
We can scale the basis elements to normalize all coefficients to $1$. A posteriori, we can use the symmetry of $Q$ in $c$ and $d$ to exchange the role of these edges, which yields the coefficient quiver
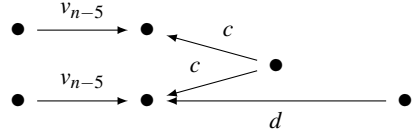
\[
   \beginpgfgraphicnamed{fig17}
   \begin{tikzpicture}[>=latex]
  \matrix (m) [matrix of math nodes, row sep=0em, column sep=2.9em, text height=1ex, text depth=0ex]
   {        &     &     &     &     &\bullet&\bullet&       &        \\   
      \qquad&\quad&\quad&\quad&\quad&       &       &\bullet&        \\   
            &     &     &     &     &\bullet&\bullet&       &\bullet \\}; 
   \path[->,font=\scriptsize]
   (m-1-6) edge node[auto] {$v_{n-5}$} (m-1-7)
   (m-2-8) edge node[auto,swap] {$c$} (m-1-7)
   (m-2-8) edge node[auto,swap] {$c$} (m-3-7)
   (m-3-6) edge node[auto] {$v_{n-5}$} (m-3-7)
   (m-3-9) edge node[auto] {$d$} (m-3-7);
  \end{tikzpicture}
\endpgfgraphicnamed
\]
which agrees with the coefficient quiver of the theorem.

\bigskip\noindent There are analogous cases of reflections at the vertices $q_0$, $q_a$ and $q_b$, which can be deduced from the previous cases by using the symmetry 
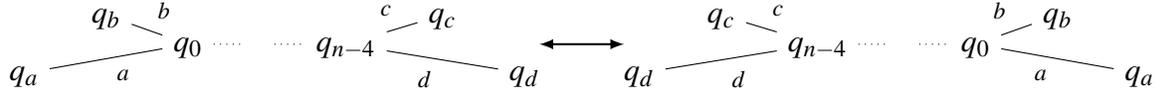
\[
   \beginpgfgraphicnamed{fig18}
   \begin{tikzpicture}[>=latex]
  \matrix (m) [matrix of math nodes, row sep=-0.2em, column sep=1em, text height=1ex, text depth=0ex]
   {      & q_b &     &     &       & q_c &      &   &     & q_c &       &     &       & q_b &       \\
          &     & q_0 &  \  &q_{n-4}&     &  \   &   &  \  &     &q_{n-4}&  \  &q_{0}  &     &  \    \\
      q_a &     &     &     &       &     &  q_d &   & q_d &     &       &     &       &     &  q_a  \\};
   \path[thick,<->]
   (m-2-7) edge node[auto] {} (m-2-9);
   \path[-,font=\scriptsize]
   (m-3-1) edge node[auto,swap] {$a$} (m-2-3)
   (m-1-2) edge node[auto] {$b$} (m-2-3)
   (m-2-5) edge node[auto] {$c$} (m-1-6)
   (m-2-5) edge node[auto,swap] {$d$} (m-3-7)
   (m-3-9) edge node[auto,swap] {$d$} (m-2-11)
   (m-1-10) edge node[auto] {$c$} (m-2-11)
   (m-2-13) edge node[auto] {$b$} (m-1-14)
   (m-2-13) edge node[auto,swap] {$a$} (m-3-15);
   \path[dotted,font=\scriptsize]
   (m-2-3) edge node[auto] {} (m-2-4)
   (m-2-4) edge node[auto] {} (m-2-5)
   (m-2-11) edge node[auto] {} (m-2-12)
   (m-2-12) edge node[auto] {} (m-2-13);
  \end{tikzpicture}
\endpgfgraphicnamed
\]

These steps cover all situations that occur when we apply successively reflections $s_{p_1},\dotsc, s_{p_r}$ to the simple projective $S_{q_a}$. We conclude that, up to an automorphism of the underlying Dynkin diagram, every preprojective representation of $Q$ of defect $-1$ has a basis of the described form, and the coefficient quiver of the theorem indeed determines a preprojective representation of defect $-1$. This finishes the proof of the theorem.
\end{proof}


\subsection{Defect $-2$} 
\label{subsection: defect -2}

A coefficient quiver for indecomposable representations of defect $-2$ can be constructed by using reflections at sources to simple projective representations whose support is one of the vertices $q_0, \dotsc,q_{n-4}$. Since we do not need to know the explicit form of these coefficient quiver, we forgo to give a description. We refer to section 1.6 of the sequel \cite{LW15} of this paper for more information about representations of defect $-2$. The reader will, in particular, find information about how to glue together two coefficient quivers of defect $-1$-representations to yield a coefficient quiver for a defect $-2$-representation.


\subsection{Positive defect} 
\label{subsection: positive defect}

The dual representation of a representation $M$ with coefficient quiver $\Gamma$ comes with a dual coefficient quiver $\Gamma^\op$, which is obtained from $\Gamma$ by reversing all arrows. Therefore, coefficient quivers for representations of positive defect are obtained from the coefficient quivers of their duals, which are representations of negative defect. We refer to section 1.8 of the sequel \cite{LW15} to this paper for more details.


\subsection{Exceptional tubes of rank $2$} 
\label{subsection: bases for exceptional tubes of rank 2}

If $n=4$, then all of the three exceptional tubes are of rank $2$. If $n\geq5$, then only two of them are of rank $2$. All of the rank $2$-tubes behave completely analogous. Each of the tubes contains two quasi-simple representations, and every other representation of a tube contains a unique quasi-simple subrepresentation, which is contained in the same tube.

The coefficient quivers of the quasi-simples are of the following shape. We apply the same conventions for our illustrations as in the previous section, i.e.\ the dashed arrows are part of the coefficient quiver if and only if the corresponding arrow of $Q$ is oriented in the indicated direction.
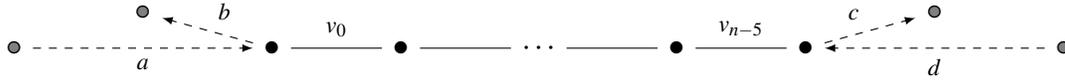
\[
   \beginpgfgraphicnamed{fig3}
   \begin{tikzpicture}[>=latex]
  \matrix (m) [matrix of math nodes, row sep=0em, column sep=2.9em, text height=1ex, text depth=0ex]
   {          &\ocirc &       &       &        &       &       &\ocirc &       \\   
      \ocirc &       &\bullet&\bullet&\dotsb  &\bullet&\bullet&       &\ocirc \\   
};
   \path[-,font=\scriptsize]
   (m-2-3) edge node[auto] {$v_{0}$} (m-2-4)
   (m-2-6) edge node[auto] {$v_{n-5}$} (m-2-7)
   ;
   \path[->,dashed,font=\scriptsize]
   (m-2-3) edge node[auto,swap] {$b$} (m-1-2)
   (m-2-1) edge node[auto,swap] {$a$} (m-2-3)
   (m-2-9) edge node[auto] {$d$} (m-2-7)
   (m-2-7) edge node[auto] {$c$} (m-1-8)
   ;
   \path[-,font=\scriptsize]
   (m-2-4) edge node[auto] {} (m-2-5)
   (m-2-5) edge node[auto,swap] {} (m-2-6)
   ;
  \end{tikzpicture}
\endpgfgraphicnamed
\]
The corresponding illustration of its Auslander-Reiten translate, which is the other quasi-simple representation of the same tube, can be illustrated in the same way, but with all dashed arrows reversed. 

In the following we assume that the vertices $q_a,q_b,q_c$ and $q_d$ are sources of $Q$. The coefficient quivers for all other orientations of $Q$ are easily deduced from this case by applying reflections to $q_a,q_b,q_c$ and $q_d$.

Let $\delta(m,n)$ be the dimension vector with $\delta(m,n)_{q_i}=m+n$ for $i=0,\ldots,n-4$, $\delta(m,n)_{q_{a}}=\delta(m,n)_{q_c}=m$ and $\delta(m,n)_{q_b}=\delta(m,n)_{q_d}=n$. 
Then $\delta(0,1)$ and $\delta(1,0)$ are quasi-simple and $\delta(n-1,n)=\delta(0,1)^{n}+\delta(1,0)^{n-1}$ and $\delta(n,n-1)=\delta(0,1)^{n-1}+\delta(1,0)^{n}$ are real roots. Note that we have $\ext(\delta(0,1),\delta(1,0))=\ext(\delta(1,0),\delta(0,1))=1$. Moreover, if $n\geq 2$ the corresponding roots are not Schurian. Finally, the imaginary roots are $\delta(n,n)=\delta(0,1)^n+\delta(1,0)^n$. In the following, we use the notation $\delta(n,n)^i$ with $i\in\{1,2\}$ to distinguish between different representations of this dimension. The tube of rank $2$ looks like
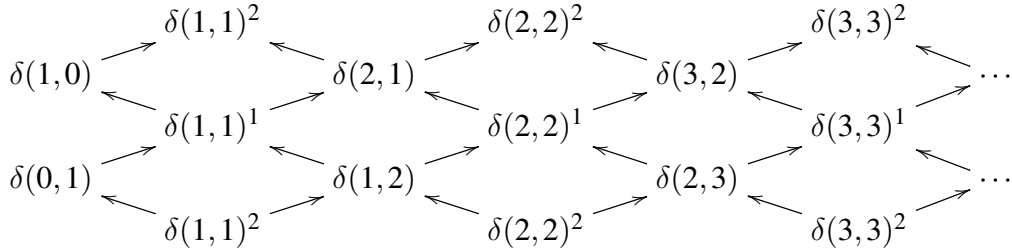
\[
\xymatrix@R0pt@C20pt{
                  &\delta(1,1)^2               &                         &\delta(2,2)^2              &                         &\delta(3,3)^2       &     \\
\delta(1,0)\ar[ru]&                            &\delta(2,1)\ar[ru]\ar[lu]&                           &\delta(3,2)\ar[lu]\ar[ru]&                    &\dotsb\ar[lu]\\
                  &\delta(1,1)^1\ar[ru]\ar[lu] &                         &\delta(2,2)^1\ar[ru]\ar[lu]&                         &\delta(3,3)^1\ar[lu]\ar[ru]&     \\
\delta(0,1)\ar[ru]&                            &\delta(1,2)\ar[ru]\ar[lu]&                           &\delta(2,3)\ar[ru]\ar[lu]&                    &\dotsb\ar[lu]\\
                  &\delta(1,1)^2\ar[ru]\ar[lu] &                         &\delta(2,2)^2\ar[ru]\ar[lu]&                         &\delta(3,3)^2\ar[lu]\ar[ru]&        }
\]
where the top and bottom row are identified.

We obtain a coefficient quiver of an arbitrary representation in this tube by glueing the coefficient quivers of the two quasi-simples in an appropriate way. We order the basis elements of the coefficient quiver in such a way that we can find the quasi-simple subrepresentation at the top and the factor representation at the bottom. 

The quasi-simple factor representation of an imaginary representations $\delta(1,1)^i$ differs from its quasi-simple subrepresentation. Up to a permutation of $a$, $b$, $c$ and $d$, $\delta(1,1)^i$ has the coefficient quiver
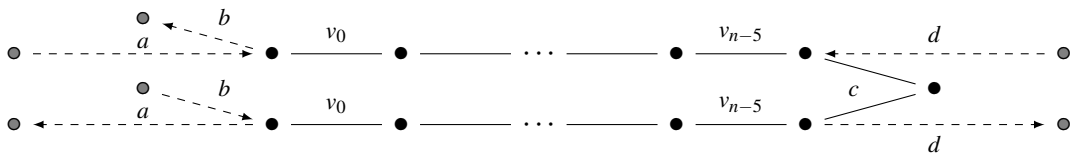
\[
   \beginpgfgraphicnamed{fig7}
   \begin{tikzpicture}[>=latex]
  \matrix (m) [matrix of math nodes, row sep=0em, column sep=2.9em, text height=1ex, text depth=0ex]
   {         &\ocirc &       &       &        &       &       &       &       \\   
      \ocirc &       &\bullet&\bullet&\dotsb  &\bullet&\bullet&       &\ocirc \\   
             &\ocirc &       &       &        &       &       &\bullet&       \\   
      \ocirc &       &\bullet&\bullet&\dotsb  &\bullet&\bullet&       &\ocirc \\   
};
   \path[-,font=\scriptsize]
   (m-2-3) edge node[auto] {$v_{0}$} (m-2-4)
   (m-2-4) edge node[auto] {} (m-2-5)
   (m-2-5) edge node[auto,swap] {} (m-2-6)
   (m-2-6) edge node[auto] {$v_{n-5}$} (m-2-7)
   (m-3-8) edge node[auto] {} (m-2-7)
   (m-3-8) edge node[auto,swap] {$c$} (m-4-7)
   (m-4-3) edge node[auto] {$v_{0}$} (m-4-4)
   (m-4-4) edge node[auto] {} (m-4-5)
   (m-4-5) edge node[auto,swap] {} (m-4-6)
   (m-4-6) edge node[auto] {$v_{n-5}$} (m-4-7)
   ;
   \path[->,dashed,font=\scriptsize]
   (m-2-3) edge node[auto,swap] {$b$} (m-1-2)
   (m-2-1) edge node[above=-2pt] {$a$} (m-2-3)
   (m-2-9) edge node[auto,swap] {$d$} (m-2-7)
   (m-4-7) edge node[auto,swap] {$d$} (m-4-9)
   (m-3-2) edge node[auto] {$b$} (m-4-3)
   (m-4-3) edge node[above=-2pt] {$a$} (m-4-1)
   ;
  \end{tikzpicture}
\endpgfgraphicnamed
\]
The representations $\delta(1,2)$ and $\delta(2,1)$ have the same quasi-simple representation as a factor and subrepresentation, and their coefficient quivers are of the form
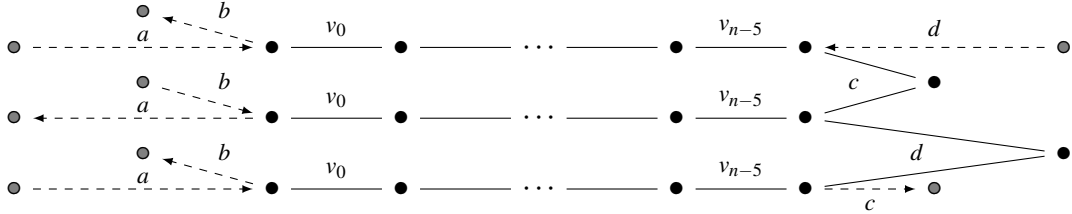
\[
   \beginpgfgraphicnamed{fig22}
   \begin{tikzpicture}[>=latex]
  \matrix (m) [matrix of math nodes, row sep=0em, column sep=2.9em, text height=1ex, text depth=0ex]
   {         &\ocirc &       &       &        &       &       &       &       \\   
      \ocirc &       &\bullet&\bullet&\dotsb  &\bullet&\bullet&       &\ocirc \\   
             &\ocirc &       &       &        &       &       &\bullet&       \\   
      \ocirc &       &\bullet&\bullet&\dotsb  &\bullet&\bullet&       &       \\   
             &\ocirc &       &       &        &       &       &       &\bullet\\   
      \ocirc &       &\bullet&\bullet&\dotsb  &\bullet&\bullet&\ocirc &       \\   
};
   \path[-,font=\scriptsize]
   (m-2-3) edge node[auto] {$v_{0}$} (m-2-4)
   (m-2-4) edge node[auto] {} (m-2-5)
   (m-2-5) edge node[auto,swap] {} (m-2-6)
   (m-2-6) edge node[auto] {$v_{n-5}$} (m-2-7)
   (m-3-8) edge node[auto] {} (m-2-7)
   (m-3-8) edge node[auto,swap] {$c$} (m-4-7)
   (m-4-3) edge node[auto] {$v_{0}$} (m-4-4)
   (m-4-4) edge node[auto] {} (m-4-5)
   (m-4-5) edge node[auto,swap] {} (m-4-6)
   (m-4-6) edge node[auto] {$v_{n-5}$} (m-4-7)
   (m-5-9) edge node[auto] {} (m-4-7)
   (m-5-9) edge node[auto,swap] {$d$} (m-6-7)
   (m-6-3) edge node[auto] {$v_{0}$} (m-6-4)
   (m-6-4) edge node[auto] {} (m-6-5)
   (m-6-5) edge node[auto,swap] {} (m-6-6)
   (m-6-6) edge node[auto] {$v_{n-5}$} (m-6-7)
   ;
   \path[->,dashed,font=\scriptsize]
   (m-2-3) edge node[auto,swap] {$b$} (m-1-2)
   (m-2-1) edge node[above=-2pt] {$a$} (m-2-3)
   (m-2-9) edge node[auto,swap] {$d$} (m-2-7)
   (m-3-2) edge node[auto] {$b$} (m-4-3)
   (m-4-3) edge node[above=-2pt] {$a$} (m-4-1)
   (m-6-3) edge node[auto,swap] {$b$} (m-5-2)
   (m-6-1) edge node[above=-2pt] {$a$} (m-6-3)
   (m-6-7) edge node[auto,swap] {$c$} (m-6-8)
   ;
  \end{tikzpicture}
\endpgfgraphicnamed
\]
up to a permutation of $a$, $b$, $c$ and $d$. We obtain coefficient quivers for larger roots in a rank $2$-tube by inserting an appropriate number of copies of the quiver
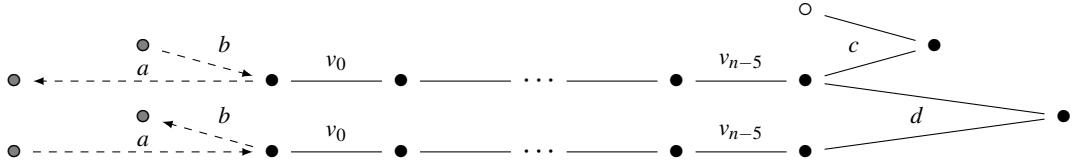
\[
   \beginpgfgraphicnamed{fig23}
   \begin{tikzpicture}[>=latex]
  \matrix (m) [matrix of math nodes, row sep=0em, column sep=2.9em, text height=1ex, text depth=0ex]
   {         &       &       &       &        &       &       &       &       \\   
             &       &       &       &        &       &\circ  &       &       \\   
             &\ocirc &       &       &        &       &       &\bullet&       \\   
      \ocirc &       &\bullet&\bullet&\dotsb  &\bullet&\bullet&       &       \\   
             &\ocirc &       &       &        &       &       &       &\bullet\\   
      \ocirc &       &\bullet&\bullet&\dotsb  &\bullet&\bullet&       &       \\   
};
   \path[-,font=\scriptsize]
   (m-3-8) edge node[auto] {} (m-2-7)
   (m-3-8) edge node[auto,swap] {$c$} (m-4-7)
   (m-4-3) edge node[auto] {$v_{0}$} (m-4-4)
   (m-4-4) edge node[auto] {} (m-4-5)
   (m-4-5) edge node[auto,swap] {} (m-4-6)
   (m-4-6) edge node[auto] {$v_{n-5}$} (m-4-7)
   (m-5-9) edge node[auto] {} (m-4-7)
   (m-5-9) edge node[auto,swap] {$d$} (m-6-7)
   (m-6-3) edge node[auto] {$v_{0}$} (m-6-4)
   (m-6-4) edge node[auto] {} (m-6-5)
   (m-6-5) edge node[auto,swap] {} (m-6-6)
   (m-6-6) edge node[auto] {$v_{n-5}$} (m-6-7)
   ;
   \path[->,dashed,font=\scriptsize]
   (m-3-2) edge node[auto] {$b$} (m-4-3)
   (m-4-3) edge node[above=-2pt] {$a$} (m-4-1)
   (m-6-3) edge node[auto,swap] {$b$} (m-5-2)
   (m-6-1) edge node[above=-2pt] {$a$} (m-6-3)
   ;
  \end{tikzpicture}
\endpgfgraphicnamed
\]
in the middle of one of the above quivers.


\subsection{Exceptional tubes of rank $n-2$} 
\label{subsection: bases for exceptional tubes of rank n-2}

If $n=4$, then the following construction is valid for each of the three tubes of rank $2=4-2$, up to a permutation of $a$, $b$, $c$ and $d$. If $n\geq5$, then there is a unique tube of rank $n-2$. We begin with a description of the $n-2$ quasi-simple representations of a tube of rank $n-2$. 

Roughly speaking, the quasi-simple representations are the thin representations whose support correspond to the maximal equioriented-oriented subgraphs of $Q$ w.r.t.\ to our way of ordering the vertices from left to right. More explicitly, the coefficient quivers of the quasi-simple representations are of the forms
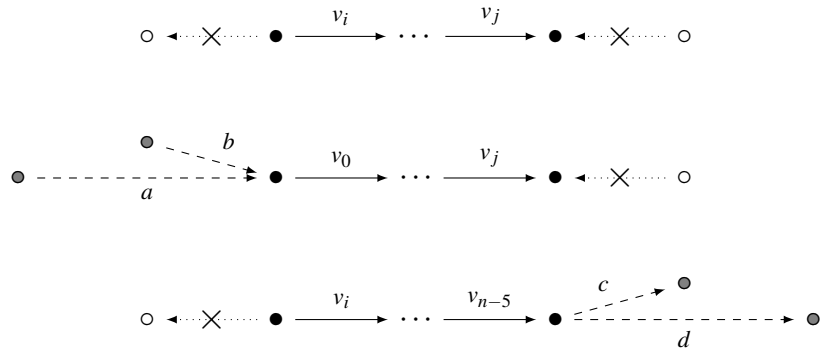
\[
   \beginpgfgraphicnamed{fig40}
   \begin{tikzpicture}[>=latex]
  \matrix (m) [matrix of math nodes, row sep=0em, column sep=2.9em, text height=1ex, text depth=0ex]
   {        &\circ  &\bullet& \dotsb&\bullet&\circ  &      \\   
       \    &       &       &       &       &       &      \\   
       \    &       &       &       &       &       &      \\   
            &\ocirc &       &       &       &       &      \\   
     \ocirc &       &\bullet&\dotsb &\bullet&\circ  &      \\   
       \    &       &       &       &       &       &      \\   
       \    &       &       &       &       &       &      \\   
            &       &       &       &       &\ocirc &      \\   
            & \circ &\bullet&\dotsb &\bullet&       &\ocirc\\   
};
   \path[->,font=\scriptsize]
   (m-1-3) edge node[auto] {$v_i$} (m-1-4)
   (m-1-4) edge node[auto] {$v_j$} (m-1-5)
   ;
   \path[->,dotted,font=\scriptsize]
   (m-1-3) edge[-x-=0.5] node[auto,swap] {} (m-1-2)
   (m-1-6) edge[-x-=0.5] node[auto,swap] {} (m-1-5)
   ;
   \path[->,font=\scriptsize]
   (m-5-3) edge node[auto] {$v_0$} (m-5-4)
   (m-5-4) edge node[auto] {$v_j$} (m-5-5)
   ;
   \path[->,dashed,font=\scriptsize]
   (m-5-1) edge node[auto,swap] {$a$} (m-5-3)
   (m-4-2) edge node[auto] {$b$} (m-5-3)
   ;
   \path[->,dotted,font=\scriptsize]
   (m-5-6) edge[-x-=0.5] node[auto,swap] {} (m-5-5)
   ;
   \path[->,font=\scriptsize]
   (m-9-3) edge node[auto] {$v_i$} (m-9-4)
   (m-9-4) edge node[auto] {$v_{n-5}$} (m-9-5)
   ;
   \path[->,dashed,font=\scriptsize]
   (m-9-5) edge node[auto,swap] {$d$} (m-9-7)
   (m-9-5) edge node[auto] {$c$} (m-8-6)
   ;
   \path[->,dotted,font=\scriptsize]
   (m-9-3) edge[-x-=0.5] node[auto,swap] {} (m-9-2)
   ;
  \end{tikzpicture}
\endpgfgraphicnamed
\]
up to reversing all arrows, where $0\leq i\leq j\leq n-4$. Note that this includes the following extremal cases (up to reversing all arrows and interchanging $a,b,v_0$ with $d,c,v_{n-4}$).
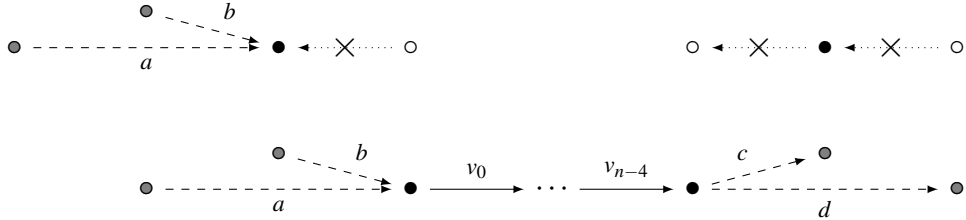
\[
   \beginpgfgraphicnamed{fig127}
   \begin{tikzpicture}[>=latex]
  \matrix (m) [matrix of math nodes, row sep=0em, column sep=3em, text height=1ex, text depth=0ex]
   {   \    &\ocirc &       &       &       &       &       &       &      \\   
      \ocirc&       &\bullet&\circ  &       &\circ  &\bullet&\circ  &\quad \\   
       \    &       &       &       &       &       &       &       &      \\   
       \    &       &       &       &       &       &       &       &      \\   
            &       &\ocirc &       &       &       &\ocirc &       &      \\   
            &\ocirc &       &\bullet&\dotsb &\bullet&       &\ocirc &      \\   
};
   \path[->,font=\scriptsize]
   (m-6-4) edge node[auto] {$v_0$} (m-6-5)
   (m-6-5) edge node[auto] {$v_{n-4}$} (m-6-6)
   ;
   \path[->,dashed,font=\scriptsize]
   (m-2-1) edge node[auto,swap] {$a$} (m-2-3)
   (m-1-2) edge node[auto] {$b$} (m-2-3)
   (m-6-2) edge node[auto,swap] {$a$} (m-6-4)
   (m-5-3) edge node[auto] {$b$} (m-6-4)
   (m-6-6) edge node[auto,swap] {$d$} (m-6-8)
   (m-6-6) edge node[auto] {$c$} (m-5-7)
   ;
   \path[->,dotted,font=\scriptsize]
   (m-2-4) edge[-x-=0.5] node[auto,swap] {} (m-2-3)
   (m-2-7) edge[-x-=0.5] node[auto,swap] {} (m-2-6)
   (m-2-8) edge[-x-=0.5] node[auto,swap] {} (m-2-7)
   ;
  \end{tikzpicture}
\endpgfgraphicnamed
\]

We continue with a description of the tube of rank $n-2$. The roots of the tube depend on the orientation of $Q$. In the following, we will describe the exceptional roots for $\widetilde D_n$ in subspace orientation, i.e.\ the arrows $a$ and $b$ are orientated to the right and all the other arrows of $Q$ are oriented to the left, according to our conventions for illustrations. 

For $1\leq i\leq j\leq n-4$, let $d^{i,j}$ and $\hat d^{i,j}$ be the dimension vectors with 
\[
 d^{i,j}_q \ = \ \begin{cases} 1 \quad \text{if }q\in\{q_a,q_b,q_c,q_d,q_i,\ldots,q_j\}, \\
                               2 \quad \text{otherwise},                                      \end{cases}  
 \quad \text{and} \quad
 \hat d^{i,j}_q \ = \ \begin{cases} 1 \quad \text{if }q\in\{q_i,,\ldots,q_j\}, \\
                                    0 \quad \text{otherwise}.                        \end{cases}
\]
For $1\leq j\leq n-4$, let $d^{0,j}$ and $\hat d^{0,j}$ be the dimension vectors with 
\[
 d^{0,j}_q=\begin{cases}1 \quad \text{if } q\in\{q_c,q_d,q_0,\ldots,q_j\},       \\
                        0 \quad \text{if } q\in\{q_a,q_b\},\\
                        2 \quad \text{otherwise},        \end{cases}
 \quad \text{and} \quad
 \hat d^{0,j}_q=\begin{cases}1 \quad \text{if }q\in\{q_a,q_b,q_0,\ldots,q_j\}, \\
                             0 \quad \text{otherwise}.                             \end{cases}
\]
Then we have $\ext(d^{i,j},\hat d^{i,j})=\ext(\hat d^{i,j},d^{i,j})=1$. These are the exceptional roots of the tube, and the corresponding roots for other orientations of $Q$ can be derived by applying an appropriate sequence of reflection functors.

Independent of the orientation of $Q$, the vectors $\delta^{i,j}(s,t):=(d^{i,j})^{s}+(\hat d^{i,j})^{t}$ are roots if $s\in\{t-1,t,t+1\}$. If we have $(s,t)\in\{(1,0),(0,1)\}$, then $\delta^{i,j}(s,t)$ are Schur roots. There are $n-2$ chains of irreducible inclusions, which are Auslander-Reiten translates of the chain 
\[
 M^{0,0}(0,1) \ \longrightarrow \ M^{0,1}(0,1) \ \longrightarrow \ \dotsb \ \longrightarrow \ M^{0,n-4}(0,1) \ \longrightarrow \ M(1,1) \ \longrightarrow \ M^{0,0}(1,2) \ \longrightarrow \ \dotsb
\]
where $M^{i,j}(s,s+1)$ is the unique indecomposable representation with dimension vector $\delta^{i,j}(s,s+1)$ and $M(s,s)$ is the unique imaginary root representation that fits into the chain. Since every indecomposable representation of the tube belongs to a unique chain, the chains determine the shape of the tube. For instance for $n=6$, the tube looks like
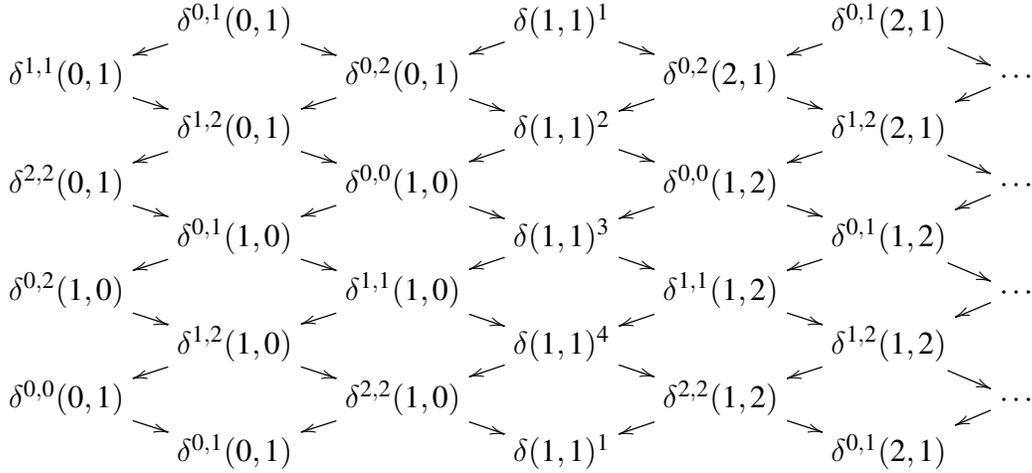
\[
\begin{xy}
\xymatrix@R0pt@C13pt{
&\delta^{0,1}(0,1)\ar[ld]\ar[rd]&&\delta(1,1)^1\ar[ld]\ar[rd]&&\delta^{0,1}(2,1)\ar[ld]\ar[rd]&\\
\delta^{1,1}(0,1)\ar[rd]&&\delta^{0,2}(0,1)\ar[ld]\ar[rd]&&\delta^{0,2}(2,1)\ar[ld]\ar[rd]&&\dotsb\ar[ld]\\
&\delta^{1,2}(0,1)\ar[ld]\ar[rd]&&\delta(1,1)^2\ar[ld]\ar[rd]&&\delta^{1,2}(2,1)\ar[ld]\ar[rd]&\\
\delta^{2,2}(0,1)\ar[rd]&&\delta^{0,0}(1,0)\ar[ld]\ar[rd]&&\delta^{0,0}(1,2)\ar[ld]\ar[rd]&&\dotsb\ar[ld]\\
&\delta^{0,1}(1,0)\ar[ld]\ar[rd]&&\delta(1,1)^3\ar[ld]\ar[rd]&&\delta^{0,1}(1,2)\ar[ld]\ar[rd]&\\
\delta^{0,2}(1,0)\ar[rd]&&\delta^{1,1}(1,0)\ar[ld]\ar[rd]&&\delta^{1,1}(1,2)\ar[ld]\ar[rd]&&\dotsb\ar[ld]\\
&\delta^{1,2}(1,0)\ar[ld]\ar[rd]&&\delta(1,1)^4\ar[ld]\ar[rd]&&\delta^{1,2}(1,2)\ar[ld]\ar[rd]&\\
\delta^{0,0}(0,1)\ar[rd]&&\delta^{2,2}(1,0)\ar[ld]\ar[rd]&&\delta^{2,2}(1,2)\ar[ld]\ar[rd]&&\dotsb\ar[ld]\\
&\delta^{0,1}(0,1)&&\delta(1,1)^1&&\delta^{0,1}(2,1)&\\}
\end{xy}
\]
where the bottom and the top-row are identified. From this, it is clear that each representation in this tube contains a unique quasi-simple subrepresentation and a unique quasi-simple factor representation. We construct coefficient quivers for these representations by the following recursive procedure. For the quasi-simples, we use the coefficient quivers as described above. Let $M$ be any representation in the tube, let
\[
 N_0 \hookrightarrow N_1 \hookrightarrow\ldots\hookrightarrow N_n=M 
\]
the (unique) chain of irreducible inclusions such that $N_0$ is quasi-simple and let $N'$ be the (unique) quasi-simple factor of $M$. Assume that we have constructed the coefficient quiver for $N_{n-1}$. In order to construct the coefficient quiver for $M$, we glue $N'$ to $N_{n-1}$ by inserting an arrow that represents a non-trivial element of $\Ext(N',N_{n-1})$. The start of this arrow is a vertex of $N'$ and its target is a vertex of $N_{n-1}$. 

Using the appropriate choices for these arrows yields a snake shaped coefficient quivers of the form
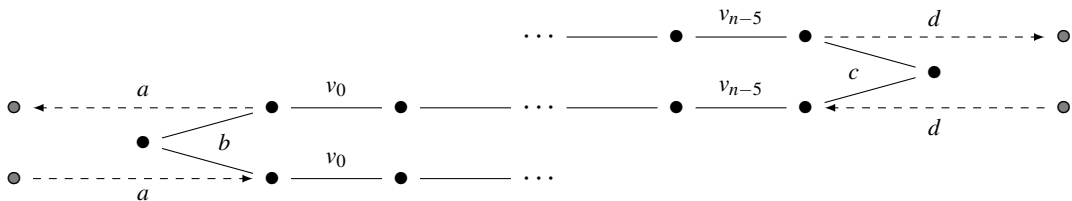
\[
 \beginpgfgraphicnamed{fig29}
 \begin{tikzpicture}[>=latex]
  \matrix (m) [matrix of math nodes, row sep=0em, column sep=2.9em, text height=1ex, text depth=0ex]
   {          &       &       &       &        &       &       &       &       \\   
              &       &       &       &\dotsb  &\bullet&\bullet&       &\ocirc \\   
              &       &       &       &        &       &       &\bullet&       \\   
      \ocirc  &       &\bullet&\bullet&\dotsb  &\bullet&\bullet&       &\ocirc \\   
              &\bullet&       &       &        &       &       &       &       \\   
      \ocirc  &       &\bullet&\bullet&\dotsb  &       &       &       &       \\   
};
   \path[-,font=\scriptsize]
   (m-2-6) edge node[auto] {$v_{n-5}$} (m-2-7)
   (m-2-7) edge node[auto,swap] {} (m-3-8)
   (m-3-8) edge node[auto,swap] {$c$} (m-4-7)
   (m-4-3) edge node[auto] {$v_0$} (m-4-4)
   (m-4-6) edge node[auto] {$v_{n-5}$} (m-4-7)
   (m-4-3) edge node[auto] {} (m-5-2)
   (m-5-2) edge node[auto] {$b$} (m-6-3)
   (m-6-3) edge node[auto] {$v_0$} (m-6-4)
   ;
   \path[->,dashed,font=\scriptsize]
   (m-2-7) edge node[auto] {$d$} (m-2-9)
   (m-4-9) edge node[auto] {$d$} (m-4-7)
   (m-4-3) edge node[auto,swap] {$a$} (m-4-1)
   (m-6-1) edge node[auto,swap] {$a$} (m-6-3)
   ;
   \path[-,font=\scriptsize]
   (m-2-5) edge node[auto,swap] {} (m-2-6)
   (m-4-4) edge node[auto] {} (m-4-5)
   (m-4-5) edge node[auto] {} (m-4-6)
   (m-6-4) edge node[auto,swap] {} (m-6-5)
   ;
  \end{tikzpicture}
\endpgfgraphicnamed
\]
where the top row ends in one of the situations
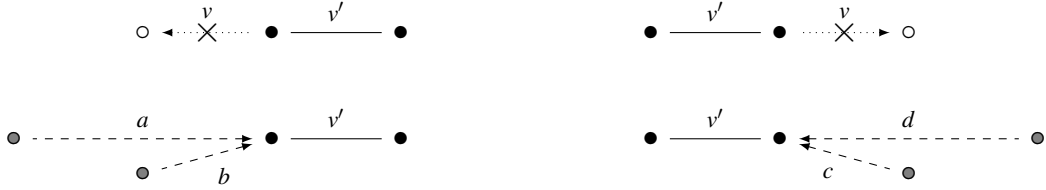
\[
\beginpgfgraphicnamed{fig28}
   \begin{tikzpicture}[>=latex]
  \matrix (m) [matrix of math nodes, row sep=0em, column sep=2.9em, text height=1ex, text depth=0ex]
   {          &\circ  &\bullet&\bullet&         &\bullet&\bullet& \circ &       \\   
              &  \    &  \    &       &         &       &  \    &  \    &       \\   
              &  \    &  \    &       &         &       &  \    &  \    &       \\   
     \ocirc   &       &\bullet&\bullet&         &\bullet&\bullet&       &\ocirc \\   
              &\ocirc &       &       &   \     &       &       &\ocirc &       \\   
   };
   \path[<-,dotted,font=\scriptsize]
   (m-1-2) edge[-x-=0.5] node[above=1pt] {$v$} (m-1-3)
   (m-1-8) edge[-x-=0.5] node[above=1pt] {$v$}  (m-1-7)
   ;
   \path[-,font=\scriptsize]
   (m-1-3) edge node[auto] {$v'$} (m-1-4)
   (m-1-6) edge node[auto] {$v'$} (m-1-7)
   ;
   \path[<-,dashed,font=\scriptsize]
   (m-4-3) edge node[auto,swap] {$a$} (m-4-1)
   (m-4-3) edge node[auto] {$b$} (m-5-2)
   (m-4-7) edge node[auto] {$d$} (m-4-9)
   (m-4-7) edge node[auto,swap] {$c$} (m-5-8)
   ;
   \path[-,font=\scriptsize]
   (m-4-3) edge node[auto] {$v'$} (m-4-4)
   (m-4-6) edge node[auto] {$v'$} (m-4-7)
   ;
  \end{tikzpicture}
\endpgfgraphicnamed
\]
and the bottom row ends in one of the situations
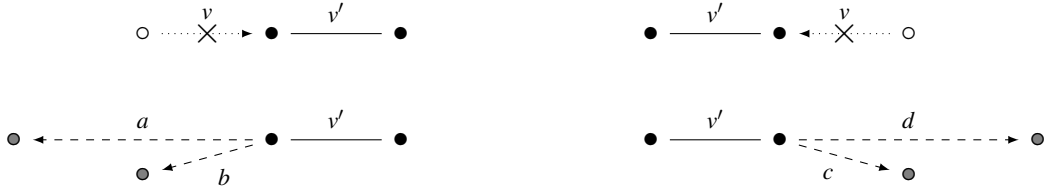
\[
\beginpgfgraphicnamed{fig27}
   \begin{tikzpicture}[>=latex]
  \matrix (m) [matrix of math nodes, row sep=0em, column sep=2.9em, text height=1ex, text depth=0ex]
   {          &\circ  &\bullet&\bullet&         &\bullet&\bullet& \circ &       \\   
              &  \    &  \    &       &         &       &  \    &  \    &       \\   
              &  \    &  \    &       &         &       &  \    &  \    &       \\   
     \ocirc   &       &\bullet&\bullet&         &\bullet&\bullet&       &\ocirc \\   
              &\ocirc &       &       &   \     &       &       &\ocirc &       \\   
   };
   \path[->,dotted,font=\scriptsize]
   (m-1-2) edge[-x-=0.5] node[above=1pt] {$v$} (m-1-3)
   (m-1-8) edge[-x-=0.5] node[above=1pt] {$v$}  (m-1-7)
   ;
   \path[-,font=\scriptsize]
   (m-1-3) edge node[auto] {$v'$} (m-1-4)
   (m-1-6) edge node[auto] {$v'$} (m-1-7)
   ;
   \path[->,dashed,font=\scriptsize]
   (m-4-3) edge node[auto,swap] {$a$} (m-4-1)
   (m-4-3) edge node[auto] {$b$} (m-5-2)
   (m-4-7) edge node[auto] {$d$} (m-4-9)
   (m-4-7) edge node[auto,swap] {$c$} (m-5-8)
   ;
   \path[-,font=\scriptsize]
   (m-4-3) edge node[auto] {$v'$} (m-4-4)
   (m-4-6) edge node[auto] {$v'$} (m-4-7)
   ;
  \end{tikzpicture}
\endpgfgraphicnamed
\]
where $v\in\{v_0,\dotsc,v_{n-5}\}$ and $v'\in\{b,v_0,\dotsc,v_{n-5},c\}$. Conversely, every coefficient quiver of this shape determines an indecomposable representation in the tube of rank $n-2$.


\subsection{Homogeneous tubes} 
\label{subsection: bases for homogeneous tubes}

The homogeneous tubes are all of rank $1$, which means that they contain a unique quasi-simple representation. Its dimension vector is $\delta$ and its coefficient quiver is 
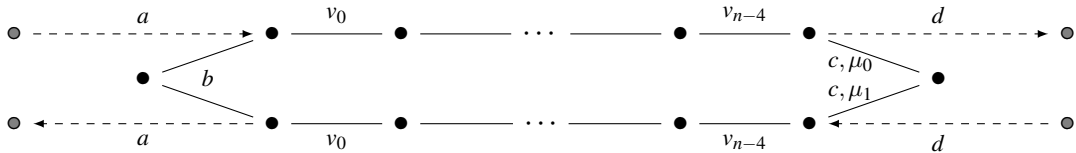
\[
 \beginpgfgraphicnamed{fig26}
 \begin{tikzpicture}[>=latex]
  \matrix (m) [matrix of math nodes, row sep=0.3em, column sep=2.9em, text height=1ex, text depth=0ex]
   {  \ocirc  &       &\bullet&\bullet&\dotsb  &\bullet&\bullet&       &\ocirc \\   
              &\bullet&       &       &        &       &       &\bullet&       \\   
      \ocirc  &       &\bullet&\bullet&\dotsb  &\bullet&\bullet&       &\ocirc \\   
};
   \path[-,font=\scriptsize]
   (m-1-3) edge node[auto] {$v_0$} (m-1-4)
   (m-1-4) edge node[auto] {} (m-1-5)
   (m-1-5) edge node[auto] {} (m-1-6)
   (m-1-6) edge node[auto] {$v_{n-4}$} (m-1-7)
   (m-2-2) edge node[below=1pt] {$b$} (m-1-3)
   (m-2-2) edge node[auto] {} (m-3-3)
   (m-2-8) edge node[below left=-4pt] {$c,\mu_0$} (m-1-7)
   (m-2-8) edge node[above left=-4pt] {$c,\mu_1$} (m-3-7)
   (m-3-3) edge node[auto,swap] {$v_0$} (m-3-4)
   (m-3-4) edge node[auto] {} (m-3-5)
   (m-3-5) edge node[auto] {} (m-3-6)
   (m-3-6) edge node[auto,swap] {$v_{n-4}$} (m-3-7)
   ;
   \path[->,dashed,font=\scriptsize]
   (m-1-1) edge node[auto] {$a$} (m-1-3)
   (m-3-3) edge node[auto] {$a$} (m-3-1)
   (m-1-7) edge node[auto] {$d$} (m-1-9)
   (m-3-9) edge node[auto] {$d$} (m-3-7)
   ;
  \end{tikzpicture}
\endpgfgraphicnamed
\]
for some non-zero complex weights $\mu_0$ and $\mu_1$ such that $\mu_0\neq\mu_1$ in case that the vertices $q_b$ and $q_c$ of $Q$ are both sinks or are both sources and $\mu_0\neq-\mu_1$ in case that one of $q_b$ and $q_c$ is a source and the other vertex is a sink.

Note that two pairs of weights $(\mu_0,\mu_1)$ and $(\mu_0',\mu_1')$ yield isomorphic representations if and only if $[\mu_0:\mu_1]=[\mu_0':\mu_1']$ is the same point in $\P^1(\C)$. In fact, the three ``forbidden'' values $[0:1]$, $[1:0]$ and $[1:\pm 1]$ determine representations in the three exceptional tubes; this yields a parametrization of all tubes by $\P^1(\C)$.

A coefficient quiver for an arbitrary representation in a homogeneous tube, which has dimension vector $r\delta$ for some $r\geq1$, is obtained by glueing together $r$ copies of the coefficient quiver for the quasi-simple in the same tube with $r-1$ arrows with label $a$, which connect to the obvious positions. For more details, cf.\ section 1.7 in the sequel \cite{LW15} to this paper.


\pagebreak

\begin{small}
\bibliographystyle{plain}

\def\cprime{$'$}

\end{small}

\end{document}